\documentclass[12pt]{amsart}
\usepackage{amscd,amssymb}
\usepackage{amsthm,amsmath,amssymb}
\usepackage[matrix,arrow, curve]{xy}
\usepackage{enumerate}

\newtheorem{theorem}{Theorem}[section]

\newtheorem{lemma}[theorem]{Lemma}
\newtheorem{corollary}[theorem]{Corollary}

\newtheorem*{question*}{Question}
\newtheorem*{conjecture*}{Conjecture}

\theoremstyle{definition}
\newtheorem{example}[theorem]{Example}
\newtheorem{definition}[theorem]{Definition}

\theoremstyle{remark}
\newtheorem{remark}[theorem]{Remark}

\makeatletter\@addtoreset{equation}{section} \makeatother

\author{Ivan Cheltsov and Jesus Martinez-Garcia}

\title{Unstable polarized del Pezzo surfaces}
\date{18/01/2018}

\keywords{$K$-stability, cscK metric, del Pezzo surface, slope stability, flop.}

\thanks{This article was written while the authors were visiting the Max Planck Institute for Mathematics.
We would like to thank the institute for the excellent working conditions. Jesus Martinez-Garcia is supported by the Simons Foundation under the Simons Collaboration on Special Holonomy in Geometry, Analysis and Physics (grant \#488631, Johannes Nordstr\"om).}

\begin{document}

\begin{abstract}

We provide new examples of $K$-unstable polarized smooth del Pezzo surfaces
using a flopped version first used by Cheltsov and Rubinstein of the test configurations introduced by Ross and Thomas.
As an application, we provide new obstructions for the existence of constant scalar curvature K\"ahler metrics
on polarized smooth del Pezzo surfaces.\end{abstract}

\sloppy

\maketitle

All varieties are assumed to be algebraic, projective and defined over $\mathbb{C}$.

\section{Introduction}
\label{section:intro}

$K$-stability is an algebraic notion of polarized varieties which has been of great importance
in the study of the existence of canonical metrics on complex varieties.
This is mainly because of the following

\begin{conjecture*}[Yau--Tian--Donaldson]
\label{conjecture:YTD}
Let $X$ be a smooth variety, and let $L$ be an ample line bundle on it.
Then $X$ admits a constant scalar curvature K\"ahler (cscK) metric in $c_1(L)$ if and only if the pair $(X,L)$ is $K$-polystable.
\end{conjecture*}

It is known in different degrees of generality that $K$-polystability is a necessary condition for the existence of a cscK metric, with the most general result due to Berman, Darvas and Lu \cite{BDL16} following work of Darvas and Rubinstein \cite{DR17}.
For smooth Fano varieties polarized by anticanonical line bundles, Conjecture \ref{conjecture:YTD} was recently proved by Chen, Donaldson and Sun in \cite{CDS}.

In spite of the above (conjectural) characterizations, deciding whether a given polarized variety is $K$-stable is a problem of considerable difficulty.
In this paper, we study this problem for del Pezzo surfaces polarized by ample $\mathbb{Q}$-divisors.
Using $\mathbb{Q}$-divisors does not affect the original problem, since $K$-stability is preserved when we scale the polarization positively.

Let $S$ be a smooth del Pezzo surface, and let $L$ be an ample $\mathbb{Q}$-divisor on it.
Recall that $S$ is toric if and only if $K_S^2\geqslant 6$.
In this case, the problem we plan to consider is completely solved.
In the non-toric case, few results in this direction are known.
For instance, if $S$ is not toric, then it admits a K\"ahler--Einstein metric by Tian's theorem \cite{Tian1990}, so that $(S,-K_S)$ is $K$-stable.
Moreover, a result of LeBrun and Simanca \cite{LeBrun-Simanca} implies that the same holds for every divisor $L$ that is \emph{close enough} to $-K_S$.
On the other hand, many $K$-unstable pairs $(S,L)$ have been constructed by Ross and Thomas in \cite{RT1,RT2}.

In the prequel to this article \cite{Cheltsov-Jesus}, we gave a simple condition on $L$ that guarantees that $(S,L)$ is $K$-stable.
The goal of this article is to obtain new simple conditions on $L$ that guarantee that $(S,L)$ is $K$-unstable. In addition our technique recovers all previous obstructions to $K$-stability on polarized del Pezzo surfaces $(S,L)$ found in the literature. 

To present our results, it is convenient to split ample $\mathbb{Q}$-divisors on $S$ into three major types:
$\mathbb{P}^2$-type, $\mathbb{F}_1$-type, and $\mathbb{P}^1\times\mathbb{P}^1$-type.
To be precise, up to a positive scaling of $L$, one always has
$$
L\sim_{\mathbb{Q}} -K_S+bB +\sum_{i=1}^{r} a_iF_i,
$$
where $F_1,\ldots F_{r}$ are disjoint $(-1)$-curves, $B$ is a smooth rational curve such that $B^2=0$, and $b,a_1,\ldots,a_{r}$ are some non-negative rational numbers such that $1>a_{r}\geqslant\cdots\geqslant a_1\geqslant 0$.
Moreover, if $b\ne 0$, then the curve $B$ is disjoint from the $(-1)$-curves $F_1,\ldots,F_{r}$.
Since the $(-1)$-curves are disjoint, their contraction gives a birational morphism $\phi\colon S\to\widehat{S}$.
We say that $L$ is of $\mathbb{P}^2$-type, $\mathbb{F}_1$-type or $\mathbb{P}^1\times\mathbb{P}^1$-type
in the case when $\widehat{S}=\mathbb{P}^2$, $\widehat{S}=\mathbb{F}_1$ or $\widehat{S}=\mathbb{P}^1\times\mathbb{P}^1$, respectively.
In particular, if $L$ is of $\mathbb{P}^2$-type, then $r=9-K_{S}^2$.
Similarly, if $L$ is $\mathbb{F}_1$-type or $\mathbb{P}^1\times\mathbb{P}^1$-type, then $r=8-K_S^2$.
It is easy to see that every ample $\mathbb{Q}$-divisor on $S$ is of one of these three types.
To make the types mutually exclusive, we also require  $b>0$ and $a_1>0$ in the $\mathbb{F}_1$-case,
and we ask that $b>0$ or $a_1>0$ in the $\mathbb{P}^1\times\mathbb{P}^1$-case.

We believe that our newly introduced language may shed a new light on this problem.
The indication of this can be seen from the $K$-polystability criterion in the case $K_S^2=6$.
Translating it into our language we recover a result originally due to Donaldson and Wang--Zhou:

\begin{theorem}[{\cite{SD-toric-surfaces,toric}}]
\label{theorem:toric-popular}
Suppose that $K_S^2=6$.
Then $(S,L)$ is $K$-polystable and it accepts a cscK metric in $c_1(L)$ if and only if either $L$ is of $\mathbb{P}^2$-type and $a_1=a_2=a_3$,
or $L$ is of $\mathbb{P}^1\times\mathbb{P}^1$-type and $a_1=a_2$.
\end{theorem}
The above result was originally proven using properties of toric geometry. In Example~\ref{example:d-6}, we used our technique of \emph{flop slope test configurations}, which have a non-toric nature, to prove K-instability.

In particular, if $K_S^2=6$ and $L$ is of $\mathbb{F}_1$-type, then $(S,L)$ is always $K$-unstable.
By the aforementioned result of LeBrun--Simanca \cite{LeBrun-Simanca}, this is no longer true in the non-toric case.
Nevertheless, in this case we prove a somewhat similar result:

\begin{theorem}
\label{theorem:nice-inequality}
Suppose that $K_S^2\leqslant 5$, the divisor $L$ is of $\mathbb{F}_1$-type or $\mathbb{P}^1\times\mathbb{P}^1$-type, and
$$
a_1^2+6-K_S^2<\sum_{i=2}^{r}a_i^2.
$$
Then $(S,L)$ is $K$-unstable for $b\gg 0$ and $S$ does not accept a cscK metric in $c_1(L)$.
\end{theorem}

In their seminal works \cite{RT1,RT2}, Ross and Thomas introduced the notion of \emph{slope stability} as an obstruction to $K$-stability.
In particular, their \cite[Example~5.30]{RT1} implies

\begin{theorem}
\label{theorem:Ross-Thomas}
Suppose that $K_S^2\leqslant 6$. If $L$ is of $\mathbb{P}^2$-type and $a_2\gg a_1$, then $(S,L)$ is $K$-unstable and $S$ does not accept a cscK metric in $c_1(L)$.
Similarly, if $L$ is of $\mathbb{F}_1$-type and $a_1\gg 0$, then $(S,L)$ is $K$-unstable and $S$ does not accept a metric in $c_1(L)$.
\end{theorem}

In \cite{CheRu}, Cheltsov and Rubinstein considered a modification of Ross and Thomas construction using flops.
In this article, we adapt their method to obtain several criteria for $K$-instability of the pair $(S,L)$.
In particular, we prove that $(S,L)$ is $K$-unstable for every ample $\mathbb{Q}$-divisor $L$ in the case when $K_S^2=7$,
and we give a short proof of the \emph{only if} part of Theorem~\ref{theorem:toric-popular}.
Moreover, we improve Ross--Thomas result by obtaining the following

\begin{theorem}
\label{theorem:degree-6}
Suppose that $K_S^2\leqslant 5$.
If $L$ is of $\mathbb{P}^2$-type and $a_4\gg a_3>a_1$, then $(S,L)$ is $K$-unstable and $S$ does not accept a metric in $c_1(L)$.
Similarly, if $L$ is of $\mathbb{P}^1\times\mathbb{P}^1$-type and $a_3\gg a_2>a_1$, then $(S,L)$ is $K$-unstable and $S$ does not accept a metric in $c_1(L)$.
Finally, if $L$ is of $\mathbb{F}_1$-type and $a_3\gg a_2$, then $(S,L)$ is $K$-unstable and $S$ does not accept a metric in $c_1(L)$.
\end{theorem}

We use the same approach to prove the following quantitative result.

\begin{theorem}
\label{theorem:main}
Suppose that $K_S^2\leqslant 5$.
If $L$ is of $\mathbb{P}^2$-type and
\begin{equation}
\label{equation:P2-a2-a1}
a_2-a_1\geqslant\left\{%
\aligned
&0.8717\ \text{if}\ K_S^2=1,\\
&0.8469\ \text{if}\ K_S^2=2,\\
&0.8099\ \text{if}\ K_S^2=3,\\
&0.7488\ \text{if}\ K_S^2=4,\\
&0.6248\ \text{if}\ K_S^2=5,\\
\endaligned\right.\\
\end{equation}
then $(S,L)$ is $K$-unstable and $S$ does not accept a metric in $c_1(L)$. Similarly, if $L$ is of $\mathbb{P}^2$-type and
\begin{equation}
\label{equation:P2-a3-a1}
a_3-a_1\geqslant\left\{%
\aligned
&0.9347\ \text{if}\ K_S^2=1,\\
&0.9206\ \text{if}\ K_S^2=2,\\
&0.8985\ \text{if}\ K_S^2=3,\\
&0.8595\ \text{if}\ K_S^2=4,\\
&0.6798\ \text{if}\ K_S^2=5,\\
\endaligned\right.\\
\end{equation}
then $(S,L)$ is $K$-unstable and $S$ does not accept a metric in $c_1(L)$. Likewise, if $L$ is of $\mathbb{P}^1\times\mathbb{P}^1$-type and
\begin{equation}
\label{equation:P1xP1-a2-a1}
a_2-a_1\geqslant\left\{%
\aligned
&0.9305\ \text{if}\ K_S^2=1,\\
&0.9150\ \text{if}\ K_S^2=2,\\
&0.8911\ \text{if}\ K_S^2=3,\\
&0.8480\ \text{if}\ K_S^2=4,\\
&0.7452\ \text{if}\ K_S^2=5,\\
\endaligned\right.\\
\end{equation}
then $(S,L)$ is $K$-unstable and $S$ does not accept a metric in $c_1(L)$. Finally, $(S,L)$ is $K$-unstable and $S$ does not accept a metric in $c_1(L)$ if $L$ is of $\mathbb{F}_1$-type and
\begin{equation}
\label{equation:F1-a2-a1}
a_2-a_1\geqslant\left\{%
\aligned
&0.9347\ \text{if}\ K_S^2=1,\\
&0.9206\ \text{if}\ K_S^2=2,\\
&0.8985\ \text{if}\ K_S^2=3,\\
&0.8595\ \text{if}\ K_S^2=4,\\
&0.7701\ \text{if}\ K_S^2=5.\\
\endaligned\right.\\
\end{equation}
\end{theorem}

Observe that only the differences $a_3-a_1$ and $a_2-a_1$ are considered in Theorem~\ref{theorem:main}.
There is a good reason for this.
For example, if $L$ is of $\mathbb{P}^2$-type, $K_S^2\leqslant 5$ and $a_1=a_2=a_3$,
then $S$ admits a cscK metric in $c_1(L)$ for $1>a_4\gg a_3$, so that $(S,L)$ is $K$-stable.
Similarly, if $L$ is of $\mathbb{P}^1\times\mathbb{P}^1$-type, $K_S^2\leqslant 5$ and $a_1=a_2$,
then $(S,L)$ is $K$-stable for $1>a_3\gg a_2$.
This follows from the \emph{lifting} results of Arezzo, Pacard, Rollin and Singer (see \cite{AP1,AP2,APS,RS}). A related non-exhaustive list of lifting results for cscK metrics in non-del Pezzo situations includes LeBrun--Singer \cite{LeBrun-Singer}, Kim--Pontecorvo \cite{Kim-Pontecorvo}, and Kim--LeBrun--Pontecorvo \cite{Kim-LeBrun-Pontecorvo}.

At this point, it is probably a good idea to have some \emph{critical} commentary of our results, reflecting on our contribution to the problem and the  relation of our work to that of the aforementioned authors. Ideally, a complete solution to the problem of the existence of cscK on a del Pezzo surface $S$ would consist of having some parametrization of the ample cone $\mathrm{Amp}(S)$, so that one could tell for each ample line bundle $L\in \mathrm{Amp}(S)$ if a cscK metric exists in $c_1(L)$ or not. This viewpoint is informed by the case of toric del Pezzo surfaces (i.e. when $K_S^2\geqslant 6$), where the problem is completely solved (see theorems \ref{theorem:toric-popular} and \ref{theorem:Ross-Thomas}). Aside of the toric case, the aforementioned existence results \cite{AP1,AP2,APS,RS} are \emph{qualititative}, and follow the following lifting strategy: it is assumed that a cscK metric $\omega_{S,L}$ is known to exist in $c_1(L)$, where $L$ is and ample line bundle of a del Pezzo surface $S$. Another del Pezzo surface $S'$ is obtained as the blow-up $\pi\colon S'\rightarrow S$ of $S$ and the metric $\omega_{S,L}$ is lifted to a cscK metric $\omega_{S',L'}\in c_1(L')$ where $L'$ is an ample line bundle of $S'$. However, it does not seem to be possible to say much about where $L'$ lies with respect to $\pi^*(L)$ even assuming that a parametrization of $\mathrm{Amp}(S')$ is available. In \cite{Cheltsov-Jesus}, we gave existence results where, if desired, coordinates for $L$ in $\mathrm{Amp}(S)$ would be easy to obtain. In the current article, we have taken a complementary approach, providing results on \emph{non-existence} of cscK metrics in $c_1(L)$ for $L\in \mathrm{Amp}(S)$.

The first task we undertake is to give an effective parametrization of the ample cone $\mathrm{Amp}(S)$. Then, we construct test configurations that  obstruct the K-stability of $(S,L)$ for $L$  in a precise region of $\mathrm{Amp}(S)$, described using the parametrization. The test configurations constructed generalize slope test configurations introduced by Ross--Thomas and Ross--Panov \cite{RT1, RT2, RP} using deformations to the normal cone by performing flops on the slope test configurations. Ross--Thomas's test configurations are limiting, in the sense that they do not detect K-instability for many ample line bundles, even in the toric case, for instance in the setting in Theorem \ref{theorem:toric-popular}. The approach we follow was first pioneered by the first author and Rubinstein in \cite{CheRu} in a very special case, where $L\sim -K_S-(1-\beta)C$ for $C$ a curve and $\beta$ a small parameter. Their motivation was studying asymptotical K-stability of asymptotically log Fano pairs $(S,C)$. We generalize their work to any polarization $L$. Once it has been established that a pair $(S,L)$ is not K-semistable, one can use the fact that K-stability is an obstruction to the existence of cscK metrics \cite{DR17,BDL16} to determine that there is no cscK metric in $c_1(L)$.

While a complete characterization of the ample cone according to K-stability is the main motivation for this article and \cite{Cheltsov-Jesus}, it is probably an unattainable goal. This is due to two reasons: firstly, the ample cone becomes very complicated when $K_S^2$ is small and detecting obstructions to K-stability using our method eventually boils down to describing the negative locus of many polynomials of degree $4$ in $9-K_S^2$ variables. As a result, a simple description of such loci is not expected. The second complication has to do with the intrinsic limitations of the method. As it happened with the slope test configurations of Ross--Thomas, it is expected that some polarizations are not obstructed by our \emph{flop slope test configurations}. While we have not found any method in the literature which obstructs a polarization that we fail to obstruct, this is no reason to believe better methods will be develop in the future to obstruct polarizations that we are unable to tackle. Nevertheless, our article recovers all currently known obstructions to the existence of cscK metrics on polarized del Pezzo surfaces, including Theorem \ref{theorem:toric-popular} and Ross--Thomas's obstructions, and provides many new ones (see theorems \ref{theorem:nice-inequality}, \ref{theorem:degree-6} and \ref{theorem:main}).

Let us outline the structure of this article.
In~sections~\ref{section:stability} and \ref{section:DF-invariant}, we recall the $K$-stability setting,
including the flopped version of slope stability.
In~Section~\ref{section:del-Pezzo}, we study ample divisors on del Pezzo surfaces, giving coordinates to elements of the ample cone,
and prove Theorem~\ref{theorem:degree-6}.
In~Section~\ref{section:technical}, we prove technical results on ample divisors on del Pezzo surfaces.
In~Section~\ref{section:unstable-del-pezzo}, we prove Theorems~\ref{theorem:nice-inequality}~and~\ref{theorem:main}.
The proof of Theorem~\ref{theorem:main} relies on symbolic computations presented in Appendix~\ref{section:Maple}.

\section{What is $K$-polystability?}
\label{section:stability}

Let $X$ be an $n$-dimensional smooth projective variety, and let $L$ be an ample line bundle on it.
In this section we will remind the reader of the notion of $K$-stability of the pair $(X,L)$,
which was originally defined by Tian in \cite{Tian1997}.
A more refined, algebro-geometric definition was introduced by Donaldson in \cite{Donaldson}, which eventually led to Conjecture~\ref{conjecture:YTD}.

First we need to define the notion of \emph{test configuration}.
We will always assume that the total space of the test configuration is normal (see \cite{LiXu} for an explanation).

\begin{definition}
\label{definition:TC}
A \emph{test configuration} of $(X,L)$ consists of
\begin{itemize}
\item a normal variety $\mathcal{U}$ with a $\mathbb{G}_m$-action,

\item a flat $\mathbb{G}_m$-equivariant map ${p}_{\mathcal{U}}\colon \mathcal{U}\to\mathbb{A}^1$, where $\mathbb{G}_m$ acts on $\mathbb{A}^1$ naturally,

\item a $\mathbb{G}_m$-equivariant ${p}_{\mathcal{U}}$-ample line bundle $\mathcal{L}_{\mathcal{U}}\to\mathcal{U}$
such that there exists a positive integer $r$, called \emph{exponent},
and a $\mathbb{G}_m$-equivariant isomorphism
\begin{equation}
\label{equation:isomorphism}
\Big(\mathcal{U}\setminus{p}_{\mathcal{U}}^{-1}(0),\mathcal{L}_{\mathcal{U}}\Big\vert_{\mathcal{U}\setminus{p}_{\mathcal{U}}^{-1}(0)}\Big)\cong\Big(X\times\big(\mathbb{A}^1\setminus\{0\}\big), p_1^*(L^{\otimes r})\Big),
\end{equation}
with the natural action of the group $\mathbb{G}_m$ on $\mathbb{A}^1\setminus\{0\}$ and the trivial action on $X$,
where $p_1\colon X\times(\mathbb{A}^1\setminus\{0\})\to X$ is the projection to the first factor.
\end{itemize}
We also say that $(\mathcal{U},\mathcal{L}_\mathcal{U},{p}_{\mathcal{U}})$ a \emph{product test configuration}
if $\mathcal{U}\cong X\times\mathbb{A}^1$ and $\mathcal{L}_{\mathcal{U}}=p_1^*(L^{\otimes r})$.
A product test configuration is \emph{trivial} if $\mathbb{G}_m$ acts trivially on the left factor of $X\times\mathbb{A}^1$.
\end{definition}

Given an arbitrary test configuration $(\mathcal{U},\mathcal{L}_{\mathcal{U}},{p}_{\mathcal{U}})$ of the pair $(X,L)$ with exponent $r$,
one can naturally compactify it by gluing $(\mathcal{U},\mathcal{L}_{\mathcal{U}})$ with $(X\times(\mathbb{P}^1\setminus \{0\}),p_1^*(L^{\otimes r}))$ as follows.
In the $\mathbb{G}_m$-equivariant isomorphism \eqref{equation:isomorphism},
each $t\in\mathbb{G}_m$ acts on its right hand side by
$$
t\circ\big(\{p\}\times\{a\},s\big)=\big(\{p\}\times\{ta\},s\big)
$$
for any $p\in X$, $a\in \mathbb{A}^1$ and $s\in(\mathcal{L}_{\mathcal{U}})_p$.
So, we can define the gluing map using the diagram
$$
\xymatrix{
\big(\mathcal{U}, \mathcal{L}_{\mathcal{U}}\big)&&\big(X\times\mathbb{P}^1\setminus\{0\},p_1^*(L^{\otimes r})\big)\\
\Big(\mathcal{U}\setminus{p}_{\mathcal{U}}^{-1}(0),\mathcal{L}_{\mathcal{U}}\Big\vert_{\mathcal{U}\setminus{p}_{\mathcal{U}}^{-1}(0)}\Big)\ar@{->}[rr]^{\phi}\ar@{^{(}->}[u]&&\Big(X\times \big(\mathbb{A}^1\setminus\{0\}\big), p_1^*(L^{\otimes r})\Big),\ar@{^{(}->}[u]}
$$
where the map $\phi$ is given by
$$
\phi\colon \big(p,a,s\big)\longmapsto\big(\{a^{-1}\circ p\}\times\{a\}, a^{-1}\circ s\big),
$$
where $\mathbb{G}_m$ only acts by multiplication on the factor $\mathbb{P}^1\setminus\{0\}$ of $(X\times \mathbb{P}^1\setminus\{0\},p_1^*(L^{\otimes r}))$.
Using this gluing map, we obtain a triple $(\mathcal{X},\mathcal{L},{p})$ consisting of
\begin{itemize}
\item a normal projective variety $\mathcal{X}$ with a $\mathbb{G}_m$-action,

\item a flat $\mathbb{G}_m$-equivariant map ${p}\colon\mathcal{X}\to\mathbb{P}^1$ such that ${p}_{\mathcal{U}}^{-1}(t)\cong X$ for every $t\in\mathbb{P}^1\setminus\{0\}$,

\item a $\mathbb{G}_m$-equivariant ${p}$-ample line bundle $\mathcal{L}\to\mathcal{X}$,
such that
$$
\mathcal{L}\Big\vert_{{p}^{-1}(t)}\cong L^{\otimes r}
$$
for every $t\in\mathbb{P}^1\setminus\{0\}$, where we identify ${p}_{\mathcal{U}}^{-1}(t)$ with $X$.
\end{itemize}
For further details and examples, see \cite[Section~8.1]{LiXu}.

\begin{remark}
\label{remark:Li-Xu}
In \cite{LiXu}, the triple $(\mathcal{X},\mathcal{L},{p})$ is called \emph{$\infty$-trivial compactification} of the test configuration $(\mathcal{U},\mathcal{L}_{\mathcal{U}},{p}_{\mathcal{U}})$.
Since we will always work with compactified test configurations in this article,
we will simply call the triple $(\mathcal{X},\mathcal{L},{p})$ a \emph{test configuration} of the pair $(X,L)$.
Moreover, if the original test configuration $(\mathcal{U},\mathcal{L}_{\mathcal{U}},{p}_{\mathcal{U}})$ is trivial,
then we say that the test configuration $(\mathcal{X},\mathcal{L},{p})$ is \emph{trivial}.
In this case, $\mathcal{X}\cong X\times\mathbb{P}^1$ and $\mathcal{L}=p_1^*(L^{\otimes r})$.
Similarly, if $(\mathcal{U},\mathcal{L}_{\mathcal{U}},{p}_{\mathcal{U}})$ is a product test configuration,
then we say that $(\mathcal{X},\mathcal{L},{p})$ is a \emph{product test configuration}.
In this case, we have ${p}^{-1}(0)\cong X$, so that ${p}$ is an isotrivial fibration.
\end{remark}

Using the compactified test configuration $(\mathcal{X},\mathcal{L},{p})$, Li and Xu gave an intersection formula for the Donaldson--Futaki invariant of the original test configuration.
This formula first appeared in work of Odaka \cite{Odaka} and Wang \cite{Wang}, c.f. \cite[Proposition~6]{LiXu} for a new proof.
We will use this formula as a definition of the Donaldson--Futaki invariant.
\emph{The slope} of the pair $(X,L)$ is
$$
\nu(L)=\frac{-K_X\cdot L^{n-1}}{{L}^{n}}.
$$
The (normalized) \emph{Donaldson-Futaki invariant} of the (compactified) test configuration $(\mathcal{X},\mathcal{L},{p})$ with exponent $r$
is the number
\begin{equation}
\label{equation:DF-definition}
\mathrm{DF}\big(\mathcal{X},\mathcal{L},{p}\big)=\frac{1}{r^{n}}\Bigg(\frac{n}{n+1}\frac{1}{r}\nu(L)\mathcal{L}^{n+1}+\mathcal{L}^n\cdot\Big(K_{\mathcal{X}}-{p}^*\big(K_{\mathbb{P}^1}\big)\Big)\Bigg),
\end{equation}
where $n$ is the dimension of the variety $X$.
Observe that the number $\mathrm{DF}(\mathcal{X},\mathcal{L},{p})$ does not change
if we replace $\mathcal{L}$ by $\mathcal{L}+{p}^*(D)$ for any line bundle $D$ on $\mathbb{P}^1$.
Moreover, if the test configuration $(\mathcal{X},\mathcal{L},{p})$ is trivial,
then the formula \eqref{equation:DF-definition} gives $\mathrm{DF}(\mathcal{X},\mathcal{L},{p})=0$.

\begin{definition}
\label{definition:K-stability}
The pair $(X,L)$ is said to be \emph{$K$-polystable} if
$\mathrm{DF}(\mathcal{X},\mathcal{L},{p})\geqslant 0$
for every non-trivial test configuration $(\mathcal{X},\mathcal{L},{p})$,
and $\mathrm{DF}(\mathcal{X},\mathcal{L},{p})=0$ only if $(\mathcal{X},\mathcal{L},{p})$ is a product test configuration .
Similarly, the pair $(X,L)$ is said to be \emph{$K$-stable} if $\mathrm{DF}\big(\mathcal{X},\mathcal{L},{p}\big)>0$
for every non-trivial test configuration $(\mathcal{X},\mathcal{L},{p})$.
Finally, if $\mathrm{DF}(\mathcal{X},\mathcal{L},{p})\geqslant 0$ for every test configuration $(\mathcal{X},\mathcal{L},{p})$,
then $(X,L)$ is said to be \emph{$K$-semistable}.
\end{definition}

If the pair $(X,L)$ is not $K$-semistable, then $\mathrm{DF}(\mathcal{X},\mathcal{L},{p})<0$
for some test configuration $(\mathcal{X},\mathcal{L},{p})$ of the pair $(X,L)$.
In this case, we say that $(X,L)$ is \emph{$K$-unstable}, and $(\mathcal{X},\mathcal{L},{p})$ is a \emph{destabilizing} test configuration.

\begin{remark}
\label{remark:stability-polystability}
The $K$-polystability of the pair $(X,L)$ implies its $K$-semistability.
Similarly, the $K$-stability of the pair $(X,L)$ implies its $K$-polystability.
Moreover, if the group $\mathrm{Aut}(X,L)$ is finite, then all product test configurations of the pair $(X,L)$ are trivial,
so that $(X,L)$ is $K$-stable if and only if it is $K$-polystable.
\end{remark}

The pair $(X,L)$ is $K$-polystable (respectively, $K$-stable or $K$-semistable) if and only if
the pair $(X,L^{\otimes k})$ is $K$-polystable (respectively, $K$-stable or $K$-semistable) for some positive integer $k$.
Thus, we can adapt both Definitions~\ref{definition:TC} and \ref{definition:K-stability} to the case when $L$ is an ample $\mathbb{Q}$-divisor on the variety $X$.
This gives us notions of $K$-polystability, $K$-stability, $K$-semistability and $K$-instability for varieties polarized by ample $\mathbb{Q}$-divisors.
Similarly, we can assume that $\mathcal{L}$ in the test configuration $(\mathcal{X},\mathcal{L},{p})$ is a $p$-ample $\mathbb{Q}$-divisor on $\mathcal{X}$.
Because of this, we can assume that $r=1$ in the formula \eqref{equation:DF-definition} for the Donaldson-Futaki invariant.

\section{Slope stability and Atiyah flops}
\label{section:DF-invariant}

Let $S$ be a smooth projective surface, and let $L$ be an ample $\mathbb{Q}$-divisor on the surface $S$.
In this section, we will compute the Donaldson--Futaki invariant of some explicit test configurations of the pair $(S,L)$.
One of them is a very special case of a much more general construction studied by Ross and Thomas in \cite{RT1,RT2}.
Namely, fix a smooth irreducible curve $Z$ in the surface $S$.
By a slight abuse of notation, let us identify the curve $Z$ with the curve $Z\times \{0\}$ in the product $S\times\mathbb{P}^1$.
Let $\pi_Z\colon \mathcal{X}\to S\times\mathbb{P}^1$ be the blow-up of the curve $Z$.
Denote the exceptional divisor of $\pi_Z$ by $E_Z$, let $p=p_{\mathbb{P}^1}\circ\pi_Z$, and let
$$
\mathcal{L}_{\lambda}=\big(p_S\circ\pi_Z\big)^*(L)-\lambda E_Z,
$$
where $\lambda$ is a positive rational number.
Denote by $\sigma(S,L,Z)$ the \emph{Seshadri constant} of the pair $(S,L)$ along $Z$.
Recall that $\sigma(S,L,Z)$ is usually very easy to compute, since
$$
\sigma(S,L,Z)=\mathrm{sup}\Big\{\mu\in\mathbb{Q}_{>0}\ \Big\vert\ \text{the divisor}\ L-\mu Z\ \text{is nef}\Big\}.
$$
By \cite[Proposition~4.1]{RT2}, if $\lambda<\sigma(S,L,Z)$, then $\mathcal{L}_{\lambda}$ is $p$-ample (see also \cite[Lemma~2.2]{CheRu}),
so that $(\mathcal{X},\mathcal{L}_{\lambda},p)$ is a (compactified) test configuration of the pair $(S,L)$.
This test configuration is often called a \emph{slope test configuration} centred at $Z$.
If $\mathrm{DF}(\mathcal{X},\mathcal{L}_{\lambda},p)<0$ for some $\lambda<\sigma(S,L,Z)$,
then $(S,L)$ is said to be \emph{slope unstable}. This implies, in particular, that $(S,L)$ is $K$-unstable.

One good thing about the test configuration $(\mathcal{X},\mathcal{L}_{\lambda},p)$ is that its Donaldson--Futaki invariant is very easy to compute.
Namely, let $g(Z)$ be the genus of the curve $Z$, and let
\begin{equation}
\label{equation:DF-polynomial}
\mathfrak{DF}(\lambda)=\frac{2}{3}\nu(L)\Big(\lambda^3Z^2-3\lambda^2L\cdot Z\Big)+\lambda^2\big(2-2g(Z)\big)+2\lambda L\cdot Z.
\end{equation}
Recall from Section~\ref{section:stability} that $\nu(L)=\frac{-K_S\cdot L}{L^2}$ is the slope of the pair $(S,L)$.

\begin{lemma}[{\cite{RT1,RT2}}]
\label{lemma:DF-slope}
If $\lambda<\sigma(S,L,Z)$, then $\mathrm{DF}(\mathcal{X},\mathcal{L}_{\lambda},p)=\mathfrak{DF}(\lambda)$.
\end{lemma}

\begin{proof}
Since $-E_Z^3$ is the degree of the normal bundle of $Z$ in $S\times \mathbb{P}^1$, we get
$$
\mathcal{L}_{\lambda}^3=
3\lambda^2\Big((p_{\mathbb{P}^1}\circ\pi_{Z})^*\big(L\big)\Big)\cdot E_Z^2-\lambda^3E_Z^3=-3\lambda^2L\cdot Z+\lambda^3Z^2.
$$
Moreover, we have
\begin{multline*}
\mathcal{L}_{\lambda}^2\cdot\Big(K_{\mathcal{X}}-{p}^*\big(K_{\mathbb{P}^1}\big)\Big)=\Big((p_{\mathbb{P}^1}\circ\pi_{Z})^*\big(K_S\big)+E_Z\Big)\cdot \Big((p_{X}\circ\pi_{Z})^*\big(L\big)-\lambda E_Z\Big)^2=\\
=-2\lambda\left(p_{X}\circ\pi_{Z}\right)^*\left(L\right)\cdot E_Z^2+\lambda^2\left(p_{\mathbb{P}^1}\circ \pi_{Z}\right)^*(K_S)\cdot E_Z^2+\lambda^2E_Z^3=\\
=2\lambda L\cdot Z-\lambda^2K_S\cdot Z-\lambda^2 Z^2.
\end{multline*}
Now the result follows from substituting in \eqref{equation:DF-definition} the above identities.
\end{proof}

If $\lambda<\sigma(S,L,Z)$ and $Z^2\geqslant 0$, then $\mathrm{DF}(\mathcal{X},\mathcal{L}_{\lambda},p)\geqslant 0$ by \cite[Theorem~1.3]{RP}.
Thus, if we want $(\mathcal{X},\mathcal{L}_{\lambda},p)$ to be a destabilizing test configuration, then $Z^2$ must be negative.
In particular, if $S$ is a del Pezzo surface, then $Z$ must be a $(-1)$-curve.
In this case, we have finitely many choices for the curve $Z$.

\begin{example}[{\cite[Example 5.27]{RT1}}]
\label{example:K-unstable-F1}
Suppose that $S\cong\mathbb{F}_1$, and $Z$ is the $(-1)$-curve.
Then $\mathrm{DF}(\mathcal{X},\mathcal{L}_{\lambda},p)<0$ for some $\lambda<\sigma(S,L,Z)$.
Indeed, denote by $f$ the fiber of the natural projection $\mathbb{F}_1\to\mathbb{P}^1$.
Up to positive scaling, either $L\sim_{\mathbb{Q}}-K_S+aZ$
for some non-negative rational number $a<1$, or
$L\sim_{\mathbb{Q}}-K_S+bf$
for some positive rational number $b$.
In the former case, we have $\nu(L)=\frac{8+a}{8+2a-a^2}$ and $\sigma(S,L,Z)=2+a$,
so that Lemma~\ref{lemma:DF-slope} implies that $\mathrm{DF}(\mathcal{X},\mathcal{L}_\lambda,p)<0$ for some $\lambda<2+a$,
because
\begin{multline*}
\lim_{\lambda\to1+a}\mathfrak{DF}(\lambda)=\lim_{\lambda\to 1+a}\frac{2}{3}\cdot\frac{8+a}{8+2a-a^2}\left(-3\lambda^2(1-a)-\lambda^3\right)+2\lambda^2+2\lambda(1-a)=\\
=\frac{4}{3}\cdot \frac{a^3+3a^2-4}{4-a}<0,
\end{multline*}
for all $a\in[0,1)$. Similarly, in the latter case, we have $\nu(L)=\frac{4+b}{4+2b}$ and $\sigma(S,L,e)=2$,
so that
$$
\lim_{\lambda\to 2}\mathfrak{DF}(\lambda)=\lim_{\lambda\to 2}\frac{2}{3}\cdot\frac{4+b}{4+2b}\left(-3\lambda^2(1+b)-\lambda^3\right)+2\lambda^2+2\lambda(1+b)=-\frac{8}{3}\frac{1+b}{2+b}<0,
$$
which implies that $\mathrm{DF}(\mathcal{X},\mathcal{L}_\lambda,p)<0$ for some $\lambda<2$.
\end{example}

The leading term of the cubic polynomial $\mathfrak{DF}(\lambda)$ defined in \eqref{equation:DF-polynomial} is  $\frac{2}{3}\nu(L)Z^2$.
Thus, if $Z^2<0$ and $S$ is del Pezzo surface, then $\mathfrak{DF}(\lambda)<0$ for $\lambda\gg 0$.
Unfortunately, in this case $\mathfrak{DF}(\lambda)$ is usually positive for $\lambda<\sigma(S,L,Z)$,
simply because the Seshadri constant $\sigma(S,L,Z)$ is often \emph{too small}.
In particular, this happens in the case when $S$ is a blow-up of $\mathbb{P}^2$ in two distinct points and $L=-K_S$ (see \cite[Example 7.6]{RP}).
On the other hand, if $S$ is a blow-up of $\mathbb{P}^2$ in two distinct points,
then the group $\mathrm{Aut}_0(S,L)$ is not reductive for every ample divisor $L$,
which implies that $(S,L)$ does not admit a constant scalar curvature K\"ahler metric by Matsushima's obstruction. In particular, by Donaldson's resolution of Conjecture \ref{conjecture:YTD} for toric surfaces, $(S, L)$ is not equivariantly K-polystable, and in particular $(S, L)$ is not K-polystable.

Recall that the pseudo-effective cone is the closure of the cone of effective divisors. There is another famous threshold that one can relate to the triple $(S,L,Z)$,
which is commonly known as \emph{the pseudo-effective threshold}.
It can be defined as
\begin{equation}
\label{equation:tau}
\tau(S,L,Z)=\mathrm{sup}\Big\{\mu\in\mathbb{Q}_{>0}\ \Big\vert\ \text{the divisor}\ L-\mu Z\ \text{is pseudo-effective}\Big\}.
\end{equation}
Since nef divisors are pseudo-effective, we always have $\sigma(S,L,Z)\leqslant\tau(S,L,Z)$,
and the inequality is strict in many interesting cases.
In \cite{CheRu}, Cheltsov and Rubinstein  introduced a birational modification to the slope test configuration
in order to increase the value of $\lambda$ up to the pseudo-effective threshold.
Let us briefly describe their construction.

Suppose that there exists a birational morphism
$\pi\colon S\to\overline{S}$ such that the surface $\overline{S}$ is smooth, and $\pi$ is a blow-up of $k>0$ distinct points $O_1,\ldots,O_k$ in the surface $\overline{S}$.
Moreover, we assume that the image of the curve $Z$ in the surface $\overline{S}$ is a smooth curve that contains all these points.
Let $\overline{Z}=\pi(Z)$, and denote by $C_1,\ldots,C_k$ the $\pi$-exceptional curves that are mapped to the points $O_1,\ldots,O_k$, respectively.
For every point $O_i\in\overline{S}$, let $\Gamma_i$ be the curve $O_i\times\mathbb{P}^1$ in the product $\overline{S}\times\mathbb{P}^1$.
Then there exists a commutative diagram
$$
\xymatrix{
S\times\mathbb{P}^1\ar@{->}[d]_{p_{S}}\ar@{->}[rrr]^{\pi_\Gamma}&&& \overline{S}\times\mathbb{P}^1\ar@{->}[d]^{q_{\overline{S}}}\\%
S\ar@{->}[rrr]_{\pi}&&&\overline{S}}
$$
where $q_{\overline{S}}$ is a natural projection, and $\pi_\Gamma$ is the blow-up of the the curves $\Gamma_1,\ldots,\Gamma_k$.
Let us expand this commutative diagram by adding the threefold $\mathcal{X}$, the blow-up $\pi_Z$, and few other birational maps.
Namely, recall that we identified $Z$ with the curve $Z\times \{0\}$ in the product $S\times\mathbb{P}^1$.
Similarly, let us identify the curve $\overline{Z}$ with the curve $\overline{Z}\times \{0\}$ in the product $\overline{S}\times\mathbb{P}^1$,
so that $Z$ is a proper transform of the curve $\overline{Z}$ via the blow-up $\pi_{\Gamma}$.
Thus, the threefold $\mathcal{X}$ is obtained from $\overline{S}\times\mathbb{P}^1$ by blowing up the curves $\Gamma_1,\ldots,\Gamma_k$,
with a consecutive blow-up of the proper transform of the curve $\overline{Z}$.
If we change the order of blow-ups here (first blow-up the curve $\overline{Z}$, and then blow-up the proper transform of the curves $\Gamma_1,\ldots,\Gamma_k$),
we obtain another (smooth) threefold, which differs from $\mathcal{X}$ by exactly $r$ simple flops.
To be precise, let $\pi_{\overline{Z}}\colon\overline{\mathcal{X}}\to\overline{S}\times\mathbb{P}^1$ be the blow-up of $\overline{S}\times\mathbb{P}^1$ along the curve $\overline{Z}$,
and denote by $\overline{\Gamma}_1,\ldots,\overline{\Gamma}_k$ the proper transforms on $\overline{\mathcal{X}}$ of the curves $\Gamma_1,\ldots,\Gamma_k$, respectively.
Let $\pi_{\overline{\Gamma}}\colon\widehat{\mathcal{X}}\to\overline{\mathcal{X}}$ be the blow-up of the curves $\overline{\Gamma}_1,\ldots,\overline{\Gamma}_k$.
Then there exists a commutative diagram:
\begin{equation}
\label{equation:big-diagram} \xymatrix{
&&&&&\widehat{\mathcal{X}}\ar@{->}[d]^{\pi_{\overline{\Gamma}}}\ar@/^5pc/@{->}[ddddr]^{q}\\%
&&\mathcal{X}\ar@/_3pc/@{->}[dddl]_{p}\ar@{->}[d]_{\pi_{Z}}\ar@{-->}[rrru]^{\rho}&&&\overline{\mathcal{X}}\ar@{->}[d]^{\pi_{\overline{Z}}}&&&\\%
&&S\times\mathbb{P}^1
\ar@{->}[d]_{p_{S}}\ar@/_1pc/@{->}[ddl]_{p_{\mathbb{P}^1}}\ar@{->}[rrr]^{\pi_\Gamma}&&&
\overline{S}\times\mathbb{P}^1\ar@{->}[d]^{q_{\overline{S}}}\ar@/^2pc/@{->}[ddr]^{q_{\mathbb{P}^1}}&&&\\%
&&S\ar@{->}[rrr]_{\pi}&&&\overline{S}&&&\\
&\mathbb{P}^1\ar@{=}[rrrrr]&&&&&\mathbb{P}^1.&&}
\end{equation} %
Here $\rho$ is a composition of flops, $q_{\mathbb{P}^1}$ is the natural projection,
and $q=q_{\mathbb{P}^1}\circ\pi_{\overline{Z}}\circ\pi_{\overline{\Gamma}}$.

Let us describe the curves flopped by $\rho$.
To do this, identify the surface $S$ with the fibre of $p_{\mathbb{P}^1}$ over the point $0\in\mathbb{P}^1$,
and denote by $S_{0}$ its  proper transform on the threefold $\mathcal{X}$.
Then $S_0\cong S$, and the union $S_0\cup E_{Z}$ is the fibre of $p$ over the point $0\in\mathbb{P}^1$.
Denote the proper transforms of the curves $C_1,\ldots,C_k\subset S$ on the threefold $\mathcal{X}$
by $\mathcal{C}_1,\ldots,\mathcal{C}_k$, respectively.
Then the curves $\mathcal{C}_1,\ldots,\mathcal{C}_k$ are contained in $S_0$.
These are the curves flopped by $\rho$.
Observe that each $\mathcal{C}_i$ is a smooth rational curve that is contained in $S_0$,
and its normal bundle in $\mathcal{X}$ is isomorphic to $\mathcal{O}_{\mathbb{P}^1}(-1)\oplus\mathcal{O}_{\mathbb{P}^1}(-1)$ (see \cite[Lemma 4.1]{CheRu}).
Thus, the map $\rho$ is a composition of $r$ simple flops, commonly known as Atiyah flops.

\begin{remark}
\label{remark:anti-flop}
Let us identify the surface $\overline{S}$ with the fibre of $q_{\mathbb{P}^1}$ over the point $0\in\mathbb{P}^1$,
and denote its proper transforms on the threefolds $\overline{\mathcal{X}}$ and $\widehat{\mathcal{X}}$ by $\overline{S}_0$ an $\widehat{S}_0$, respectively.
Then $\overline{S}_0\cong\widehat{S}_0\cong\overline{S}$, and $\rho$ maps $S_{0}$ onto $\widehat{S}_0$,
Moreover, the map $\rho$ induces a birational morphism $S_0\to\widehat{S}_0$ that contracts the curves $\mathcal{C}_1,\ldots,\mathcal{C}_r$,
which is just the morphism $\pi\colon S\to\overline{S}$, since $S_0\cong S$ and $\widehat{S}_0\cong\overline{S}$.
Let $E_{\overline{Z}}$ be the $\pi_{\overline{Z}}$-exceptional surface.
Then $\rho^{-1}$ flops the proper transforms in $\widehat{\mathcal X}$ of the fibers of the projection $E_{\overline{Z}}\to\overline{Z}$ over the points $O_1\times \{0\},\ldots,O_k\times \{0\}$.
\end{remark}

Let $\widehat{\mathcal{L}}_\lambda=\rho_*(\mathcal{L}_\lambda)$.
When is $\widehat{\mathcal{L}}_\lambda$ $q$-ample?
To answer this question, let $\overline{L}=\pi_*(L)$.
Then $\overline{L}$ is an ample $\mathbb{Q}$-divisor on the surface $\overline{S}$.
Observe that $\sigma(S,L,Z)\leqslant\sigma(\overline{S},\overline{L},\overline{Z})$ and
$$
L\sim_{\mathbb{Q}}\pi^*\big(\overline{L}\big)-\sum_{i=1}^k(L\cdot C_i)C_i.%
$$
Then $\sigma(S,L,Z)\leqslant L\cdot C_i$, since $(L-\lambda Z)\cdot C_i=L\cdot C_i-\lambda$.
Furthermore, we have
$$
L-\lambda Z\sim_{\mathbb{Q}}\pi^*\big(\overline{L}-\lambda\overline{Z}\big)+\sum_{i=1}^k(\lambda-L\cdot C_i)C_i
$$
is pseudo-effective if $L\cdot C_i<\lambda<\sigma(\overline S, \overline L, \overline Z)$ for every $i$. Thus, if $L\cdot C_i<\sigma(\overline{S},\overline{L},\overline{Z})$ for every $i$,
then $\sigma(S,L,Z)<\sigma(\overline{S},\overline{L},\overline{Z})\leqslant\tau(S,L,Z)$.

\begin{lemma}[{\cite[Lemma 4.7]{CheRu}}]
\label{lemma:q-ample}
If $L\cdot C_i<\lambda<\sigma(\overline{S},\overline{L},\overline{Z})$ for every $i$, then $\widehat{\mathcal{L}}_\lambda$ is $q$-ample.
\end{lemma}

Hence, if $L\cdot C_i<\lambda<\sigma(\overline{S},\overline{L},\overline{Z})$ for every $i$,
then $(\widehat{\mathcal{X}}, \widehat{\mathcal{L}}_\lambda,q)$ is a test configuration of the~pair~$(S,L)$.
Its Donaldson--Futaki invariant is also easy to compute.
Namely, let
\begin{align}
\label{equation:DF-flope-polynomial}
\widehat{\mathfrak{DF}}(\lambda)&=\mathfrak{DF}(\lambda)+\frac{2}{3}\nu(L)\Bigg(\sum_{i=1}^k \big(\lambda -L\cdot C_i\big)^3\Bigg)=\\
																&=\frac{2}{3}\nu(L)\Bigg(\lambda^3Z^2-3\lambda^2L\cdot Z+\sum_{i=1}^k \big(\lambda -L\cdot C_i\big)^3\Bigg)+\lambda^2\big(2-2g(Z)\big)+2\lambda L\cdot Z,\nonumber
\end{align}
where $\mathfrak{DF}(\lambda)$ is the rational function defined in \eqref{equation:DF-polynomial}.

\begin{theorem}
\label{theorem:flop-slope}
If $L\cdot C_i<\lambda<\sigma(\overline{S},\overline{L},\overline{Z})$ for every $i$, then $\mathrm{DF}(\widehat{\mathcal{X}},\widehat{\mathcal{L}}_\lambda,q)=\widehat{\mathfrak{DF}}(\lambda)$.
\end{theorem}

\begin{proof}
By \cite[Lemma A.3]{CheRu}, we have
$$
\rho_*(H_1)\cdot\rho_*(H_2)\cdot\rho_*(H_3)=H_1\cdot H_2\cdot H_3-\sum_{i=1}^k\big(H_1\cdot C_i\big)\big(H_2\cdot C_i\big)\big(H_3\cdot C_i\big)
$$
for any three $\mathbb{Q}$-divisors $H_1$, $H_2$, $H_3$ on the threefold $\mathcal{X}$.
Therefore, we have
$$
\big(\widehat{\mathcal{L}}_\lambda\big)^3=\big(\mathcal{L}_\lambda\big)^3-\sum_{i=1}^k\big(\mathcal{L}_\lambda\cdot C_i\big)^3.
$$
Similarly, as $C_i$ are floppable curves contained in a fibre of $p$, then $K_{\mathcal{X}}\cdot C_i=0=p^*(K_{\mathbb{P}^1})\cdot C_i$ and we have
$$
\big(\widehat{\mathcal{L}}_\lambda\big)^2\cdot\Big(K_{\widehat{\mathcal{X}}}-q^*\big(K_{\mathbb{P}^1}\big)\Big)=
\big(\mathcal{L}_\lambda\big)^2\cdot\Big(K_{\mathcal{X}}-p^*\big(K_{\mathbb{P}^1}\big)\Big).
$$
Since $\mathcal{L}_{\lambda}\cdot C_i=L\cdot C_i-\lambda$, the assertion follows from \eqref{equation:DF-definition} and Lemma~\ref{lemma:DF-slope}.
\end{proof}

\begin{corollary}
\label{corollary:flop-slope}
Suppose that $L\cdot C_i<\sigma(\overline{S},\overline{L},\overline{Z})$ for every $i$,
and $\widehat{\mathfrak{DF}}(\sigma(\overline{S},\overline{L},\overline{Z}))<0$.
Then there is a positive rational number $\lambda$ such that
$L\cdot C_i<\lambda<\sigma(\overline{S},\overline{L},\overline{Z})$ for every $i$,
and $\mathrm{DF}(\widehat{\mathcal{X}},\widehat{\mathcal{L}}_\lambda,q)<0$.
In particular, the pair $(S,L)$ is $K$-unstable.
\end{corollary}

In this article we will apply this corollary to polarized smooth del Pezzo surfaces.
Which curve $Z$ should we choose in this case? Should it be a $(-1)$-curve?
If the answer is positive, then which $(-1)$-curve should we choose?
Once the curve $Z$ is chosen, how should we choose the contraction $\pi\colon S\to\overline{S}$?
Is it uniquely determined by the curve $Z$?
We will answer all these questions in the remaining part of this article.
But first, let us show how to apply Corollary~\ref{corollary:flop-slope} in the simplest case.

\begin{example}[{cf. \cite[6.1]{CheRu}}]
\label{example:K-unstable-F1-fiber}
As in Example~\ref{example:K-unstable-F1}, suppose that $S\cong\mathbb{F}_1$.
Let $Z$ be a fiber of the natural projection $\mathbb{F}_1\to\mathbb{P}^1$,
let $C_1$ be the $(-1)$-curve, and let $\pi\colon S\to\overline{S}$ be the contraction of $C_1$, so that $k=1$, $\overline{S}\cong\mathbb{P}^2$,
and $\overline{Z}$ is a line.
Up to positive scaling, either $L\sim_{\mathbb{Q}}-K_S+aC_1$
for some non-negative rational number $a<1$, or
$L\sim_{\mathbb{Q}}-K_S+bZ$
for some positive rational number $b$.
In the former case, we have $\nu(L)=\frac{8+a}{8+2a-a^2}$ and $\sigma(\overline{S},\overline{L},\overline{Z})=3$,
so that \eqref{equation:DF-flope-polynomial} gives
$\widehat{\mathfrak{DF}}(\sigma(\overline{S},\overline{L},\overline{Z}))=\frac{2}{3}\cdot\frac{a^3+3a^2-4}{4-a}<0$.
Similarly, in the latter case, we have $\nu(L)=\frac{4+b}{4+2b}$ and  $\sigma(\overline{S},\overline{L},\overline{Z})=3+b$,
so that
$\widehat{\mathfrak{DF}}(\sigma(\overline{S},\overline{L},\overline{Z}))=-\frac{4}{3}\cdot\frac{1+b}{2+b}<0$.
Thus, in both cases, the pair $(S,L)$ is $K$-unstable by Corollary~\ref{corollary:flop-slope}.
\end{example}

Let us conclude this section by one observation inspired by \cite[Corollary~5.29]{RT1}.
To do this, fix a point $P\in S$ that is not contained in the curves $Z,C_1,\ldots,C_k$.
Let $g\colon S^{\prime}\to S$ be the blow-up of the point $P$, and let $G$ be the exceptional curve of the blow-up $g$.
Denote by $Z^{\prime},C_1^{\prime},\ldots,C_k^{\prime}$ the proper transforms of the curves $Z,C_1,\ldots,C_k$
on the surface $S^{\prime}$, respectively.
Let $\overline{P}=\pi(P)$. Then $\overline{P}\not\in\overline{Z}$, and there exists a commutative diagram
$$
\xymatrix{
S^\prime\ar@{->}[rr]^{g}\ar@{->}[d]_{\pi^\prime}&& S\ar@{->}[d]^{\pi}\\%
\overline{S}^\prime\ar@{->}[rr]_{\overline{g}}&&\overline{S}}
$$
where $\pi^{\prime}$ is a contraction of the curves $C_1^{\prime},\ldots,C_k^{\prime}$,
and $\overline{g}$ is the blow-up of the point $\overline{P}$.
Note that the $\overline{g}$-exceptional curve is the proper transform of the curve $G$ on the surface $\overline{S}^\prime$.
Denote this curve by $\overline{G}$.
Merging this commutative diagram together with the large commutative diagram \eqref{equation:big-diagram}, we obtain the even larger commutative diagram
$$
\xymatrix{
\mathcal{X}^\prime\ar@/_4pc/@{->}[ddddr]_{p^\prime}\ar@{-->}[d]_{\rho^{\prime}}\ar@{->}[rr]^{\pi_{Z^{\prime}}}&&S^\prime\times\mathbb{P}^1\ar@{->}[rrrr]^{h}\ar@{->}[ddd]_{\pi_{\Gamma^\prime}}\ar@{->}[dr]^{p_{S^\prime}}&&&&S\times\mathbb{P}^1\ar@{->}[ddd]^{\pi_{\Gamma}}\ar@{->}[dl]_{p_{S}}&&\mathcal{X}\ar@{->}[ll]_{\pi_{Z}}\ar@{-->}[d]^{\rho}\ar@/^4pc/@{->}[ddddl]^{p}\\
\widehat{\mathcal{X}}^\prime\ar@/_1pc/@{->}[dddr]_{q^\prime}\ar@{->}[dr]_{\pi_{\overline{\Gamma}^\prime}}&&&S^\prime\ar@{->}[rr]^{g}\ar@{->}[d]_{\pi^\prime}&& S\ar@{->}[d]^{\pi}&&&\widehat{\mathcal{X}}\ar@/^1pc/@{->}[dddl]^{q}\ar@{->}[dl]^{\pi_{\overline{\Gamma}}}\\%
&\overline{\mathcal{X}}^\prime\ar@{->}[rd]_{\pi_{\overline{Z}^{\prime}}}&&\overline{S}^\prime\ar@{->}[rr]_{\overline{g}}&&\overline{S}&&\overline{\mathcal{X}}\ar@{->}[dl]^{\pi_{\overline{Z}}}&\\
&&\overline{S}^\prime\times\mathbb{P}^1\ar@{->}[dl]^{q_{\mathbb{P}^1}^\prime}\ar@{->}[rrrr]_{\overline{h}}\ar@{->}[ru]_{q_{\overline{S}^\prime}}&&&&\overline{S}\times\mathbb{P}^1\ar@{->}[dr]_{q_{\mathbb{P}^1}}\ar@{->}[lu]^{q_{\overline{S}}}&&\\
&\mathbb{P}^1\ar@{=}[rrrrrr]&&&&&&\mathbb{P}^1.&}
$$
Here $h$ is the blow-up of the curve $P\times\mathbb{P}^1$,
$\overline{h}$ is the blow-up of the curve $\overline{P}\times\mathbb{P}^1$,
and the maps $\pi_{Z^\prime}$, $\pi_{\overline{Z}^\prime}$, $\pi_{\Gamma^\prime}$, $\pi_{\overline{\Gamma}^\prime}$,
$\rho^\prime$, $q_{\mathbb{P}^1}^\prime$, $q_{\overline{S}^\prime}$, $p_{S^{\prime}}$, $p^\prime$ and $q^\prime$
are defined similarly to the maps  $\pi_{Z}$, $\pi_{\overline{Z}}$, $\pi_{\Gamma}$, $\pi_{\overline{\Gamma}}$,
$\rho$, $q_{\mathbb{P}^1}$, $q_{\overline{S}}$, $p_S$, $p$ and $q$, respectively.
To get their detailed description, one just has to add $\prime$ to every geometrical object involved
in the definition of the maps $\pi_{Z}$, $\pi_{\overline{Z}}$, $\pi_{\Gamma}$, $\pi_{\overline{\Gamma}}$,
$\rho$, $q_{\mathbb{P}^1}$, $q_{\overline{S}}$, $p$ and $q$,
We leave this to the reader.

To polarize the surface $S^\prime$, choose a positive rational number $\varepsilon$, and let $L^\prime=g^*(L)-\varepsilon G$.
Then $L^\prime$ is ample provided that the number $\varepsilon$ is small enough.
Let us assume that this is the case.  Let $E_{Z^\prime}$ be the exceptional divisor of $\pi_{Z^\prime}$, and let
$\mathcal{L}_{\lambda}^\prime=(p_{S^\prime}\circ\pi_{Z^\prime})^*(L^\prime)-\lambda E_{Z^\prime}$,
where $\lambda$ is a positive rational number.
If $\lambda<\sigma(S^\prime,L^\prime,Z^\prime)$, then the divisor $\mathcal{L}_{\lambda}^\prime$ is $p^\prime$-ample by \cite[Proposition~4.1]{RT2},
so that $(\mathcal{X}^\prime,\mathcal{L}_{\lambda}^\prime,p^\prime)$ is a slope test configuration of the pair $(S^\prime,L^\prime)$.
In this case, its Donaldson--Futaki invariant is given by Lemma~\ref{lemma:DF-slope}.
Namely, we have
$$
\mathrm{DF}\big(\mathcal{X}^\prime,\mathcal{L}_{\lambda}^\prime,p^\prime\big)=\frac{2}{3}\nu(L^\prime)\Big(\lambda^3Z^2-3\lambda^2L\cdot Z\Big)+\lambda^2\big(2-2g(Z)\big)+2\lambda L\cdot Z.
$$
Here we used the fact that the point $P$ is not contained in the curve $Z$.
This assumption also implies that
$$
\lim_{\varepsilon\to 0^+}\sigma\big(S^\prime,L^\prime,Z^\prime\big)=\sigma(S,L,Z).
$$
Moreover, we have $\nu(L)=\frac{-K_{S}\cdot L}{L\cdot L}$ and
$\nu(L^\prime)=\frac{-K_{S^\prime}\cdot L^\prime}{L^\prime\cdot L^\prime}=\frac{-K_S\cdot L-\varepsilon}{L^2-\varepsilon^2}$.
This gives

\begin{corollary}[{\cite[Corollary~5.29]{RT1}}]
\label{corollary:RT-blow-up}
Suppose that $\lambda<\sigma(S,L,Z)$ and $\mathrm{DF}\big(\mathcal{X},\mathcal{L}_{\lambda},p)<0$.
Then $\lambda<\sigma(S^\prime,L^\prime,Z^\prime)$ and $\mathrm{DF}(\mathcal{X}^\prime,\mathcal{L}_{\lambda}^\prime,p^\prime)<0$
for sufficiently small $\varepsilon>0$.
\end{corollary}

Note that this corollary together with Example~\ref{example:K-unstable-F1} imply Theorem~\ref{theorem:Ross-Thomas}.
A similar corollary exists for the flopped version of the slope test configuration described in Section~\ref{section:DF-invariant}.
To present it here, let $\widehat{\mathcal{L}}_\lambda^\prime=\rho^\prime_*(\mathcal{L}_\lambda^\prime)$ and $\overline{L}^\prime=\pi^\prime_{*}(L^\prime)$.
Then $\overline{L}^\prime$ is ample. Since $L\cdot C_i=L^\prime\cdot C_i^\prime$ for every curve $C_i$,
we have
$$
L^\prime\sim_{\mathbb{Q}}g^*(L)-\varepsilon G\sim_{\mathbb{Q}}(\pi^\prime)^*\big(\overline{L}^\prime\big)-\sum_{i=1}^k(L\cdot C_i)C_i^\prime.
$$
By Lemma~\ref{lemma:q-ample}, if $L\cdot C_i<\lambda<\sigma(\overline{S},\overline{L},\overline{Z})$ for every $i$,
then $\widehat{\mathcal{L}}_\lambda$ is $q$-ample, so that
$(\widehat{\mathcal{X}}, \widehat{\mathcal{L}}_\lambda,q)$ is a test configuration of the pair $(S,L)$.
Similarly, if $L\cdot C_i<\lambda<\sigma(\overline{S}^\prime,\overline{L}^\prime,\overline{Z}^\prime)$ for every~$i$,
then $\widehat{\mathcal{L}}_\lambda^\prime$ is $q^\prime$-ample, so that
$(\widehat{\mathcal{X}}^\prime, \widehat{\mathcal{L}}_\lambda^\prime, q^\prime)$ is a test configuration of the pair $(S^\prime,L^\prime)$.
In this case, its Donaldson--Futaki invariant $\widehat{\mathrm{DF}}(\widehat{\mathcal{X}}^\prime,\widehat{\mathcal{L}}_\lambda^\prime,q^\prime)$
is given by the formula
$$
\frac{2}{3}\nu(L^\prime)\Big(\lambda^3Z^2-3\lambda^2L\cdot Z\Big)+\lambda^2\big(2-2g(Z)\big)+2\lambda L\cdot Z+\frac{2}{3}\nu(L^\prime)\Bigg(\sum_{i=1}^k \big(\lambda -L\cdot C_i\big)^3\Bigg)
$$
by Theorem~\ref{theorem:flop-slope}.
As above, we have
$$
\lim_{\varepsilon\to 0^{+}}\sigma\big(\overline{S}^\prime,\overline{L}^\prime,\overline{Z}^\prime\big)=\sigma\big(\overline{S},\overline{L},\overline{Z}\big).
$$
This gives

\begin{corollary}
\label{corollary:blow-up}
Suppose that $L\cdot C_i<\lambda<\sigma(\overline{S},\overline{L},\overline{Z})$ for every $i$, and
$\widehat{\mathrm{DF}}(\widehat{\mathcal{X}}, \widehat{\mathcal{L}}_\lambda, q)<0$.
Then $L\cdot C_i<\lambda<\sigma\big(\overline{S}^\prime,\overline{L}^\prime,\overline{Z}^\prime\big)$ for every $i$,
and $\widehat{\mathrm{DF}}(\widehat{\mathcal{X}}^\prime, \widehat{\mathcal{L}}_\lambda^\prime, q^\prime)<0$
for sufficiently small~$\varepsilon$.
\end{corollary}

We will use this corollary to prove Theorem~\ref{theorem:degree-6}.

\section{Ample divisors on del Pezzo surfaces}
\label{section:del-Pezzo}

In this section we describe basic facts about smooth del Pezzo surfaces.
The simplest examples of such surfaces are $\mathbb{P}^2$, $\mathbb{P}^1\times\mathbb{P}^1$ and the first Hirzebruch surface $\mathbb{F}_1$.
To work with them, we fix notations that we will use throughout the remaining part of this article.
Namely, we denote by $\ell$ the class of a line in $\mathbb{P}^2$.
For $\mathbb{P}^1\times\mathbb{P}^1$, we denote by $f_1$ and $f_2$ the fibres of the two distinct projections $\mathbb{P}^1\times\mathbb{P}^1\to\mathbb{P}^1$.
Similarly, for the surface $\mathbb{F}_1$, we denote by $e$ the unique $(-1)$-curve in $\mathbb{F}_1$,
and we denote by $f$ the class of a fibre of the natural morphism $\mathbb{F}_1\to\mathbb{P}^1$.

\begin{remark}
\label{remark:nef-ample}
Note that the divisor $af_1+bf_2$ on $\mathbb{P}^1\times\mathbb{P}^1$ is nef (respectively, ample) if and only if $a\geqslant 0$ and $b\geqslant 0$ (respectively, $a>0$ and $b>0$).
The classes $f_1$ and $f_1$ also generate the Mori cone $\overline{\mathbb{NE}}(\mathbb{P}^1\times\mathbb{P}^1)$.
Similarly, a divisor $ae+bf$ on $\mathbb{F}_1$ is nef (respectively, ample) if and only if $b\geqslant a\geqslant 0$ (respectively, $b>a>0$).
The classes $e$ and $f$ generate the Mori cone $\overline{\mathbb{NE}}(\mathbb{F}_1)$.
\end{remark}

Now let $S$ be a smooth del Pezzo surface such that $K_S^2\leqslant 7$,
so that $S\not\cong\mathbb{P}^2$, $S\not\cong\mathbb{P}^1\times\mathbb{P}^1$ and $S\not\cong\mathbb{F}_1$.
Then the Mori cone $\overline{\mathbb{NE}}(S)$ is a polyhedral cone that is generated by all $(-1)$-curves on the surface $S$, i.e. smooth rational curves with self-intersection $-1$. Recall that there is a finite number of $(-1)$-curves on any del Pezzo surface.
The description of these curves is well-known.
Nevertheless, we decided to partially present it here, because we will need this description later.

First, we choose a birational morphism $\gamma\colon S\to\mathbb{P}^2$ that contracts $9-K_S^2\geqslant 2$ disjoint $(-1)$-curves.
Such morphism always exists, since we assume that $K_S^2\leqslant 7$.
However, it is never unique for $K_S^2\leqslant 6$.
We let $r=9-K_S^2$, and denote the $\gamma$-exceptional curves by $E_1,\ldots,E_r$.
Let $L_{ij}$ be the proper transform of the line in $\mathbb{P}^2$ that contains the points $\gamma(E_i)$ an $\gamma(E_j)$, where $1\leqslant i<j\leqslant r$.
Then
$$
L_{ij}\sim\gamma^*\big(l\big)-E_i-E_j,
$$
and each $L_{ij}$ is a $(-1)$-curve.
In fact, if $r\leqslant 4$, then these are all $(-1)$-curves on $S$ aside of the curves $E_1,\ldots,E_r$.
If $r\geqslant 5$, let $C_{i_1i_2i_3i_4i_5}$ be the proper transform of the conic in $\mathbb{P}^2$ that contains $\gamma(E_{i_1})$, $\gamma(E_{i_2})$, $\gamma(E_{i_3})$, $\gamma(E_{i_4})$ and $\gamma(E_{i_5})$ for $1\leqslant i_1<i_2<i_3<i_4<i_5\leqslant r$.
Then  each $C_{i_1i_2i_3i_4i_5}$ is also a $(-1)$-curve and
$$
C_{i_1i_2i_3i_4i_5}\sim\gamma^*\big(2l\big)-E_{i_1}-E_{i_2}-E_{i_3}-E_{i_4}-E_{i_5}.
$$
If $r=5$ or $r=6$, then $C_i, L_{ij}, E_i$ describe all the $(-1)$-curves in $S$. If $r=7$, then we denote by $Z_i$ the proper transform of the cubic in $\mathbb{P}^2$ that contains the points
$\gamma(E_1)$, $\gamma(E_2)$, $\gamma(E_3)$, $\gamma(E_4)$, $\gamma(E_5)$, $\gamma(E_6)$, $\gamma(E_7)$, and $Z_i$ is singular at the point $\gamma(E_i)$.
In this case, each $Z_i$ is a $(-1)$-curve, and
$$
Z_{i}\sim\gamma^*\big(3l\big)-\Big(E_1-E_2-E_3-E_4-E_5-E_6-E_6-E_7\Big)-E_i.
$$
We have described all $(-1)$-curves on $S$ in the case when $K_S^2\geqslant 2$.
If $K_S^2=1$, then $S$ contains many more $(-1)$-curves.
For example, the class
$$
\gamma^*\big(3l\big)-\Big(E_1-E_2-E_3-E_4-E_5-E_6-E_6-E_7-E_8\Big)-E_i+E_j
$$
contains a unique $(-1)$-curve for every $i\ne j$, which we denote by $Z_{ij}$.
This curve is the proper transform of the cubic in $\mathbb{P}^2$ that contains
all points $\gamma(E_1)$, $\gamma(E_2)$, $\gamma(E_3)$, $\gamma(E_4)$, $\gamma(E_5)$, $\gamma(E_6)$, $\gamma(E_7)$, $\gamma(E_8)$ except for $\gamma(E_j)$,
which is singular at the point $\gamma(E_i)$. There is also a unique $(-1)$-curve defined in each of the following classes:
\begin{align*}
\gamma^*(4l)-&\sum_{l=1}^{8}E_l-E_i-E_j-E_k\qquad 1\leqslant i<j<k\leqslant 8,\\
\gamma^*(5l)-&\sum_{\substack{l=1\\
l\neq i,j}}^{8}E_l-\sum_{l=1}^8E_l\qquad 1\leqslant i<j\leqslant 8,\\
\gamma^*(6l)-&E_i-2\sum_{l=1}^8E_l\qquad 1\leqslant i\leqslant 8,\\
\end{align*}
completing the description of all $(-1)$-curves when $K_S^2=1$.

Let $L$ be an ample $\mathbb{Q}$-divisor on the surface $S$.
Then $L\cdot C>0$ for every $(-1)$-curve $C$ on the surface $S$.
In fact, the latter condition is equivalent to the ampleness of the divisor $L$.
Moreover, we have
\begin{equation}
\label{equation:non-canonical}
L\sim_{\mathbb{Q}}\gamma^{*}\big(\varepsilon l\big)-\sum_{i=1}^{r}\varepsilon_iE_i
\end{equation}
for some positive rational numbers $\varepsilon,\varepsilon_1,\ldots,\varepsilon_r$.
Unfortunately, this $\mathbb{Q}$-rational equivalence is not canonical, since the contraction $\gamma$ is not unique for $K_S^2\leqslant 6$.
There is a \emph{better} way to work with ample divisors on $S$.
To present it, we let
$$
\mu_{L}=\inf\Big\{\lambda\in\mathbb{Q}_{>0}\ \Big|\ K_S+\lambda L\in\overline{\mathrm{NE}}(S)\Big\}.
$$
Then $\mu_L$ is a positive rational number, known as the \emph{Fujita invariant} of $(S,L)$.
Let $\Delta_L$ be the \emph{smallest} face of the Mori cone $\overline{\mathrm{NE}}(X)$ that contains $K_S+\mu_L L$.
If $\Delta_L=0$, then $\mu_LL\sim_{\mathbb{Q}} -K_{S}$.
If $\mathrm{dim}(\Delta_L)\neq 0$, then $K_X+\mu_LL$ is a non-zero effective divisor.
Applying the Minimal Model Program, we obtain a morphism $\phi\colon S\to Y$ where $Y$ is smooth and such that
$\phi$ contracts all curves contained in $\Delta_L$, i.e. $\phi$ is the contraction of the face $\Delta_L$.
Then either
\begin{itemize}
\item $\phi$ is a birational morphism that contracts $\mathrm{dim}(\Delta_L)\leqslant r$ disjoint $(-1)$-curves, or
\item $\mathrm{dim}(\Delta_L)=r$, $Y\cong\mathbb{P}^1$ and general fiber of $\phi$ is $\mathbb{P}^1$, i.e. $\phi$ is a conic bundle.
\end{itemize}

It seems quite natural to split ample divisors in $\mathrm{Amp}(S)$ according to the type of contraction $\phi$.
However, we prefer to use a slightly different splitting into types that is based
on the contraction of one of the \emph{largest} faces of the Mori cone $\overline{\mathrm{NE}}(S)$ that contains $K_X+\mu_L L$. The contraction of a face of maximum dimension guarantees that the image of $\phi$ will not contain any $(-1)$-curve (in fact, the image of $\phi$ may not even be a surface, but $\mathbb P^1$, giving $\phi$ the structure of a conic bundle).
Namely, observe that if $\phi$ is birational and $Y\cong\mathbb{P}^1\times\mathbb{P}^1$, then
\begin{equation}
\label{equation:equivalence-birational-special}
\mu_LL\sim_{\mathbb{Q}} -K_S+\sum_{i=1}^{r-1} a_iF_i,
\end{equation}
where $F_1,\ldots,F_{r-1}$ are disjoint $(-1)$-curves contracted by $\phi$, and each $a_i$ is a positive rational number such that $a_i<1$.
Similarly, if $\phi$ is birational and $Y\not\cong\mathbb{P}^1\times\mathbb{P}^1$,
then there exists a (possibly non-unique) birational morphism $\psi\colon Y\to\mathbb{P}^2$ such that the composition $\psi\circ\phi$ is a contraction of $r$ disjoint $(-1)$-curves $F_1,\ldots,F_{r}$,
which generate a maximal face of the Mori cone $\overline{\mathrm{NE}}(X)$ that contains $K_S+\mu_L L$.
In this case, we have
\begin{equation}
\label{equation:equivalence-birational}
\mu_LL\sim_{\mathbb{Q}} -K_S+\sum_{i=1}^{r} a_iF_i,
\end{equation}
where each $a_i$ is a non-negative rational number such that $a_i<1$.
Observe that $F_i\in\Delta_L$ if and only if $a_i>0$,
and, a priori, the contraction $\psi\circ\phi$ does not need to coincide with $\gamma$.
Note that in both cases, we have a very simple formula for the slope $\nu(L)$ of the pair $(S,L)$.
Namely, we have
$$
\nu(L)=\frac{-K_S\cdot L}{L^2}=\mu_L\frac{d+\sum_{i=1}^{r}a_i}{d+2\sum_{i=1}^{r}a_i-\sum_{i=1}^{r}a_i^2}.
$$

If $\phi$ is a conic bundle, then $Y\cong\mathbb{P}^1$,
and $\Delta_L$ is a maximal face of the Mori cone $\overline{\mathrm{NE}}(X)$ that contains $K_X+\mu_L L$.
Note that the morphism $\phi$ has exactly $r-1=8-K_S^2$ reducible fibers,
each of them consisting of two $(-1)$-curves, and the face $\Delta_L$ is generated by these $(-1)$-curves.
Then we have
\begin{equation}
\label{equation:equivalence-conic-bundle}
\mu_LL\sim_{\mathbb{Q}} -K_S+bB+\sum_{i=1}^{r-1} a_iF_i,
\end{equation}
where $B$ is a general fiber of $\phi$, and $F_1,\ldots,F_{r-1}$ are disjoint $(-1)$-curves contained in the singular fibers of $\phi$,
each $a_i$ is a non-negative rational number such that $a_i<1$, and $b$ is a positive rational number.
Then
$$
\nu(L)=\frac{-K_S\cdot L}{L^2}=\mu_L\frac{d+2b+\sum_{i=1}^{r-1}a_i}{d+4b+2\sum_{i=1}^{r-1}a_i-\sum_{i=1}^{r-1}a_i^2}.
$$
In addition, there exists a commutative diagram
\begin{equation}
\label{equation:diagram}
\xymatrix{
&S\ar@{->}[dl]_{\psi}\ar@{->}[dr]^{\phi}&\\
\widehat{S}\ar@{->}[rr]_{\omega}&&\mathbb{P}^1}
\end{equation}
where $\psi$ is a birational morphism that contracts the curves $F_1,\ldots,F_{r-1}$, and $\omega$ is a natural projection.
Then either $\widehat{S}\cong\mathbb{F}_1$ or $\widehat{S}\cong\mathbb{P}^1\times\mathbb{P}^1$.
Observe that the morphism $\psi$ in \eqref{equation:diagram} is uniquely determined by $L$ only if every $a_i$ in \eqref{equation:equivalence-conic-bundle} is positive.
Thus, if at least one of the numbers $a_1,\ldots,a_{r-1}$ in \eqref{equation:equivalence-conic-bundle} is not positive, then we may assume that $\widehat{S}=\mathbb{P}^1\times\mathbb{P}^1$.

\begin{definition}
\label{definition:types}
We say that
\begin{itemize}
\item the divisor $L$ is of \emph{$\mathbb{P}^2$-type} if $\phi$ is birational and $Y\not\cong\mathbb{P}^1\times\mathbb{P}^1$;

\item the divisor $L$ is of \emph{$\mathbb{P}^1\times\mathbb{P}^1$-type} if either $\phi$ is a conic bundle and $\widehat{S}\cong\mathbb{P}^1\times\mathbb{P}^1$, or $\phi$ is birational and $Y\cong\mathbb{P}^1\times\mathbb{P}^1$;

\item the divisor $L$ is of \emph{$\mathbb{F}_1$-type} if $\phi$ is a conic bundle, $\widehat{S}\cong\mathbb{F}_1$, and every $a_i$ in \eqref{equation:equivalence-conic-bundle} is positive.
\end{itemize}
\end{definition}

We will always assume that $0\leqslant a_1\leqslant \cdots\leqslant a_r<1$ if $L$ is of $\mathbb{P}^2$-type.
Similarly, if $L$ is of $\mathbb{P}^1\times\mathbb{P}^1$-type or of $\mathbb{F}_1$-type,
then we will assume that $0\leqslant a_1\leqslant \cdots\leqslant a_{r-1}<1$.

\begin{remark}
\label{remark:types}
Suppose that $L$ is of $\mathbb{P}^1\times\mathbb{P}^1$-type.
Then we can combine numerical equivalences
\eqref{equation:equivalence-birational-special} and \eqref{equation:equivalence-conic-bundle} together
by allowing $b$ to be zero in \eqref{equation:equivalence-conic-bundle}.
Thus, if $b=0$ in \eqref{equation:equivalence-conic-bundle},
then $f$ is birational and $Y\cong\mathbb{P}^1\times\mathbb{P}^1$, so that every $a_i$ in \eqref{equation:equivalence-conic-bundle} is positive.
\end{remark}

If $K_S^2=7$, then $r=2$, and $\gamma$ is uniquely determined.
Thus, if $L$ is of $\mathbb P^2$-type, then we may assume that $F_1=E_1$ and $F_2=E_2$.
Similarly, if $L$ is $\mathbb{F}_1$-type, then we may assume that $F_1=E_1$.
If $L$ is of $\mathbb{P}^1\times\mathbb{P}^1$-type, then $F_1=L_{12}$.
In this case, we may assume that $B\sim L_{12}+E_1\sim\gamma^*(l)-E_2$.

If $K_S^2\leqslant 6$, we can choose the contraction $\gamma\colon S\to\mathbb{P}^2$ according to the type of the divisor $L$.
Namely, if $L$ is of $\mathbb{P}^2$-type, then we can assume that $\gamma=\psi\circ\phi$ and $E_i=F_i$ for every $i$.
Similarly, if $L$ is of either $\mathbb{F}_1$-type or $\mathbb{P}^1\times\mathbb{P}^1$-type, then we can assume that
$$
B\sim L_{1r}+E_1\sim\gamma^{*}(\ell)-E_r,
$$
so that $B$ is a proper transform of a general line in $\mathbb{P}^2$ passing through the point $\gamma(E_r)$.
Similarly, if $L$ is of $\mathbb{F}_1$-type, then we can assume that
$F_i=E_i$ for every $i$ such that $r-1\geqslant i\geqslant 1$,
so that $\gamma$ is a composition of $\psi$ with the birational morphism $\mathbb{F}_1\to\mathbb{P}^2$, which contracts the curve $\psi(E_r)$.
If $L$ is of $\mathbb{P}^1\times\mathbb{P}^1$-type and $r=3$, then we can assume that $F_1=E_1$ and $F_2=L_{2r}$.
Finally, if $L$ is of $\mathbb{P}^1\times\mathbb{P}^1$-type and $r\geqslant 4$,
then we can assume that
$F_1=E_1$, $F_2=L_{2r}$, and $F_i=E_i$ for every $i$ such that $r-1\geqslant i\geqslant 3$.

Let us illustrate the introduced language by two examples that show how to apply
Corollary~\ref{corollary:flop-slope} to the pair $(S,L)$ in the case when $S$ is toric (cf. Examples~\ref{example:K-unstable-F1} and \ref{example:K-unstable-F1-fiber}).

\begin{example}
\label{example:d-7}
Suppose that $K_S^2=7$.
Let $Z=L_{12}$, $C_1=E_1$, $C_2=E_2$, $\overline{S}=\mathbb{P}^2$, and let $\pi\colon S\to\overline{S}$ be the contraction of the curves $C_1$ and $C_2$ and $\overline Z\sim l$.
Then we can use the notations and assumptions of Section~\ref{section:DF-invariant}.
We claim that there is a positive rational number $\lambda$ such that
$L\cdot C_1<\lambda$, $L\cdot C_2<\lambda$, $\lambda<\sigma(\overline{S},\overline{L},\overline{Z})$
and $\mathrm{DF}(\widehat{\mathcal{X}},\widehat{\mathcal{L}}_\lambda,q)<0$,
which implies, in particular, that $(S,L)$ is $K$-unstable.
By Corollary~\ref{corollary:flop-slope}, it is enough to show that
$L\cdot C_2<\sigma(\overline{S},\overline{L},\overline{Z})$, $L\cdot C_1<\sigma(\overline{S},\overline{L},\overline{Z})$,
and $\widehat{\mathfrak{DF}}(\sigma(\overline{S},\overline{L},\overline{Z}))<0$.
To do this, we may assume that $\mu_L=1$.
Using \eqref{equation:DF-flope-polynomial}, we see that
$$
\widehat{\mathfrak{DF}}(\lambda)=\nu(L)\Big(-3\lambda^2L\cdot Z-\lambda^3+(\lambda-L\cdot C_1)^3+(\lambda-L\cdot C_2)^3\Big)+2\lambda^2+2\lambda L\cdot Z.
$$
If $L$ is of $\mathbb{P}^2$-type, then $L\cdot Z =1+a_1+a_2$, $L\cdot C_1=1-a_1$, $L\cdot C_2=1-a_2$,
$\overline{L}\sim_{\mathbb{Q}}3\ell$, which implies that $\sigma(\overline{S},\overline{L},\overline{Z})=3>L\cdot C_1\geqslant L\cdot C_2$,
so that
$$
\widehat{\mathfrak{DF}}\big(\sigma(\overline{S},\overline{L},\overline{Z})\big)=\frac{2}{3}\frac{(1+a_1+a_2)(a_1^3+3a_1^2-6a_1a_2+6a_1+a_2^3+6a_2+3a_2^2-14)}{7+2(a_1+a_2)-a_1^2-a_2^2}<0,
$$
because
\begin{multline*}
a_1^3+3a_1^2-6a_1a_2+6a_1+a_2^3+6a_2+3a_2^2-14=-(a_2-a_1)^3-\\
-3(1-a_2)(a_2-a_1)^2-3(a_2-a_1)(1-a_2)^2-6(a_2-a_1)(1-a_2+a_1)-\\
-2(1-a_2)^3-3(a_2-a_1)a_2-3(1-a_2)(a_1+a_2+4)\leqslant -2(1-a_2)^3<0.
\end{multline*}
If $L$ is of $\mathbb{F}_1$-type,
then $L\cdot Z=1+a_1$, $L\cdot C_1=1-a_1$, $L\cdot C_2=1+b$ and $\overline{L}\sim_{\mathbb{Q}}(3+b)\ell$,
which implies that $\sigma(\overline{S},\overline{L},\overline{Z})=3+b>L\cdot C_2>L\cdot C_1$,
so that
\begin{multline*}
\widehat{\mathfrak{DF}}(\sigma(\overline{S},\overline{L},\overline{Z}))=\frac{2}{3}\frac{1+a_1}{7+4b+2a_1-a_1^2}+\frac{2}{3}\frac{1+a_1}{7+4b+2a_1-a_1^2}(3a_1^2-3)b^2+\\
+\frac{2}{3}\frac{1+a_1}{7+4b+2a_1-a_1^2}\Big((-16-6a_1+12a_1^2+2a_1^3)b-8a_1+a_1^4+9a_1^2+4a_1^3-14\Big)<0.
\end{multline*}
Finally, if $L$ is of $\mathbb{P}^1\times\mathbb{P}^1$-type, then $L\cdot Z =1-a_1$,
$L\cdot C_1=1+a_1$, $L\cdot C_2=1+a_1+b$ and $\overline{L}\sim_{\mathbb{Q}}(3+b+a_1)\ell$,
which implies that $\sigma(\overline{S},\overline{L},\overline{Z})=3+a_1+b>L\cdot C_2\geqslant L\cdot C_1$ and
$$
\widehat{\mathfrak{DF}}\big(\sigma(\overline{S},\overline{L},\overline{Z})\big)=\frac{2}{3}\frac{(a_1-1)(2a_1^3+4a_1^2b+3a_1b^2+10a_1^2+16a_1b+3b^2+22a_1+16b+14)}{7+4b+2a_1-a_1^2},
$$
so that $\widehat{\mathfrak{DF}}(\sigma(\overline{S},\overline{L},\overline{Z}))<0$.
\end{example}

\begin{example}
\label{example:d-6}
Suppose that $K_S^2=6$.
Then it follows from \cite{toric} (see also \cite[Example~3.2]{LeBrun-Simanca}) that
$(S,L)$ is $K$-polystable if and only if $\varepsilon_1=\varepsilon_2=\varepsilon_3$ or $\varepsilon=\varepsilon_1+\varepsilon_2+\varepsilon_3$ in \eqref{equation:non-canonical}.
Thus, if $L$ is of $\mathbb{P}^2$-type, then $(S,L)$ is $K$-polystable if and only if $a_1=a_2=a_3$,
because
$$
\varepsilon_i=L\cdot E_i=\frac{1-a_i}{\mu_L}
$$
and $\varepsilon=L\cdot\gamma^*(l)=\frac{3}{\mu_L}$ in this case.
Similarly, if $L$ is of $\mathbb{F}_1$-type, then
$\varepsilon_1=\frac{1-a_1}{\mu_L}$, $\varepsilon_2=\frac{1-a_2}{\mu_L}$ and
$\varepsilon_3=\frac{1+b}{\mu_L}$ and $\varepsilon=\frac{3+b}{\mu_L}$,
which implies that $(S,L)$ is not $K$-polystable, because $a_1>0$, $a_2>0$ and $b>0$ in this case.
Finally, if $L$ is of $\mathbb{P}^1\times\mathbb{P}^1$-type, then
$\varepsilon_1=\frac{1-a_1}{\mu_L}$, $\varepsilon_2=\frac{1+a_2}{\mu_L}$,
$\varepsilon_3=\frac{1+a_2+b}{\mu_L}$ and $\varepsilon=\frac{3+a_2+b}{\mu_L}$,
so that $(S,L)$ is $K$-polystable if and only if $a_1=a_2$.
Thus, we see that $(S,L)$ is $K$-polystable if and only if either $L$ is of $\mathbb{P}^2$-type and $a_1=a_2=a_3$,
or $L$ is of $\mathbb{P}^1\times\mathbb{P}^1$-type and $a_1=a_2$.
In fact, if none of these conditions is satisfied, then $(S,L)$ can be destabilized
by the flopped version of the slope test configuration described in Section~\ref{section:DF-invariant}.
To show this, let $Z$ be one of the $(-1)$-curves on the surface $S$,
let $C_1$ and $C_2$ be two disjoint $(-1)$-curves that intersect $Z$,
and let $\pi\colon S\to\overline{S}$ be the contraction of the curves $C_1$ and $C_2$.
Then $\overline{S}\cong\mathbb{F}_1$ and $\overline{Z}\sim f+e$.
Let us use the notations and assumptions of Section~\ref{section:DF-invariant}.
As in Example~\ref{example:d-7}, it is enough to show that we can choose $Z$ such that
$L\cdot C_1<\sigma(\overline{S},\overline{L},\overline{Z})$, $L\cdot C_2<\sigma(\overline{S},\overline{L},\overline{Z})$
and $\widehat{\mathfrak{DF}}(\sigma(\overline{S},\overline{L},\overline{Z}))<0$.
To do this, we may assume that $\mu_L=1$.
If $L$ is of $\mathbb{P}^2$-type, then we let $Z=L_{12}$, $C_1=E_1$ and $C_2=E_2$,
so that $L\cdot Z=1+a_1+a_2$, $L\cdot C_1=1-a_1$, $L\cdot C_2=1-a_2$, $\overline{Z}\sim f+e$
and $\overline{L}\sim_{\mathbb{Q}}(2+a_3)e+3f$,
which implies that
$$
\sigma(\overline{S},\overline{L},\overline{Z})=2+a_3>L\cdot C_2\geqslant L\cdot C_1,
$$
and it follows from \eqref{equation:DF-flope-polynomial} that
$$
\widehat{\mathfrak{DF}}\big(\sigma(\overline{S},\overline{L},\overline{Z})\big)=\frac{2(a_1+a_2+a_3)h_1(a_1,a_2,a_3)}{3(6+2(a_1+a_2+a_3)-a_1^2-a_2^2-a_3^2)},
$$
where $h_1(a_1,a_2,a_3)=a_1^3+a_2^3-2a_3^3+3a_1^2-6a_1a_2+3a_1a_3+3a_2^2+3a_2a_3-6a_3^2+3a_1+3a_2-6a_3$.
If $a_1<a_3$, then
\begin{multline*}
h_1(a_1,a_2,a_3)=-3a_1^2(a_2-a_1)-6a_1^2(a_3-a_2)-3a_1(a_2-a_1)^2-12a_1(a_2-a_1)(a_3-a_2)-\\
-6a_1(a_3-a_2)^2-(a_2-a_1)^3-6(a_2-a_1)^2(a_3-a_2)-6(a_2-a_1)(a_3-a_2)^2-\\
-2(a_3-a_2)^3-3a_1(a_2-a_1)-6a_1(a_3-a_2)-9(a_2-a_1)(a_3-a_2)-\\
-6(a_3-a_2)^2-3(a_2-a_1)-6(a_3-a_2)\leqslant 3(a_2-a_1)-6(a_3-a_2)<0,
\end{multline*}
so that $\widehat{\mathfrak{DF}}(\sigma(\overline{S},\overline{L},\overline{Z}))<0$ in this case.
Similarly, if $L$ is of $\mathbb{F}_1$-type, we let $Z=L_{13}$, $C_1=E_1$ and $C_2=E_3$,
so that $L\cdot Z=1+a_1$, $L\cdot C_1=1-a_1$, $L\cdot C_2=1+b$, $\overline Z\sim e+f$ and $\overline{L}\sim_{\mathbb{Q}}(2+b+a_2)e+(3+b)f$,
which implies that $\sigma(\overline{S},\overline{L},\overline{Z})=2+b+a_2>L\cdot C_2>L\cdot C_1$,
so that
$$
\widehat{\mathfrak{DF}}\big(\sigma(\overline{S},\overline{L},\overline{Z})\big)=\frac{2(a_1+a_2)h_2(a_1,a_2,b)}{3(6+4b+2(a_1+a_2)-a_1^2-a_2^2)},
$$
where $h_2(a_1,a_2,b)$ is the polynomial
$$
3(a_1-a_2)b^2+(2a_1^2+a_1a_2-4a_2^2+9(a_1-a_2)-3)b + a_1^3-2a_2^3+3a_1^2+3a_1a_2-6a_2^2+3a_1-6a_2.
$$
Since $a_1\leqslant a_2$, it follows that $h(a_1,a_2,b)<0$ unless $b=a_1=a_2=0$,
and we conclude that $\widehat{\mathfrak{DF}}(\sigma(\overline{S},\overline{L},\overline{Z}))<0$ in this case as well.
Finally, if $L$ is of $\mathbb{P}^1\times\mathbb{P}^1$-type,
let $Z=E_1$, $C_1=L_{13}$ and $C_2=L_{12}$,
so that $L\cdot Z=1-a_1$, $L\cdot C_1=1+a_1$, $L\cdot C_2=1+b+a_1$, $\overline{Z}\sim f+e$
and $\overline{L}\sim_{\mathbb{Q}}(2+b+a_1+a_2)e+(3+b+a_1)f$,
so that
$$
\sigma(\overline{S},\overline{L},\overline{Z})=2+b+a_1+a_2>L\cdot C_2\geqslant L\cdot C_1,
$$
which implies that
$$
\widehat{\mathfrak{DF}}\big(\sigma(\overline{S},\overline{L},\overline{Z})\big)=\frac{2(a_1-a_2)h_3(a_1,a_2,b)}{3(6+4b+2(a_1+a_2)-a_1^2-a_2^2)},
$$
where $h_3(a_1,a_2,b)$ is the polynomial
\begin{multline*}
2a_1^3+4a_1^2+4a_1^2b+4a_1^2a_2+7a_1a_2b+3a_1b^2+2a_2^3+4a_2^2b+3a_2b^2+\\
+6a_1^2+4a_2^2a_1+12a_1a_2+9a_1b+6a_2^2+9a_2b+6a_1+6a_2+3b.
\end{multline*}
This shows that $\widehat{\mathfrak{DF}}(\sigma(\overline{S},\overline{L},\overline{Z}))<0$ provided $a_1\ne a_2$.
\end{example}

\begin{proof}[Proof of Theorem~\ref{theorem:degree-6}]
The assertion follows from Example~\ref{example:d-6}, Corollary~\ref{corollary:blow-up} and the fact that $S$ is a del Pezzo surface.
\end{proof}

In the remaining part of this article we will apply Corollary~\ref{corollary:flop-slope} to the pair $(S,L)$ in the case when $S$ is not toric.
To do this in a concise way, we need to prove several very explicit technical results about polarized del Pezzo surfaces.
We will do this in the next section.

\section{Seshadri constants and pseudo-effective thresholds}
\label{section:technical}

Let us use all assumptions and notations of Section~\ref{section:del-Pezzo}.
Suppose, in addition, that $K_S^2\leqslant 5$, so that $S$ is not toric.
In this case, the group $\mathrm{Aut}(S)$ is known to be finite,
and the pair $(S,-K_S)$ is $K$-stable by Tian's theorem \cite{Tian1990}.

Let $Z$ be a $(-1)$-curve in the del Pezzo surface $S$,
and let $\mu$ be a positive rational number.
In this section, we study numerical properties of the divisor $L-\mu Z$.
Since this problem depend on the scaling of $L$ in an obvious way,
we will assume, for simplicity, that that the Fujita invariant of the pair $(S,L)$ equals $1$, i.e. $\mu_L=1$.

The first threshold that controls the numerical properties of the divisor $L-\mu Z$ is the Seshadri constant $\sigma(S,L,Z)$.
In our case, it can be computed as follows:
\begin{equation}
\label{equation:sigma}
\sigma(S,L,Z)=\mathrm{min}\Bigg\{\frac{L\cdot C}{Z\cdot C}\ \Bigg\vert\ C\ \text{is a $(-1)$-curve on $S$ such that}\ C\cap Z\ne\emptyset\Bigg\}.
\end{equation}
Using this formula, one can easily compute $\sigma(S,L,Z)$.
The second threshold one can relate to the triple $(S,L,Z)$ is the pseudo-effective threshold $\tau(S,L,Z)$ defined in \eqref{equation:tau}.
Observe that $\sigma(S,L,Z)\leqslant\tau(S,L,Z)$.

\begin{remark}
\label{remark:sigma-tau-downstairs}
If $S\cong\mathbb P^1\times\mathbb P^1$ or $S\cong \mathbb P^2$, then $\tau(S,L,Z)=\sigma(S,L,Z)$.
Similarly, if $S\cong\mathbb{F}_1$ and $\sigma(S,L,Z)=\frac{L\cdot f}{Z\cdot f}$, then $\sigma(S,L,Z)= \tau(S,L,Z)$.
\end{remark}

For the simple reason of applying results of Section~\ref{section:DF-invariant} to the pair $(S,L)$,
we are mostly interested in the case when $\sigma(S,L,Z)<\mu<\tau(S,L,Z)$.
Because of this, we will always assume that
$$
\sigma(S,L,Z)\leqslant\mu\leqslant\tau(S,L,Z).
$$
Then $L-\mu Z$ is a pseudo-effective divisor. Moreover, it is not nef if $\mu>\sigma(S,L,Z)$.
Taking its Zariski decomposition (see \cite[Theorem I:2.3.19]{Lazarsfeld}), we see that there exists a birational morphism $\pi\colon S\to\overline{S}$ that contracts $k\geqslant 0$ disjoint $(-1)$-curves $C_1,\ldots,C_k$ such that
\begin{equation}
L-\mu Z\sim_{\mathbb{Q}}\pi^*\big(\overline{L}-\mu\overline{Z}\big)+\sum_{i=1}^{k}c_iC_i,
\label{eq:Zariski-decomposition}
\end{equation}
where $\overline{L}=\pi_{*}(L)$ and $\overline{Z}=\pi(Z)$,
the divisor $\overline{L}-\mu\overline{Z}$ is nef, and $c_1,\ldots,c_k$ are some positive rational numbers.
Then $\overline{S}$ is a smooth del Pezzo surface, $K_{\overline{S}}^2=K_S^2+k$ and $\overline{Z}^2=-1+k$.
Note that $k=0$ if and only if the divisor  $L-\mu Z$ is nef.
In this case, the morphism $\pi$ is an isomorphism.
If $k\geqslant 1$, from \eqref{equation:sigma}, we see that
$$
\sigma\big(\overline{S},\overline{L},\overline{Z}\big)\geqslant\mu>\frac{L\cdot C_i}{Z\cdot C_i}\geqslant\sigma(S,L,Z)
$$
for every $i$.
Without loss of generality, we may assume that $\frac{L\cdot C_i}{Z\cdot C_i}\leqslant \frac{L\cdot C_j}{Z\cdot C_j}$ for $i<j$.

\begin{remark}
\label{remark:sigma-tau-upstairs}
Let $C_1,\ldots,C_m$ be disjoint $(-1)$-curves on $S$ such that $C_i\cdot Z=1$ for every~$i$.
Let $\eta\colon S\rightarrow\widetilde{S}$ be the contraction of the curves $C_1,\ldots,C_m$, let $\widetilde{L}=\eta_*(L)$, and let $\widetilde{Z}=\eta(Z)$.
Then $\widetilde{Z}$ is smooth and
\begin{equation}
\label{eq:sigma-tau-upstairs}
L-\sigma(\widetilde S, \widetilde L, \widetilde Z)Z\sim_{\mathbb Q}\eta^*(\widetilde L-\sigma(\widetilde S, \widetilde L, \widetilde Z)\widetilde Z)+\sum_{i=1}^m\big(\sigma(\widetilde{S},\widetilde{L},\widetilde{Z})-L\cdot C_i\big)C_i.
\end{equation}
Suppose that $\sigma(\widetilde{S},\widetilde{L},\widetilde{Z})\geqslant L\cdot C_i$ for every $i$,
and suppose also that $\sigma(\widetilde{S},\widetilde{L},\widetilde{Z})=\tau(\widetilde{S},\widetilde{L},\widetilde{Z})$.
Then $\tau(S,L,Z)=\sigma(\widetilde{S},\widetilde{L},\widetilde{Z})$.
Moreover, if we also have $\sigma(\widetilde{S},\widetilde{L},\widetilde{Z})>L\cdot C_i$ for every $i$,
then \eqref{eq:sigma-tau-upstairs} is the Zariski decomposition of $L-\tau(S,L,Z)Z$.
In this case we may assume that $\eta=\pi$,
because the Zariski decomposition is unique by \cite[Theorem I:2.3.19]{Lazarsfeld}.
\end{remark}

In Section~\ref{section:unstable-del-pezzo}, we will apply Corollary~\ref{corollary:flop-slope} to the pair $(S,L)$ using the curve $Z$, the contraction $\pi\colon S\to\overline{S}$,
and a positive rational number $\lambda$ such that $0\ll\lambda<\mu$.
To do this we need the curve $\overline{Z}$ to be smooth.
This is always the case when $K_S^2\geqslant 3$, because then any two $(-1)$-curves intersect at most at one point.
However, this is an additional condition in the case $K_S^2\leqslant 2$.
Recall that we assume that $K_S^2\leqslant 5$.

\begin{lemma}
\label{lemma:lambda-E1-P2-type}
Suppose that $L$ is of $\mathbb{P}^2$-type, $Z=E_1$, $\mu=\tau(S,L,Z)$ and
\begin{equation}
\label{equation:a3}
a_3\geqslant
\left\{%
\aligned
&\frac{2}{3}\ \text{if}\ K_S^2=1,\\
&\frac{3}{5}\ \text{if}\ K_S^2=2,\\
&\frac{1}{2}\ \text{if}\ K_S^2=3,\\
&\frac{1}{3}\ \text{if}\ K_S^2=4,\\
&0\ \text{if}\ K_S^2=5.\\
\endaligned\right.\\
\end{equation}
Then $k=r-1$, we may assume that $C_1=L_{12}$, $C_2=L_{13}$, $C_3=L_{14},\ldots,C_k=L_{1r}$, and
\begin{multline*}
\sigma(S,L,Z)=L\cdot L_{12}=1+a_1+a_2\leqslant L\cdot L_{13}=1+a_1+a_3\leqslant\cdots\\
\cdots\leqslant L\cdot L_{1r}=1+a_1+a_r<2+a_1=\sigma\big(\overline{S},\overline{L},\overline{Z}\big)=\tau(S,L,Z).
\end{multline*}
If~$K_S^2$ is even, then $\overline{S}=\mathbb{F}_1$. If $K_S^2$ is odd, then $\overline{S}=\mathbb{P}^1\times\mathbb{P}^1$.
The curve $\overline{Z}$ is smooth.
\end{lemma}

\begin{proof}
Let $\eta\colon S\to \widetilde{S}$ be the contraction of $L_{12},\ldots, L_{1r}$,
let $\widetilde{L}=\eta_*(L)$, and let $\widetilde{Z}=\eta(Z)$.
Then $\widetilde{Z}^2=r-2$, and the curve $\widetilde Z$ is smooth.
Moreover, either $\widetilde{S}\cong\mathbb{F}_1$ or $\widetilde{S}\cong\mathbb{P}^1\times\mathbb{P}^1$.
In the former case, we have $\eta(E_i)\sim f$ for every $i\geqslant 2$.
Similarly, in the latter case, we may assume that $\eta(E_i)\sim f_2$ for every $i\geqslant 2$.

Suppose that $K_S^2$ is even.
Then there is a $(-1)$-curve $E\subset S$ disjoint from the curves $L_{12},\ldots, L_{1r}$,
which implies that $\widetilde{S}\cong\mathbb{F}_1$ and $\eta(E)\sim e$.
Indeed, if $K_S^2=4$, then $E=C_{12345}$.
Similarly, if $K_S^2=2$, then $E=Z_1$.
Moreover $\widetilde{Z}\cdot f=1$, from which we can deduce that
$\widetilde{Z}\sim e+\frac{r-1}{2}f$, which in turn implies that
$$
\widetilde{L}\sim_{\mathbb{Q}}(2+a_1)e+\Big(3+\frac{r-1}{2}a_1+\sum_{i=2}^{r}a_i\Big)f.
$$
Using Remark~\ref{remark:sigma-tau-downstairs} and inequality~\eqref{equation:a3}, we conclude that
$$\tau\big(\widetilde{S},\widetilde{L},\widetilde{Z}\big)=\sigma\big(\widetilde{S},\widetilde{L},\widetilde{Z}\big)=\min\Bigg\{\frac{3+\frac{r-1}{2}a_1+a_2+\cdots+a_r}{\frac{r-1}{2}},2+a_1\Bigg\}=2+a_1.
$$

Suppose now that $K_S^2$ is odd.
We claim that the surface $S$ contains an irreducible curve $C$ such that $\eta(C)\cdot\eta(E_2)=1$ and $\eta(C)\cdot\eta(C)=0$.
Indeed, if $K_S^2=1$, then $C=Z_{1,8}$.
Similarly, if $K_S^2=3$, then $C=C_{12345}$.
Finally, if $K_S^2=5$, then $C=L_{23}$.
Thus, we see that $\widetilde{S}\cong\mathbb{P}^1\times\mathbb{P}^1$.
Then $\widetilde{Z}\sim f_1+\frac{r-2}{2}f_2$, which implies
$$
\widetilde{L}\sim_{\mathbb{Q}}(2+a_1)f_1+\Big(2+\frac{r-2}{2}a_1+\sum_{i=2}^{r}a_i\Big)f_2.
$$
Using Remark~\ref{remark:sigma-tau-downstairs} and~\eqref{equation:a3}, we deduce that
$$
\tau\big(\widetilde{S},\widetilde{L},\widetilde{Z}\big)=\sigma\big(\widetilde{S},\widetilde{L},\widetilde{Z}\big)=\min\Bigg\{\frac{2+\frac{r-2}{2}a_1+a_2+\cdots+a_r}{\frac{r-2}{2}},2+a_1\Bigg\}=2+a_1.
$$

Hence, we see that $\tau(\widetilde{S},\widetilde{L},\widetilde{Z})=\sigma(\widetilde{S},\widetilde{L},\widetilde{Z})=2+a_1$ in all cases.
On the other hand, we have
\begin{equation}
\label{equation:P2-type-Zariski}
L-(2+a_1)Z\sim_{\mathbb{Q}}\eta^*\big(\widetilde{L}-(2+a_1)\widetilde{Z}\big)+\sum_{i=2}^r(1-a_i)L_{1i}.
\end{equation}
Using Remark~\ref{remark:sigma-tau-upstairs}, we see that $\mu=\tau(S,L,Z)=\sigma(\widetilde{S},\widetilde{L},\widetilde{Z})=2+a_1$,
and \eqref{equation:P2-type-Zariski} is the Zariski decomposition of the divisor $L-\mu Z$.
Since the Zariski decomposition is unique by \cite[Theorem I:2.3.19]{Lazarsfeld},
we may assume that $\eta=\pi$ and $\widetilde{S}=\overline{S}$, so that $k=r-1$.
Hence, we may also assume that $C_1=L_{12}$, $C_2=L_{13}$, $C_3=L_{14},\ldots,C_k=L_{1r}$.
Note that \eqref{equation:P2-type-Zariski} and \eqref{equation:sigma} imply that $\sigma(S,L,Z)=1+a_1+a_2$.
\end{proof}

\begin{lemma}
\label{lemma:lambda-E1-F1-type}
Suppose that $L$ is of $\mathbb{F}_1$-type, $Z=E_1$, $\mu=\tau(S,L,Z)$ and
\begin{equation}
\label{equation:a3-F1-type}
a_3\geqslant\frac{5-K_{S}^2}{6-K_{S}^2}.
\end{equation}
Then
$$
\sigma(S,L,Z)=L\cdot L_{1r}=1+a_1\leqslant L\cdot L_{1i}=1+b+a_1+a_i<2+a_1+b=\sigma\big(\overline{S},\overline{L},\overline{Z}\big)=\tau(S,L,Z)
$$
for every $i$ such that $2\leqslant i<r$.
One has $k=r-1$, $C_1=L_{1r}$, and $C_i=L_{1i}$ for $r>i\geqslant 2$.
If $K_S^2$ is even, then $\overline{S}=\mathbb{F}_1$. If $K_S^2$ is odd, then $\overline{S}=\mathbb{P}^1\times\mathbb{P}^1$.
The curve $\overline{Z}$ is smooth.
\end{lemma}

\begin{proof}
Observe first that  $2+a_1+b>L\cdot L_{1i}$ for every $i\geqslant 2$, because 
$$
L\sim_{\mathbb Q}-K_S+\sum_{i=1}^{r-1}a_iE_i+b(L_{1r}+E_1).
$$ 
Let $\eta\colon S\to\widetilde{S}$ be the contraction of the curves $L_{12},\ldots,L_{1r}$,
let $\widetilde{L}=\eta_{*}(L)$ and $\widetilde{Z}=\eta(Z)$.
Then $\widetilde{Z}$ is smooth and $\widetilde{Z}^2=r-2$.
Moreover, either $\widetilde{S}=\mathbb{F}_1$ or $\widetilde{S}=\mathbb{P}^1\times\mathbb{P}^1$.
Furthermore, if $K_S^2=9-r$ is even, then $\widetilde{Z}^2$ is odd, so that $\widetilde{S}=\mathbb{F}_1$.
Arguing as in the proof of Lemma~\ref{lemma:lambda-E1-P2-type},
we see that $\widetilde{S}=\mathbb{P}^1\times\mathbb{P}^1$ if $K_S^2$ is odd.
Let $\widetilde{E}_i=\eta(E_i)$. Then
$$
\widetilde{E}_i\cdot\widetilde{E}_j=
\left\{%
\aligned
&0\ \text{if}\ i\geqslant 2\ \text{and}\ j\geqslant 2,\\
&1\ \text{if}\ j>i=1\ \text{or}\ i>j=1,\\
&r-2\ \text{if}\ i=j=1.\\
\endaligned\right.\\
$$
Therefore, if $\widetilde{S}=\mathbb{F}_1$, then $\widetilde Z =\widetilde{E}_1\sim e+\frac{r-1}{2}f$, and $\widetilde{E}_i\sim f$ for every $i\geqslant 2$.
In this case, we have $\eta_{*}(B)\sim\widetilde{Z}$, so that
$$
\widetilde{L}\sim (2+b+a_1)e+\Big(3+\frac{r-1}{2}b+\frac{r-1}{2}a_1+a_2+\cdots+a_{r-1}\Big)f,
$$
which implies
$$
\tau\big(\widetilde{S},\widetilde{L},\widetilde{Z}\big)=\sigma\big(\widetilde{S},\widetilde{L},\widetilde{Z}\big)=\mathrm{min}\Bigg\{2+b+a_1,\frac{2+\frac{r-2}{2}b+\frac{r-2}{2}a_1+a_2+\cdots+a_{r-1}}{\frac{r-2}{2}}\Bigg\}=2+b+a_1,
$$
because of \eqref{equation:a3-F1-type} and Remark~\ref{remark:sigma-tau-downstairs}.
Similarly, if $\widetilde{S}=\mathbb{P}^1\times\mathbb{P}^1$, then we may assume that
$\widetilde{E}_i\sim f_2$ for every $i\geqslant 2$, so that $\widetilde{E}_1\sim f_1+\frac{r-2}{2}f_2$.
In this case, we have
$$
\widetilde{L}\sim (2+b+a_1)f_1+\Big(2+\frac{r-1}{2}b+\frac{r-2}{2}a_1+a_2+\cdots+a_{r-2}\Big)f_2,
$$
which implies  that $\tau(\widetilde{S},\widetilde{L},\widetilde{Z})=\sigma(\widetilde{S},\widetilde{L},\widetilde{Z})=2+b+a_1$, because of \eqref{equation:a3-F1-type} and Remark~\ref{remark:sigma-tau-downstairs}.
On the other hand, we have
\begin{equation}
\label{equation:F1-case-Zariski-decomposition}
L-(2+b+a_1)Z\sim_{\mathbb{Q}}\eta^{*}\big(\widetilde{L}-(2+b+a_1)\widetilde{Z}\big)+\sum_{i=1}^{r}(2+b+a_1-L\cdot L_{1i})L_{1i}.
\end{equation}
Using Remark~\ref{remark:sigma-tau-upstairs}, we see that $\tau(S,L,Z)=2+b+a_1$,
and \eqref{equation:F1-case-Zariski-decomposition} is the Zariski decomposition of the divisor $L-\mu Z$.
Since the Zariski decomposition is unique,
we may assume that $\eta=\pi$ and $\widetilde{S}=\overline{S}$, so that $k=r-1$, $C_1=L_{1r}$, and $C_i=L_{1i}$ for for every $i$ such that $2\leqslant i<r$.
Thus, to complete the proof of the lemma, we have to show that $\sigma(S,L,Z)=1+a_1$.
This follows easily from \eqref{equation:F1-case-Zariski-decomposition} and \eqref{equation:sigma}.
\end{proof}

\begin{lemma}
\label{lemma:lambda-E1-P1-P1-type-d-5}
Suppose that $K_S^2=5$, $L$ is of $\mathbb{P}^1\times\mathbb{P}^1$-type, $Z=E_1$, $\mu=\tau(S,L,Z)$. Then
\begin{multline*}
\sigma(S,L,Z)=L\cdot L_{14}=1+a_1\leqslant L\cdot L_{12}=1+b+a_1\leqslant\\
\leqslant L\cdot L_{13}=1+b+a_1+a_2+a_3<2+b+a_1+a_2=\sigma\big(\overline{S},\overline{L},\overline{Z}\big)=\tau(S,L,Z).
\end{multline*}
Moreover, one has $k=3$, $\overline{S}=\mathbb{P}^1\times\mathbb{P}^1$, $C_1=L_{14}$, $C_2=L_{12}$, $C_3=L_{13}$.
\end{lemma}

\begin{proof}
Recall that 
$$
L\sim_{\mathbb Q} -K_S+a_1E_1+a_2L_{24}+a_3E_3+b(L_{13}+E_1),
$$
and the only $(-1)$-curves on the surface $S$ that intersect $Z$ are the curves $L_{12}$, $L_{13}$ and $L_{14}$.
Intersecting $L$ with these curves, we see that $\sigma(S,L,Z)=L\cdot L_{14}=1+a_1$ by \eqref{equation:sigma}.

Let $\eta\colon S\to\widetilde{S}$ be the contraction of the curves $L_{12}$, $L_{13}$ and $L_{14}$.
Then $\widetilde{S}\cong\mathbb{P}^1\times\mathbb{P}^1$.
Let $\widetilde{Z}=\eta(Z)$, $\widetilde{L}_{24}=\eta(L_{24})$ and $\widetilde{E}_3=\eta(E_3)$ .
Then $\widetilde{L}_{24}\cdot\widetilde{E}_3=1$ and $\widetilde{L}_{24}^2=\widetilde{E}_3^2=0$.
Thus, we may assume that $\widetilde{L}_{24}\sim f_1$ and $\widetilde{E}_3\sim f_2$.
Since $\widetilde{Z}^2=2$, we have $\widetilde{Z}\sim f_1+f_2$.

Let $\widetilde{L}=\eta_{*}(L)$.
Since $\eta_{*}(B)\sim\widetilde{Z}$, we have
$$
\widetilde{L}\sim_{\mathbb{Q}} -K_{\widetilde{S}}+(b+a_1)\widetilde{Z}+a_2\widetilde{L}_{24}+a_3\widetilde{E}_3\sim_{\mathbb{Q}} (2+b+a_1+a_2)f_1+(2+b+a_1+a_3)f_2,
$$
so that $\sigma(\widetilde{S},\widetilde{L},\widetilde{Z})=\tau(\widetilde{S},\widetilde{L},\widetilde{Z})=2+b+a_1+a_2$.
Moreover, we have
\begin{multline}
\label{equation:d-5-P1-P1-Zariski}
L-(2+b+a_1+a_2)Z\sim_{\mathbb{Q}}\eta^{*}\big(\widetilde{L}-(2+b+a_1+a_2)\widetilde{Z}\big)+\\
+(2+b+a_1+a_2-L\cdot L_{14})L_{14}+(2+b+a_1+a_2-L\cdot L_{12})L_{12}+\\
+(2+b+a_1+a_2-L\cdot L_{13})L_{13}.
\end{multline}
Using Remark~\ref{remark:sigma-tau-upstairs}, we see that $\tau(S,L,Z)=2+b+a_1+a_2$,
and \eqref{equation:d-5-P1-P1-Zariski} is the Zariski decomposition of the divisor $L-\mu Z$, so that $k=3$.
Thus, we may assume that $\eta=\pi$, $\widetilde{S}=\overline{S}$, and also $C_1=L_{14}$, $C_2=L_{12}$ and $C_3=L_{13}$.
\end{proof}

\begin{lemma}
\label{lemma:lambda-E1-P1-P1-type-d-4}
Suppose that $K_S^2=4$, $L$ is of $\mathbb{P}^1\times\mathbb{P}^1$-type, $Z=E_1$, $\mu=\tau(S,L,Z)$. Then
\begin{multline*}
\sigma(S,L,Z)=L\cdot L_{15}=1+a_1\leqslant L\cdot L_{12}=1+b+a_1\leqslant \\
\leqslant L\cdot L_{13}=1+b+a_1+a_2+a_3\leqslant L\cdot L_{14}=1+b+a_1+a_2+a_4<\\
<\mathrm{min}\Bigg\{\frac{3}{2}+b+a_1+\frac{a_2+a_3+a_4}{2},2+b+a_1+a_2\Bigg\}=\sigma\big(\overline{S},\overline{L},\overline{Z}\big)=\tau(S,L,Z).
\end{multline*}
Moreover, one has $k\geqslant 4$, $C_1=L_{15}$, $C_2=L_{12}$, $C_3=L_{13}$, $C_4=L_{14}$ and
\begin{enumerate}[$(a)$]
\item either $k=4$, $\overline{S}=\mathbb{F}_1$, $a_3+a_4\geqslant 1+a_2$ and $\tau(S,L,Z)=2+b+a_1+a_2$,

\item or $k=5$, $C_5=C_{12345}$, $\overline{S}=\mathbb{P}^2$, $a_3+a_4<1+a_2$,
and
$$
L\cdot C_{12345}=1+b+a_1+a_3+a_4<\tau(S,L,Z)=\frac{3}{2}+b+a_1+\frac{a_2+a_3+a_4}{2}.
$$
\end{enumerate}
\end{lemma}

\begin{proof}
Recall that 
$$
L\sim_{\mathbb Q}-K_S+a_1E_1+a_2L_{25}+\sum_{i=3}^4a_iE_i+b(L_{15}+E_1).
$$ 
Observe that the only $(-1)$-curves on $S$ that intersect $Z$ are the curves $E_1$, $L_{12}$, $L_{13}$, $L_{14}$, $L_{15}$ and $C_{12345}$.
Intersecting the divisor $L$ with these curves, we get
\begin{multline*}
L\cdot L_{15}=1+a_1\leqslant L\cdot L_{12}=1+b+a_1\leqslant L\cdot L_{13}=1+b+a_1+a_2+a_3\leqslant\\
\leqslant L\cdot L_{14}=1+b+a_1+a_2+a_4\leqslant L\cdot C_{12345}=1+b+a_1+a_3+a_4,
\end{multline*}
which implies that $\sigma(S,L,Z)=1+a_1$ by \eqref{equation:sigma}.

Note that the curves $L_{12}$, $L_{13}$, $L_{14}$, $L_{15}$ and $C_{12345}$ are disjoint.
Let $\eta\colon S\to\mathbb{P}^2$ be the contraction of these curves, $\widetilde{L}=\eta_{*}(L)$, and $\widetilde{Z}=\eta(Z)$.
Then $\widetilde{Z}$ is a conic and
$$
\widetilde{L}\sim_{\mathbb{Q}} (3+2b+2a_1+a_2+a_3+a_4)\ell,
$$
so that $\sigma(\widetilde{S},\widetilde{L},\widetilde{Z})=\tau(\widetilde{S},\widetilde{L},\widetilde{Z})=\frac{3}{2}+b+a_1+\frac{a_2+a_3+a_4}{2}$.
Moreover, $a_2+1>a_3+a_4$, if and only if $\sigma(\widetilde{S},\widetilde{L},\widetilde{Z})>L\cdot C_{12345}=1+b+a_1+a_3+a_4$.
Observe that
\begin{multline}
\label{equation:d-4-P1-P1-case-Zariski-decomposition-1}
L-\sigma\big(\widetilde{S},\widetilde{L},\widetilde{Z}\big)Z\sim_{\mathbb{Q}}\sum_{i=2}^5\Big(\sigma\big(\widetilde{S},\widetilde{L},\widetilde{Z}\big)-L\cdot L_{1i}\Big)L_{1i}+\\
+\Big(\sigma\big(\widetilde{S},\widetilde{L},\widetilde{Z}\big)-L\cdot C_{12345}\Big)C_{12345}.
\end{multline}
Using Remark~\ref{remark:sigma-tau-upstairs}, we see that if $a_2+1>a_3+a_4$,
then $\tau(S,L,Z)=\tau(\widetilde{S}, \widetilde{L},\widetilde{Z})=\sigma(\widetilde{S},\widetilde{L},\widetilde{Z})$,
so that \eqref{equation:d-4-P1-P1-case-Zariski-decomposition-1} is the Zariski decomposition of the divisor $L-\mu Z$.
Hence, if $a_2+1>a_3+a_4$, then we may assume that $\eta=\pi$, $\widetilde{S}=\overline{S}$, $C_1=L_{15}$, $C_2=L_{12}$, $C_3=L_{13}$, $C_4=L_{14}$ and $C_5=C_{12345}$.

To complete the proof, we now assume that $a_2+1\leqslant a_3+a_4$.
Let $\upsilon\colon S\to\mathbb{F}_1$ be the contraction of the curves $L_{12}$, $L_{13}$, $L_{14}$ and $L_{15}$,
let $\widehat{L}=\upsilon_{*}(L)$ and $\widehat{Z}=\upsilon(Z)$. Then
\begin{multline}
\label{equation:d-4-P1-P1-case-Zariski-decomposition-2}
L-(2+b+a_1+a_2)Z\sim_{\mathbb{Q}}\upsilon^{*}\big(\widehat{L}-(2+b+a_1+a_2)\widehat{Z}\big)+\\
+(2+b+a_1+a_2-L\cdot L_{15})L_{15}+(2+b+a_1+a_2-L\cdot L_{12})L_{12}+\\
+(2+b+a_1+a_2-L\cdot L_{13})L_{13}+(2+b+a_1+a_2-L\cdot L_{14})L_{14},
\end{multline}
where the coefficients on the right hand side of \eqref{equation:d-4-P1-P1-case-Zariski-decomposition-2} are all positive.
Furthermore, $\sigma(E_i)\sim f$ if $i\geqslant 2$, $\sigma(E_2)\cdot \widehat Z =1$ and $\widehat Z^2=3$, so it follows that $\widehat{Z}\sim e+2f$. 
Similarly $\sigma(L_{25})\sim e+f$. Hence, we have 
$$
\widehat{L}\sim_{\mathbb{Q}} (2+b+a_1+a_2)e+(3+2b+2a_1+a_2+a_3+a_4)f.
$$
Since $1+a_2\leqslant a_3+a_4$, it follows from Remark~\ref{remark:sigma-tau-downstairs} that 
$$
\sigma(\widehat{S},\widehat{L},\widehat{Z})=\tau(\widehat{S},\widehat{L},\widehat{Z})=2+b+a_1+a_2.
$$
Using Remark~\ref{remark:sigma-tau-upstairs}, we see that $\tau(S,L,Z)=2+b+a_1+a_2$, 
and the Zariski decomposition of the divisor $L-\mu Z$ is given by \eqref{equation:d-4-P1-P1-case-Zariski-decomposition-2}.
Hence, we may assume that $\upsilon=\pi$ and $\widehat{S}=\overline{S}$, so that $k=4$ in this case.
Moreover, we may also assume that $C_1=L_{15}$, $C_2=L_{12}$, $C_3=L_{13}$ and $C_4=L_{14}$.
This completes the proof of the lemma.
\end{proof}

\begin{lemma}
\label{lemma:lambda-E1-P1-P1-type-d-3}
Suppose that $K_S^2=3$, $L$ is of $\mathbb{P}^1\times\mathbb{P}^1$-type, $Z=E_1$, $\mu=\tau(S,L,Z)$. Then
\begin{multline*}
\sigma(S,L,Z)=L\cdot L_{16}=1+a_1\leqslant L\cdot L_{12}=1+b+a_1\leqslant\\
\leqslant L\cdot L_{13}=1+b+a_1+a_2+a_3\leqslant L\cdot L_{14}=1+b+a_1+a_2+a_4\leqslant \\
\leqslant\min\Bigg\{1+b+a_1+\frac{a_2+a_3+a_4+a_5}{2},2+b+a_1+a_2\Bigg\}=\sigma\big(\overline{S},\overline{L},\overline{Z}\big)=\tau(S,L,Z).
\end{multline*}
Moreover, one has $k\geqslant 1$, $C_1=L_{16}$, and one of the following cases holds:
\begin{enumerate}[$(a)$]
\item if $a_1=a_2=a_3=a_4=a_5=0$, then $k=1$ and $\mu=1+b$;
\item if $a_2=a_3=a_4=a_5>0$, then $k=2$, $C_2=L_{12}$ and $\mu=1+b+a_1+2a_2$;
\item if $a_2=a_3<a_4=a_5$, then $k=3$, $C_2=L_{12}$, $C_3=L_{13}$ and $\mu=1+b+a_1+a_2+a_5$;
\item if $a_2+a_5=a_3+a_4$ and $a_4<a_5$, then $k=4$, $C_2=L_{12}$, $C_3=L_{13}$, $C_4=L_{14}$ and $\mu=1+b+a_1+a_2+a_5$;
\item if $a_2+a_5<a_3+a_4$ and $a_3+a_4+a_5<2+a_2$, then $k=5$, $C_2=L_{12}$, $C_3=L_{13}$, $C_4=L_{14}$, $C_5=L_{15}$, $\overline{S}=\mathbb{P}^1\times\mathbb{P}^1$ and $\mu=1+b+a_1+\frac{a_2+a_3+a_4+a_5}{2}$;
\item if $a_2+a_5<a_3+a_4$ and $a_3+a_4+a_5\geqslant 2+a_2$, then $k=5$, $C_2=L_{12}$, $C_3=L_{13}$, $C_4=L_{14}$, $C_5=L_{15}$, $\overline{S}=\mathbb{P}^1\times\mathbb{P}^1$ and $\mu=2+b+a_1+a_2$;
\item if $a_2+a_5>a_3+a_4$, then $k=5$, $C_2=L_{12}$, $C_3=L_{13}$, $C_4=L_{14}$, $C_5=C_{12346}$,  $\overline{S}=\mathbb{F}_1$ and $\mu=1+b+a_1+\frac{a_2+a_3+a_4+a_5}{2}$.
\end{enumerate}
\end{lemma}

\begin{proof}
Recall that 
$$
L\sim_{\mathbb Q}-K_S+a_1E_1+a_2L_{26}+\sum_{i=3}^5a_iE_i+b(L_{16}+E_1). 
$$
Observe also that the only $(-1)$-curves on $S$ that intersect $Z$ are the curves $L_{12}$, $L_{13}$, $L_{14}$, $L_{15}$, $L_{16}$, $C_{12345}$, $C_{12346}$, $C_{12356}$, $C_{12456}$ and $C_{13456}$.
Moreover, we have
\begin{multline*}
L\cdot L_{16}=1+a_1\leqslant L\cdot L_{12}=1+b+a_1\leqslant L\cdot L_{13}=1+b+a_1+a_2+a_3\leqslant\\
\leqslant L\cdot L_{14}=1+b+a_1+a_2+a_4\leqslant L\cdot C_{12346}=1+b+a_1+a_3+a_4\leqslant \\
\leqslant L\cdot C_{12356}=1+b+a_1+a_3+a_5\leqslant L\cdot C_{12456}=1+b+a_1+a_4+a_5<\\
<L\cdot C_{13456}=1+b+a_1+a_2+a_3+a_4+a_5\leqslant L\cdot C_{12345}=1+2b+a_1+a_2+a_3+a_4+a_5,
\end{multline*}
and $L\cdot L_{14}\leqslant L\cdot L_{15}=1+b+a_1+a_2+a_5\leqslant L\cdot C_{12356}$.
Then $\sigma(S,L,Z)=1+a_1$ by \eqref{equation:sigma}.

If $a_2+a_5\leqslant a_3+a_4$, let $\eta\colon S\to\widetilde{S}$ be the contraction of the curves $L_{16}$, $L_{12}$, $L_{13}$, $L_{14}$ and $L_{15}$.
Similarly, if $a_2+a_5>a_3+a_4$, let $\eta\colon S\to\widetilde{S}$ be the contraction of the curves $L_{16}$, $L_{12}$, $L_{13}$, $L_{14}$ and $C_{12346}$.
In both cases, let $\widetilde{L}=\eta_{*}(L)$ and $\widetilde{Z}=\eta(Z)$.
Similarly, denote by $\widetilde{L}_{26}$, $\widetilde{E}_3$, $\widetilde{E}_4$, $\widetilde{E}_5$
and $\widetilde{C}_{12345}$ the images on $\widetilde{S}$ of the curves $L_{26}$, $E_3$, $E_4$, $E_5$ and $C_{12345}$, respectively.
If $a_2+a_5\leqslant a_3+a_4$, then $\widetilde{Z}^{2}=4$, $\widetilde{L}_{26}^2=2$,
$\widetilde{E}_3^2=\widetilde{E}_4^2=\widetilde{E}_5^2=\widetilde{C}_{12345}^2=0$, $\widetilde Z\cdot \widetilde E_3=1$,
the curves $\widetilde{E}_3$, $\widetilde{E}_4$ and $\widetilde{E}_5$ are disjoint,
and $\widetilde{C}_{12345}\cdot\widetilde{E}_2=1$, so that $\widetilde{S}\cong\mathbb{P}^1\times\mathbb{P}^1$.
In this case, we may assume that $\widetilde{E}_3\sim\widetilde{E}_4\sim\widetilde{E}_5\sim f_2$,
which implies that  $\widetilde{Z}\sim f_1+2f_2$ and $\widetilde{L}_{26}\sim f_1+f_2$, so that
$$
\widetilde{L}\sim_{\mathbb{Q}} (2+b+a_1+a_2)f_1+(2+2b+2a_1+a_2+a_3+a_4+a_5)f_2,
$$
which in turns implies that  
$$
\sigma\big(\widetilde{S},\widetilde{L},\widetilde{Z}\big)=\tau(\widetilde{S},\widetilde{L},\widetilde{Z})=\min\Big\{2+b+a_1+a_2,1+b+a_1+\frac{a_2+a_3+a_4+a_5}{2}\Big\}.
$$
Similarly, if $a_2+a_5>a_3+a_4$, then
$\widetilde{Z}^{2}=4$, $\widetilde{E}_5^2=-1$, $\widetilde{L}_{26}^2=\widetilde{E}_3^2=\widetilde{E}_4^2=1$, and $\widetilde E_5\cdot \widetilde L_{26}=\widetilde E_5\cdot \widetilde E_3=\widetilde E_5\cdot \widetilde E_4=\widetilde Z\cdot \widetilde E_5=0$,
which implies that $\widetilde{S}\cong\mathbb{F}_1$ and $\widetilde E_5\sim e$.
In this case, we have $\widetilde{Z}\sim 2e+2f$ and
$\widetilde{L}_{26}\sim\widetilde{E}_3\sim\widetilde{E}_4\sim e+f$,
so that
$$
\widetilde{L}\sim_{\mathbb{Q}} (2+2b+2a_1+a_2+a_3+a_4+a_5)e+(3+2b+2a_1+a_2+a_3+a_4)f,
$$
which also gives $\sigma(\widetilde{S},\widetilde{L},\widetilde{Z})=\tau(\widetilde{S},\widetilde{L},\widetilde{Z})=1+b+a_1+\frac{a_2+a_3+a_4+a_5}{2}$ by Remark~\ref{remark:sigma-tau-downstairs}.

If $a_2+a_5\leqslant a_3+a_4$, then
\begin{multline}
\label{equation:d-3-P1-P1-case-Zariski-decomposition-1}
L-\sigma\big(\widetilde{S},\widetilde{L},\widetilde{Z}\big)Z\sim_{\mathbb{Q}}\eta^{*}\Big(\widetilde{L}-\sigma\big(\widetilde{S},\widetilde{L},\widetilde{Z}\big)\widetilde{Z}\Big)+\sum_{i=2}^6\Big(\sigma\big(\widetilde{S},\widetilde{L},\widetilde{Z}\big)-L\cdot L_{1i}\Big)L_{1i}.
\end{multline}
and the coefficients in front of each $L_{1i}$ in \eqref{equation:d-3-P1-P1-case-Zariski-decomposition-1} are all non-negative,
so that Remark~\ref{remark:sigma-tau-upstairs} gives $\tau(S,L,Z)=\sigma(\widetilde{S},\widetilde{L},\widetilde{Z})$.
Moreover, if $a_2+a_5<a_3+a_4$, then all coefficients in front of each $L_{1i}$ in \eqref{equation:d-3-P1-P1-case-Zariski-decomposition-1} are positive,
so that \eqref{equation:d-3-P1-P1-case-Zariski-decomposition-1} is the Zariski decomposition of the divisor $L-\mu Z$ by Remark~\ref{remark:sigma-tau-upstairs}.
In this case, we may assume that $\eta=\pi$ and $\widetilde{S}=\overline{S}$, $C_1=L_{16}$, $C_2=L_{12}$, $C_3=L_{13}$, $C_4=L_{14}$ and $C_5=L_{15}$,
which gives us the cases (e)~and~(f).
Similarly, if $a_2+a_5>a_3+a_4$, then $\tau(S,L,Z)=\sigma(\widetilde{S},\widetilde{L},\widetilde{Z})$ by Remark~\ref{remark:sigma-tau-upstairs},
since all the coefficients in front of the right hand side of
\begin{multline}
\label{equation:d-3-P1-P1-case-Zariski-decomposition-2}
L-\sigma\big(\widetilde{S},\widetilde{L},\widetilde{Z}\big)Z\sim_{\mathbb{Q}}\eta^{*}\Big(\widetilde{L}-\sigma\big(\widetilde{S},\widetilde{L},\widetilde{Z}\big)\widetilde{Z}\Big)+\Big(\sigma\big(\widetilde{S},\widetilde{L},\widetilde{Z}\big)-L\cdot L_{16}\Big)L_{16}+\\
+\sum_{i=2}^4\Big(\sigma\big(\widetilde{S},\widetilde{L},\widetilde{Z}\big)-L\cdot L_{1i}\Big)L_{1i}+\Big(\sigma\big(\widetilde{S},\widetilde{L},\widetilde{Z}\big)-L\cdot C_{12346}\Big)C_{12346},
\end{multline}
are positive.
Using Remark~\ref{remark:sigma-tau-upstairs} again, we see that if $a_2+a_5>a_3+a_4$,
then \eqref{equation:d-3-P1-P1-case-Zariski-decomposition-2} is the Zariski decomposition of the divisor $L-\mu Z$.
In this case we may assume that $\eta=\pi$, $\widetilde{S}=\overline{S}$, $C_1=L_{16}$, $C_2=L_{12}$, $C_3=L_{13}$, $C_4=L_{14}$ and $C_5=C_{12346}$,
which gives us the case~(g).

To complete the proof, we may assume that $a_2+a_5=a_3+a_4$. Then $\mu=1+b+a_1+a_2+a_5$, so that
$(L-\mu Z)\cdot L_{16}=-b-a_2-a_5<0$,
$(L-\mu Z)\cdot L_{12}=-a_2-a_5$,
$(L-\mu Z)\cdot L_{13}=a_3-a_5$ and
$(L-\mu Z)\cdot L_{14}=a_4-a_5$.
Let $\upsilon\colon S\to\widehat{S}$ be the contraction of the curve $L_{16}$,
and those curves (if any) among $L_{12}$, $L_{13}$ and $L_{14}$ that have negative intersection with $L-\mu Z$.
Let $\widehat{L}=\upsilon_{*}(L)$ and $\widehat{Z}=\upsilon(Z)$, so that
\begin{equation}
\label{equation:d-3-P1-P1-case-Zariski-decomposition-3}
L-\mu Z\sim_{\mathbb{Q}}\upsilon^{*}\big(\widehat{L}-\mu\widehat{Z}\big)+(\mu-L\cdot L_{16})L_{16}+\sum_{i=2}^4(\mu-L\cdot L_{1i})L_{1i}.
\end{equation}
By Remark~\ref{remark:sigma-tau-upstairs}, the Zariski decomposition of the divisor $L-\mu Z$ is \eqref{equation:d-3-P1-P1-case-Zariski-decomposition-3},
so that we may assume that $\eta=\pi$, $\widetilde{S}=\overline{S}$ and $C_1=L_{16}$.
If $\mu-L\cdot L_{12}=0$, then $k=1$, which is case (a).
Moreover, if $\mu-L\cdot L_{14}>0$, then $k=4$, and we may also assume that $C_2=L_{12}$, $C_3=L_{13}$ and $C_4=L_{14}$,
which is case (d).
Similarly, if $\mu-L\cdot L_{14}=0$ and $\mu-L\cdot L_{13}>0$, then $k=3$, and we may assume that $C_2=L_{12}$ and $C_3=L_{13}$,
which is case (c).
Finally, if $\mu-L\cdot L_{13}=0$, then $k=2$, and we may assume that $C_2=L_{12}$, which is case (b).
This completes the proof of the lemma.
\end{proof}

If $K_S^2=2$, let us denote the $(-1)$-curve curve $Z_1$ also as $Z_{17}$.

\begin{lemma}
\label{lemma:E1-P1-P1-type-d-1-2}
Suppose that $K_S^2\leqslant 2$, $L$ is of $\mathbb{P}^1\times\mathbb{P}^1$-type, and $Z=E_1$.
Then
\begin{multline*}
\sigma(S,L,Z)=L\cdot L_{1r}=1+a_1\leqslant L\cdot L_{12}=1+b+a_1\leqslant \\
\leqslant L\cdot L_{13}=1+b+a_1+a_2+a_3\leqslant L\cdot L_{14}=1+b+a_1+a_2+a_4\leqslant \\
\leqslant L\cdot L_{15}=1+b+a_1+a_2+a_5\leqslant L\cdot L_{16}=1+b+a_1+a_2+a_6.
\end{multline*}
Moreover, one has $L\cdot L_{14}\leqslant L\cdot C_{1234r}=1+b+a_1+a_3+a_4$.
Furthermore, one has
$$
\frac{L\cdot Z_{17}}{Z\cdot Z_{17}}=\frac{1+2b+2a_1+a_2+a_3+a_4+a_5+a_6}{2},
$$
and the curve $Z_{17}$ is disjoint from the $(-1)$-curves $L_{1r}$, $L_{12}$, $L_{13}$, $L_{14}$, $L_{15}$, $L_{16}$ and $C_{1234r}$.
Let $C$ be a $(-1)$-curve on the surface $S$ such that the curve $C$ intersects the curve~$Z$, $C$ is not one of $E_1$, $L_{1r}$, $L_{12}$, $L_{13}$, $L_{14}$, $L_{15}$, $L_{16}$, $C_{1234r}$, $Z_{17}$ and $a_3\geqslant\frac{2}{3}$.
Then
$$
\frac{L\cdot C}{Z\cdot C}\geqslant\min\Bigg\{L\cdot L_{16}, L\cdot C_{1234r}, \frac{L\cdot Z_{17}}{Z\cdot Z_{17}}\Bigg\}.
$$
\end{lemma}

\begin{proof} 
Recall that 
$$
L\sim_{\mathbb Q}-K_S+a_1E_1+a_2L_{2r}+\sum_{i=3}^{r-1}a_iE_i+b(L_{1r}+E_1),
$$
and observe that all assertions of the lemma are obvious except for the last one. Let us prove it.
Note that $C\cdot B\geqslant C\cdot Z\geqslant 1$, since $B\sim E_1+L_{1r}$ and $C\ne L_{1r}$.
Moreover, by looking at the classes of the list of $(-1)$-curves in $S$, we have $C\cdot Z\leqslant 4-K_S^2$.
Furthermore, the surface $S$ contains a unique $(-1)$-curve $Z^\prime$ such that $Z^\prime\cdot Z=4-K_S^2$.
We also have $Z+Z^\prime\sim -(3-K_S^2)K_{S}$. Indeed, if $K_S^2=2$, then $Z^\prime=Z_{17}$ and if $K_S^2=1$, then $\gamma(Z^\prime)$ is a sextic curve that has triple singular point at $\gamma(E_1)$,
and double points in the points $\gamma(E_2)$, $\gamma(E_3)$, $\gamma(E_4)$, $\gamma(E_5)$, $\gamma(E_6)$,
$\gamma(E_7)$ and $\gamma(E_8)$.

Suppose first that $C\cdot Z=1$. Then
$$
\frac{L\cdot C}{Z\cdot C}=L\cdot C\geqslant 1+b+a_1+a_{2}C\cdot L_{2r}+\sum_{i=3}^{r-1}a_iC\cdot E_i.
$$
Thus, if $C\cdot(L_{2r}+E_3+\cdots+E_{r-1}\big)\geqslant 3$,
then
$$
\frac{L\cdot C}{Z\cdot C}\geqslant 1+b+a_1+a_{2}+a_3+a_4\geqslant 2+b+a_1+a_2\geqslant 1+b+a_1+a_2+a_6=L\cdot L_{16}.
$$
Going through the list of $(-1)$-curves on $S$, we see that if $C\cdot(L_{2r}+E_3+\cdots+E_{r-1})\leqslant 2$,
then either $C=L_{17}$ and $K_S^2=1$, or $C$ is one of the curves $C_{12ijr}$ with $2<i<j<r$ and $(i,j)\ne (3,4)$.
In the former case, we have $L\cdot C\geqslant L\cdot L_{16}$.
In the latter case, we have $L\cdot C\geqslant L\cdot C_{1234r}$.

Suppose now that $C\cdot Z\geqslant 2$.
If $K_S^2=2$, then $C\cdot Z=2$ as $C\cdot Z\leqslant 4-K^2_S$ and $C=Z'=Z_{17}$.
Thus, to complete the proof, we may assume that $K_S^2=1$ and $3\geqslant C\cdot Z\geqslant 2$.
If $C\cdot Z=3$, then $C=Z^\prime\sim_{\mathbb{Q}}-2K_{S}-Z$, so that
$$
\frac{L\cdot C}{Z\cdot C}=\frac{L\cdot (-2K_S-Z)}{3}
=\frac{1+4b+3a_1+2a_2+2a_3+2a_4+2a_5+2a_6+2a_7}{3}\geqslant\frac{L\cdot Z_{17}}{Z\cdot Z_{17}}
$$
because $a_3\geqslant\frac{2}{3}$. Thus, we may assume that $C\cdot Z=2$.
If $C=Z_{1i}$ with $2\leqslant i\leqslant 8$, then
$$
L\cdot C\geqslant 1+2b+2a_1+a_2+a_3+a_4+a_5+a_6+a_7-a_i\geqslant L\cdot Z_{17}.
$$
For the other three types of classes of possible $(-1)$-curves introduced in section \ref{section:del-Pezzo} it is straight forward to see that $L\cdot C\geqslant L\cdot Z_{17}$.
\end{proof}

Now we are ready to complete this section by proving

\begin{lemma}
\label{lemma:lambda-E1-P1-P1-type-d-1-2}
Suppose that $K_S^2\leqslant 2$, $L$ is of $\mathbb{P}^1\times\mathbb{P}^1$-type, $Z=E_1$, and $a_3\geqslant\frac{2}{3}$. Then
$\sigma(S,L,Z)<\mu\leqslant\tau(S,L,Z)$, $k\geqslant 2$, $C_1=L_{1r}$, $C_2=L_{12}$,
$\mu=\sigma(\overline{S},\overline{L},\overline{Z})$ and the curve $\overline{Z}$ is smooth, where
$$\min\Bigg\{\frac{1+2b+2a_1+a_2+a_3+a_4+a_5+a_6}{2}, 1+b+a_1+a_2+a_6\Bigg\}.$$
Moreover one of the following cases holds:
\begin{enumerate}[$(a)$]
\item if $1+a_2+a_3\geqslant a_4+a_5+a_6$, then $k=2$ and $\mu=\frac{1+2b+2a_1+a_2+a_3+a_4+a_5+a_6}{2}$;

\item if $1+a_2+a_3<a_4+a_5+a_6$ and $a_3+a_5+a_6\leqslant 1+a_2+a_4$, then
$k=3$, $C_3=L_{13}$ and $\mu=\frac{1+2b+2a_1+a_2+a_3+a_4+a_5+a_6}{2}$;

\item if $a_3+a_5+a_6>1+a_2+a_4$, $a_3+a_4+a_6\leqslant 1+a_2+a_5$ and $a_2+a_5+a_6\leqslant 1+a_3+a_4$,
then $k=4$, $C_3=L_{13}$,  $C_4=L_{14}$ and $\mu=\frac{1+2b+2a_1+a_2+a_3+a_4+a_5+a_6}{2}$;

\item if $a_3+a_4+a_6>1+a_2+a_5$ and $1+a_2+a_6\geqslant a_3+a_4+a_5$,
then $k=5$, $C_3=L_{13}$, $C_4=L_{14}$, $C_5=L_{15}$ and $\mu=\frac{1+2b+2a_1+a_2+a_3+a_4+a_5+a_6}{2}$;

\item if $a_3+a_4+a_6>1+a_2+a_5$ and $1+a_2+a_6<a_3+a_4+a_5$,
then $k=5$, $C_3=L_{13}$, $C_4=L_{14}$, $C_5=L_{15}$ and $\mu=1+b+a_1+a_2+a_6$;

\item if $a_2+a_5+a_6>1+a_3+a_4$, then
$k=5$, $C_3=L_{13}$, $C_4=L_{14}$,  $C_5=C_{1234r}$ and $\mu=\frac{1+2b+2a_1+a_2+a_3+a_4+a_5+a_6}{2}$.
\end{enumerate}
\end{lemma}

\begin{proof}
Using Lemma~\ref{lemma:E1-P1-P1-type-d-1-2}, we see that
$$
\frac{L\cdot Z_{17}}{Z\cdot Z_{17}}=\frac{1+2b+2a_1+a_2+a_3+a_4+a_5+a_6}{2}\geqslant 1+b+a_1=L\cdot L_{12}\geqslant \sigma(S,L,Z),
$$
because $a_3\geqslant\frac{2}{3}$.
So that $\mu$ is the smallest number among $\frac{L\cdot Z_{17}}{Z\cdot Z_{17}}$ and $L\cdot L_{16}$,
so that $\mu>\sigma(S,L,Z)$.
We will show later that $\mu\leqslant\tau(S,L,Z)$.
Observe that  $\mu>L\cdot L_{1r}$ and $\mu>L\cdot L_{12}$.
However, we do not know whether $\mu$ is larger than
the remaining intersections $L\cdot L_{13}$, $L\cdot L_{14}$, $L\cdot L_{15}$, $L\cdot L_{16}$ and $L\cdot C_{1234r}$ or not,
because $\frac{L\cdot Z_{17}}{Z\cdot Z_{17}}$ can be small.
This explains the several cases we may have.

Suppose first that either $\mu>L\cdot L_{15}$ or $\mu>L\cdot C_{1234r}$ (or both).
Note that $\mu>L\cdot L_{15}$ if and only if $a_6>a_5$ and $a_3+a_4+a_6>1+a_2+a_5$.
Similarly, $\mu>L\cdot C_{1234r}$ if and only if $a_2+a_6>a_3+a_4$ and $a_2+a_5+a_6>1+a_3+a_4$.
In particular, we must have $a_2+a_5\ne a_3+a_4$.
If $a_2+a_5<a_3+a_4$, let $\eta\colon S\to\widetilde{S}$ be the contraction of the curves $L_{1r}$, $L_{12}$, $L_{13}$, $L_{14}$ and $L_{15}$.
Similarly, if $a_2+a_5>a_3+a_4$, let $\eta\colon S\to\widetilde{S}$ be the contraction of the curves
$L_{1r}$, $L_{12}$, $L_{13}$, $L_{14}$ and $C_{1234r}$.
Denote by $\widetilde{E}_5$, $\widetilde{E}_6$, $\widetilde{E}_7$, $\widetilde{L}_{16}$,
$\widetilde{L}_{17}$,  $\widetilde{Z}_{15}$,  $\widetilde{Z}_{16}$ and  $\widetilde{Z}_{17}$
the images on $\widetilde{S}$ of the curves $E_5$, $E_6$, $E_7$, $L_{16}$,
$L_{17}$,  $Z_{15}$,  $Z_{16}$ and $Z_{17}$, respectively.
Then $\widetilde{S}$ is a smooth del Pezzo surface and $K_{\widetilde{S}}^2=K_S^2+5$.
If $K_S^2=1$ ($K_S^2=2$, respectively) and $a_2+a_5<a_3+a_4$,
then all $(-1)$-curves on $\widetilde{S}$ are $\widetilde{E}_6$, $\widetilde{E}_7$,
$\widetilde{L}_{16}$, $\widetilde{L}_{17}$, $\widetilde{Z}_{16}$ and  $\widetilde{Z}_{17}$
($\widetilde{E}_6$, $\widetilde{L}_{16}$ and $\widetilde{Z}_{17}$, respectively).
Similarly, if $d=1$ ($d=2$, respectively) and $a_2+a_5>a_3+a_4$, then all $(-1)$-curves on $\widetilde{S}$ are
$\widetilde{E}_5$, $\widetilde{E}_6$, $\widetilde{E}_7$, $\widetilde{Z}_{15}$,  $\widetilde{Z}_{16}$ and  $\widetilde{Z}_{17}$
($\widetilde{E}_5$, $\widetilde{E}_6$, $\widetilde{Z}_{17}$, respectively).

Let $\widetilde{L}=\eta_{*}(L)$ and $\widetilde{Z}=\eta(Z)$.
Then $\widetilde{Z}$ is smooth, and $\mu=\sigma(\widetilde{S},\widetilde{L},\widetilde{Z})$.
The latter follows from the intersection of the divisor $\widetilde{L}-\mu\widetilde{Z}$ with $(-1)$-curves on $\widetilde{S}$.
For example, if $a_2+a_5>a_3+a_4$, then $\mu=\frac{1+2b+2a_1+a_2+a_3+a_4+a_5+a_6}{2}$, which implies that
$(\widetilde{L}-\mu\widetilde{Z})\cdot\widetilde{Z}_{17}=0$.
Similarly, if $a_2+a_5\leqslant a_3+a_4$ and $1+a_2+a_6<a_3+a_4+a_5$, then $\mu=1+b+a_1+a_2+a_6$, which implies that
$(\widetilde{L}-\mu\widetilde{Z})\cdot\widetilde{L}_{16}=0$. In particular $\widetilde L-\mu \widetilde Z$ is nef.
On the other hand, if $a_2+a_5< a_3+a_4$, then $L\cdot L_{15}\leqslant L\cdot C_{1234r}$, and
\begin{equation}
\label{equation:Zariski-decomposition-1}
L-\mu Z\sim_{\mathbb{Q}}\eta^{*}\big(\widetilde{L}-\mu\widetilde{Z}\big)+(\mu-L\cdot L_{1r})L_{1r}+\sum_{i=2}^5(\mu-L\cdot L_{1i})L_{1i},
\end{equation}
where $\mu-L\cdot L_{1i}>0$ for every $i\in\{2,3,4,5,r\}$,
as $\mu>L\cdot L_{15}$.
If $a_2+a_5<a_3+a_4$, then
\begin{multline}
\label{equation:Zariski-decomposition-2}
L-\mu Z\sim_{\mathbb{Q}}\eta^{*}\big(\widetilde{L}-\mu\widetilde{Z}\big)+(\mu-L\cdot L_{1r})L_{1r}+\\
+\sum_{i=2}^4(\mu-L\cdot L_{1i})L_{1i}+(\mu-L\cdot C_{1234r})C_{1234r},
\end{multline}
where $\mu-L\cdot L_{1i}>0$ for every $i\in\{2,3,4,r\}$ and $\mu-L\cdot C_{1234r}>0$.
Therefore, the  divisor $L-\mu Z$ is pseudo-effective in both cases.
In particular, we see that $\mu\leqslant\tau(S,L,Z)$.
Moreover, \eqref{equation:Zariski-decomposition-1} (respectively \eqref{equation:Zariski-decomposition-2}) is the Zariski decomposition of the divisor $L-\mu Z$ in the case when $a_2+a_5<a_3+a_4$ (respectively when $a_2+a_5>a_3+a_4$).
Since the Zariski decomposition of $L-\mu Z$ is unique,
we may assume that $\eta=\pi$ and $\widetilde{S}=\overline{S}$, so that $k=5$ in this case.
Thus, we may assume that $C_1=L_{1r}$, $C_2=L_{12}$, $C_3=L_{13}$, $C_4=L_{14}$.
If $a_2+a_5<a_3+a_4$, then $C_5=L_{15}$, so that we are either in the case (d) or in the case (e).
If $a_2+a_5>a_3+a_4$, then $C_5=C_{1234r}$, which is the case (f).
This proves the required assertion in the case when $\mu>L\cdot L_{15}$ or $\mu>L\cdot C_{1234r}$.

Now we suppose that $\mu\leqslant L\cdot L_{15}$ and $\mu\leqslant L\cdot C_{1234r}$.
The former inequality implies that $a_3+a_4+a_6\leqslant 1+a_2+a_5$, so that, in particular, $a_3+a_4+a_5\leqslant 1+a_2+a_6$.
Thus, we have
$$
\mu=\frac{1+2b+2a_1+a_2+a_3+a_4+a_5+a_6}{2}=\frac{L\cdot Z_{17}}{Z\cdot Z_{17}}.
$$
Then $(L-\mu Z)\cdot Z_{17}\geqslant 0$, $(L-\mu Z)\cdot C_{1234r}\geqslant 0$,
$(L-\mu Z)\cdot L_{15}\geqslant 0$, $(L-\mu Z)\cdot L_{16}\geqslant 0$.

Let us use the same notations as in the previous case with one exception:
now assume that $\eta\colon S\to\widetilde{S}$ is the contraction of those curves among
$L_{1r}$, $L_{12}$, $L_{13}$, $L_{14}$ that have negative intersection with $L-\mu Z$.
In particular, $\eta$ contracts $L_{1r}$ and $L_{12}$, since we already know that $(L-\mu Z)\cdot L_{1r}<0$ and $(L-\mu Z)\cdot L_{12}<0$.
We claim that $\widetilde{L}-\mu\widetilde{Z}$ is nef.
Indeed, let $\widetilde{C}$ be a $(-1)$-curve on $\widetilde{S}$,
and let $C$ be its proper transform on the surface $S$.
Then $(\widetilde{L}-\mu\widetilde{Z})\cdot\widetilde{C}=(L-\mu Z)\cdot C\geqslant 0$
by Lemma~\ref{lemma:E1-P1-P1-type-d-1-2}. This implies that $\widetilde{L}-\mu\widetilde{Z}$ is nef.
Now arguing as in the previous case, we see that we can assume that $\eta=\pi$ and $\widetilde{S}=\overline{S}$.

If $L_{1r}$ and $L_{12}$ are the only curves among $L_{1r}$, $L_{12}$, $L_{13}$, $L_{14}$ that have negative intersection with $L-\mu Z$,
then we get $k=2$, and we may assume that $C_1=L_{1r}$ and $C_2=L_{12}$.
In this case, we have $(L-\mu Z)\cdot L_{13}\geqslant 0$, which can be rewritten as $1+a_2+a_3\geqslant a_4+a_5+a_6$,
which gives us the case (a).
Similarly, if $(L-\mu Z)\cdot L_{13}<0$ and $(L-\mu Z)\cdot L_{14}\geqslant 0$,
then
$1+a_2+a_3<a_4+a_5+a_6$ and $a_3+a_5+a_6\leqslant 1+a_2+a_4$, respectively.
In this case, we have $k=3$, and  we may assume that $C_1=L_{1r}$, $C_2=L_{12}$ and $C_3=L_{13}$,
which is the case (b).
Finally, if both $(L-\mu Z)\cdot L_{13}<0$ and $(L-\mu Z)\cdot L_{14}<0$, then $a_3+a_5+a_6>1+a_2+a_4$ and $k=4$.
In this case $\eta$ contracts all $4$ curves $L_{1r}$, $L_{12}$, $L_{13}$, $L_{14}$,
so that we may assume that $C_1=L_{1r}$, $C_2=L_{12}$, $C_3=L_{13}$,  $C_4=L_{14}$,
which is the case (c). This completes the proof of the lemma.
\end{proof}

\section{Computing Donaldson--Futaki invariants}
\label{section:unstable-del-pezzo}

In this section, we will prove Theorems~\ref{theorem:nice-inequality} and \ref{theorem:main}.
Namely, let $S$ be a smooth del Pezzo surface such that $K_S^2\leqslant 5$, and let $L$ be an ample $\mathbb{Q}$-divisor on it.
We will apply the results of Section~\ref{section:DF-invariant} to the pair $(S,L)$
using Lemmas~\ref{lemma:lambda-E1-P2-type}, \ref{lemma:lambda-E1-F1-type}
\ref{lemma:lambda-E1-P1-P1-type-d-5}, \ref{lemma:lambda-E1-P1-P1-type-d-4}, \ref{lemma:lambda-E1-P1-P1-type-d-3}, \ref{lemma:lambda-E1-P1-P1-type-d-1-2}.
To do this, let us use notations and assumptions of Sections~\ref{section:DF-invariant}, \ref{section:del-Pezzo} and \ref{section:technical}.
As usual, we may assume that $\mu_L=1$, where $\mu_L$ is the Fujita invariant of $(S,L)$.

Observe that the inequality \eqref{equation:a3} (respectively \eqref{equation:a3-F1-type}) follows from \eqref{equation:P2-a2-a1} or \eqref{equation:P2-a3-a1} (respectively \eqref{equation:F1-a2-a1}).
Similarly, the inequality $a_3\geqslant\frac{2}{3}$ follows from \eqref{equation:P1xP1-a2-a1}.
Thus, we assume that \eqref{equation:a3} holds in the case when $L$ is of $\mathbb{P}^2$-type, \eqref{equation:a3-F1-type} holds if $L$ if is of $\mathbb F_1$-type and $a_3\geqslant\frac{2}{3}$ if $L$ is of $\mathbb F_1$-type.

Let $Z=E_1$, and let $\mu$ be the number defined in Lemmas~\ref{lemma:lambda-E1-P2-type}, \ref{lemma:lambda-E1-F1-type}
\ref{lemma:lambda-E1-P1-P1-type-d-5}, \ref{lemma:lambda-E1-P1-P1-type-d-4}, \ref{lemma:lambda-E1-P1-P1-type-d-3}, \ref{lemma:lambda-E1-P1-P1-type-d-1-2}.
Then $\mu=\tau(S,L,Z)$ except the case when $K_S^2\leqslant 2$ and $L$ is a divisor of $\mathbb{P}^1\times\mathbb{P}^1$-type.
In this case, we have
$$
\mu=\min\Bigg\{\frac{1+2b+2a_1+a_2+a_3+a_4+a_5+a_6}{2}, 1+b+a_1+a_2+a_6\Bigg\},
$$
so that $\mu\leqslant\tau(S,L,Z)$ by Lemma~\ref{lemma:lambda-E1-P1-P1-type-d-1-2}.
Moreover, there exists a birational morphism
$\pi\colon S\to\overline{S}$ that contracts a disjoint union of $(-1)$-curves $C_1,\ldots,C_k$,
canonically determined by $(S,L,Z)$ and which are described in Lemmas~\ref{lemma:lambda-E1-P2-type}, \ref{lemma:lambda-E1-F1-type}
\ref{lemma:lambda-E1-P1-P1-type-d-5}, \ref{lemma:lambda-E1-P1-P1-type-d-4}, \ref{lemma:lambda-E1-P1-P1-type-d-3}, \ref{lemma:lambda-E1-P1-P1-type-d-1-2}.
In each case, we have $\mu=\sigma(\overline{S},\overline{L},\overline{Z})$,
where $\overline{L}=\pi_{*}(Z)$ and $\overline{Z}=\pi(Z)$.
Here $\sigma(\overline{S},\overline{L},\overline{Z})$ is the Seshadri constant of the pair $(\overline{S},\overline{L})$
with respect to the curve $\overline{Z}$.
Moreover, it follows from Lemmas~\ref{lemma:lambda-E1-P2-type}, \ref{lemma:lambda-E1-F1-type}
\ref{lemma:lambda-E1-P1-P1-type-d-5}, \ref{lemma:lambda-E1-P1-P1-type-d-4}, \ref{lemma:lambda-E1-P1-P1-type-d-3}, \ref{lemma:lambda-E1-P1-P1-type-d-1-2}
that the curve $\overline{Z}$ is smooth, and $L\cdot C_i<\sigma(\overline{S},\overline{L},\overline{Z})$ for every $i$.
Thus, it follows from Corollary~\ref{corollary:flop-slope} that $(S,L)$ is not $K$-stable if
$\widehat{\mathfrak{DF}}(\mu)<0$,
where $\widehat{\mathfrak{DF}}$ is the rational function defined in \eqref{equation:DF-flope-polynomial}.
The goal is to show that $\widehat{\mathfrak{DF}}(\mu)<0$
provided that the divisor $L$ satisfies the hypotheses of Theorems~\ref{theorem:nice-inequality} and \ref{theorem:main}.

To simplify computations, let $\mathfrak{D}=\frac{3}{2}\widehat{\mathfrak{DF}}(\mu)L^2$,
so that $\mathfrak{D}$ has the same sign as $\widehat{\mathfrak{DF}}$.
Using \eqref{equation:DF-flope-polynomial} and $L\cdot E_1=1-a_1$, we get
\begin{equation}
\label{equation:D}
\mathfrak{D}=-K_{S}\cdot L\Big(-\mu^3-3\mu^2(1-a_1)\Big)+3\mu^2L^2+3\mu L^2(1-a_1)-K_{S}\cdot L\Bigg(\sum_{i=1}^k \big(\mu-L\cdot C_i\big)^3\Bigg),
\end{equation}
where $k$, each $L\cdot C_i$, and $\mu=\sigma(\overline{S},\overline{L},\overline{Z})$ are given by Lemmas~\ref{lemma:lambda-E1-P2-type}, \ref{lemma:lambda-E1-F1-type}
\ref{lemma:lambda-E1-P1-P1-type-d-5}, \ref{lemma:lambda-E1-P1-P1-type-d-4}, \ref{lemma:lambda-E1-P1-P1-type-d-3}, or \ref{lemma:lambda-E1-P1-P1-type-d-1-2}.
If $L$ is of $\mathbb{F}_1$-type or $\mathbb{P}^1\times\mathbb{P}^1$-type, then
$\mathfrak{D}=\mathfrak{A}\cdot b^2+\mathfrak{B}\cdot b+\mathfrak{C}$
for some functions $\mathfrak{A}$, $\mathfrak{B}$ and $\mathfrak{C}$ that depend only on $a_1,\ldots,a_{r-1}$.
For instance, if $K_S^2=5$ and  $L$ is of $\mathbb{P}^1\times\mathbb{P}^1$-type,
then \eqref{equation:D} and Lemma~\ref{lemma:lambda-E1-P1-P1-type-d-5} imply that
$\mathfrak{D}$ is the polynomial
\begin{multline}
\label{equation:polynomial-P1xP1-d-5}
\Big(3a_1^2+3-3a_2^2-3a_3^2\Big)b^2+\\
+\Big(4a_1^3+3a_1^2a_2+3a_1^2a_3-3a_1a_2^2-3a_1a_3^2-4a_2^3-6a_2a_3^2-2a_3^3+6a_1^2-9a_2^2-9a_3^2+6a_1+3a_2+3a_3+8\Big)b+\\
+5+4a_1+4a_2+4a_3+2a_1^4-2a_2^4-a_3^4+4a_1^3-7a_2^3-2a_3^3+3a_1^2a_2a_3-3a_1a_2a_3^2+3a_1^2+\\
+3a_1a_2-6a_2^2-6a_3^2+2a_1^3a_2+2a_1^3a_3-2a_1a_2^3-a_1a_3^3+a_2^3a_3-a_2a_3^3+3a_1^2a_2+\\
+3a_1^2a_3-6a_1a_2^2-6a_1a_3^2+3a_2^2a_3-12a_2a_3^2+3a_1a_3+9a_2a_3-3a_2^2a_3^2.
\end{multline}
If $K_S^2=4$ and  $L$ is of $\mathbb{F}_1$-type, then
Lemma~\ref{lemma:lambda-E1-F1-type} implies that
$\mathfrak{A}=3a_1^2+6-3a_2^2-3a_3^2-3a_4^2$,
\begin{multline*}
\mathfrak{B}=4a_1^3+3a_1^2a_2+3a_1^2a_3+3a_1^2a_4-3a_1a_2^2-3a_1a_3^2-\\
-3a_1a_4^2-2a_2^3-2a_3^3-2a_4^3+3a_1^2-9a_2^2-9a_3^2-9a_4^2+9a_1+3a_2+3a_3+3a_4+16,
\end{multline*}
and
\begin{multline*}
\mathfrak{C}=8+8a_1+8a_2+8a_3+8a_4+2a_1^4-a_2^4-a_3^4-a_4^4+2a_1^3-a_2^3-a_3^3-a_4^3-9a_4^2-\\
-9a_3^2-9a_2^2+2a_1^3a_2+2a_1^3a_3+2a_1^3a_4-a_1a_2^3-a_1a_3^3-a_1a_4^3-a_2^3a_3-a_2^3a_4-a_2a_3^3-\\
-a_2a_4^3-a_3^3a_4-a_3a_4^3+3a_1^2a_2+3a_1^2a_3+3a_1^2a_4-6a_1a_2^2-6a_1a_3^2-6a_1a_4^2+3a_2^2a_3+3a_2^2a_4+\\
+3a_2a_3^2+3a_2a_4^2+3a_3^2a_4+3a_3a_4^2+3a_1a_2+3a_1a_3+3a_1a_4-6a_2a_3-6a_2a_4-6a_3a_4.
\end{multline*}
Similarly, if $S$ is a smooth cubic surface, $L$ is of $\mathbb{P}^1\times\mathbb{P}^1$-type,
$a_2+a_5<a_3+a_4$ and $a_3+a_4+a_5\geqslant 2+a_2$, then Lemma~\ref{lemma:lambda-E1-P1-P1-type-d-3}
gives $k=5$, $C_1=L_{16}$, $C_2=L_{12}$, $C_3=L_{13}$, $C_4=L_{14}$, $C_5=L_{15}$, $\overline{S}=\mathbb{P}^1\times\mathbb{P}^1$ and $\mu=2+b+a_1+a_2$.
In this case, we have
$L\cdot C_1=1+a_1$, $L\cdot C_2=1+b+a_1$, $L\cdot C_{3}=1+b+a_1+a_2+a_3$,
$L\cdot C_{4}=1+b+a_1+a_2+a_4$ and
$L\cdot C_{5}=1+b+a_1+a_2+a_5$, so that \eqref{equation:D} gives
$\mathfrak{A}=3a_1^2+9-3a_2^2-3a_3^2-3a_4^2-3a_5^2$,
\begin{multline*}
\mathfrak{B}=4a_1^3+3a_1^2a_2+3a_1^2a_3+3a_1^2a_4+3a_1^2a_5-3a_1a_2^2-3a_1a_3^2-\\
-3a_1a_4^2-3a_1a_5^2-4a_2^3-6a_2a_3^2-6a_2a_4^2-6a_2a_5^2-2a_3^3-2a_4^3-\\
-2a_5^3-9a_2^2-9a_3^2-9a_4^2-9a_5^2+12a_1+15a_2+3a_3+3a_4+3a_5+24,
\end{multline*}
and
\begin{multline*}
\mathfrak{C}=9+12a_1+12a_2+12a_3+12a_4+12a_5-3a_1^2-6a_2^2-12a_3^2-12a_4^2-12a_5^2-\\
-a_4^4-a_5^4-9a_2^3+3a_1^2a_3+3a_1^2a_4+3a_1^2a_5-6a_1a_2^2-6a_1a_3^2-6a_1a_4^2-6a_1a_5^2+\\
+3a_2^2a_4+3a_2^2a_5-12a_2a_3^2-12a_2a_4^2-12a_2a_5^2+3a_3^2a_4+3a_3^2a_5+3a_3a_4^2+3a_3a_5^2+\\
+3a_4a_5^2+3a_1a_3+3a_1a_4+3a_1a_5+9a_2a_3+9a_2a_4+9a_2a_5-6a_3a_4-6a_3a_5-6a_4a_5-\\
-3a_2^2a_3^2-3a_2^2a_4^2-3a_2^2a_5^2+9a_1a_2+2a_1^3a_2+2a_1^3a_3+2a_1^3a_4+2a_1^3a_5-2a_1a_2^3-\\
-a_1a_3^3-a_1a_4^3-a_1a_5^3+a_2^3a_3+a_2^3a_4+a_2^3a_5-a_2a_3^3-a_2a_4^3-a_2a_5^3-a_3^3a_4-\\
-a_3^3a_5+3a_2^2a_3-a_3a_4^3-a_3a_5^3-a_4^3a_5-a_4a_5^3-3a_1^2a_2+3a_1^2a_2a_3-3a_1a_2a_3^2-\\
-3a_1a_2a_5^2-3a_1a_2a_4^2+3a_4^2a_5+3a_1^2a_2a_5+3a_1^2a_2a_4+2a_1^4-2a_2^4-a_3^4.
\end{multline*}
We clearly see the pattern for the polynomial $\mathfrak{A}$.
Indeed, if $L$ is of $\mathbb{F}_1$-type or $\mathbb{P}^1\times\mathbb{P}^1$-type, then
$$
\mathfrak{A}=3a_1^2+3r-9-3a_2^2-\cdots-3a_{r-1}^2.
$$
Thus, if $L$ is of $\mathbb{F}_1$-type or $\mathbb{P}^1\times\mathbb{P}^1$-type, and
$a_1^2+r-3<a_2^2+\cdots+a_{r-1}^2$,
then 
$$
\mathfrak{D}(a_1,\ldots,a_{r-1},b)<0
$$ for $b\gg 0$.
This proves Theorem~\ref{theorem:nice-inequality}.

Now let us denote by $\mathfrak{D}_{\mathbb{P}^2}$ the polynomial
\begin{multline}
\label{equation:polynomial-P2}
5+2a_4^3+2a_5^3+2a_6^3+2a_7^3+2a_8^3+2a_1^4-a_2^4-a_3^4-a_4^4-a_5^4-a_6^4+\\
-a_7^4-a_8^4+3a_1a_6-6a_1a_8^2-a_2a_5^3+3a_2a_6^2+3a_3a_6^2+3a_6a_2^2+2a_4a_1^3+3a_3a_7^2+\\
+20a_6-18a_3^2-18a_4^2-18a_5^2-18a_6^2-18a_7^2-18a_8^2+3a_1a_2+3a_1a_4+3a_1a_5+\\
20a_7+3a_1a_7+3a_1a_8-a_1a_2^3-a_1a_3^3-a_1a_4^3-a_1a_5^3-a_1a_6^3-a_1a_7^3+3a_8a_6^2-\\
-a_1a_8^3-6a_1a_2^2-6a_1a_3^2-6a_1a_4^2-6a_1a_5^2-6a_1a_6^2-6a_1a_7^2-6a_2a_3+3a_8a_4^2+\\
+20a_8-6a_2a_5-6a_2a_6-6a_2a_7-6a_2a_8+2a_2a_1^3-a_2a_3^3-a_2a_4^3-a_2a_6^3+3a_8a_5^2+\\
+2a_3^3-a_2a_7^3-a_2a_8^3+3a_2a_1^2+3a_2a_3^2+3a_2a_4^2+3a_2a_5^2+3a_2a_7^2+3a_2a_8^2-\\
-6a_3a_5-6a_3a_6-6a_3a_7-6a_3a_8+2a_3a_1^3-a_3a_2^3-a_3a_4^3-a_3a_5^3-a_3a_6^3+\\
-4a_1^3+3a_1a_3-a_3a_7^3-a_3a_8^3+3a_3a_1^2+3a_3a_2^2+3a_3a_4^2+3a_3a_5^2+3a_3a_8^2-\\
-6a_3a_4-6a_5a_6-6a_4a_6-a_5a_7^3+3a_7a_2^2-a_8a_6^3+2a_7a_1^3+3a_8a_7^2+3a_5a_8^2+\\
+20a_2-6a_4a_5-6a_4a_7-6a_4a_8-a_4a_2^3-a_4a_3^3-a_4a_5^3-a_4a_6^3-a_4a_7^3+\\
+2a_2^3-a_4a_8^3+3a_4a_1^2+3a_4a_2^2+3a_4a_3^2+3a_4a_5^2+3a_4a_6^2+3a_4a_7^2+3a_4a_8^2+\\
+20a_3-6a_5a_7-6a_5a_8+2a_5a_1^3-a_5a_2^3-a_5a_3^3-a_5a_4^3-a_5a_6^3-a_5a_8^3+3a_5a_1^2-\\
-18a_2^2+3a_5a_2^2+3a_5a_3^2+3a_5a_4^2+3a_5a_6^2+3a_5a_7^2-6a_6a_7-6a_6a_8+2a_6a_1^3+\\
+20a_1+-a_6a_2^3-a_6a_3^3-a_6a_4^3-a_6a_5^3-a_6a_7^3-a_6a_8^3+3a_6a_1^2+3a_6a_3^2+3a_8a_2^2+\\
+3a_6a_4^2+3a_6a_5^2+3a_6a_7^2+3a_6a_8^2-6a_7a_8-a_7a_2^3-a_7a_3^3-a_7a_4^3+3a_8a_3^2-\\
-9a_1^2+20a_4-a_7a_5^3-a_7a_6^3-a_7a_8^3+3a_7a_1^2+3a_7a_3^2+3a_7a_4^2+3a_7a_5^2+3a_7a_6^2+\\
+20a_5-6a_2a_4+3a_7a_8^2+2a_8a_1^3-a_8a_2^3-a_8a_3^3-a_8a_4^3-a_8a_5^3-a_8a_7^3+3a_8a_1^2.\\
\end{multline}
If  $L$ is of $\mathbb{P}^2$-type, then
$\mathfrak{D}$ equals $\mathfrak{D}_{\mathbb{P}^2}$,
$\mathfrak{D}_{\mathbb{P}^2}(a_1,a_2,a_3,a_4,a_5,a_6,1)$,
$\mathfrak{D}_{\mathbb{P}^2}(a_1,a_2,a_3,a_4,a_5,1,1)$,
$\mathfrak{D}_{\mathbb{P}^2}(a_1,a_2,a_3,a_4,1,1,1)$ and
$\mathfrak{D}_{\mathbb{P}^2}(a_1,a_2,a_3,1,1,1,1)$ in the case when $K_S^2=1$,
$K_S^2=2$, $K_S^2=3$, $K_S^2=4$ and $K_S^2=5$, respectively.
This follows from \eqref{equation:D} and Lemma~\ref{lemma:lambda-E1-P2-type}.
Now, by Lemmas~\ref{lemma:Maple-P2-1} and \ref{lemma:Maple-P2-2}, we get $\mathfrak{D}<0$ when $L$ is of $\mathbb{P}^2$-type and \eqref{equation:P2-a2-a1} or \eqref{equation:P2-a3-a1} hold.

To deal with an ample $\mathbb{Q}$-divisor $L$ of $\mathbb{F}_1$-type, let us denote by $\mathfrak{D}_{\mathbb{F}_1}$ the polynomial
\begin{multline}
\label{equation:polynomial-F1}
-\Big(1+2b+\sum_{i=1}^{7}a_i\Big)\Big((2+a_1+b)^3+3(1-a_1)(2+a_1+b)^2\Big)+\\
+3\big(2+a_1+b\big)^2\Big(1+4b+2\sum_{i=1}^{7}a_i-\sum_{i=1}^{7}a_i^2\Big)+3(1-a_1)(2+a_1+b)\Big(1+4b+2\sum_{i=1}^{7}a_i-\sum_{i=1}^{7}a_i^2\Big)+\\
+\Big(1+2b+\sum_{i=1}^{7}a_i\Big)\Big((1+b)^3+(1-a_2)^3+(1-a_3)^3+(1-a_4)^3+(1-a_5)^3+(1-a_6)^3+(1-a_7)^3\Big).
\end{multline}
Then $\mathfrak{D}=\mathfrak{D}_{\mathbb{F}_1}$ in the case when $L$ is of $\mathbb{F}_1$-type and $K_S^2=1$.
Indeed, if $K_S^2=1$, then it follows from Lemma~\ref{lemma:lambda-E1-F1-type} that
$\mu=2+a_1+b$, $k=7$, $C_1=L_{18}$, $C_2=L_{12}$, $C_3=L_{13}$, $C_4=L_{14}$,
$C_5=L_{15}$, $C_6=L_{16}$ and $C_7=L_{17}$, so that
$L\cdot C_1=1+a_1$,
$L\cdot C_2=1+a_1+a_2+b$,
$L\cdot C_3=1+a_1+a_3+b$,
$L\cdot C_4=1+a_1+a_4+b$,
$L\cdot C_5=1+a_1+a_5+b$,
$L\cdot C_6=1+a_1+a_6+b$ and $L\cdot C_7=1+a_1+a_7+b$.
Thus, in this case, it follows from  \eqref{equation:D} that $\mathfrak{D}=\mathfrak{D}_{\mathbb{F}_1}$.
Similarly, one can deduce from Lemma~\ref{lemma:lambda-E1-F1-type} and \eqref{equation:D} that
$\mathfrak{D}$ equals $\mathfrak{D}_{\mathbb{F}_1}(a_1,a_2,a_3,a_4,a_5,a_6,1,b)$,
$\mathfrak{D}_{\mathbb{F}_1}(a_1,a_2,a_3,a_4,a_5,1,1,b)$,
$\mathfrak{D}_{\mathbb{F}_1}(a_1,a_2,a_3,a_4,1,1,1,b)$,
$\mathfrak{D}_{\mathbb{F}_1}(a_1,a_2,a_3,1,1,1,1,b)$ in the case when $K_S^2=2$, $K_S^2=3$, $K_S^2=4$ and $K_S^2=5$,
respectively.
Thus, it follows from  Lemma~\ref{lemma:Maple-F1} that $\mathfrak{D}<0$ in the case when $L$ is of $\mathbb{F}_1$-type
and \eqref{equation:F1-a2-a1} holds.

If $K_S^2=5$ and $L$ is of $\mathbb{P}^1\times\mathbb{P}^1$-type,
then $\mathfrak{D}$ is the polynomial \eqref{equation:polynomial-P1xP1-d-5}.
In this case, we have $\mathfrak{D}<0$ by Lemma~\ref{lemma:Maple-P1-P1-d-5} provided that \eqref{equation:P1xP1-a2-a1} holds.
Similarly, if $K_S^2=4$, $L$ is a divisor of $\mathbb{P}^1\times\mathbb{P}^1$-type, and $a_3+a_4\geqslant 1+a_2$, then
it follows from Lemma~\ref{lemma:lambda-E1-P1-P1-type-d-4} that
$\mathfrak{D}$ is given by
\begin{multline}
\label{equation:polynomial-P1xP1-d-4-a}
\Big(3a_1^2-3a_2^2-3a_3^2-3a_4^2+6\Big)b^2+4a_1^3b+3a_1^2a_2b+3a_1^2a_3b+3a_1^2a_4b-9a_2^2b-\\
-9a_3^2b-3a_1a_2^2b-3a_1a_3^2b-3a_1a_4^2b-4a_2^3b-6a_2a_3^2b-6a_2a_4^2b-2a_3^3b-2a_4^3b+3a_1^2b-\\
-9a_4^2b+9a_1b+9a_2b+3a_3b+3a_4b+16b+8+8a_1+8a_2+8a_3+8a_4+2a_1^4-2a_2^4-\\
-a_3^4-a_4^4+2a_1^3-8a_2^3-a_3^3-a_4^3-9a_4^2-9a_3^2-6a_2^2+3a_1^2a_2a_4-3a_1a_2a_4^2-3a_1a_2a_3^2+\\
-12a_2a_4^2+3a_1^2a_2a_3+2a_1^3a_2+2a_1^3a_3+2a_1^3a_4-2a_1a_2^3-a_1a_3^3-a_1a_4^3+a_2^3a_3+a_2^3a_4-a_2a_3^3-\\
-a_2a_4^3-a_3^3a_4-a_3a_4^3+3a_1^2a_3+3a_1^2a_4-6a_1a_2^2-6a_1a_3^2-6a_1a_4^2+3a_2^2a_3+3a_2^2a_4-12a_2a_3^2-\\
+3a_3^2a_4+3a_3a_4^2+3a_1a_3+3a_1a_4+9a_2a_3+9a_2a_4-6a_3a_4-3a_2^2a_3^2-3a_2^2a_4^2+6a_1a_2.
\end{multline}
If $K_S^2=4$ and $a_3+a_4\leqslant 1+a_2$, then $\mathfrak{D}$ is
\begin{multline}
\label{equation:polynomial-P1xP1-d-4-b}
\Big(3a_1^2-3a_2^2-3a_3^2-3a_4^2+6\Big)b^2+4a_1^3b+3a_1^2a_2b+3a_1^2a_3b+3a_1^2a_4b-3a_1a_2^2b-\\
-3a_1a_3^2b-3a_1a_4^2b-2a_2^3b-2a_3^3b-2a_4^3b+3a_1^2b-9a_2^2b-9a_3^2b-9a_4^2b+9a_1b+3a_2b+\\
+3a_3b+3a_4b+16b+8+8a_1+8a_2+8a_3+8a_4+3a_1^2a_2+2a_1^4-a_2^4-a_3^4-a_4^4+2a_1^3-a_2^3-\\
-a_3^3-a_4^3+3a_1a_2-9a_4^2-9a_3^2-9a_2^2-6a_1a_2^2+3a_1a_3-6a_4a_2-6a_3a_2+2a_1^3a_2+3a_3^2a_4+\\
+2a_1^3a_3+2a_1^3a_4-a_1a_2^3+3a_3a_4^2-a_1a_3^3-a_1a_4^3-a_2^3a_3-a_2^3a_4-a_2a_3^3-a_2a_4^3-a_3^3a_4-\\
-a_3a_4^3+3a_1^2a_3+3a_1a_4+3a_1^2a_4-6a_1a_3^2-6a_1a_4^2+3a_2^2a_3+3a_2^2a_4+3a_2a_3^2+3a_2a_4^2-6a_3a_4.
\end{multline}
In both cases, \eqref{equation:P1xP1-a2-a1} implies $\mathfrak{D}<0$ by Lemmas~\ref{lemma:Maple-P1-P1-d-4-a} and \ref{lemma:Maple-P1-P1-d-4-b}.

If $K_S^2=3$, $L$ is of $\mathbb{P}^1\times\mathbb{P}^1$-type, and $a_3+a_4+a_5\geqslant 2+a_2$,
then $\mathfrak{D}$ is the polynomial
\begin{multline}
\label{equation:polynomial-P1xP1-d-3-b}
-\Big(3+2b+\sum_{i=1}^{5}a_i\Big)\Big((2+b+a_1+a_2)^3+3(1-a_1)(2+b+a_1+a_2)^2\Big)+\\
+\Big(3(2+b+a_1+a_2)^2+3(1-a_1)(2+b+a_1+a_2)\Big)\Big(3+4b+2\sum_{i=1}^{5}a_i-\sum_{i=1}^{5}a_i^2\Big)+\\
+\Big(3+2b+\sum_{i=1}^{5}a_i\Big)\Big((1+b+a_2)^3+(1+a_2)^3+(1-a_3)^3+(1-a_4)^3+(1-a_5)^3\Big).\\
\end{multline}
This follows from Lemma~\ref{lemma:lambda-E1-P1-P1-type-d-3}.
Similarly, if $a_3+a_4+a_5\leqslant 2+a_2$, then $\mathfrak{D}$ is
\begin{multline}
\label{equation:polynomial-P1xP1-d-3-a}
-\frac{(3+2b+a_1+a_2+a_3+a_4+a_5)(2+2b+2a_1+a_2+a_3+a_4+a_5)^3}{8}-\\
-\frac{3(1-a_1)(3+2b+a_1+a_2+a_3+a_4+a_5)(2+2b+2a_1+a_2+a_3+a_4+a_5)^2}{4}+\\
+\frac{3(2+2b+2a_1+a_2+a_3+a_4+a_5)^2}{4}\Big(3+4b+2\sum_{i=1}^{5}a_i-\sum_{i=1}^{5}a_i^2\Big)+\\
+\frac{3(1-a_1)(2+2b+2a_1+a_2+a_3+a_4+a_5)}{2}\Big(3+4b+2\sum_{i=1}^{5}a_i-\sum_{i=1}^{5}a_i^2\Big)+\\
+\frac{(b+a_2+a_3+a_4+a_5)^3}{8}\Big(3+2b+\sum_{i=1}^{5}a_i\Big)+\frac{(a_2+a_3+a_4+a_5)^3}{8}\Big(3+2b+\sum_{i=1}^{5}a_i\Big)+\\
+\frac{(a_4+a_5-a_2-a_3)^3}{8}\Big(3+2b+\sum_{i=1}^{5}a_i\Big)+\frac{(a_3+a_5-a_2-a_4)^3}{8}\Big(3+2b+\sum_{i=1}^{5}a_i\Big)+\\
+\frac{|a_3+a_4-a_2-a_5|^3}{8}\Big(3+2b+\sum_{i=1}^{5}a_i\Big).
\end{multline}
In both cases, \eqref{equation:P1xP1-a2-a1} implies that $\mathfrak{D}<0$ by Lemmas~\ref{lemma:Maple-P1-P1-d-3-a} and \ref{lemma:Maple-P1-P1-d-3-b}.

If $K_S^2\leqslant 2$ and $L$ is of $\mathbb{P}^1\times\mathbb{P}^1$-type,
then we can derive the formulas for $\mathfrak{D}$ using \eqref{equation:D} and Lemma~\ref{lemma:lambda-E1-P1-P1-type-d-1-2}.
To present them in a compact way, let us denote by $\mathfrak{F}$ the polynomial
\begin{multline}
\label{equation:polynomial-P1xP1-d-1-a}
-\frac{1}{8}\Big(1+2b+2a_1+\sum_{i=2}^{6}a_i\Big)^3\Big(1+2b+\sum_{i=1}^{7}a_i\Big)+\\
-\frac{3}{4}(1-a_1)\Big(1+2b+2a_1+\sum_{i=2}^{6}a_i\Big)^2\Big(1+2b+\sum_{i=1}^{7}a_i\Big)+\\
+\frac{3}{4}\Big(1+2b+2a_1+\sum_{i=2}^{6}a_i\Big)^2\Big(1+4b+2\sum_{i=1}^{7}a_i-\sum_{i=1}^{7}a_i^2\Big)+\\
+\frac{3}{2}(1-a_1)\Big(1+2b+2a_1+\sum_{i=2}^{6}a_i\Big)\Big(1+4b+2\sum_{i=1}^{7}a_i-\sum_{i=1}^{7}a_i^2\Big)+\\
+\frac{1}{8}\Big(1+2b+\sum_{i=1}^{7}a_i\Big)\Big(-1+2b+\sum_{i=2}^{6}a_i\Big)^3+\frac{1}{8}\Big(1+2b+\sum_{i=1}^{7}a_i\Big)\Big(-1+\sum_{i=2}^{6}a_i\Big)^3.
\end{multline}
If $K_S^2=1$, $L$ is of $\mathbb{P}^1\times\mathbb{P}^1$-type, and $1+a_2+a_3\geqslant a_4+a_5+a_6$
then $\mathfrak{D}=\mathfrak{F}$ by Lemma~\ref{lemma:lambda-E1-P1-P1-type-d-1-2}
Similarly, if $K_S^2=1$, $1+a_2+a_3\leqslant a_4+a_5+a_6$ and $a_3+a_5+a_6\leqslant 1+a_2+a_4$,
then $\mathfrak{D}$ is the polynomial
\begin{equation}
\label{equation:polynomial-P1xP1-d-1-b}
\mathfrak{F}+\frac{1}{8}\Big(1+2b+\sum_{i=1}^{7}a_i\Big)\big(a_4+a_5+a_6-1-a_2-a_3\big)^3.
\end{equation}
Likewise, if $K_S^2=1$, $a_3+a_5+a_6\geqslant 1+a_2+a_4$, $a_3+a_4+a_6\leqslant 1+a_2+a_5$ and $a_2+a_5+a_6\leqslant 1+a_3+a_4$,
then $\mathfrak{D}$ is the polynomial
\begin{equation}
\label{equation:polynomial-P1xP1-d-1-c}
\mathfrak{F}+\frac{1}{8}\Big(1+2b+\sum_{i=1}^{7}a_i\Big)\Big(\big(a_4+a_5+a_6-1-a_2-a_3\big)^3+\big(a_3+a_5+a_6-1-a_2-a_4\big)^3\Big)^3.
\end{equation}
If $K_S^2=1$, $a_3+a_4+a_6\geqslant 1+a_2+a_5$, $1+a_2+a_6\geqslant a_3+a_4+a_5$ and $a_2+a_5\leqslant a_3+a_4$,
then $\mathfrak{D}$ is the polynomial
\begin{multline}
\label{equation:polynomial-P1xP1-d-1-d}
\mathfrak{F}+\frac{1}{8}\Big(1+2b+\sum_{i=1}^{7}a_i\Big)\big(a_4+a_5+a_6-1-a_2-a_3\big)^3+\\
+\frac{1}{8}\Big(1+2b+\sum_{i=1}^{7}a_i\Big)\big(a_3+a_5+a_6-1-a_2-a_4\big)^3+\frac{1}{8}\Big(1+2b+\sum_{i=1}^{7}a_i\Big)\big(a_3+a_4+a_6-1-a_2-a_5\big)^3.
\end{multline}
If $K_S^2=1$, $a_2+a_5+a_6\geqslant 1+a_3+a_4$ and $a_2+a_5\geqslant a_3+a_4$,
then $\mathfrak{D}$ is the polynomial
\begin{multline}
\label{equation:polynomial-P1xP1-d-1-f}
\mathfrak{F}+\frac{1}{8}\Big(1+2b+\sum_{i=1}^{7}a_i\Big)\big(a_4+a_5+a_6-1-a_2-a_3\big)^3+\\
+\frac{1}{8}\Big(1+2b+\sum_{i=1}^{7}a_i\Big)\big(a_3+a_5+a_6-1-a_2-a_4\big)^3+\frac{1}{8}\Big(1+2b+\sum_{i=1}^{7}a_i\Big)\big(a_3+a_4+a_6-1-a_2-a_5\big)^3+\\
+\frac{1}{8}\Big(1+2b+\sum_{i=1}^{7}a_i\Big)\big(a_2+a_5+a_6-1-a_3-a_4\big)^3.
\end{multline}
Finally, if $K_S^2=1$, $a_3+a_4+a_6\geqslant 1+a_2+a_5$, $1+a_2+a_6\leqslant a_3+a_4+a_5$ and $a_2+a_5\leqslant a_3+a_4$,
then $\mathfrak{D}$ is the polynomial
\begin{multline}
\label{equation:polynomial-P1xP1-d-1-e}
-\Big(1+2b+\sum_{i=1}^{7}a_i\Big)\big(4+b-2a_1+a_2+a_6\big)\big(1+b+a_1+a_2+a_6\big)^2+\\
+3\Big(\big(1+b+a_1+a_2+a_6\big)^2+\big(1-a_1\big)\big(1+b+a_1+a_2+a_6\big)\Big)\Big(1+4b+2\sum_{i=1}^{7}a_i-\sum_{i=1}^{7}a_i^2\Big)+\\
+\Big(1+2b+\sum_{i=1}^{7}a_i\Big)\Big(\big(b+a_2+a_6\big)^3+\big(a_2+a_6\big)^3+\big(a_6-a_3\big)^3+\big(a_6-a_4\big)^3+\big(a_6-a_5\big)^3\Big).\\
\end{multline}
This gives the formulas for $\mathfrak{D}$ in the case when $K_S^2=1$ and $L$ is of $\mathbb{P}^1\times\mathbb{P}^1$-type.
In these cases, if $a_2-a_1\geqslant 0.9347$, then $\mathfrak{D}<0$
by Lemmas~\ref{lemma:Maple-P1-P1-d-2-a}, \ref{lemma:Maple-P1-P1-d-2-b}, \ref{lemma:Maple-P1-P1-d-2-c},
\ref{lemma:Maple-P1-P1-d-2-d}, \ref{lemma:Maple-P1-P1-d-2-e} and \ref{lemma:Maple-P1-P1-d-2-f}.

If $K_S^2=2$ and $L$ is of $\mathbb{P}^1\times\mathbb{P}^1$-type,
then the formulas for $\mathfrak{D}$ are obtained from
\eqref{equation:polynomial-P1xP1-d-1-a}, \eqref{equation:polynomial-P1xP1-d-1-b}, \eqref{equation:polynomial-P1xP1-d-1-c},
\eqref{equation:polynomial-P1xP1-d-1-d}, \eqref{equation:polynomial-P1xP1-d-1-f} and \eqref{equation:polynomial-P1xP1-d-1-e}
by letting $a_7=1$.
In this case, if $a_2-a_1\geqslant 0.9206$, then $\mathfrak{D}<0$ by Lemmas~\ref{lemma:Maple-P1-P1-d-2-a}, \ref{lemma:Maple-P1-P1-d-2-b}, \ref{lemma:Maple-P1-P1-d-2-c},
\ref{lemma:Maple-P1-P1-d-2-d}, \ref{lemma:Maple-P1-P1-d-2-e} and \ref{lemma:Maple-P1-P1-d-2-f}.

We see that $\mathfrak{D}<0$ in the following cases: when $L$ is of $\mathbb{P}^2$-type and either \eqref{equation:P2-a2-a1} or  \eqref{equation:P2-a3-a1} holds,
when $L$ is of $\mathbb{F}_1$-type and \eqref{equation:F1-a2-a1} holds,
when $L$ is of $\mathbb{P}^1\times\mathbb{P}^1$-type and \eqref{equation:P1xP1-a2-a1} holds.
As we already explained above, this implies Theorem~\ref{theorem:main}.

\appendix

\section{Symbolic computations}
\label{section:Maple}

The proof of Theorem~\ref{theorem:main} relies on computations which use symbolic algebra packages. 
The length limitations of journals make it impractical to include such computations in original articles. 
On the other hand, the code used to perform computations is hardly ever maintained or preserved after several years, 
making it impossible to verify results decades later, letting the reader to rely on the good faith and skills of the authors. 
In reality this is hardly a new problem of the 21st century. 
Indeed, let us recall the following quote of one of the articles of Sylvester \cite{Sylvester} from 1871:
\begin{quote}
\emph{The manuscript sheets containing the original calculations [...] are  deposited in the iron safe of the Johns Hopkins University, Baltimore, where  they can be seen and examined, or copied, by any one interested in the subject.
}\end{quote}
Similarly, the online platform \emph{arXiv} allows us to preserve our computations. 
The proofs in this article ultimately require verifying that certain polynomials of degree $4$ in up to $8$ variables are negative under suitable conditions. 
The appendix in the online version of this article contains all details of the proofs of the following lemmas, 
where such positivity is claimed, while the version submitted for publication only contains the proofs of three lemmas, 
each serving as an example of the three different approaches used in the proofs.

Let $a_1$, $a_2$, $a_3$, $a_4$, $a_5$, $a_6$, $a_7$, $a_8$, $b$ be real numbers such that $0\leqslant a_1\leqslant a_2\leqslant\ldots\leqslant a_n<1$ and $b\geqslant 0$.
Let $s_1=a_{2}-a_1$, $s_2=a_{3}-a_2$, $s_3=a_{4}-a_3$, $s_4=a_{5}-a_4$, $s_5=a_{6}-a_5$, $s_6=a_{7}-a_6$ and $s_7=a_{8}-a_7$. 
For every polynomial $f$ in $\mathbb{R}[a_1,a_2,a_3,a_4,a_5,a_6,a_7,a_8,b]$,
let us denote by $\widehat{f}$ the polynomial in $\mathbb{R}[a_1,s_1,s_2,s_3,s_4,s_5,s_6,s_7,b]$ obtained from $f$ using the corresponding change of variables.

\begin{lemma}
\label{lemma:Maple-P2-1}
Let $f$ be the polynomial~\eqref{equation:polynomial-P2}.
Then the following assertions hold:
\begin{itemize}
\item $f(a_1,a_2,a_3,a_4,1,1,1,1)<0$ when $a_2-a_1\geqslant 0.6248$;
\item $f(a_1,a_2,a_3,a_4,a_5,1,1,1)<0$ when $a_2-a_1\geqslant 0.7488$;
\item $f(a_1,a_2,a_3,a_4,a_5,a_6,1,1)<0$ when $a_2-a_1\geqslant 0.8099$;
\item $f(a_1,a_2,a_3,a_4,a_5,a_6,a_7,1)<0$ when $a_2-a_1\geqslant 0.8469$;
\item $f(a_1,a_2,a_3,a_4,a_5,a_6,a_7,a_8)<0$ when $a_2-a_1\geqslant 0.8717$.
\end{itemize}
\end{lemma}

\begin{proof}
Let $f_5=f(a_1,a_2,a_3,a_4,1,1,1,1)$. Then $\widehat{f}_5(0,x,0,0)=-9x^4+12x^3-36x^2+12x+5$.
This polynomial has one positive root.
Denote it by $\gamma_5$. Then $\gamma_5\approx 0.6247798071$ and
\begin{multline*}
\widehat{f}_5(a_1,x+\gamma_{5},s_2,s_3)=-4a_1^4-26a_1^3s_2-13a_1^3s_3-39a_1^3\gamma_{5}-39a_1^3x-36a_1^2s_2^2-\\
-42a_1^2s_3x-63a_1^2x^2-20a_1s_2^3-30a_1s_2^2s_3-66a_1s_2^2\gamma_{5}-66a_1s_2^2x-24a_1s_2s_3^2-8s_2^3s_3-\\
-84a_1^2s_2x-12s_3^2\gamma_{5}^2-66a_1s_2s_3\gamma_{5}-66a_1s_2s_3x-78a_1s_2\gamma_{5}^2-78a_1s_2x^2-7a_1s_3^3-\\
-36a_1^2s_2s_3-27a_1s_3^2\gamma_{5}-27a_1s_3^2x-39a_1s_3x^2-39a_1\gamma_{5}^3-12s_3\gamma_{5}^3-39a_1x^3-4s_2^4-9s_2^2s_3^2-\\
-15a_1^2s_3^2-27s_2^2s_3x-30s_2^2x^2-5s_2s_3^3-21s_2s_3^2\gamma_{5}-21s_2s_3^2x-30s_2s_3x^2-72s_2\gamma_{5}^2x-\\
-42a_1^2s_3\gamma_{5}-21a_1s_3-63a_1x-18s_2^2-18s_2s_3-6s_3^2-s_3^4-6s_3^3\gamma_{5}-24s_2\gamma_{5}^3-2s_3^3-18s_2^3x-\\
-24s_2x^3-6s_3^3x-24s_3^2\gamma_{5} x-12s_3^2x^2-36s_3\gamma_{5}^2x-12s_3x^3-9x^4-6a_1s_3^2-3s_2s_3^2-
\end{multline*}

\begin{multline*}
-48s_2x(1-\gamma_{5})-a_1^2(24-7a_1-9\gamma_{5}-3s_3)-9xa_1^2(14\gamma_{5}-1)-6a_1^2s_2(14\gamma_{5}-1)-24s_3x(1-\gamma_{5})-\\
-63a_1^2\gamma_{5}^2-a_1s_2x(13\gamma_{5}-2)-2s_2^3(9\gamma_{5}+1)-4s_3(6\gamma_{5}-1-3\gamma_{5}^2)-a_1(63\gamma_{5}-16-18\gamma_{5}^2)-\\
-6a_1s_3x(13\gamma_{5}-2)-6s_2^2\gamma_{5}(5\gamma_{5}-2)-12s_2^2x(5\gamma_{5}-1)-3s_2^2s_3(9\gamma_{5}-1)-12s_3x^2(3\gamma_{5}-1)-\\
-3a_1s_3\gamma_{5}(13\gamma_{5}-4)-24a_1s_2(2-\gamma_{5})-8s_2(6\gamma_{5}-1-3\gamma_{5}^2)-9a_1\gamma_{5} x(13\gamma_{5}-4)-\\
-9a_1x^2(13\gamma_{5}-2)-24s_2x^2(3\gamma_{5}-1)-12s_2s_3x(5\gamma_{5}-1)-3s_2s_3\gamma_{5}(10\gamma_{5}-4)-\\
-12(3\gamma_{5}^3-3\gamma_{5}^2+6\gamma_{5}-1)x-(54\gamma_{5}^2-36\gamma_{5}+36)x^2-12(3\gamma_{5}-1)x^3.
\end{multline*}
All coefficients of this polynomial are negative.
This shows that $f_5<0$ when $a_2-a_1>\gamma_5$.
In particular, if $a_2-a_1\geqslant 0.6248$, then $f_5<0$.

Let $f_4=f(a_1,a_2,a_3,a_4,a_5,1,1,1)$. Then
$$
\widehat{f}_4(0,x,0,0,0)=-8(2x^4-4x^3+9x^2-4x-1).
$$
Denote the unique positive root of this polynomial by $\gamma_4$. Then $\gamma_4\approx 0.7487226925$ and
\begin{multline*}
\widehat{f}_4(a_1,x+\gamma_{4},s_2,s_3,s_4)=-(68\gamma_4^3-72\gamma_4^2+132\gamma_4-40)a_1-(39\gamma_4-15)s_2^3-\\
-(64\gamma_4^3-96\gamma_4^2+144\gamma_4-32)x-(16\gamma_4^3-24\gamma_4^2+36\gamma_4-8)s_4-(96\gamma_4^2-96\gamma_4+72)x^2-\\
-(32\gamma_4^3-48\gamma_4^2+72\gamma_4-16)s_3-(7\gamma_4+1)s_4^3-(15\gamma_4^2-6\gamma_4+9)s_4^2-(84\gamma_4-36)s_2s_4x-\\
-5s_3s_4^3-(72\gamma_4-24)s_3s_4x-(20\gamma_4-4)s_3^3-(108\gamma_4-30)a_1^2s_3-(126\gamma_4-54)s_2^2x-\\
-(102\gamma_4-36)a_1s_4x-(66\gamma_4-18)s_2s_3s_4-(168\gamma_4-72)s_2s_3x-(90\gamma_4-24)a_1s_2s_4-\\
-(306\gamma_4-108)a_1s_2x-(78\gamma_4-12)a_1s_3s_4-(204\gamma_4-72)a_1s_3x-(180\gamma_4-48)a_1s_2s_3-\\
-s_4^4-(63\gamma_4^2-54\gamma_4+45)s_2^2-(48\gamma_4^3-72\gamma_4^2+108\gamma_4-24)s_2-(42\gamma_4^2-36\gamma_4+30)s_2s_4-\\
-(36\gamma_4^2-24\gamma_4+24)s_3^2-(68\gamma_4-22)a_1^3-(108\gamma_4^2-60\gamma_4+60)a_1^2-(96\gamma_4^2-96\gamma_4+72)s_3x-\\
-(144\gamma_4-72)s_2x^2-(144\gamma_4^2-144\gamma_4+108)s_2x-(30\gamma_4-6)s_3^2s_4-(36\gamma_4^2-24\gamma_4+24)s_3s_4-\\
-(216\gamma_4-60)a_1^2x-(54\gamma_4-15)a_1^2s_4-(135\gamma_4-36)a_1s_2^2-(153\gamma_4^2-108\gamma_4+99)a_1s_2-\\
-78a_1s_3^2x-(162\gamma_4-45)a_1^2s_2-(78\gamma_4-12)a_1s_3^2-36s_3s_4x^2-42a_1^2s_3^2-42a_1s_2^3-\\
-21s_2s_3s_4^2-(102\gamma_4^2-72\gamma_4+66)a_1s_3-27s_2s_3^2s_4-8s_3^3s_4-9s_3^2s_4^2-68a_1^3x-108a_1^2x^2-\\
-30s_2^2s_3s_4-(51\gamma_4^2-36\gamma_4+33)a_1s_4-(204\gamma_4-72)a_1x^2-(204\gamma_4^2-144\gamma_4+132)a_1x-\\
-(78\gamma_4-30)s_2^2s_3-54a_1^2s_4x-27s_2s_4^2x-24s_3s_4^2x-30s_3^2s_4x-51a_1s_4x^2-39s_2^2s_4x-22a_1s_3^3-\\
-(39\gamma_4-15)s_2^2s_4-(66\gamma_4-18)s_2s_3^2-(84\gamma_4^2-72\gamma_4+60)s_2s_3-(27\gamma_4-3)s_2s_4^2-\\
-33a_1s_4^2\gamma_4-24s_3s_4^2\gamma_4-10a_1^4-16s_4x^3-15s_4^2x^2-7s_4^3x-12s_2^3s_4-12s_2^2s_4^2-6s_2s_4^3-\\
-17a_1^3s_4-18a_1^2s_4^2-8a_1s_4^3-90a_1s_2s_4x-78a_1s_3s_4x-66s_2s_3s_4x-72a_1s_2s_3s_4-42s_2s_4x^2-\\
-(64\gamma_4-32)x^3-42a_1s_2^2s_4-135a_1s_2^2x-(30\gamma_4-6)s_4^2x-68a_1x^3-84a_1s_2^2s_3-\\
-102a_1s_3x^2-108a_1^2s_3x-162a_1^2s_2x-84s_2s_3x^2-16x^4-9s_2^4-4s_3^4-39s_2^3x-20s_3^3x-\\
-(72\gamma_4-24)s_3^2x-(96\gamma_4-48)s_3x^2-(48\gamma_4^2-48\gamma_4+36)s_4x-(48\gamma_4-24)s_4x^2-
\end{multline*}

\begin{multline*}
-180a_1s_2s_3x-33a_1s_4^2x-42a_1^2s_3s_4-27a_1s_3s_4^2-30a_1s_2s_4^2-33a_1s_3^2s_4-48a_1^2s_2s_4-\\
-48s_2x^3-32s_3x^3-63s_2^2x^2-36s_3^2x^2-24s_2^3s_3-30s_2^2s_3^2-18s_2s_3^3-72a_1^2s_2^2-\\
-96a_1^2s_2s_3-72a_1s_2s_3^2-78s_2^2s_3x-153a_1s_2x^2-66s_2s_3^2x-34a_1^3s_3-51a_1^3s_2.
\end{multline*}
Observe that all coefficients of this polynomial are negative.
This implies that $f_4<0$ when $a_2-a_1>\gamma_4$.
In particular, we also have $f_4<0$ when $a_2-a_1\geqslant 0.7488$.

Let $f_3=f(a_1,a_2,a_3,a_4,a_5,a_6,1,1)$. Then
$$
\widehat{f}_3(0,x,0,0,0,0)=-25x^4+60x^3-120x^2+60x+9.
$$
Let $\gamma_3$ be the unique positive root of this polynomial. Then $\gamma_3\approx 0.8098960196$ and
\begin{multline*}
\widehat{f}_3(a_1,x+\gamma_{3},s_2,s_3,s_4,s_5)=-(105\gamma_3^3-150\gamma_3^2+225\gamma_3-72)a_1-\\
-(105\gamma_3-45)a_1^3-(68\gamma_3-36)s_2^3-(60\gamma_3^3-108\gamma_3^2+144\gamma_3-36)s_3-\\
-20s_2s_4^3-78s_2s_4^2x-84s_2^2s_3s_4-24s_2s_4s_5^2-30a_1s_4s_5^2-\\
-(120\gamma_3-72)s_4x^2-(36\gamma_3-12)s_5^2x-(120\gamma_3^2-144\gamma_3+96)s_4x-\\
-(180\gamma_3^2-216\gamma_3+144)s_3x-(33\gamma_3-9)s_4^2s_5-(84\gamma_3-36)s_4^2x-\\
-81a_1^2s_3^2-84s_3^2s_4x-108a_1s_2^2s_4-84a_1s_2s_4^2-90a_1s_4^2x-\\
-(30\gamma_3-6)s_3s_5^2-(96\gamma_3^2-96\gamma_3+72)s_3s_4-(48\gamma_3^2-48\gamma_3+36)s_3s_5-\\
-78a_1s_3s_4s_5-66s_2s_3s_4s_5-114a_1s_2s_5x-102a_1s_3s_5x-90a_1s_4s_5x-\\
-(72\gamma_3^2-72\gamma_3+54)s_3^2-(315\gamma_3^2-300\gamma_3+225)a_1x-(102\gamma_3-54)s_2^2s_4-\\
-48s_3s_5x^2-78a_1s_3s_4^2-45a_1s_3^2s_5-36a_1s_2s_5^2-42s_2^2s_3s_5-\\
-(42\gamma_3-18)s_3^3-(42\gamma_3^2-36\gamma_3+30)s_4s_5-(72\gamma_3-24)s_3s_4s_5-(27\gamma_3-3)s_4s_5^2-\\
-(84\gamma_3-36)s_4s_5x-(96\gamma_3-48)s_3s_5x-90s_2s_3s_5x-(153\gamma_3-81)s_2^2s_3-\\
-153s_2^2s_3x-105a_1x^3-165a_1^2x^2-198a_1^2s_3x-189a_1s_3x^2-(315\gamma_3-150)a_1x^2-\\
-(108\gamma_3-60)s_2s_5x-(204\gamma_3-72)a_1s_3s_4-108a_1^2s_3s_4-30s_3^2s_4s_5-\\
-264a_1^2s_2x-42s_4s_5x^2-(180\gamma_3-108)s_3x^2-18s_5^2x^2-16s_2^3s_5-9a_1s_5^3-\\
-(240\gamma_3^2-288\gamma_3+192)s_2x-(180\gamma_3-84)s_2s_3s_4-(108\gamma_3^2-120\gamma_3+84)s_2s_4-\\
-342a_1s_2s_3x-(100\gamma_3^3-180\gamma_3^2-40s_4x^3-(264\gamma_3-108)a_1^2s_2-\\
-8s_4^3s_5-21a_1^2s_5^2-27s_4s_5^2x-20s_5x^3-8s_5^3x-9s_4^2s_5^2-\\
-(324\gamma_3-180)s_2s_3x-78s_2s_4s_5x-(51\gamma_3-27)s_2^2s_5-30s_2s_4^2s_5-\\
-5s_4s_5^3-120a_1^2s_2^2-(39\gamma_3-6)a_1s_5^2-39s_2s_3^2s_5-(42\gamma_3^2-36\gamma_3+30)s_4^2-\\
-72s_3s_4^2x-24s_3^3s_4-54a_1^2s_3s_5-33a_1s_3s_5^2-36s_2^2s_4^2-84a_1s_2s_4s_5-
\end{multline*}

\begin{multline*}
-(192\gamma_3-96)s_3s_4x-(228\gamma_3-96)a_1s_2s_4-4s_4^4-132a_1^2s_4x-\\
-228a_1s_2s_4x-126a_1s_4x^2-102s_2^2s_4x-108s_2s_4x^2-96s_3s_4x^2-\\
-66s_2s_3s_4^2-90a_1s_3^2s_4-(72\gamma_3-24)s_3s_4^2-120a_1^2s_2s_4-78s_2s_3^2s_4-\\
-s_5^4-(108\gamma_3^2-120\gamma_3+84)s_2^2-(80\gamma_3^3-144\gamma_3^2+192\gamma_3-48)s_2-\\
-(102\gamma_3-36)a_1s_3s_5-(378\gamma_3-180)a_1s_3x-(342\gamma_3-144)a_1s_2s_3-\\
-39a_1s_5^2x-54s_2s_5x^2-63a_1s_5x^2-51s_2^2s_5x-21a_1^3s_5-\\
-(90\gamma_3-24)a_1s_4s_5-(252\gamma_3-120)a_1s_4x-(126\gamma_3-60)a_1s_5x-\\
-(165\gamma_3^2-135\gamma_3+108)a_1^2-(78\gamma_3-30)s_2s_4^2-(33\gamma_3-9)s_2s_5^2-\\
-60a_1^2s_2s_5-(54\gamma_3^2-60\gamma_3+42)s_2s_5-(240\gamma_3-144)s_2x^2-\\
-(42\gamma_3-18)s_3^2s_5-(84\gamma_3-36)s_3^2s_4-(144\gamma_3-72)s_3^2x-\\
-(198\gamma_3-81)a_1^2s_3-(132\gamma_3-54)a_1^2s_4-(66\gamma_3-27)a_1^2s_5-\\
-(228\gamma_3-96)a_1s_2^2-(90\gamma_3-42)s_2s_3s_5-(150\gamma_3^2-180\gamma_3+120)x^2-\\
-(40\gamma_3^3-72\gamma_3^2+96\gamma_3-24)s_4-(18\gamma_3^2-12\gamma_3+12)s_5^2-\\
+240\gamma_3-60)x-(330\gamma_3-135)a_1^2x-42s_4^2x^2-22s_4^3x-32s_2^3s_4-\\
-30s_3^2s_4^2-18s_3s_4^3-8s_5^3\gamma_3-42a_1^3s_4-48a_1^2s_4^2-24a_1s_4^3-\\
-204a_1s_3s_4x-180s_2s_3s_4x-192a_1s_2s_3s_4-(60\gamma_3-36)s_5x^2-\\
-(60\gamma_3^2-72\gamma_3+48)s_5x-(252\gamma_3^2-240\gamma_3+180)a_1s_2-96a_1s_2s_3s_5-\\
-66a_1^2s_5x-33s_2s_5^2x-42s_3^2s_5x-30s_3s_5^2x-33s_4^2s_5x-\\
-72s_3s_4s_5x-(22\gamma_3-6)s_4^3-(216\gamma_3-120)s_2s_4x-(78\gamma_3-30)s_2s_4s_5-\\
-135s_2s_3^2x-(20\gamma_3^3-36\gamma_3^2+48\gamma_3-12)s_5-(100\gamma_3-60)x^3-\\
-(216\gamma_3-120)s_2^2x-(90\gamma_3-24)a_1s_4^2-(153\gamma_3-54)a_1s_3^2-\\
-84a_1^3s_2-15s_2^2s_5^2-7s_2s_5^3-12s_3^3s_5-12s_3^2s_5^2-6s_3s_5^3-\\
-(189\gamma_3^2-180\gamma_3+135)a_1s_3-(126\gamma_3^2-120\gamma_3+90)a_1s_4-\\
-72a_1s_2^3-45a_1s_3^3-162a_1s_2^2s_3-180a_1^2s_2s_3-144a_1s_2s_3^2-\\
-(63\gamma_3^2-60\gamma_3+45)a_1s_5-(114\gamma_3-48)a_1s_2s_5-(504\gamma_3-240)a_1s_2x-\\
-(135\gamma_3-63)s_2s_3^2-(162\gamma_3^2-180\gamma_3+126)s_2s_3-153a_1s_3^2x-228a_1s_2^2x-\\
-162s_2s_3x^2-25x^4-16s_2^4-9s_3^4-68s_2^3x-42s_3^3x-105a_1^3x-\\
-80s_2x^3-60s_3x^3-108s_2^2x^2-72s_3^2x^2-48s_2^3s_3-63s_2^2s_3^2-39s_2s_3^3-\\
-18a_1^4-21s_3s_4s_5^2-27s_3s_4^2s_5-27s_2s_3s_5^2-36a_1s_4^2s_5-\\
-252a_1s_2x^2-48a_1^2s_4s_5-54a_1s_2^2s_5-36s_2^2s_4s_5-63a_1^3s_3.
\end{multline*}
Since all coefficients of this polynomial are negative, we have $f_3<0$ when $a_2-a_1>\gamma_3$.
Hence, if $a_2-a_1\geqslant 0.8099$, then $f_3<0$ as well.

Let $f_2=f(a_1,a_2,a_3,a_4,a_5,a_6,a_7,1)$. Then
$$
\widehat{f}_2(0,x,0,0,0,0,0)=-4(9x^4-24x^3+45x^2-24x-2).
$$
Let us denote by $\gamma_2$ the unique positive root of this polynomial.
Then $\gamma_2\approx 0.8468219906$.
Then $\widehat{f}_2(a_1,x+\gamma_{2},s_2,s_3,s_4,s_5,s_6)$ is a polynomial in $a_1$, $x$, $s_2$, $s_3$, $s_4$, $s_5$ and $s_6$,
which can be expanded as
\begin{multline*}
-120s_3s_5x^2-102a_1s_5^2x-156a_1^2s_5x-126s_2^2s_5x-132s_2s_5x^2-150a_1s_5x^2-90s_2s_5^2x-108s_3^2s_5x-\\
-108s_4s_5x^2-84s_3s_5^2x-90s_4^2s_5x-78s_4s_5^2x-4s_5^4-48s_5x^3-48s_5^2x^2-24s_5^3x-16s_3^3s_6-15s_3^2s_6^2-\\
-7s_3s_6^3-12s_4^3s_6-12s_4^2s_6^2-6s_4s_6^3-8s_5^3s_6-9s_5^2s_6^2-5s_5s_6^3-54a_1^2s_5^2-26a_1s_5^3-40s_2^3s_5-\\
-50a_1^3s_5-42s_2^2s_5^2-22s_2s_5^3-32s_3^3s_5-36s_3^2s_5^2-20s_3s_5^3-24s_4^3s_5-30s_4^2s_5^2-18s_4s_5^3-24a_1^2s_6^2-\\
-10a_1s_6^3-20s_2^3s_6-18s_2^2s_6^2-8s_2s_6^3-96s_2^2s_4s_5-72s_2s_4s_5^2-90a_1s_3s_5^2-132a_1s_2^2s_5-84s_2s_4^2s_5-\\
-78s_3s_4^2s_5-108s_2^2s_3s_5-102s_2s_3^2s_5-114a_1s_3^2s_5-84s_3^2s_4s_5-96a_1s_4^2s_5-120a_1^2s_4s_5-96a_1s_2s_5^2-\\
-144a_1^2s_2s_5-66s_3s_4s_5^2-84a_1s_4s_5^2-132a_1^2s_3s_5-(42\gamma_2-18)s_6^2x-276a_1s_2s_5x-252a_1s_3s_5x-\\
-228s_2s_3s_5x-204s_2s_4s_5x-192s_3s_4s_5x-240a_1s_2s_3s_5-216a_1s_2s_4s_5-204a_1s_3s_4s_5-\\
-(126\gamma_2-60)a_1s_3s_6-(252\gamma_2-120)a_1s_3s_5-(171\gamma_2-72)a_1s_4^2-(378\gamma_2-180)a_1s_3s_4-\\
-(120\gamma_2^3-240\gamma_2^2+300\gamma_2-80)s_2-(72\gamma_2-40)s_3^3-(330\gamma_2-210)s_2^2x-(228\gamma_2-132)s_2s_3^2-\\
-(90\gamma_2-42)s_2s_5^2-(750\gamma_2-420)a_1s_2x-(225\gamma_2^2-252\gamma_2+171)a_1s_4-(120\gamma_2^2-144\gamma_2+96)s_3s_5-\\
-(138\gamma_2-72)a_1s_2s_6-(276\gamma_2-144)a_1s_2s_5-(414\gamma_2-216)a_1s_2s_4-(375\gamma_2^2-420\gamma_2+285)a_1s_2-\\
-(60\gamma_2^2-72\gamma_2+48)s_3s_6-(102\gamma_2-36)a_1s_5^2-(288\gamma_2-192)s_3x^2-(288\gamma_2^2-384\gamma_2+240)s_3x-\\
-(252\gamma_2-120)a_1s_3^2-(72\gamma_2-48)s_6x^2-(600\gamma_2-336)a_1s_3x-78s_2s_3s_5^2-180s_2s_3s_4s_5-\\
-(252\gamma_2-156)s_2^2s_3-(144\gamma_2^3-288\gamma_2^2+360\gamma_2-96)x-(114\gamma_2-48)a_1s_4s_6-(228\gamma_2-96)a_1s_4s_5-\\
-(36\gamma_2-12)s_3s_6^2-(216\gamma_2-120)s_4s_5x-(78\gamma_2-30)s_4s_5s_6-(120\gamma_2-72)s_3s_6x-s_6^4-\\
-(150\gamma_2^2-168\gamma_2+114)a_1s_5-39s_2s_6^2x-54s_3^2s_6x-36s_3s_6^2x-45s_4^2s_6x-33s_4s_6^2x-\\
-36s_5^2s_6x-30s_5s_6^2x-75a_1s_6x^2-45a_1s_6^2x-66s_2s_6x^2-60s_3s_6x^2-54s_4s_6x^2-48s_5s_6x^2-\\
-63s_2^2s_6x-78a_1^2s_6x-42s_2s_4^2s_6-21s_4s_5s_6^2-42s_3^2s_4s_6-39a_1s_5^2s_6-33a_1s_5s_6^2-42s_2^2s_5s_6-\\
-48s_2^2s_4s_6-27s_4s_5^2s_6-66a_1s_2^2s_6-39a_1s_3s_6^2-42a_1s_2s_6^2-36a_1s_4s_6^2-66a_1^2s_3s_6-24s_3s_5s_6^2-\\
-54s_2^2s_3s_6-27s_3s_4s_6^2-30s_2s_4s_6^2-54a_1^2s_5s_6-48a_1s_4^2s_6-33s_2s_3s_6^2-39s_3s_4^2s_6-\\
-57a_1s_3^2s_6-30s_4^2s_5s_6-27s_2s_5s_6^2-36s_3^2s_5s_6-72a_1^2s_2s_6-33s_2s_5^2s_6-30s_3s_5^2s_6-51s_2s_3^2s_6-\\
-72s_2s_4s_5s_6-60a_1^2s_4s_6-(150\gamma_2-76)a_1^3-(45\gamma_2-12)a_1s_6^2-(234\gamma_2^2-234\gamma_2+168)a_1^2-\\
-(108\gamma_2^2-120\gamma_2+84)s_4s_5-(114\gamma_2-66)s_2s_3s_6-(228\gamma_2-132)s_2s_3s_5-(342\gamma_2-198)s_2s_3s_4-\\
-(450\gamma_2-252)a_1x^2-(153\gamma_2-81)s_2s_4^2-(360\gamma_2-216)s_3s_4x-(150\gamma_2^3-252\gamma_2^2+342\gamma_2-112)a_1-\\
-(150\gamma_2-84)a_1s_6x-(300\gamma_2-168)a_1s_5x-(102\gamma_2-36)a_1s_5s_6-(216\gamma_2^2-288\gamma_2+180)x^2-
\end{multline*}
\begin{multline*}
-(264\gamma_2^2-336\gamma_2+216)s_2s_3-(390\gamma_2-195)a_1^2s_2-(105\gamma_2-65)s_2^3-(450\gamma_2-252)a_1s_4x-\\
-(552\gamma_2-288)a_1s_2s_3-(102\gamma_2-54)s_2s_4s_6-(204\gamma_2-108)s_2s_4s_5-(528\gamma_2-336)s_2s_3x-\\
-(90\gamma_2-42)s_4^2s_5-(78\gamma_2-30)s_4s_5^2-(45\gamma_2-21)s_4^2s_6-(300\gamma_2^2-336\gamma_2+228)a_1s_3-\\
-(162\gamma_2-90)s_4^2x-(96\gamma_2-48)s_3s_4s_6-(39\gamma_2-15)s_2s_6^2-120a_1s_2s_3s_6-108a_1s_2s_4s_6-\\
-96a_1s_2s_5s_6-102a_1s_3s_4s_6-90a_1s_3s_5s_6-84a_1s_4s_5s_6-90s_2s_3s_4s_6-78s_2s_3s_5s_6-\\
-66s_3s_4s_5s_6-552a_1s_2s_3x-138a_1s_2s_6x-126a_1s_3s_6x-114a_1s_4s_6x-102a_1s_5s_6x-\\
-300a_1s_3x^2-114s_2s_3s_6x-(120\gamma_2^2-144\gamma_2+96)s_3^2-153a_1s_3s_4^2-198a_1^2s_3s_4-135s_2s_3s_4^2-\\
-312a_1^2s_3x-198a_1s_2^2s_4-216a_1^2s_2s_4-171a_1s_3^2s_4-162s_2^2s_3s_4-75a_1^3s_4-90a_1^2s_4^2-48a_1s_4^3-\\
-60s_2^3s_4-72s_2^2s_4^2-42s_2s_4^3-48s_3^3s_4-63s_3^2s_4^2-39s_3s_4^3-162a_1s_2s_4^2-153s_2s_3^2s_4-\\
-(468\gamma_2-234)a_1^2x-(90\gamma_2-42)s_2s_5s_6-(450\gamma_2^2-504\gamma_2+342)a_1x-(189\gamma_2-117)s_2^2s_4-\\
-228a_1s_4s_5x-(234\gamma_2-117)a_1^2s_4-(126\gamma_2-78)s_2^2s_5-(144\gamma_2^2-192\gamma_2+120)s_5x-\\
-(96\gamma_2^3-192\gamma_2^2+240\gamma_2-64)s_3-72s_4x^3-81s_4^2x^2-45s_4^3x-153s_2s_4^2x-189s_2^2s_4x-\\
-225a_1s_4x^2-198s_2s_4x^2-234a_1^2s_4x-144s_3s_4^2x-162s_3^2s_4x-180s_3s_4x^2-171a_1s_4^2x-\\
-(396\gamma_2-252)s_2s_4x-28a_1^4-(156\gamma_2-78)a_1^2s_5-(345\gamma_2-180)a_1s_2^2-(63\gamma_2-39)s_2^2s_6-\\
-342s_2s_3s_4x-414a_1s_2s_4x-378a_1s_3s_4x-360a_1s_2s_3s_4-(198\gamma_2^2-252\gamma_2+162)s_2s_4-\\
-(216\gamma_2^2-288\gamma_2+180)s_4x-(132\gamma_2^2-168\gamma_2+108)s_2s_5-24s_6x^3-21s_6^2x^2-9s_6^3x-\\
-264s_2s_3x^2-(144\gamma_2-96)x^3-(24\gamma_2^3-48\gamma_2^2+60\gamma_2-16)s_6-(360\gamma_2^2-480\gamma_2+300)s_2x-\\
-(54\gamma_2^2-60\gamma_2+42)s_4s_6-(21\gamma_2^2-18\gamma_2+15)s_6^2-(162\gamma_2-90)s_3^2s_4-(36\gamma_2-12)s_5^2s_6-\\
-(144\gamma_2-96)s_5x^2-(192\gamma_2-96)s_3s_4s_5-(9\gamma_2-1)s_6^3-25a_1^3s_6-(30\gamma_2-6)s_5s_6^2-\\
-(165\gamma_2^2-210\gamma_2+135)s_2^2-(216\gamma_2-144)s_4x^2-(264\gamma_2-168)s_2s_5x-(84\gamma_2-36)s_3s_5^2-\\
-(48\gamma_2^2-48\gamma_2+36)s_5s_6-(312\gamma_2-156)a_1^2s_3-(48\gamma_2^3-96\gamma_2^2+120\gamma_2-32)s_5-\\
-(240\gamma_2-144)s_3s_5x-(84\gamma_2-36)s_3s_5s_6-(66\gamma_2^2-84\gamma_2+54)s_2s_6-(75\gamma_2^2-84\gamma_2+57)a_1s_6-\\
-(108\gamma_2-60)s_3^2s_5-(78\gamma_2-39)a_1^2s_6-9s_4^4-(144\gamma_2-72)s_3s_4^2-(48\gamma_2^2-48\gamma_2+36)s_5^2-\\
-102s_2s_4s_6x-90s_2s_5s_6x-96s_3s_4s_6x-84s_3s_5s_6x-78s_4s_5s_6x-(33\gamma_2-9)s_4s_6^2-\\
-36x^4-25s_2^4-16s_3^4-345a_1s_2^2x-252a_1s_3^2x-390a_1^2s_2x-375a_1s_2x^2-252s_2^2s_3x-\\
-(108\gamma_2-60)s_4s_6x-(96\gamma_2-48)s_5s_6x-228s_2s_3^2x-240a_1s_2s_3^2-288a_1^2s_2s_3-264a_1s_2^2s_3-\\
-(360\gamma_2-240)s_2x^2-150a_1x^3-120s_2x^3-96s_3x^3-165s_2^2x^2-120s_3^2x^2-105s_2^3x-72s_3^3x-\\
-150a_1^3x-234a_1^2x^2-80s_2^3s_3-108s_2^2s_3^2-68s_2s_3^3-180a_1^2s_2^2-132a_1^2s_3^2-110a_1s_2^3-\\
-76a_1s_3^3-(24\gamma_2-8)s_5^3-100a_1^3s_3-125a_1^3s_2-(72\gamma_2^3-144\gamma_2^2+180\gamma_2-48)s_4-\\
-(54\gamma_2-30)s_3^2s_6-(81\gamma_2^2-90\gamma_2+63)s_4^2-(240\gamma_2-144)s_3^2x-(132\gamma_2-84)s_2s_6x-\\
-(45\gamma_2-21)s_4^3-(180\gamma_2^2-216\gamma_2+144)s_3s_4-(96\gamma_2-48)s_5^2x-(72\gamma_2^2-96\gamma_2+60)s_6x.
\end{multline*}
Since all coefficients of this polynomial are negative,
one has $f_2<0$ when $a_2-a_1>\gamma_2$.
In particular, if $a_2-a_1\geqslant 0.8469$, then $f_2<0$, because $\gamma_2<0.8469$.

To complete the proof, we have to show that $f<0$ in the case when $a_2-a_1\geqslant 0.8717$.
To do this, observe first that
$$
\widehat{f}(0,x,0,0,0,0,0,0)=-49x^4+140x^3-252x^2+140x+5.
$$
This polynomial has unique positive root.
Denote it by $\gamma_1$.
Then $\gamma_1\approx 0.871635634$.
Moreover, the polynomial $\widehat{f}_1(a_1,x+\gamma_{1},s_2,s_3,s_4,s_5,s_6,s_7)$ can be expanded as
\begin{multline*}
-378s_2s_4s_5x-132a_1s_2s_4s_7-120a_1s_2s_5s_7-108a_1s_2s_6s_7-126a_1s_3s_4s_7-\\
-75s_2^2s_7x-45s_2s_7^2x-66s_3^2s_7x-42s_3s_7^2x-57s_4^2s_7x-39s_4s_7^2x-48s_5^2s_7x-\\
-39s_6^2s_7x-33s_6s_7^2x-(45\gamma_1-21)s_2s_7^2-216s_3s_5x^2-270a_1^2s_5x-225s_2^2s_5x-\\
-234s_2s_5x^2-261a_1s_5x^2-171s_2s_5^2x-198s_3^2s_5x-198s_4s_5x^2-162s_3s_5^2x-171s_4^2s_5x-\\
-153s_4s_5^2x-9s_5^4-84s_5x^3-90s_5^2x^2-48s_5^3x-40s_3^3s_6-42s_3^2s_6^2-22s_3s_6^3-32s_4^3s_6-\\
-(450\gamma_1-255)a_1^2s_3-24s_5^3s_6-30s_5^2s_6^2-18s_5s_6^3-99a_1^2s_5^2-51a_1s_5^3-72s_2^3s_5-\\
-36s_4^2s_6^2-81s_2^2s_5^2-45s_2s_5^3-60s_3^3s_5-72s_3^2s_5^2-42s_3s_5^3-48s_4^3s_5-63s_4^2s_5^2-\\
-414s_2s_3s_5x-60a_1^2s_6^2-28a_1s_6^3-48s_2^3s_6-48s_2^2s_6^2-24s_2s_6^3-180s_2^2s_4s_5-144s_2s_4s_5^2-\\
-90a_1s_5s_6s_7-171a_1s_3s_5^2-234a_1s_2^2s_5-162s_2s_4^2s_5-153s_3s_4^2s_5-198s_2^2s_3s_5-\\
-24s_7^2x^2-10s_7^3x-207a_1s_3^2s_5-180a_1s_4^2s_5-216a_1^2s_4s_5-180a_1s_2s_5^2-153s_2s_3s_5^2-\\
-(264\gamma_1-168)s_4s_6x-252a_1^2s_2s_5-135s_3s_4s_5^2-162a_1s_4s_5^2-(60\gamma_1^2-72\gamma_1+48)s_5s_7-\\
-(42\gamma_1-18)s_3s_7^2-(114\gamma_1-66)s_2s_5s_7-486a_1s_2s_5x-450a_1s_3s_5x-414a_1s_4s_5x-\\
-360s_3s_4s_5x-432a_1s_2s_3s_5-396a_1s_2s_4s_5-378a_1s_3s_4s_5-342s_2s_3s_4s_5-\\
-84a_1^2s_2s_7-78a_1^2s_3s_7-63s_2s_3^2s_7-30s_2s_6s_7^2-42s_3^2s_6s_7-51s_3s_4^2s_7-33s_3s_4s_7^2-\\
-42s_3s_5^2s_7-30s_3s_5s_7^2-33s_3s_6^2s_7-27s_3s_6s_7^2-42s_4^2s_5s_7-36s_4^2s_6s_7-39s_4s_5^2s_7-\\
-27s_4s_5s_7^2-30s_4s_6^2s_7-24s_4s_6s_7^2-30s_5^2s_6s_7-27s_5s_6^2s_7-(252\gamma_1-156)s_2s_4s_6-\\
-4s_6^4-(504\gamma_1-360)s_2x^2-(171\gamma_1-99)s_2s_5^2-(78\gamma_1^2-108\gamma_1+66)s_2s_7-\\
-144a_1s_2s_3s_7-(396\gamma_1-252)s_4s_5x-(120\gamma_1^2-144\gamma_1+96)s_5s_6-(114\gamma_1-66)s_4^2s_6-\\
-96s_2s_4s_5s_7-84s_2s_4s_6s_7-78s_2s_5s_6s_7-(234\gamma_1^2-324\gamma_1+198)s_2s_5-(76\gamma_1-44)s_4^3-\\
-(168\gamma_1-120)s_6x^2-(153\gamma_1-81)s_4s_5^2-(72\gamma_1^2-96\gamma_1+60)s_3s_7-(870\gamma_1-540)a_1s_3x-\\
-(66\gamma_1-42)s_3^2s_7-(174\gamma_1-108)a_1s_7x-(780\gamma_1-540)s_2s_3x-(414\gamma_1-270)s_2s_3s_5-\\
-(156\gamma_1^2-216\gamma_1+132)s_2s_6-51a_1s_7^2x-78s_2s_7x^2-72s_3s_7x^2-66s_4s_7x^2-60s_5s_7x^2-\\
-87a_1s_7x^2-(180\gamma_1-108)s_5^2x-(540\gamma_1-306)a_1^2s_2-(87\gamma_1^2-108\gamma_1+69)a_1s_7-\\
-(96\gamma_1-48)s_3s_6^2-(26\gamma_1-10)s_6^3-(522\gamma_1^2-648\gamma_1+414)a_1s_2-(252\gamma_1-156)s_2s_4^2-\\
-90a_1^2s_7x-114a_1s_3s_5s_7-102a_1s_3s_6s_7-108a_1s_4s_5s_7-96a_1s_4s_6s_7-102s_2s_3s_5s_7-\\
-20s_4s_6^3-36s_5s_7^2x-87a_1^3s_5-21s_5s_6s_7^2-234a_1^2s_3s_5-39s_4s_5^3-189a_1s_5^2x-189s_2s_3^2s_5-\\
-(108\gamma_1-60)s_6^2x-(300\gamma_1-204)s_2^2s_4-(110\gamma_1-70)s_3^3-(414\gamma_1-216)a_1s_4s_5-\\
-(162\gamma_1-90)s_3s_5^2-(276\gamma_1-144)a_1s_4^2-(300\gamma_1-168)a_1s_3s_6-114s_2s_3s_4s_7-
\end{multline*}

\begin{multline*}
-84s_3s_5^2s_6-126s_2s_3^2s_6-144a_1^2s_4s_6-(312\gamma_1-216)s_2s_6x-(57\gamma_1-33)s_4^2s_7-\\
-102a_1s_5^2s_6-90a_1s_5s_6^2-108s_2^2s_5s_6-120s_2^2s_4s_6-78s_4s_5^2s_6-156a_1s_2^2s_6-\\
-90s_4s_6^2x-96s_5^2s_6x-84s_5s_6^2x-174a_1s_6x^2-114a_1s_6^2x-156s_2s_6x^2-144s_3s_6x^2-\\
-120s_5s_6x^2-150s_2^2s_6x-180a_1^2s_6x-108s_2s_4^2s_6-66s_4s_5s_6^2-108s_3^2s_4s_6-\\
-102a_1s_3s_6^2-108a_1s_2s_6^2-96a_1s_4s_6^2-156a_1^2s_3s_6-72s_3s_5s_6^2-132s_2^2s_3s_6-\\
-336a_1^2s_2s_4-(102\gamma_1-54)s_4s_5s_7-78s_2s_5s_6^2-96s_3^2s_5s_6-168a_1^2s_2s_6-\\
-138a_1s_3^2s_6-78s_3s_4s_6^2-84s_2s_4s_6^2-132a_1^2s_5s_6-120a_1s_4^2s_6-90s_2s_3s_6^2-\\
-162s_3^2s_4s_5-(810\gamma_1-480)a_1s_2s_3-(486\gamma_1-288)a_1s_2^2-(252\gamma_1-180)s_5x^2-\\
-(336\gamma_1-240)s_4x^2-(522\gamma_1-324)a_1s_5x-102s_2s_6^2x-132s_3^2s_6x-96s_3s_6^2x-114s_4^2s_6x-\\
-(552\gamma_1-360)s_2s_3s_4-(196\gamma_1^3-420\gamma_1^2+504\gamma_1-140)x-(294\gamma_1^2-420\gamma_1+252)x^2-\\
-(216\gamma_1^2-288\gamma_1+180)s_3s_5-(270\gamma_1-153)a_1^2s_5-(198\gamma_1^2-252\gamma_1+162)s_4s_5-\\
-(348\gamma_1^2-432\gamma_1+276)a_1s_4-(132\gamma_1^2-168\gamma_1+108)s_4^2-(90\gamma_1-51)a_1^2s_7-s_7^4-\\
-(51\gamma_1-18)a_1s_7^2-(420\gamma_1^2-600\gamma_1+360)s_3x-(90\gamma_1-42)s_4s_6s_7-(600\gamma_1-336)a_1s_3s_4-\\
-(252\gamma_1^2-360\gamma_1+216)s_5x-(96\gamma_1-48)s_5^2s_6-(150\gamma_1-102)s_2^3-(84\gamma_1-60)s_7x^2-\\
-(234\gamma_1^2-324\gamma_1+198)s_2^2-(648\gamma_1-384)a_1s_2s_4-(486\gamma_1-288)a_1s_2s_5-\\
-(1044\gamma_1-648)a_1s_2x-(108\gamma_1-60)s_6s_7x-(180\gamma_1^2-240\gamma_1+150)s_3^2-\\
-(288\gamma_1-192)s_3s_6x-(48\gamma_1-24)s_7^2x-(54\gamma_1^2-60\gamma_1+42)s_6s_7-(33\gamma_1-9)s_6s_7^2-\\
-(168\gamma_1^3-360\gamma_1^2+432\gamma_1-120)s_2-(576\gamma_1-384)s_3s_4x-312a_1^2s_3s_4-\\
-(375\gamma_1-210)a_1s_3^2-(144\gamma_1^2-192\gamma_1+120)s_3s_6-(609\gamma_1^2-756\gamma_1+483)a_1x-\\
-(114\gamma_1-48)a_1s_6^2-(84\gamma_1-36)s_5s_6s_7-(150\gamma_1-102)s_2^2s_6-(264\gamma_1-168)s_3^2s_4-\\
-(120\gamma_1-72)s_5s_7x-(240\gamma_1-144)s_5s_6x-(120\gamma_1-72)s_3s_4s_7-(84\gamma_1-36)s_5s_6^2-\\
-(174\gamma_1^2-216\gamma_1+138)a_1s_6-(261\gamma_1^2-324\gamma_1+207)a_1s_5-(126\gamma_1-60)a_1s_5s_7-\\
-(375\gamma_1-255)s_2^2s_3-114s_2s_5s_7x-102s_2s_6s_7x-120s_3s_4s_7x-108s_3s_5s_7x-96s_3s_6s_7x-\\
-102s_4s_5s_7x-90s_4s_6s_7x-84s_5s_6s_7x-162a_1s_2s_7x-150a_1s_3s_7x-138a_1s_4s_7x-126a_1s_5s_7x-\\
-116a_1^3s_4-138s_2s_3s_7x-126s_2s_4s_7x-(435\gamma_1^2-540\gamma_1+345)a_1s_3-(102\gamma_1-54)s_2s_6^2-\\
-(225\gamma_1-153)s_2^2s_5-(156\gamma_1-108)s_2s_7x-(348\gamma_1-216)a_1s_6x-(252\gamma_1-120)a_1s_5s_6-\\
-240a_1s_2s_5s_6-252a_1s_3s_4s_6-228a_1s_3s_5s_6-216a_1s_4s_5s_6-228s_2s_3s_4s_6-204s_2s_3s_5s_6-\\
-180s_3s_4s_5s_6-810a_1s_2s_3x-324a_1s_2s_6x-300a_1s_3s_6x-276a_1s_4s_6x-252a_1s_5s_6x-\\
-276s_2s_3s_6x-(609\gamma_1-378)a_1x^2-252a_1s_3s_4^2-228s_2s_3s_4^2-312a_1s_2^2s_4-60s_2^2s_4s_7-\\
-276a_1s_3^2s_4-264s_2^2s_3s_4-(138\gamma_1-90)s_2s_3s_7-(140\gamma_1^3-300\gamma_1^2+360\gamma_1-100)s_3-\\
-84s_4^2s_5s_6-54s_6s_7x^2-132s_4s_6x^2-29a_1^3s_7-(162\gamma_1-96)a_1s_2s_7-102s_3s_4^2s_6-\\
-(36\gamma_1-12)s_5s_7^2-114a_1s_6s_7x-(10\gamma_1-2)s_7^3-(468\gamma_1-324)s_2^2x-90s_2s_5^2s_6-\\
-144a_1^2s_4^2-80a_1s_4^3-96s_2^3s_4-120s_2^2s_4^2-72s_2s_4^3-80s_3^3s_4-108s_3^2s_4^2-68s_3s_4^3-\\
-24s_2^3s_7-21s_2^2s_7^2-9s_2s_7^3-20s_3^3s_7-18s_3^2s_7^2-8s_3s_7^3-16s_4^3s_7-15s_4^2s_7^2-7s_4s_7^3-
\end{multline*}

\begin{multline*}
-264a_1s_2s_4^2-252s_2s_3^2s_4-(276\gamma_1-180)s_2s_3s_6-(240\gamma_1-144)s_3s_4^2-(48\gamma_1-24)s_5^3-\\
-(114\gamma_1-48)a_1s_6s_7-(54\gamma_1^2-60\gamma_1+42)s_6^2-(203\gamma_1^3-378\gamma_1^2+483\gamma_1-160)a_1-\\
-(345\gamma_1-225)s_2s_3^2-(132\gamma_1^2-168\gamma_1+108)s_4s_6-(324\gamma_1-192)a_1s_2s_6-27a_1^2s_7^2-\\
-(196\gamma_1-140)x^3-(216\gamma_1-120)s_3s_5s_6-(28\gamma_1^3-60\gamma_1^2+72\gamma_1-20)s_7-(108\gamma_1-60)s_3s_5s_7-\\
-(90\gamma_1^2-108\gamma_1+72)s_5^2-(468\gamma_1-324)s_2s_5x-(240\gamma_1-144)s_3s_4s_6-(390\gamma_1^2-540\gamma_1+330)s_2s_3-\\
-12s_5^3s_7-12s_5^2s_7^2-6s_5s_7^3-8s_6^3s_7-9s_6^2s_7^2-5s_6s_7^3-112s_4x^3-132s_4^2x^2-76s_4^3x-\\
-252s_2s_4^2x-300s_2^2s_4x-348a_1s_4x^2-312s_2s_4x^2-360a_1^2s_4x-240s_3s_4^2x-264s_3^2s_4x-\\
-276a_1s_4^2x-40a_1^4-(171\gamma_1-99)s_4^2s_5-(132\gamma_1-84)s_4s_7x-552s_2s_3s_4x-648a_1s_2s_4x-\\
-600a_1s_3s_4x-576a_1s_2s_3s_4-(39\gamma_1-15)s_4s_7^2-(360\gamma_1-216)s_3s_4s_5-(132\gamma_1-84)s_3^2s_6-\\
-(168\gamma_1^2-240\gamma_1+144)s_6x-(75\gamma_1-51)s_2^2s_7-28s_7x^3-(150\gamma_1-84)a_1s_3s_7-56s_6x^3-\\
-54s_6^2x^2-26s_6^3x-(126\gamma_1-78)s_2s_4s_7-(420\gamma_1-300)s_3x^2-72a_1^2s_4s_7-48s_2^2s_6s_7-\\
-(432\gamma_1-288)s_3s_5x-45s_2s_5^2s_7-66s_2^2s_3s_7-60a_1s_4^2s_7-51a_1s_5^2s_7-39a_1s_5s_7^2-\\
-60a_1^2s_6s_7-54s_3^2s_4s_7-48a_1s_2s_7^2-39s_2s_3s_7^2-78a_1s_2^2s_7-36a_1s_6s_7^2-36s_2s_6^2s_7-\\
-42a_1s_6^2s_7-54s_2^2s_5s_7-42a_1s_4s_7^2-33s_2s_5s_7^2-36s_2s_4s_7^2-66a_1^2s_5s_7-69a_1s_3^2s_7-\\
-(48\gamma_1-24)s_5^2s_7-(180\gamma_1-102)a_1^2s_6-(203\gamma_1-115)a_1^3-(288\gamma_1^2-384\gamma_1+240)s_3s_4-\\
-288a_1s_2s_3s_6-264a_1s_2s_4s_6-(66\gamma_1^2-84\gamma_1+54)s_4s_7-(56\gamma_1^3-120\gamma_1^2+144\gamma_1-40)s_6-\\
-(112\gamma_1^3-240\gamma_1^2+288\gamma_1-80)s_4-(312\gamma_1^2-432\gamma_1+264)s_2s_4-(96\gamma_1-48)s_3s_6s_7-\\
-(696\gamma_1-432)a_1s_4x-(39\gamma_1-15)s_6^2s_7-(198\gamma_1-126)s_3^2s_5-(624\gamma_1-432)s_2s_4x-\\
-(450\gamma_1-252)a_1s_3s_5-(630\gamma_1-357)a_1^2x-(102\gamma_1-54)s_2s_6s_7-(315\gamma_1^2-357\gamma_1+240)a_1^2-\\
-(228\gamma_1-132)s_2s_5s_6-16s_4^4-90s_3s_4s_5s_7-78s_3s_4s_6s_7-72s_3s_5s_6s_7-66s_4s_5s_6s_7-\\
-228s_2s_5s_6x-240s_3s_4s_6x-216s_3s_5s_6x-204s_4s_5s_6x-(504\gamma_1^2-720\gamma_1+432)s_2x-\\
-49x^4-36s_2^4-25s_3^4-486a_1s_2^2x-375a_1s_3^2x-540a_1^2s_2x-522a_1s_2x^2-375s_2^2s_3x-435a_1s_3x^2-\\
-192s_2s_4s_5s_6-450a_1^2s_3x-390s_2s_3x^2-345s_2s_3^2x-360a_1s_2s_3^2-420a_1^2s_2s_3-390a_1s_2^2s_3-\\
-(336\gamma_1^2-480\gamma_1+288)s_4x-(138\gamma_1-72)a_1s_4s_7-203a_1x^3-168s_2x^3-140s_3x^3-234s_2^2x^2-\\
-180s_3^2x^2-150s_2^3x-110s_3^3x-203a_1^3x-315a_1^2x^2-120s_2^3s_3-165s_2^2s_3^2-105s_2s_3^3-\\
-252a_1^2s_2^2-195a_1^2s_3^2-156a_1s_2^3-115a_1s_3^3-(276\gamma_1-144)a_1s_4s_6-(144\gamma_1-96)s_3s_7x-\\
-54s_2s_4^2s_7-(84\gamma_1^3-180\gamma_1^2+216\gamma_1-60)s_5-145a_1^3s_3-174a_1^3s_2-(378\gamma_1-234)s_2s_4s_5-\\
-(264\gamma_1-168)s_4^2x-(90\gamma_1-42)s_4s_6^2-(24\gamma_1^2-24\gamma_1+18)s_7^2-(360\gamma_1-240)s_3^2x-\\
-48s_3^2s_5s_7-11a_1s_7^3-90s_2s_3s_6s_7-58a_1^3s_6-45a_1s_3s_7^2-252s_2s_4s_6x-288s_3s_4x^2-\\
-(84\gamma_1^2-120\gamma_1+72)s_7x-(189\gamma_1-90)a_1s_5^2-(204\gamma_1-108)s_4s_5s_6-(360\gamma_1-204)a_1^2s_4.
\end{multline*}
Since coefficients of this polynomial are negative,
one has $f<0$ provided that $a_2-a_1>\gamma_1$.
Hence, if $a_2-a_1\geqslant 0.8717$, then we also have $f<0$, since $0.8717>\gamma_1$.
\end{proof}

\begin{lemma}
\label{lemma:Maple-P2-2}
Let $f$ be the polynomial~\eqref{equation:polynomial-P2}.
Then the following assertions hold:
\begin{itemize}
\item $f(a_1,a_2,a_3,a_4,1,1,1,1)<0$ when $a_3-a_1\geqslant 0.7698$;
\item $f(a_1,a_2,a_3,a_4,a_5,1,1,1)<0$ when $a_3-a_1\geqslant 0.8595$;
\item $f(a_1,a_2,a_3,a_4,a_5,a_6,1,1)<0$ when $a_3-a_1\geqslant 0.8985$;
\item $f(a_1,a_2,a_3,a_4,a_5,a_6,a_7,1)<0$ when $a_3-a_1\geqslant 0.9206$;
\item $f(a_1,a_2,a_3,a_4,a_5,a_6,a_7,a_8)<0$ when $a_3-a_1\geqslant 0.9347$.
\end{itemize}
\end{lemma}

\begin{proof}
Let $g_1(a_1,s_1,t,s_3,s_4,s_5,s_6,s_7)=\widehat{f}(a_1,s_1,t-s_1,s_3,s_4,s_5,s_6,s_7)$ for $t=a_3-a_1$.
Then $g_1(0,0,x,0,0,0,0,0)=-36x^4+102x^3-198x^2+120x+5$ has one positive root.
Denote it by $\gamma_1$. Then $\gamma_1\approx 0.9346473877$, and $g_1(a_1,s_1,x+\gamma_{1},s_3,s_4,s_5,s_6,s_7)$ equals
\begin{multline*}
-378a_1s_3s_4s_5-(144\gamma_1+9s_1-81)s_4s_5^2-78a_1^2s_3s_7-42s_3^2s_6s_7-51s_3s_4^2s_7-\\
-(132\gamma_1+6s_1-90)s_3s_7x-(192\gamma_1+12s_1-108)s_4s_5s_6-63s_3^2s_7x-39s_3s_7^2x-\\
-144s_4s_5^2x-36s_4s_7^2x-45s_5^2s_7x-33s_5s_7^2x-36s_6^2s_7x-30s_6s_7^2x-198s_3s_5x^2-\\
-180a_1s_5^2x-252a_1^2s_5x-234a_1s_5x^2-189s_3^2s_5x-180s_4s_5x^2-153s_3s_5^2x-162s_4^2s_5x-\\
-9s_5^4-72s_5x^3-81s_5^2x^2-45s_5^3x-40s_3^3s_6-42s_3^2s_6^2-22s_3s_6^3-32s_4^3s_6-36s_4^2s_6^2-\\
-342s_3s_4s_5x-20s_4s_6^3-24s_5^3s_6-30s_5^2s_6^2-18s_5s_6^3-99a_1^2s_5^2-51a_1s_5^3-60s_3^3s_5-72s_3^2s_5^2-\\
-42s_3s_5^3-48s_4^3s_5-63s_4^2s_5^2-39s_4s_5^3-60a_1^2s_6^2-28a_1s_6^3-171a_1s_3s_5^2-153s_3s_4^2s_5-\\
-207a_1s_3^2s_5-162s_3^2s_4s_5-180a_1s_4^2s_5-216a_1^2s_4s_5-135s_3s_4s_5^2-162a_1s_4s_5^2-234a_1^2s_3s_5-\\
-(45\gamma_1+3s_1-24)s_5^3-(30\gamma_1+3s_1-9)s_6s_7^2-(108\gamma_1+6s_1-48)a_1s_6s_7-396a_1s_4s_5x-\\
-432a_1s_3s_5x-42s_3s_5^2s_7-30s_3s_5s_7^2-33s_3s_6^2s_7-27s_3s_6s_7^2-42s_4^2s_5s_7-36s_4^2s_6s_7-\\
-33s_3s_4s_7^2-39s_4s_5^2s_7-27s_4s_5s_7^2-30s_4s_6^2s_7-24s_4s_6s_7^2-30s_5^2s_6s_7-27s_5s_6^2s_7-21s_5s_6s_7^2-\\
-(33\gamma_1+3s_1-12)s_5s_7^2-(9\gamma_1+s_1-2)s_7^3-(36\gamma_1+3s_1-15)s_6^2s_7-(102\gamma_1+6s_1-60)s_3s_5s_7-\\
-(96\gamma_1+6s_1-54)s_6s_7x-87a_1^3s_5-(360\gamma_1+15s_1-255)s_3x^2-126a_1s_3s_4s_7-114a_1s_3s_5s_7-\\
-102a_1s_3s_6s_7-108a_1s_4s_5s_7-96a_1s_4s_6s_7-90a_1s_5s_6s_7-(528\gamma_1+24s_1-360)s_3s_4x-\\
-132a_1^2s_5s_6-(18s_1^2+36s_1\gamma_1+468\gamma_1^2-72s_1-576\gamma_1+414)a_1x-(105\gamma_1+5s_1-70)s_3^3-\\
-(108\gamma_1+6s_1-66)s_4^2s_6-48a_1s_7^2x-66s_3s_7x^2-60s_4s_7x^2-54s_5s_7x^2-48s_6s_7x^2-84a_1^2s_7x-\\
-78a_1s_7x^2-(96\gamma_1+6s_1-54)s_4s_5s_7-(216\gamma_1+12s_1-132)s_5s_6x-(114\gamma_1+6s_1-72)s_3s_4s_7-\\
-84s_4^2s_5s_6-4s_6^4-(4s_1^3+12s_1\gamma_1^2+96\gamma_1^3-12s_1^2-24s_1\gamma_1-204\gamma_1^2+24s_1+264\gamma_1-80)s_4-\\
-(144\gamma_1+6s_1-102)x^3-(36\gamma_1+3s_1-15)s_4s_7^2-(45\gamma_1+3s_1-24)s_5^2s_7-(48\gamma_1+3s_1-18)a_1s_7^2-\\
-(11s_1^3+18s_1^2\gamma_1+18s_1\gamma_1^2+156\gamma_1^3-18s_1^2-72s_1\gamma_1-288\gamma_1^2+69s_1+414\gamma_1-160)a_1-\\
-(153\gamma_1+9s_1-90)s_3s_5^2-126s_3^2s_6x-90s_3s_6^2x-108s_4^2s_6x-84s_4s_6^2x-90s_5^2s_6x-\\
-(468\gamma_1+18s_1-288)a_1s_5x-72s_3s_5s_6^2-78s_3s_4s_6^2-120a_1s_4^2s_6-102s_3s_4^2s_6-138a_1s_3^2s_6-\\
-78s_5s_6^2x-156a_1s_6x^2-108a_1s_6^2x-132s_3s_6x^2-120s_4s_6x^2-108s_5s_6x^2-168a_1^2s_6x-\\
-84s_3s_5^2s_6-144a_1^2s_4s_6-(204\gamma_1+12s_1-120)s_3s_5s_6-(144\gamma_1+6s_1-84)a_1s_3s_7-\\
-66s_4s_5s_6^2-108s_3^2s_4s_6-102a_1s_5^2s_6-90a_1s_5s_6^2-78s_4s_5^2s_6-102a_1s_3s_6^2-96a_1s_4s_6^2-
\end{multline*}

\begin{multline*}
-29a_1^3s_7-(6s_1^3+18s_1\gamma_1^2+144\gamma_1^3-18s_1^2-36s_1\gamma_1-306\gamma_1^2+36s_1+396\gamma_1-120)x-\\
-54s_4^2s_7x-156a_1^2s_3s_6-(396\gamma_1+18s_1-216)a_1s_4s_5-(6s_1\gamma_1+60\gamma_1^2-6s_1-78\gamma_1+54)s_4s_7-\\
-(90\gamma_1+6s_1-48)s_3s_6s_7-(12s_1\gamma_1+108\gamma_1^2-12s_1-132\gamma_1+96)s_5s_6-(468\gamma_1+18s_1-288)a_1x^2-\\
-(6s_1\gamma_1+54\gamma_1^2-6s_1-66\gamma_1+48)s_5s_7-s_7^4-(12s_1\gamma_1+120\gamma_1^2-12s_1-156\gamma_1+108)s_4s_6-\\
-69a_1s_3^2s_7-(342\gamma_1+18s_1-216)s_3s_4s_5-(18s_1\gamma_1+198\gamma_1^2-18s_1-270\gamma_1+180)s_3s_5-\\
-(264\gamma_1+12s_1-144)a_1s_4s_6-(24s_1\gamma_1+264\gamma_1^2-24s_1-360\gamma_1+240)s_3s_4-(54\gamma_1+3s_1-33)s_4^2s_7-\\
-(90\gamma_1+6s_1-48)s_5^2s_6-(18s_1\gamma_1+216\gamma_1^2-18s_1-306\gamma_1+198)s_5x-(240\gamma_1+12s_1-120)a_1s_5s_6-\\
-(108\gamma_1+6s_1-48)a_1s_6^2-(132\gamma_1+6s_1-72)a_1s_4s_7-(216\gamma_1+9s_1-153)s_5x^2-\\
-(2s_1^3+6s_1\gamma_1^2+48\gamma_1^3-6s_1^2-12s_1\gamma_1-102\gamma_1^2+12s_1+132\gamma_1-40)s_6-102s_3s_5s_7x-\\
-(24s_1\gamma_1+288\gamma_1^2-24s_1-408\gamma_1+264)s_4x-(12s_1\gamma_1+144\gamma_1^2-12s_1-204\gamma_1+132)s_6x-\\
-(84\gamma_1+6s_1-42)s_4s_6s_7-(264\gamma_1+12s_1-180)s_3s_6x-(162\gamma_1+9s_1-99)s_5^2x-\\
-(6s_1\gamma_1+48\gamma_1^2-6s_1-54\gamma_1+42)s_6s_7-(228\gamma_1+12s_1-144)s_3s_4^2-(39\gamma_1+3s_1-18)s_3s_7^2-\\
-(18s_1\gamma_1+180\gamma_1^2-18s_1-234\gamma_1+162)s_4s_5-(72\gamma_1+4s_1-44)s_4^3-(576\gamma_1+24s_1-336)a_1s_3s_4-\\
-7s_4s_7^3-12s_5^3s_7-(84\gamma_1+6s_1-42)s_4s_6^2-(240\gamma_1+12s_1-156)s_4^2x-(288\gamma_1+12s_1-204)s_4x^2-\\
-(78\gamma_1+6s_1-36)s_5s_6s_7-(432\gamma_1+18s_1-252)a_1s_3s_5-114s_3s_4s_7x-90s_3s_6s_7x-96s_4s_5s_7x-\\
-144a_1s_3s_7x-132a_1s_4s_7x-120a_1s_5s_7x-108a_1s_6s_7x-(396\gamma_1+18s_1-270)s_3s_5x-\\
-252a_1s_3s_4s_6-228a_1s_3s_5s_6-216a_1s_4s_5s_6-180s_3s_4s_5s_6-288a_1s_3s_6x-264a_1s_4s_6x-\\
-252a_1s_3s_4^2-312a_1^2s_3s_4-276a_1s_3^2s_4-116a_1^3s_4-144a_1^2s_4^2-80a_1s_4^3-80s_3^3s_4-\\
-68s_3s_4^3-(780\gamma_1+30s_1-480)a_1s_3x-(6s_1\gamma_1+66\gamma_1^2-6s_1-90\gamma_1+60)s_3s_7-\\
-115a_1s_3^3-84s_4s_6s_7x-(12s_1\gamma_1+132\gamma_1^2-12s_1-180\gamma_1+120)s_3s_6-27a_1^2s_7^2-\\
-78s_5s_6s_7x-(3s_1\gamma_1+21\gamma_1^2-3s_1-21\gamma_1+18)s_7^2-(9s_1\gamma_1+81\gamma_1^2-9s_1-99\gamma_1+72)s_5^2-\\
-(240\gamma_1+12s_1-156)s_4s_6x-11a_1s_7^3-20s_3^3s_7-18s_3^2s_7^2-8s_3s_7^3-16s_4^3s_7-15s_4^2s_7^2-\\
-(96\gamma_1+6s_1-54)s_6^2x-12s_5^2s_7^2-6s_5s_7^3-8s_6^3s_7-9s_6^2s_7^2-5s_6s_7^3-96s_4x^3-120s_4^2x^2-\\
-72s_4^3x-312a_1s_4x^2-336a_1^2s_4x-228s_3s_4^2x-252s_3^2s_4x-264s_3s_4x^2-264a_1s_4^2x-40a_1^4-\\
-(162\gamma_1+9s_1-99)s_4^2s_5-576a_1s_3s_4x-(30s_1\gamma_1+360\gamma_1^2-30s_1-510\gamma_1+330)s_3x-\\
-(120\gamma_1+6s_1-78)s_4s_7x-(624\gamma_1+24s_1-384)a_1s_4x-24s_7x^3-(126\gamma_1+6s_1-84)s_3^2s_6-\\
-(5s_1^3+15s_1\gamma_1^2+120\gamma_1^3-15s_1^2-30s_1\gamma_1-255\gamma_1^2+30s_1+330\gamma_1-100)s_3-\\
-48s_6x^3-48s_6^2x^2-24s_6^3x-72a_1^2s_4s_7-60a_1^2s_6s_7-60a_1s_4^2s_7-51a_1s_5^2s_7-39a_1s_5s_7^2-\\
-240a_1s_5s_6x-108s_3^2s_4^2-(90\gamma_1+6s_1-48)s_3s_6^2-96s_3^2s_5s_6-(312\gamma_1+12s_1-192)a_1s_6x-\\
-145a_1^3s_3-(288\gamma_1+12s_1-168)a_1s_3s_6-(18s_1\gamma_1+216\gamma_1^2-18s_1-306\gamma_1+198)x^2-\\
-(42\gamma_1+3s_1-21)s_7^2x-(63\gamma_1+3s_1-42)s_3^2s_7-(252\gamma_1+12s_1-168)s_3^2s_4-58a_1^3s_6-\\
-(72\gamma_1+3s_1-51)s_7x^2-(330\gamma_1+15s_1-225)s_3^2x-(3s_1^2+6s_1\gamma_1+78\gamma_1^2-12s_1-96\gamma_1+69)a_1s_7-\\
-72s_3s_5s_6s_7-(3s_1^3+9s_1\gamma_1^2+72\gamma_1^3-9s_1^2-18s_1\gamma_1-153\gamma_1^2+18s_1+198\gamma_1-60)s_5-\\
-48s_3^2s_5s_7-54s_3^2s_4s_7-36a_1s_6s_7^2-45a_1s_3s_7^2-42a_1s_6^2s_7-42a_1s_4s_7^2-66a_1^2s_5s_7-
\end{multline*}

\begin{multline*}
-(108\gamma_1+6s_1-66)s_5s_7x-(78\gamma_1+6s_1-36)s_5s_6^2-16s_4^4-90s_3s_4s_5s_7-78s_3s_4s_6s_7-\\
-66s_4s_5s_6s_7-228s_3s_4s_6x-204s_3s_5s_6x-192s_4s_5s_6x-36x^4-25s_3^4-360a_1s_3^2x-\\
-(24\gamma_1+2s_1-10)s_6^3-(189\gamma_1+9s_1-126)s_3^2s_5-390a_1s_3x^2-420a_1^2s_3x-105s_3^3x-\\
-(9s_1^2+18s_1\gamma_1+234\gamma_1^2-36s_1-288\gamma_1+207)a_1s_5-(360\gamma_1+18s_1-234)s_4s_5x-156a_1x^3-\\
-(27s_1^2+36s_1\gamma_1+252\gamma_1^2-51s_1-306\gamma_1+240)a_1^2-174a_1^3x-252a_1^2x^2-195a_1^2s_3^2-\\
-(360\gamma_1+15s_1-210)a_1s_3^2-(s_1^3+3s_1\gamma_1^2+24\gamma_1^3-3s_1^2-6s_1\gamma_1-51\gamma_1^2+6s_1+66\gamma_1-20)s_7-\\
-(264\gamma_1+12s_1-144)a_1s_4^2-(15s_1^2+30s_1\gamma_1+390\gamma_1^2-60s_1-480\gamma_1+345)a_1s_3-\\
-(504\gamma_1+36s_1-306)a_1^2x-(84\gamma_1+6s_1-51)a_1^2s_7-(168\gamma_1+12s_1-102)a_1^2s_6-21s_7^2x^2-\\
-(252\gamma_1+18s_1-153)a_1^2s_5-(120\gamma_1+6s_1-60)a_1s_5s_7-(144\gamma_1+6s_1-102)s_6x^2-\\
-(180\gamma_1+9s_1-90)a_1s_5^2-(228\gamma_1+12s_1-144)s_3s_4s_6-(336\gamma_1+24s_1-204)a_1^2s_49s_7^3x-\\
-(6s_1^2+12s_1\gamma_1+156\gamma_1^2-24s_1-192\gamma_1+138)a_1s_6-(15s_1\gamma_1+165\gamma_1^2-15s_1-225\gamma_1+150)s_3^2-\\
-(174\gamma_1+29s_1-115)a_1^3-(12s_1\gamma_1+120\gamma_1^2-12s_1-156\gamma_1+108)s_4^2-(420\gamma_1+30s_1-255)a_1^2s_3-\\
-120s_3x^3-(6s_1\gamma_1+48\gamma_1^2-6s_1-54\gamma_1+42)s_6^2-(6s_1\gamma_1+72\gamma_1^2-6s_1-102\gamma_1+66)s_7x-\\
-(12s_1^2+24s_1\gamma_1+312\gamma_1^2-48s_1-384\gamma_1+276)a_1s_4-(156\gamma_1+6s_1-96)a_1s_7x-\\
-165s_3^2x^2-(36\gamma_1-20-18\gamma_1^2-2s_1^2+18s_1+s_1^3-18s_1\gamma_1+6s_1^2\gamma_1+6\gamma_1^3)s_1.
\end{multline*}
Since all summands here are non-positive, we see that $f<0$ provided that $a_3-a_1>\gamma_1$.
In particular, if $a_3-a_1\geqslant 0.9347$, then $f<0$ as well.

Let us denote the polynomial $f(a_1,a_2,a_3,a_4,a_5,a_6,a_7,1)$ by $f_2$.
As in the previous case, we let $g_2(a_1,s_1,t,s_3,s_4,s_5,s_6)=\widehat{f}_2(a_1,s_1,t-s_1,s_3,s_4,s_5,s_6)$.
Then
$$
g_2(0,0,x,0,0,0,0)=-25x^4+65x^3-135x^2+80x+8.
$$
This polynomial has exactly one positive root.
Denote it by $\gamma_2$. Then $\gamma_2\approx 0.9205032589$.
Let us consider $g_2(a_1,s_1,x+\gamma_{2},s_3,s_4,s_5,s_6)$ as a polynomial in $a_1$, $x$, $s_3$, $s_4$, $s_5$ and $s_6$.
Then all its coefficients are non-positive, since it can be expanded as
\begin{multline*}
-25x^4-(24s_1^2+30s_1\gamma_2+180\gamma_2^2-39s_1-195\gamma_2+168)a_1^2-108s_3s_5x^2-96a_1s_5^2x-\\
-7s_3s_6^3-144a_1^2s_5x-132a_1s_5x^2-102s_3^2s_5x-96s_4s_5x^2-78s_3s_5^2x-84s_4^2s_5x-72s_4s_5^2x-\\
-(528\gamma_2+24s_1-288)a_1s_3x-4s_5^4-40s_5x^3-42s_5^2x^2-22s_5^3x-16s_3^3s_6-15s_3^2s_6^2-\\
-42s_3^2s_4s_6-32s_3^3s_5-36s_3^2s_5^2-20s_3s_5^3-24s_4^3s_5-30s_4^2s_5^2-18s_4s_5^3-24a_1^2s_6^2-10a_1s_6^3-\\
-66s_3s_4s_5^2-84a_1s_4s_5^2-132a_1^2s_3s_5-(120\gamma_2+6s_1-60)a_1s_3s_6-240a_1s_3s_5x-\\
-39a_1s_3s_6^2-50a_1^3s_5-216a_1s_4s_5x-180s_3s_4s_5x-204a_1s_3s_4s_5-(180\gamma_2+9s_1-117)s_4x^2-\\
-(12s_1\gamma_2+108\gamma_2^2-12s_1-132\gamma_2+96)s_3^2-(18s_1\gamma_2+180\gamma_2^2-18s_1-234\gamma_2+162)s_4x-\\
-(15s_1^2+30s_1\gamma_2+330\gamma_2^2-60s_1-360\gamma_2+285)a_1x-(3s_1\gamma_2+18\gamma_2^2-3s_1-15\gamma_2+15)s_6^2-\\
-21s_4s_5s_6^2-12s_4^3s_6-12s_4^2s_6^2-6s_4s_6^3-8s_5^3s_6-9s_5^2s_6^2-5s_5s_6^3-54a_1^2s_5^2-26a_1s_5^3-
\end{multline*}

\begin{multline*}
-33a_1s_5s_6^2-90a_1s_3s_5^2-78s_3s_4^2s_5-114a_1s_3^2s_5-84s_3^2s_4s_5-96a_1s_4^2s_5-120a_1^2s_4s_5-\\
-(192\gamma_2+12s_1-108)s_4s_5x-(9s_1^2+18s_1\gamma_2+198\gamma_2^2-36s_1-216\gamma_2+171)a_1s_4-s_6^4-\\
-(3s_1^3+9s_1\gamma_2^2+60\gamma_2^3-9s_1^2-18s_1\gamma_2-117\gamma_2^2+18s_1+162\gamma_2-48)s_4-(180\gamma_2+12s_1-96)s_3s_4s_5-\\
-66a_1^2s_3s_6-27s_5s_6^2x-66a_1s_6x^2-42a_1s_6^2x-54s_3s_6x^2-48s_4s_6x^2-42s_5s_6x^2-72a_1^2s_6x-\\
-16s_3^4-(125\gamma_2+25s_1-76)a_1^3-51s_3^2s_6x-33s_3s_6^2x-42s_4^2s_6x-30s_4s_6^2x-33s_5^2s_6x-\\
-24s_3s_5s_6^2-27s_3s_4s_6^2-54a_1^2s_5s_6-48a_1s_4^2s_6-39s_3s_4^2s_6-57a_1s_3^2s_6-30s_4^2s_5s_6-\\
-(6s_1\gamma_2+54\gamma_2^2-6s_1-66\gamma_2+48)s_3s_6-(9s_1\gamma_2+72\gamma_2^2-9s_1-81\gamma_2+63)s_4^2-\\
-27s_4s_5^2s_6-36a_1s_4s_6^2-36s_3^2s_5s_6-30s_3s_5^2s_6-60a_1^2s_4s_6-(100\gamma_2+5s_1-65)x^3-\\
-(144\gamma_2+9s_1-81)s_4^2x-(240\gamma_2+12s_1-120)a_1s_3^2-(108\gamma_2+6s_1-66)s_3s_6x-\\
-(12s_1^2+24s_1\gamma_2+264\gamma_2^2-48s_1-288\gamma_2+228)a_1s_3-(153\gamma_2+9s_1-90)s_3^2s_4-\\
-25a_1^3s_6-(84\gamma_2+6s_1-42)s_5^2x-(240\gamma_2+12s_1-156)s_3x^2-(120\gamma_2+6s_1-78)s_5x^2-\\
-(6s_1^2+12s_1\gamma_2+132\gamma_2^2-24s_1-144\gamma_2+114)a_1s_5-(78\gamma_2+6s_1-36)s_3s_5s_6-\\
-(240\gamma_2+12s_1-120)a_1s_3s_5-(s_1^3+3s_1\gamma_2^2+20\gamma_2^3-3s_1^2-6s_1\gamma_2-39\gamma_2^2+6s_1+54\gamma_2-16)s_6-\\
-(42\gamma_2+3s_1-12)a_1s_6^2-(27\gamma_2+3s_1-6)s_5s_6^2-(96\gamma_2+6s_1-54)s_4s_6x-(8\gamma_2+s_1-1)s_6^3-\\
-(360\gamma_2+18s_1-180)a_1s_3s_4-(24s_1\gamma_2+240\gamma_2^2-24s_1-312\gamma_2+216)s_3x-\\
-(15s_1\gamma_2+150\gamma_2^2-15s_1-195\gamma_2+135)x^2-(144\gamma_2+12s_1-78)a_1^2s_5-153s_3^2s_4x-\\
-(216\gamma_2+18s_1-117)a_1^2s_4-(162\gamma_2+9s_1-72)a_1s_4^2-(6s_1\gamma_2+42\gamma_2^2-6s_1-42\gamma_2+36)s_5s_6-\\
-39a_1s_5^2s_6-(135\gamma_2+9s_1-72)s_3s_4^2-(3s_1^2+6s_1\gamma_2+66\gamma_2^2-12s_1-72\gamma_2+57)a_1s_6-\\
-(96\gamma_2+6s_1-36)a_1s_5s_6-(33\gamma_2+3s_1-12)s_3s_6^2-(330\gamma_2+15s_1-180)a_1x^2-(68\gamma_2+4s_1-40)s_3^3-\\
-(264\gamma_2+12s_1-144)a_1s_5x-(22\gamma_2+2s_1-8)s_5^3-(12s_1\gamma_2+108\gamma_2^2-12s_1-132\gamma_2+96)s_3s_5-\\
-8s_6^3x-(5s_1^3+15s_1\gamma_2^2+100\gamma_2^3-15s_1^2-30s_1\gamma_2-195\gamma_2^2+30s_1+270\gamma_2-80)x-\\
-(33\gamma_2+3s_1-12)s_5^2s_6-102a_1s_3s_4s_6-90a_1s_3s_5s_6-84a_1s_4s_5s_6-66s_3s_4s_5s_6-120a_1s_3s_6x-\\
-(6s_1\gamma_2+60\gamma_2^2-6s_1-78\gamma_2+54)s_6x-108a_1s_4s_6x-96a_1s_5s_6x-153a_1s_3s_4^2-198a_1^2s_3s_4-\\
-(18s_1\gamma_2+162\gamma_2^2-18s_1-198\gamma_2+144)s_3s_4-171a_1s_3^2s_4-75a_1^3s_4-90a_1^2s_4^2-\\
-48a_1s_4^3-48s_3^3s_4-63s_3^2s_4^2-39s_3s_4^3-(60\gamma_2+3s_1-39)s_6x^2-216a_1^2s_4x-76a_1s_3^3-\\
-(12s_1\gamma_2+120\gamma_2^2-12s_1-156\gamma_2+108)s_5x-(102\gamma_2+6s_1-60)s_3^2s_5-135s_3s_4^2x-\\
-240a_1s_3^2x-(72\gamma_2+6s_1-30)s_4s_5^2-(360\gamma_2+30s_1-195)a_1^2x-(51\gamma_2+3s_1-30)s_3^2s_6-\\
-(96\gamma_2+6s_1-36)a_1s_5^2-(72\gamma_2+6s_1-39)a_1^2s_6-60s_4x^3-72s_4^2x^2-42s_4^3x-198a_1s_4x^2-\\
-162s_3s_4x^2-162a_1s_4^2x-(84\gamma_2+6s_1-42)s_5s_6x-28a_1^4-(42\gamma_2+3s_1-21)s_4^3-360a_1s_3s_4x-\\
-(4s_1^3+12s_1\gamma_2^2+80\gamma_2^3-12s_1^2-24s_1\gamma_2-156\gamma_2^2+24s_1+216\gamma_2-64)s_3-20s_6x^3-\\
-(10s_1^3+15s_1^2\gamma_2+15s_1\gamma_2^2+110\gamma_2^3-12s_1^2-60s_1\gamma_2-180\gamma_2^2+57s_1+285\gamma_2-112)a_1-\\
-(216\gamma_2+12s_1-132)s_3^2x-(6s_1\gamma_2+48\gamma_2^2-6s_1-54\gamma_2+42)s_4s_6-(72\gamma_2+6s_1-30)s_4s_5s_6-\\
-(396\gamma_2+18s_1-216)a_1s_4x-(324\gamma_2+18s_1-198)s_3s_4x-(288\gamma_2+24s_1-156)a_1^2s_3-\\
-18s_6^2x^2-(90\gamma_2+6s_1-48)s_3s_4s_6-(84\gamma_2+6s_1-42)s_4^2s_5-(30\gamma_2+3s_1-9)s_4s_6^2-
\end{multline*}

\begin{multline*}
-(108\gamma_2+6s_1-48)a_1s_4s_6-(2s_1^3+6s_1\gamma_2^2+40\gamma_2^3-6s_1^2-12s_1\gamma_2-78\gamma_2^2+12s_1+108\gamma_2-32)s_5-\\
-(132\gamma_2+6s_1-72)a_1s_6x-9s_4^4-(42\gamma_2+3s_1-21)s_4^2s_6-90s_3s_4s_6x-78s_3s_5s_6x-72s_4s_5s_6x-\\
-264a_1s_3x^2-288a_1^2s_3x-110a_1x^3-80s_3x^3-108s_3^2x^2-68s_3^3x-125a_1^3x-180a_1^2x^2-\\
-132a_1^2s_3^2-(78\gamma_2+6s_1-36)s_3s_5^2-100a_1^3s_3-(6s_1\gamma_2+42\gamma_2^2-6s_1-42\gamma_2+36)s_5^2-\\
-(36\gamma_2+3s_1-15)s_6^2x-(12s_1\gamma_2+96\gamma_2^2-12s_1-108\gamma_2+84)s_4s_5-(216\gamma_2+12s_1-132)s_3s_5x-\\
-(216\gamma_2+12s_1-96)a_1s_4s_5-(30\gamma_2-16-15s_1\gamma_2-15\gamma_2^2-s_1^2+15s_1+s_1^3+5\gamma_2^3+5s_1^2\gamma_2)s_1.
\end{multline*}
This shows that $f_2<0$ provided that $a_3-a_1>\gamma_2$.
Thus, if $a_3-a_1\geqslant 0.9206$, then $f_2<0$.

Let $f_3=f(a_1,a_2,a_3,a_4,a_5,a_6,1,1)$.
As above, let $g_3(a_1,s_1,t,s_3,s_4,s_5)$ be the polynomial $\widehat{f}_3(a_1,s_1,t-s_1,s_3,s_4,s_5)$.
Then
$$
g_3(0,0,x,0,0,0)=-16x^4+36x^3-84x^2+48x+9.
$$
This polynomial has one positive root.
Denote this root by $\gamma_3$. Then $\gamma_3\approx 0.8984970862$.
Moreover, we can expand $g_3(a_1,s_1,x+\gamma_{3},s_3,s_4,s_5)$ as
\begin{multline*}
-16x^4-(18s_1\gamma_{3}+144\gamma_{3}^2-18s_1-162\gamma_{3}+126)s_3x-(30\gamma_{3}+3s_1-9)s_4^2s_5-90a_1s_3^2s_4-\\
-(108\gamma_{3}+6s_1-48)a_1s_5x-(216\gamma_{3}+12s_1-96)a_1s_4x-(6s_1\gamma_{3}+42\gamma_{3}^2-6s_1-42\gamma_{3}+36)s_3s_5-\\
-4s_4^4-120a_1^2s_4x-66s_3s_4^2x-78s_3^2s_4x-108a_1s_4x^2-84s_3s_4x^2-84a_1s_4^2x-108a_1^2s_3s_4-78a_1s_3s_4^2-\\
-(144\gamma_{3}+9s_1-81)s_3x^2-s_5^4-(9s_1\gamma_{3}+63\gamma_{3}^2-9s_1-63\gamma_{3}+54)s_3^2-(84\gamma_{3}+6s_1-24)a_1s_4s_5-\\
-(216\gamma_{3}+12s_1-96)a_1x^2-(96\gamma_{3}+6s_1-36)a_1s_3s_5-(30\gamma_{3}+3s_1-9)s_5^2x-(144\gamma_{3}+9s_1-54)a_1s_3^2-\\
-(96\gamma_{3}+6s_1-54)s_4x^2-(6s_1^2+12s_1\gamma_{3}+108\gamma_{3}^2-24s_1-96\gamma_{3}+90)a_1s_4-(36\gamma_{3}+3s_1-6)a_1s_5^2-\\
-(12s_1^2+24s_1\gamma_{3}+216\gamma_{3}^2-48s_1-192\gamma_{3}+180)a_1x-(78\gamma_{3}+6s_1-36)s_3^2s_4-(39\gamma_{3}+3s_1-18)s_3^2s_5-\\
-(126\gamma_{3}+9s_1-63)s_3^2x-(66\gamma_{3}+6s_1-24)s_3s_4^2-(12s_1\gamma_{3}+84\gamma_{3}^2-12s_1-84\gamma_{3}+72)s_3s_4-\\
-(27\gamma_{3}+3s_1-6)s_3s_5^2-(6s_1\gamma_{3}+48\gamma_{3}^2-6s_1-54\gamma_{3}+42)s_5x-(48\gamma_{3}+3s_1-27)s_5x^2-\\
-33a_1s_3s_5^2-(9s_1^3+12s_1^2\gamma_{3}+12s_1\gamma_{3}^2+72\gamma_{3}^3-6s_1^2-48s_1\gamma_{3}-96\gamma_{3}^2+45s_1+180\gamma_{3}-72)a_1-\\
-27s_3s_4^2s_5-(7\gamma_{3}+s_1)s_5^3-(s_1^3+3s_1\gamma_{3}^2+16\gamma_{3}^3-3s_1^2-6s_1\gamma_{3}-27\gamma_{3}^2+6s_1+42\gamma_{3}-12)s_5-\\
-(12s_1\gamma_{3}+96\gamma_{3}^2-12s_1-108\gamma_{3}+84)s_4x-(6s_1\gamma_{3}+36\gamma_{3}^2-6s_1-30\gamma_{3}+30)s_4^2-\\
-(2s_1^3+6s_1\gamma_{3}^2+32\gamma_{3}^3-6s_1^2-12s_1\gamma_{3}-54\gamma_{3}^2+12s_1+84\gamma_{3}-24)s_4-96a_1s_3s_5x-\\
-36a_1s_4^2s_5-(3s_1^3+9s_1\gamma_{3}^2+48\gamma_{3}^3-9s_1^2-18s_1\gamma_{3}-81\gamma_{3}^2+18s_1+126\gamma_{3}-36)s_3-20s_4^3x-24s_3^3s_4-\\
-30s_3^2s_4^2-18s_3s_4^3-42a_1^3s_4-48a_1^2s_4^2-24a_1s_4^3-192a_1s_3s_4x-(24\gamma_{3}+3s_1-3)s_4s_5^2-78a_1s_3s_4s_5-\\
-(3s_1\gamma_{3}+15\gamma_{3}^2-3s_1-9\gamma_{3}+12)s_5^2-(12s_1\gamma_{3}+96\gamma_{3}^2-12s_1-108\gamma_{3}+84)x^2-\\
-(66\gamma_{3}+6s_1-24)s_3s_4s_5-(168\gamma_{3}+12s_1-84)s_3s_4x-(39\gamma_{3}+3s_1-18)s_3^3-(64\gamma_{3}+4s_1-36)x^3-\\
-180a_1^2s_3x-(60\gamma_{3}+6s_1-27)a_1^2s_5-(120\gamma_{3}+12s_1-54)a_1^2s_4-(84\gamma_{3}+6s_1-42)s_3s_5x-\\
-6s_3s_5^3-66s_3s_4s_5x-(72\gamma_{3}+6s_1-30)s_4^2x-(192\gamma_{3}+12s_1-72)a_1s_3s_4-(72\gamma_{3}+6s_1-30)s_4s_5x-\\
-21s_3s_4s_5^2-(21s_1^2+24s_1\gamma_{3}+120\gamma_{3}^2-27s_1-108\gamma_{3}+108)a_1^2-18a_1^4-(20\gamma_{3}+2s_1-6)s_4^3-
\end{multline*}

\begin{multline*}
-(180\gamma_{3}+18s_1-81)a_1^2s_3-(240\gamma_{3}+24s_1-108)a_1^2x-9s_3^4-39s_3^3x-84a_1^3x-120a_1^2x^2-\\
-72a_1x^3-48s_3x^3-63s_3^2x^2-81a_1^2s_3^2-45a_1s_3^3-7s_5^3x-16s_5x^3-15s_5^2x^2-60a_1^2s_5x-39s_3^2s_5x-\\
-9a_1s_5^3-12s_3^3s_5-12s_3^2s_5^2-8s_4^3s_5-9s_4^2s_5^2-5s_4s_5^3-30s_3^2s_4s_5-30a_1s_4s_5^2-36a_1s_5^2x-\\
-27s_3s_5^2x-30s_4^2s_5x-24s_4s_5^2x-42s_3s_5x^2-36s_4s_5x^2-54a_1s_5x^2-21a_1^3s_5-21a_1^2s_5^2-\\
-84a_1s_4s_5x-(4s_1^3+12s_1\gamma_{3}^2+64\gamma_{3}^3-12s_1^2-24s_1\gamma_{3}-108\gamma_{3}^2+24s_1+168\gamma_{3}-48)x-\\
-144a_1s_3^2x-(84\gamma_{3}+21s_1-45)a_1^3-(6s_1\gamma_{3}+36\gamma_{3}^2-6s_1-30\gamma_{3}+30)s_4s_5-32s_4x^3-36s_4^2x^2-\\
-162a_1s_3x^2-(84\gamma_{3}+6s_1-24)a_1s_4^2-(9s_1^2+18s_1\gamma_{3}+162\gamma_{3}^2-36s_1-144\gamma_{3}+135)a_1s_3-\\
-45a_1s_3^2s_5-(3s_1^2+6s_1\gamma_{3}+54\gamma_{3}^2-12s_1-48\gamma_{3}+45)a_1s_5-(324\gamma_{3}+18s_1-144)a_1s_3x-\\
-54a_1^2s_3s_5-48a_1^2s_4s_5-(24s_1\gamma_{3}-4s_1^3\gamma_{3}-12s_1+63a_1^3s_3+12s_1^2+4s_1\gamma_{3}^3+s_1^4-12s_1\gamma_{3}^2+12s_1^2\gamma_{3}).
\end{multline*}
All summands in this sum are non-negative.
This shows that $f_3<0$ if $a_3-a_1>\gamma_3$.
Therefore, if $a_3-a_1\geqslant 0.8985$, then $f_3<0$ as well, since $\gamma_3<0.8985$.

Let $f_4=f(a_1,a_2,a_3,a_4,a_5,1,1,1)$ and $g_4(a_1,s_1,t,s_3,s_4)=\widehat{f}_4(a_1,s_1,t-s_1,s_3,s_4)$.
Then the polynomial $g_4(0,0,x,0,0)=-9x^4+15x^3-45x^2+24x+8$ has one positive root.
Denote it by $\gamma_4$. Then $\gamma_4\approx 0.8594363$ and $g_4(a_1,s_1,x+\gamma_{4},s_3,s_4)$ can be expanded as
\begin{multline*}
-9x^4-(84\gamma_4+6s_1-24)a_1s_4x-(60\gamma_4+6s_1-18)s_3s_4x-(72\gamma_4+6s_1-12)a_1s_3s_4-\\
-(168\gamma_4+12s_1-48)a_1s_3x-(60\gamma_4+6s_1-18)s_3^2x-(12s_1\gamma_4+72\gamma_4^2-12s_1-60\gamma_4+60)s_3x-\\
-(6s_1\gamma_4+30\gamma_4^2-6s_1-18\gamma_4+24)s_3s_4-(72\gamma_4+6s_1-30)s_3x^2-(24\gamma_4+3s_1-3)s_4^2x-\\
-(6s_1\gamma_4+36\gamma_4^2-6s_1-30\gamma_4+30)s_4x-(36\gamma_4+3s_1-15)s_4x^2-(27\gamma_4+3s_1-6)s_3^2s_4-17a_1^3s_4-\\
-(21\gamma_4+3s_1)s_3s_4^2-(9s_1^2+18s_1\gamma_4+126\gamma_4^2-36s_1-72\gamma_4+99)a_1x-(126\gamma_4+9s_1-36)a_1x^2-\\
-(3s_1^2+6s_1\gamma_4+42\gamma_4^2-12s_1-24\gamma_4+33)a_1s_4-(30\gamma_4+3s_1)a_1s_4^2-(72\gamma_4+6s_1-12)a_1s_3^2-\\
-(144\gamma_4+18s_1-45)a_1^2x-(48\gamma_4+6s_1-15)a_1^2s_4-(96\gamma_4+12s_1-30)a_1^2s_3-72a_1s_3^2x-72a_1s_3s_4x-\\
-(6s_1^2+12s_1\gamma_4+84\gamma_4^2-24s_1-48\gamma_4+66)a_1s_3-96a_1^2s_3x-4s_3^4-30s_3^2x^2-18s_3^3x-51a_1^3x-\\
-72a_1^2x^2-42a_1x^3-42a_1^2s_3^2-22a_1s_3^3-34a_1^3s_3-(18s_1^2+18s_1\gamma_4+72\gamma_4^2-15s_1-45\gamma_4+60)a_1^2-\\
-(51\gamma_4+17s_1-22)a_1^3-(3s_1^3+9s_1\gamma_4^2+36\gamma_4^3-9s_1^2-18s_1\gamma_4-45\gamma_4^2+18s_1+90\gamma_4-24)x-\\
-(9s_1\gamma_4+54\gamma_4^2-9s_1-45\gamma_4+45)x^2-(36\gamma_4+3s_1-15)x^3-(3s_1\gamma_4+12\gamma_4^2-3s_1-3\gamma_4+9)s_4^2-\\
-(s_1^3+3s_1\gamma_4^2+12\gamma_4^3-3s_1^2-6s_1\gamma_4-15\gamma_4^2+6s_1+30\gamma_4-8)s_4-(6\gamma_4+s_1+1)s_4^3-s_4^4-\\
-27s_3^2s_4x-30s_3s_4x^2-30a_1s_4^2x-48a_1^2s_4x-42a_1s_4x^2-27a_1s_3s_4^2-33a_1s_3^2s_4-42a_1^2s_3s_4-\\
-12s_4x^3-12s_4^2x^2-6s_4^3x-(2s_1^3+6s_1\gamma_4^2+24\gamma_4^3-6s_1^2-12s_1\gamma_4-30\gamma_4^2+12s_1+60\gamma_4-16)s_3-\\
-21s_3s_4^2x-(6s_1\gamma_4+30\gamma_4^2-6s_1-18\gamma_4+24)s_3^2-18a_1^2s_4^2-8a_1s_4^3-8s_3^3s_4-9s_3^2s_4^2-5s_3s_4^3-\\
-(8s_1^3+9s_1^2\gamma_4+9s_1\gamma_4^2+42\gamma_4^3-36s_1\gamma_4-36\gamma_4^2+33s_1+99\gamma_4-40)a_1-(18\gamma_4+2s_1-4)s_3^3-\\
-24s_3x^3-84a_1s_3x^2-10a_1^4-(3s_1^2\gamma_4+3\gamma_4^3-9\gamma_4^2-9s_1\gamma_4+18\gamma_4+s_1^2+s_1^3-8+9s_1)s_1.
\end{multline*}
As above, we see that $f_4<0$ when $a_3-a_1>\gamma_4$. So, if $a_3-a_1\geqslant 0.8595$, then $f_4<0$.

Finally, let $f_5=f(a_1,a_2,a_3,a_4,1,1,1,1)$ and $g_5(a_1,s_1,t,s_3)=\widehat{f}_5(a_1,s_1,t-s_1,s_3)$.
Then the polynomial $g_5(0,0,x,0)=-4x^4+2x^3-18x^2+8x+5$ has unique positive root.
Denote it by $\gamma_5$. Then $g_5(a_1,s_1,x+\gamma_{5},s_3)$ can be expanded as
\begin{multline*}
-4x^4-(60\gamma_5+6s_1)a_1s_3x-(18\gamma_5+3s_1+3)s_3^2x-(24\gamma_5+3s_1-3)s_3x^2-13a_1^3s_3-\\
-(36\gamma_5+6s_1-3)a_1^2s_3-(3s_1^2+6s_1\gamma_5+30\gamma_5^2-12s_1+21)a_1s_3-(72\gamma_5+12s_1-6)a_1^2x-\\
-(6s_1\gamma_5+24\gamma_5^2-6s_1-6\gamma_5+18)s_3x-(24\gamma_5+3s_1+6)a_1s_3^2-(6s_1^2+12s_1\gamma_5+60\gamma_5^2-24s_1+42)a_1x-\\
-5s_3^3x-(60\gamma_5+6s_1)a_1x^2-(7s_1^3+6s_1^2\gamma_5+6s_1\gamma_5^2+20\gamma_5^3+6s_1^2-24s_1\gamma_5+21s_1+42\gamma_5-16)a_1-\\
-(15s_1^2+12s_1\gamma_5+36\gamma_5^2-3s_1-6\gamma_5+24)a_1^2-(26\gamma_5+13s_1-7)a_1^3-26a_1^3x-15a_1^2s_3^2-7a_1s_3^3-2s_1^3-\\
-(2s_1^3+6s_1\gamma_5^2+16\gamma_5^3-6\gamma_5^2+12s_1+36\gamma_5-8)x-(6s_1\gamma_5+24\gamma_5^2-6s_1-6\gamma_5+18)x^2-\\
-(16\gamma_5+2s_1-2)x^3-24a_1s_3^2x-30a_1s_3x^2-36a_1^2s_3x-4a_1^4-s_3^4-36a_1^2x^2-20a_1x^3-8s_3x^3-9s_3^2x^2-\\
-(s_1^3+3s_1\gamma_5^2+8\gamma_5^3-3s_1^2-6s_1\gamma_5-3\gamma_5^2+6s_1+18\gamma_5-4)s_3-(3s_1\gamma_5+9\gamma_5^2-3s_1+3\gamma_5+6)s_3^2-\\
-(5\gamma_5+s_1+2)s_3^3-(12\gamma_5+2\gamma_5^3+2s_1^2\gamma_5-6\gamma_5^2-6s_1\gamma_5+s_1^3+6s_1+12\gamma_5+6s_1-4)s_1.
\end{multline*}
Observe that every summand in this sum is non-positive, because $\gamma_5\approx 0.7697759834$.
This shows that $f_5<0$ when $a_3-a_1>\gamma_5$.
Thus, if $a_3-a_1\geqslant 0.7698$, then $f_5<0$.
\end{proof}

\begin{lemma}
\label{lemma:Maple-F1}
Let $f$ be the polynomial~\eqref{equation:polynomial-F1}.
Then the following assertions hold:
\begin{itemize}
\item $f(a_1,a_2,a_3,1,1,1,1,b)<0$ when $a_2-a_1\geqslant 0.7701$;
\item $f(a_1,a_2,a_3,a_4,1,1,1,b)<0$ when $a_2-a_1\geqslant 0.8595$;
\item $f(a_1,a_2,a_3,a_4,a_5,1,1,b)<0$ when $a_2-a_1\geqslant 0.8985$;
\item $f(a_1,a_2,a_3,a_4,a_5,a_6,1,b)<0$ when $a_2-a_1\geqslant 0.9206$;
\item $f(a_1,a_2,a_3,a_4,a_5,a_6,a_7,b)<0$ when $a_2-a_1\geqslant 0.9347$.
\end{itemize}
\end{lemma}

\begin{proof}
Denote by $f_5$ the polynomial $f(a_1,a_2,a_3,1,1,1,1,b)$, and let $g_5(a_1,s_1,s_2,b)=\widehat{f}_5$.
Then $g_5(0,x,0,b)=(3-6x^2)b^2+(8-4x^3-18x^2+6x)b-4x^4+2x^3-18x^2+8x+5$.
The~discriminant of this polynomial is equal to $-80x^6+192x^5-108x^4-112x^3+84x^2+4$.
Denote by $\delta_5$ its unique positive root.
Then $g_5(a_1,x+\delta_5,s_2,b)$ is a sum of
\begin{multline*}
-(12\delta_5+18)bs_2x-(36+36\delta_5)a_1bx-(6+30\delta_5)a_1^2s_2-(18\delta_5+12)a_1^2b-(9+21\delta_5)a_1s_2^2-\\
-12\delta_5b^2x-12a_1b^2\delta_5-27a_1s_2x^2-6bs_2x^2-18a_1bx^2-4x^4-4bx^3-(12\delta_5+18)bx^2-\\
-18a_1x^3-(8\delta_5^3-3\delta_5^2+18\delta_5-4)s_2-(16\delta_5-2)x^3-(18+18\delta_5)a_1bs_2-(6+54\delta_5)a_1s_2x-\\
-9s_2^2x^2-8s_2x^3-(54\delta_5^2+12\delta_5+30)a_1x-(6+54\delta_5)a_1x^2-18a_1^2bx-12a_1b^2x-30a_1^2s_2x-\\
-21a_1s_2^2x-6bs_2^2x-6b^2s_2x-(5\delta_5+2)s_2^3-18a_1^3x-5s_2^3x-(27\delta_5^2+6\delta_5+15)a_1s_2-\\
-6b^2x^2--(24\delta_5^2-6\delta_5+18)s_2x-(24\delta_5-3)s_2x^2-(18\delta_5+3)s_2^2x-(12\delta_5^2+36\delta_5-6)bx-\\
-(6\delta_5^2+18\delta_5-3)bs_2-9a_1s_2^2b-9a_1^2s_2b-6a_1s_2b^2-12a_1^2s_2^2-6a_1s_2^3-(6\delta_5+9)bs_2^2-\\
-2s_2^3b-(30\delta_5^2+12\delta_5+9)a_1^2-9a_1^3s_2-s_2^4-(60\delta_5+12)a_1^2x-(18\delta_5^2+36\delta_5-12)a_1b-\\
-30a_1^2x^2-3a_1^2b^2-3s_2^2b^2-(24\delta_5^2-6\delta_5+18)x^2-(18\delta_5^3+6\delta_5^2+30\delta_5-12)a_1-\\
-18a_1^3\delta_5-(16\delta_5^3-6\delta_5^2+36\delta_5-8)x-(9\delta_5^2+3\delta_5+6)s_2^2-18a_1bs_2x-6b^2s_2\delta_5\\
\end{multline*}
and the polynomial
$$
(3-6\delta_5^2)b^2-(4\delta_5^3+18\delta_5^2-6\delta_5-8)b-4\delta_5^4+2\delta_5^3-18\delta_5^2+8\delta_5+5.
$$
Observe that all coefficients of the former polynomial are negative, since $\delta_5\approx 0.7700518$.
On the other hand, the latter polynomial is not positive for all $b\geqslant 0$ by the choice of~$\delta_5$.
This shows that $g_5(a_1,s_1,s_2,b)<0$ if $s_1>\delta_5$, so that $f_5<0$ when $a_2-a_1\geqslant 0.7701$.

One can use the same arguments to prove that $f_5<0$ if $a_2-a_1\geqslant 0.7698$ and $b\geqslant 0.2308$.
Indeed, let $\alpha_5$ be the unique positive root of the polynomial $-4x^4+2x^3-18x^2+8x+5$,
let $\beta_5$ be the unique positive root of of the polynomial $8-4x^3-18x^2+6x$,
and let
$$
\gamma_5=-\frac{8-4\alpha_5^3-18\alpha_5^2+6\alpha_5}{3-6\alpha_5^2}.
$$
Then $\alpha_5\approx 0.7697759834$, $\beta_5\approx 0.774202788$,
and $\gamma_5\approx 0.230720502$.
Moreover, one has
$$
g_5(0,x,0,b)<0
$$
provided that $x>\alpha_5$ and $b>\gamma_5$. Then $g_5(a_1,x+\alpha_5,s_2,y+\gamma_5)$ can be expanded as
\begin{multline*}
-(18\alpha_5+6\gamma_5+12)a_1^2y-(54\alpha_5+18\gamma_5+6)a_1x^2-9a_1^2s_2y-\\
-3s_2^2y^2-12\alpha_5a_1y^2-\big(54\alpha_5^2+(36\gamma_5+12)\alpha_5+12\gamma_5^2+36\gamma_5+30\big)a_1x-5s_2^3x-\\
-6a_1s_2y^2-\big(27\alpha_5^2+(18\gamma_5+6)\alpha_5+6\gamma_5^2+18\gamma_5+15\big)a_1s_2-\\
-2s_2^3y-\big(24\alpha_5^2+(12\gamma_5-6)\alpha_5+6\gamma_5^2+18\gamma_5+18\big)s_2x-6s_2^2yx-\\
-\big(18\alpha_5^2+(24\gamma_5+36)\alpha_5-12\big)a_1y-(24\alpha_5+6\gamma_5-3)s_2x^2-9a_1s_2^2y-\\
-\big(8\alpha_5^3+(6\gamma_5-3)\alpha_5^2+(6\gamma_5^2+18\gamma_5+18)\alpha_5-3\gamma_5-4\big)s_2-\\
-s_2^4-\big(9\alpha_5^2+(6\gamma_5+3)\alpha_5+3(\gamma_5+2)(\gamma_5+1)\big)s_2^2-18a_1s_2yx-\\
-\big(6\alpha_5^2+(12\gamma_5+18)\alpha_5-3\big)ys_2-(6\alpha_5+6\gamma_5+9)ys_2^2-6\alpha_5s_2y^2-\\
-\big(16\alpha_5^3+(12\gamma_5-6)\alpha_5^2+(12\gamma_5^2+36\gamma_5+36)\alpha_5-6\gamma_5-8\big)x-6a_1s_2^3-\\
-4x^4-\big(12\alpha_5^2+(24\gamma_5+36)\alpha_5-6\big)yx-(21\alpha_5+9\gamma_5+9)a_1s_2^2-\\
-9a_1^3s_2-(60\alpha_5+18\gamma_5+12)a_1^2x-(12\alpha_5+12\gamma_5+18)yx^2-6s_2yx^2-\\
-3a_1^2y^2-\big(24\alpha_5^2+(12\gamma_5-6)\alpha_5+6\gamma_5^2+18\gamma_5+18\big)x^2-30a_1^2x^2-\\
-12y^2\alpha_5x-(30\alpha_5+9\gamma_5+6)a_1^2s_2-(5\alpha_5+2\gamma_5+2)s_2^3-8s_2x^3-12a_1y^2x-\\
-27a_1s_2x^2-21a_1s_2^2x-(16\alpha_5+4\gamma_5-2)x^3-18a_1yx^2-30a_1^2s_2x-6s_2y^2x-18a_1^2yx-\\
-6y^2x^2-4yx^3-18a_1^3x-18a_1^3\alpha_5-(36\alpha_5+24\gamma_5+36)a_1yx-(6\alpha_5^2-3)y^2-\\
-\big(18\alpha_5^3+(18\gamma_5+6)\alpha_5^2+(12\gamma_5^2+36\gamma_5+30)\alpha_5-12\gamma_5-12\big)a_1-9s_2^2x^2-\\
-(54\alpha_5+18\gamma_5+6)a_1s_2x-(12\alpha_5+12\gamma_5+18)s_2yx-(18\alpha_5+12\gamma_5+18)a_1ys_2-\\
-12a_1^2s_2^2-\big(30\alpha_5^2+(18\gamma_5+12)\alpha_5+3(\gamma_5+3)(\gamma_5+1)\big)a_1^2-18a_1x^3-\\
-(18\alpha_5+6\gamma_5+3)s_2^2x-\big(4\alpha_5^3+(12\gamma_5+18)\alpha_5^2-6\alpha_5-6\gamma_5-8\big)y.
\end{multline*}
All coefficients of this polynomial are negative, so that $f_5<0$ if $a_2-a_1>\alpha_5$ and $b>\gamma_5$.
Therefore, if $a_2-a_1\geqslant 0.7698$ and $b\geqslant 0.2308$, then $f_5<0$.

Now we let $f_4=f(a_1,a_2,a_3,a_4,1,1,1,b)$ and $g_4(a_1,s_1,s_2,s_3,b)=\widehat{f}_4$.
Then
$$
g_4(0,x,0,0,b)=(-9x^2+6)b^2+(-6x^3-27x^2+9x+16)b-9x^4+15x^3-45x^2+24x+8.
$$
Denote by $\alpha_4$ the unique positive root of the polynomial $-9x^4+15x^3-45x^2+24x+8$.
Then $\alpha_4\approx 0.8594363$, and the polynomial $g_4(a_1,x+\alpha_4,s_2,s_3,b)$ can be expanded as
\begin{multline*}
-4a_1^4-4s_2^4-s_3^4-(36\alpha_4^2-30\alpha_4+30)s_3x-12b^2s_2x-12bs_2^2x-\\
-(36\alpha_4-15)s_3x^2-36a_1^2s_2s_3-30a_1s_2^2s_3-24a_1s_2s_3^2-8s_2^3s_3-9s_2^2s_3^2-\\
-27a_1s_3^2x-5s_2s_3^3-26a_1^3s_2-13a_1^3s_3-36a_1^2s_2^2-15a_1^2s_3^2-20a_1s_2^3-7a_1s_3^3-\\
-6b^2s_3x-(24\alpha_4-3)s_3^2x-(72\alpha_4-30)s_2x^2-(72\alpha_4^2-60\alpha_4+60)s_2x-\\
-21s_2s_3^2x-(30\alpha_4^2-18\alpha_4+24)s_2s_3-(126\alpha_4-18)a_1^2x-(66\alpha_4-6)a_1s_2^2-\\
-(42\alpha_4-6)a_1^2s_3-(84\alpha_4-12)a_1^2s_2-(27\alpha_4+24)a_1^2b-12b^2s_2\alpha_4-\\
-(27\alpha_4-6)s_2^2s_3-(60\alpha_4-18)s_2^2x-(18\alpha_4^2+54\alpha_4-9)bx-(18\alpha_4+27)bx^2-\\
-(12\alpha_4^2+36\alpha_4-6)bs_2-(6\alpha_4+9)bs_3^2-(6\alpha_4^2+18\alpha_4-3)bs_3-\\
-(12\alpha_4+18)bs_2^2-(63\alpha_4^2-18\alpha_4+36)a_1^2-(39\alpha_4-8)a_1^3-(9\alpha_4^2-6)b^2-\\
-(39\alpha_4^3-27\alpha_4^2+81\alpha_4-32)a_1-(36\alpha_4^3-45\alpha_4^2+90\alpha_4-24)x-\\
-(54\alpha_4^2-45\alpha_4+45)x^2-(36\alpha_4-15)x^3-(12\alpha_4^3-15\alpha_4^2+30\alpha_4-8)s_3-\\
-(6\alpha_4+1)s_3^3-(12\alpha_4^2-3\alpha_4+9)s_3^2-(24\alpha_4^3-30\alpha_4^2+60\alpha_4-16)s_2-\\
-(30\alpha_4^2-18\alpha_4+24)s_2^2-(18\alpha_4-4)s_2^3-(6\alpha_4^3+27\alpha_4^2-9\alpha_4-16)b-\\
-(117\alpha_4-27)a_1x^2-(12\alpha_4+18)bs_2s_3-(24\alpha_4+36)bs_2x-(12\alpha_4+18)bs_3x-\\
-(60\alpha_4-18)s_2s_3x-(78\alpha_4-18)a_1s_3x-(156\alpha_4-36)a_1s_2x-(66\alpha_4-6)a_1s_2s_3-\\
-63a_1^2x^2-39a_1x^3-9b^2x^2-6bx^3-24s_2x^3-12s_3x^3-30s_2^2x^2-12s_3^2x^2-9x^4-\\
-30s_2s_3x^2-78a_1s_2x^2-27a_1bx^2-6bs_3x^2-39a_1s_3x^2-12bs_2x^2-42a_1^2s_3x-\\
-18a_1b^2x-27s_2^2s_3x-27a_1^2bx-39a_1^3x-18s_2^3x-6s_3^3x-(54\alpha_4+54)a_1bx-\\
-(18\alpha_4+18)a_1bs_3-(36\alpha_4+36)a_1bs_2-6bs_2^2s_3-18a_1^2bs_2-6a_1b^2s_3-12a_1b^2s_2-\\
-6b^2s_2s_3-6bs_2s_3^2-9a_1bs_3^2-18a_1bs_2^2-9a_1^2bs_3-2a_1^3b-6a_1^2b^2-6b^2s_2^2-\\
-66a_1s_2^2x-3b^2s_3^2-4bs_2^3-2bs_3^3-18a_1b^2\alpha_4-21s_2s_3^2\alpha_4-18b^2x\alpha_4-\\
-84a_1^2s_2x-6b^2s_3\alpha_4-36a_1bs_2x-18a_1bs_3x-66a_1s_2s_3x-12bs_2s_3x-18a_1bs_2s_3-\\
-(27\alpha_4^2+54\alpha_4-18)a_1b-(27\alpha_4+3)a_1s_3^2-(78\alpha_4^2-36\alpha_4+54)a_1s_2-\\
-6bs_3^2x-(117\alpha_4^2-54\alpha_4+81)a_1x-(39\alpha_4^2-18\alpha_4+27)a_1s_3.
\end{multline*}
Since all coefficients of this polynomial are negative, we see that $f_4<0$ when $a_2-a_1>\alpha_4$.
Since $\alpha_4<0.8595$, we also have $f_4<0$ in the case when $a_2-a_1\geqslant 0.8595$.

Now we let $f_3=f(a_1,a_2,a_3,a_4,a_5,1,1,b)$ and $g_3(a_1,s_1,s_2,s_3,s_4,b)=\widehat{f}_3$. Then
$$
g_3(0,x,0,0,0,b)=-(12x^2-9)b^2-(8x^3+36x^2-12x-24)b-16x^4+36x^3-84x^2+48x+9.
$$
Let $\alpha_3$ be the unique positive root of $g_3(0,x,0,0,0,0)$.
Then $\alpha_3\approx 0.8984970862$ and
\begin{multline*}
g_3(a_1,x+\alpha_3,s_2,s_3,s_4,b)=-(24\alpha_3+36)bs_2s_3-(12\alpha_3+18)bs_2s_4-10a_1^4-\\
-9s_2^4-4s_3^4-s_4^4-96a_1^2s_2s_3-84a_1s_2^2s_3-72a_1s_2s_3^2-24s_2^3s_3-30s_2^2s_3^2-18s_2s_3^3-\\
-51a_1^3s_2-34a_1^3s_3-72a_1^2s_2^2-42a_1^2s_3^2-42a_1s_2^3-22a_1s_3^3-(90\alpha_3-30)a_1s_2s_4-\\
-(144\alpha_3^2-162\alpha_3+126)s_2x-(48\alpha_3-27)s_4x^2-(108\alpha_3^2-72\alpha_3+75)a_1^2-\\
-(33\alpha_3-3)a_1s_4^2-(16\alpha_3^3-27\alpha_3^2+42\alpha_3-12)s_4-(15\alpha_3^2-9\alpha_3+12)s_4^2-\\
-(72\alpha_3-30)s_3s_4x-(84\alpha_3-42)s_2s_4x-(168\alpha_3-84)s_2s_3x-(64\alpha_3-36)x^3-\\
-(63\alpha_3^2-63\alpha_3+54)s_2^2-(48\alpha_3^3-81\alpha_3^2+126\alpha_3-36)s_2-(20\alpha_3-6)s_3^3-\\
-(36\alpha_3^2-30\alpha_3+30)s_3^2-(32\alpha_3^3-54\alpha_3^2+84\alpha_3-24)s_3-6bs_3^2s_4-9a_1^2bs_4-\\
-9a_1bs_4^2-30a_1s_2s_4^2-30s_2^2s_3s_4-21s_2s_3s_4^2-33a_1s_3^2s_4-6b^2s_2s_4-17a_1^3s_4-18a_1^2s_4^2-\\
-8a_1s_4^3-3b^2s_4^2-2bs_4^3-12s_2^3s_4-12s_2^2s_4^2-6s_2s_4^3-8s_3^3s_4-9s_3^2s_4^2-5s_3s_4^3-\\
-(108\alpha_3-36)a_1^2s_3-7s_4^3\alpha_3-(68\alpha_3-24)a_1^3-(54\alpha_3+54)a_1bs_2-(36\alpha_3+36)a_1bs_3-\\
-(36\alpha_3^2+72\alpha_3-24)a_1b-(18\alpha_3+18)a_1bs_4-(72\alpha_3+72)a_1bx-(180\alpha_3-60)a_1s_2s_3-\\
-(36\alpha_3^2-30\alpha_3+30)s_3s_4-18a_1bs_2s_4-18a_1bs_3s_4-72a_1s_2s_3s_4-12bs_2s_3s_4-\\
-(216\alpha_3-72)a_1^2x-(54\alpha_3-18)a_1^2s_4-16s_4x^3-15s_4^2x^2-7s_4^3x-36s_3s_4x^2-42s_2s_4x^2-\\
-6bs_4x^2-51a_1s_4x^2-6b^2s_4x-6bs_4^2x-30s_3^2s_4x-54a_1^2s_4x-24s_3s_4^2x-39s_2^2s_4x-27s_2s_4^2x-\\
-68a_1x^3-18a_1bs_4x-33a_1s_4^2x-(144\alpha_3-81)s_2x^2-(36\alpha_3+36)a_1^2b-(78\alpha_3-18)a_1s_3^2-\\
-(42\alpha_3^2-42\alpha_3+36)s_2s_4-(68\alpha_3^3-84\alpha_3^2+156\alpha_3-60)a_1-(39\alpha_3-18)s_2^3-\\
-(8\alpha_3^3+36\alpha_3^2-12\alpha_3-24)b-(12\alpha_3^2-9)b^2-(64\alpha_3^3-108\alpha_3^2+168\alpha_3-48)x-\\
-(96\alpha_3^2-108\alpha_3+84)x^2-(162\alpha_3-54)a_1^2s_2-(78\alpha_3-18)a_1s_3s_4-(24\alpha_3-3)s_3s_4^2-\\
-(204\alpha_3-84)a_1s_3x-(102\alpha_3-42)a_1s_4x-(30\alpha_3-9)s_3^2s_4-(27\alpha_3-6)s_2s_4^2-108a_1^2x^2-\\
-12b^2x^2-8bx^3-48s_2x^3-32s_3x^3-63s_2^2x^2-36s_3^2x^2-16x^4-84s_2s_3x^2-153a_1s_2x^2-36a_1bx^2-\\
-12bs_3x^2-102a_1s_3x^2-18bs_2x^2-108a_1^2s_3x-18b^2s_2x-18bs_2^2x-135a_1s_2^2x-66s_2s_3^2x-78a_1s_3^2x-\\
-12bs_3^2x-162a_1^2s_2x-24a_1b^2x-12b^2s_3x-78s_2^2s_3x-36a_1^2bx-68a_1^3x-39s_2^3x-20s_3^3x-\\
-6a_1b^2s_4-(12\alpha_3+18)bs_4x-(66\alpha_3-24)s_2s_3s_4-(24\alpha_3+36)bs_3x-(36\alpha_3+54)bs_2x-\\
-24bs_2s_3x-(12\alpha_3+18)bs_3s_4-(72\alpha_3-30)s_3^2x-(84\alpha_3^2-84\alpha_3+72)s_2s_3-12bs_2^2s_3-\\
-27a_1^2bs_2-12a_1b^2s_3-18a_1b^2s_2-12b^2s_2s_3-12bs_2s_3^2-18a_1bs_3^2-27a_1bs_2^2-18a_1^2bs_3-\\
-6b^2s_3s_4-6bs_3s_4^2-27a_1s_3s_4^2-48a_1^2s_2s_4-6bs_2s_4^2-42a_1^2s_3s_4-42a_1s_2^2s_4-27s_2s_3^2s_4-\\
-6bs_2^2s_4-4a_1^3b-9a_1^2b^2-9b^2s_2^2-6b^2s_3^2-6bs_2^3-4bs_3^3-(153\alpha_3^2-126\alpha_3+117)a_1s_2-\\
-(135\alpha_3-45)a_1s_2^2-(51\alpha_3^2-42\alpha_3+39)a_1s_4-(102\alpha_3^2-84\alpha_3+78)a_1s_3-\\
-(66\alpha_3-24)s_2s_3^2-(126\alpha_3-63)s_2^2x-(39\alpha_3-18)s_2^2s_4-(78\alpha_3-36)s_2^2s_3-\\
-(24\alpha_3^2+72\alpha_3-12)bx-(24\alpha_3+36)bx^2-(6\alpha_3^2+18\alpha_3-3)bs_4-(6\alpha_3+9)bs_4^2-\\
-(12\alpha_3^2+36\alpha_3-6)bs_3-(12\alpha_3+18)bs_3^2-(18\alpha_3^2+54\alpha_3-9)bs_2-(18\alpha_3+27)bs_2^2-\\
-(204\alpha_3^2-168\alpha_3+156)a_1x-(204\alpha_3-84)a_1x^2-(96\alpha_3^2-108\alpha_3+84)s_3x-\\
-90a_1s_2s_4x-(96\alpha_3-54)s_3x^2-(30\alpha_3-9)s_4^2x-12b^2s_3\alpha_3-18b^2s_2\alpha_3-24a_1b^2\alpha_3-\\
-24b^2x\alpha_3-6b^2s_4\alpha_3-(48\alpha_3^2-54\alpha_3+42)s_4x-54a_1bs_2x-36a_1bs_3x-180a_1s_2s_3x-\\
-36a_1bs_2s_3-78a_1s_3s_4x-12bs_2s_4x-12bs_3s_4x-66s_2s_3s_4x-(306\alpha_3-126)a_1s_2x.
\end{multline*}
Since all coefficients of this polynomial are negative, one has $f_3<0$ when $a_2-a_1>\alpha_3$.
The same holds for $a_2-a_1\geqslant 0.8985$, because $0.8985>\alpha_3$.

Now let us show that $f(a_1,a_2,a_3,a_4,a_5,a_6,1,b)<0$ in the case when $a_2-a_1\geqslant 0.9206$.
To to this, let $f_2=f(a_1,a_2,a_3,a_4,a_5,a_6,1,b)$ and let $g_2(a_1,s_1,s_2,s_3,s_4,s_5,b)=\widehat{f}_2$.
Then
$$
g_2(0,x,0,0,0,0,b)=-(15x^2-12)b^2-(10x^3+45x^2-15x-32)b-25x^4+65x^3-135x^2+80x+8.
$$
Let $\alpha_2$ be the unique positive root of $g_2(0,x,0,0,0,0,0)=-25x^4+65x^3-135x^2+80x+8$.
Then $\alpha_2\approx 0.9205032589$ and all coefficients of $g_3(a_1,x+\alpha_2,s_2,s_3,s_4,s_5,b)$ are negative.
Indeed, the polynomial $g_2(a_1,x+\alpha_2,s_2,s_3,s_4,s_5,b)$ can be expanded as
\begin{multline*}
-(30\alpha_2+45)bx^2-(18\alpha_2^2+54\alpha_2-9)bs_3-(18\alpha_2+27)bs_3^2-84s_2^2s_3s_4-\\
-(252\alpha_2^2-264\alpha_2+204)a_1s_2-(228\alpha_2-108)a_1s_2^2-(12\alpha_2^2+36\alpha_2-6)bs_4-27s_2s_3s_5^2-\\
-(18\alpha_2+18)a_1bs_5-(30\alpha_2^2+90\alpha_2-15)bx-(36\alpha_2+36)a_1bs_4-(12\alpha_2+18)bs_4^2-\\
-(192\alpha_2-108)s_3s_4x-(72\alpha_2-30)s_3s_4s_5-(108\alpha_2-66)s_2s_5x-(216\alpha_2-132)s_2s_4x-\\
-(48\alpha_2^2-54\alpha_2+42)s_3s_5-(45\alpha_2^2+90\alpha_2-30)ba_1-(126\alpha_2^2-132\alpha_2+102)a_1s_4-\\
-(72\alpha_2+72)a_1bs_2-(24\alpha_2^2+72\alpha_2-12)bs_2-(24\alpha_2+36)bs_2^2-(6\alpha_2^2+18\alpha_2-3)bs_5-\\
-(105\alpha_2-48)a_1^3-(165\alpha_2^2-150\alpha_2+126)a_1^2-(105\alpha_2^3-165\alpha_2^2+255\alpha_2-96)a_1-\\
-(15\alpha_2^2-12)b^2-(108\alpha_2^2-132\alpha_2+96)s_2^2-(68\alpha_2-40)s_2^3-(66\alpha_2-30)a_1^2s_5-78s_2s_3^2s_4-\\
-(198\alpha_2-90)a_1^2s_3-(315\alpha_2^2-330\alpha_2+255)a_1x-(60\alpha_2^2-78\alpha_2+54)s_5x-42s_2^2s_3s_5-36s_2^2s_4s_5-\\
-(60\alpha_2-39)s_5x^2-(36\alpha_2-15)s_5^2x-(264\alpha_2-120)a_1^2s_2-(45\alpha_2+48)a_1^2b-(78\alpha_2-36)s_2s_4^2-\\
-(96\alpha_2-54)s_3s_5x-(120\alpha_2^2-156\alpha_2+108)s_4x-(120\alpha_2-78)s_4x^2-(96\alpha_2^2-108\alpha_2+84)s_3s_4-\\
-(72\alpha_2-30)s_3s_4^2-(144\alpha_2-81)s_3^2x-(189\alpha_2^2-198\alpha_2+153)a_1s_3-(42\alpha_2^2-42\alpha_2+36)s_4s_5-\\
-(42\alpha_2-21)s_3^2s_5-(84\alpha_2-42)s_3^2s_4-(240\alpha_2^2-312\alpha_2+216)s_2x-(153\alpha_2-63)a_1s_3^2-\\
-6b^2s_5\alpha_2-12b^2s_4\alpha_2-18b^2s_3\alpha_2-24b^2s_2\alpha_2-30a_1b^2\alpha_2-(240\alpha_2-156)s_2x^2-\\
-(54\alpha_2^2-66\alpha_2+48)s_2s_5-(33\alpha_2-12)s_2s_5^2-(108\alpha_2^2-132\alpha_2+96)s_2s_4-66s_2s_3s_4^2-\\
-(90\alpha_2-30)a_1s_4s_5-(378\alpha_2-198)a_1s_3x-(135\alpha_2-72)s_2s_3^2-(216\alpha_2-132)s_2^2x-\\
-(204\alpha_2-84)a_1s_3s_4-(504\alpha_2-264)a_1s_2x-(114\alpha_2-54)a_1s_2s_5-(228\alpha_2-108)a_1s_2s_4-\\
-(324\alpha_2-198)s_2s_3x-(90\alpha_2-48)s_2s_3s_5-(180\alpha_2-96)s_2s_3s_4-(12\alpha_2+18)bs_5x-\\
-(12\alpha_2+18)bs_4s_5-(51\alpha_2-30)s_2^2s_5-(102\alpha_2-60)s_2^2s_4-(153\alpha_2-90)s_2^2s_3-\\
-24s_2s_4s_5^2-30s_3^2s_4s_5-27s_3s_4^2s_5-21s_3s_4s_5^2-42a_1^3s_4-16s_2^4-9s_3^4-4s_4^4-s_5^4-\\
-7s_2s_5^3-84a_1s_2s_4s_5-78a_1s_3s_4s_5-66s_2s_3s_4s_5-54a_1bs_2s_3-36a_1bs_2s_4-18a_1bs_2s_5-\\
-18a_1bs_4s_5-24bs_2s_3s_4-12bs_2s_3s_5-12bs_2s_4s_5-12bs_3s_4s_5-21a_1^3s_5-120a_1^2s_2^2-\\
-36a_1bs_3s_4-81a_1^2s_3^2-(180\alpha_2-117)s_3x^2-30s_2s_4^2s_5-(6\alpha_2+9)bs_5^2-192a_1s_2s_3s_4-\\
-(30\alpha_2-9)s_3s_5^2-30b^2x\alpha_2-(126\alpha_2-66)a_1s_5x-(24\alpha_2+36)bs_4x-(132\alpha_2-60)a_1^2s_4-\\
-(90\alpha_2-30)a_1s_4^2-(78\alpha_2-36)s_2s_4s_5-(330\alpha_2-150)a_1^2x-(90\alpha_2+90)ba_1x-39s_2s_3^2s_5-
\end{multline*}

\begin{multline*}
-48a_1^2s_4^2-21a_1^2s_5^2-72a_1s_2^3-45a_1s_3^3-24a_1s_4^3-9a_1s_5^3-48s_2^3s_3-32s_2^3s_4-16s_2^3s_5-\\
-(54\alpha_2+54)a_1bs_3-96a_1s_2s_3s_5-18a_1bs_3s_5-(252\alpha_2-132)a_1s_4x-(102\alpha_2-42)a_1s_3s_5-\\
-15s_2^2s_5^2-6s_3s_5^3-8s_4^3s_5-9s_4^2s_5^2-5s_4s_5^3-(84\alpha_2-42)s_4s_5x-180a_1^2s_2s_3-120a_1^2s_2s_4-\\
-108a_1^2s_3s_4-54a_1^2s_3s_5-48a_1^2s_4s_5-162a_1s_2^2s_3-108a_1s_2^2s_4-54a_1s_2^2s_5-144a_1s_2s_3^2-\\
-36a_1s_2s_5^2-90a_1s_3^2s_4-45a_1s_3^2s_5-78a_1s_3s_4^2-33a_1s_3s_5^2-36a_1s_4^2s_5-9a_1^2bs_5-\\
-30s_3^2s_4^2-18a_1^2bs_4-27a_1^2bs_3-36a_1^2bs_2-(33\alpha_2-12)s_4^2s_5-84a_1^3s_2-(36\alpha_2+54)bs_2s_3-\\
-(180\alpha_2^2-234\alpha_2+162)s_3x-18a_1^4-6a_1^3b-12a_1^2b^2-(63\alpha_2^2-66\alpha_2+51)a_1s_5-\\
-(39\alpha_2-9)a_1s_5^2-(72\alpha_2^2-81\alpha_2+63)s_3^2-(42\alpha_2-21)s_3^3-(84\alpha_2-42)s_4^2x-\\
-39s_2s_3^3-(315\alpha_2-165)a_1x^2-(48\alpha_2+72)bs_2x-(24\alpha_2+36)bs_3s_4-(12\alpha_2+18)bs_3s_5-\\
-12s_3^3s_5-(80\alpha_2^3-156\alpha_2^2+216\alpha_2-64)s_2-(100\alpha_2^3-195\alpha_2^2+270\alpha_2-80)x-\\
-(100\alpha_2-65)x^3-(20\alpha_2^3-39\alpha_2^2+54\alpha_2-16)s_5-(18\alpha_2^2-15\alpha_2+15)s_5^2-\\
-(40\alpha_2^3-78\alpha_2^2+108\alpha_2-32)s_4-(42\alpha_2^2-42\alpha_2+36)s_4^2-(22\alpha_2-8)s_4^3-18bs_3x^2-\\
-20s_2s_4^3-162s_2s_3x^2-108s_2s_4x^2-54s_2s_5x^2-96s_3s_4x^2-48s_3s_5x^2-42s_4s_5x^2-252a_1s_2x^2-\\
-126a_1s_4x^2-63a_1s_5x^2-24bs_2x^2-45a_1bx^2-84s_3^2s_4x-42s_3^2s_5x-72s_3s_4^2x-30s_3s_5^2x-\\
-27s_4s_5^2x-153a_1s_3^2x-90a_1s_4^2x-39a_1s_5^2x-24b^2s_2x-18b^2s_3x-12b^2s_4x-6b^2s_5x-\\
-132a_1^2s_4x-24bs_2^2x-18bs_3^2x-12bs_4^2x-6bs_5^2x-153s_2^2s_3x-102s_2^2s_4x-51s_2^2s_5x-\\
-78s_2s_4^2x-33s_2s_5^2x-(60\alpha_2^3-117\alpha_2^2+162\alpha_2-48)s_3-165a_1^2x^2-105a_1x^3-\\
-10bx^3-80s_2x^3-60s_3x^3-40s_4x^3-20s_5x^3-108s_2^2x^2-72s_3^2x^2-42s_4^2x^2-18s_5^2x^2-68s_2^3x-\\
-22s_4^3x-8s_5^3x-105a_1^3x-228a_1s_2s_4x-114a_1s_2s_5x-204a_1s_3s_4x-102a_1s_3s_5x-90a_1s_4s_5x-\\
-24bs_2s_4x-12bs_2s_5x-24bs_3s_4x-12bs_3s_5x-12bs_4s_5x-180s_2s_3s_4x-90s_2s_3s_5x-78s_2s_4s_5x-\\
-72a_1bs_2x-54a_1bs_3x-36a_1bs_4x-18a_1bs_5x-342a_1s_2s_3x-66a_1^2s_5x-30a_1b^2x-264a_1^2s_2x-\\
-228a_1s_2^2x-6bs_4^2s_5-6bs_4s_5^2-12bs_2s_4^2-24a_1b^2s_2-18bs_2^2s_3-18a_1b^2s_3-6a_1b^2s_5-\\
-18bs_2s_3^2-12bs_2^2s_4-36a_1bs_2^2-6bs_2^2s_5-18a_1bs_4^2-12a_1b^2s_4-6b^2s_2s_5-6b^2s_4s_5-\\
-18b^2s_2s_3-6bs_3s_5^2-6bs_3^2s_5-12b^2s_2s_4-6bs_2s_5^2-9a_1bs_5^2-12bs_3s_4^2-12b^2s_3s_4-\\
-36s_2^2s_4^2-(10\alpha_2^3+45\alpha_2^2-15\alpha_2-32)b-12b^2s_2^2-9b^2s_3^2-6b^2s_4^2-3b^2s_5^2-8bs_2^3-\\
-24s_3^3s_4-189a_1s_3x^2-12bs_3^2s_4-45a_1^2bx-42s_3^3x-33s_4^2s_5x-(150\alpha_2^2-195\alpha_2+135)x^2-\\
-27a_1bs_3^2-(8\alpha_2-1)s_5^3-36bs_2s_3x-198a_1^2s_3x-25x^4-63a_1^3s_3-30a_1s_4s_5^2-63s_2^2s_3^2-\\
-6b^2s_3s_5-(24\alpha_2+36)bs_2s_4-72s_3s_4s_5x-12bs_4x^2-(36\alpha_2+54)bs_3x-60a_1^2s_2s_5-\\
-6bs_3^3-15b^2x^2-135s_2s_3^2x-(12\alpha_2+18)bs_2s_5-6bs_5x^2-84a_1s_2s_4^2-18s_3s_4^3-12s_3^2s_5^2-\\
-4bs_4^3-2bs_5^3-(342\alpha_2-162)a_1s_2s_3-(27\alpha_2-6)s_4s_5^2-(162\alpha_2^2-198\alpha_2+144)s_2s_3.
\end{multline*}
This implies that $f_2<0$ if $a_2-a_1>\alpha_2$. In particular, if $a_2-a_1\geqslant0.9206$, then $f_2<0$.

Finally, let $g_1(a_1,s_1,s_2,s_3,s_4,s_5,s_6,b)=\widehat{f}$. Then
$$
g_1(0,x,0,0,0,0,0,0)=-36x^4+102x^3-198x^2+120x+5.
$$
This polynomial has unique positive root. Let us denote it by $\alpha_1$.
Then~$\alpha_1\approx0.934647387$.
Moreover, the polynomial $g_3(a_1,x+\alpha_1,s_2,s_3,s_4,s_5,s_6,b)$ can be expanded as
\begin{multline*}
-(48\alpha_1^3-102\alpha_1^2+132\alpha_1-40)s_5-(18\alpha_1+18)a_1bs_6-(126\alpha_1-66)a_1s_3s_6-\\
-(414\alpha_1-234)a_1s_2s_4-(600\alpha_1-360)a_1s_3x-(552\alpha_1-312)a_1s_2s_3-(108\alpha_1+108)ba_1x-\\
-(48\alpha_1+72)bs_3x-(12\alpha_1+18)bs_3s_6-(54\alpha_1+54)a_1bs_4-(72\alpha_1+72)a_1bs_3-(24\alpha_1+36)bs_4s_5-\\
-(252\alpha_1-132)a_1s_3s_5-(378\alpha_1-198)a_1s_3s_4-(12\alpha_1+18)bs_6x-(72\alpha_1^2-102\alpha_1+66)s_6x-\\
-(48\alpha_1^2-54\alpha_1+42)s_5^2-(750\alpha_1-450)a_1s_2x-(138\alpha_1-78)a_1s_2s_6-(24\alpha_1+36)bs_5x-\\
-552a_1s_2s_3x-(12\alpha_1+18)bs_5s_6-(36\alpha_1+54)bs_4x-(12\alpha_1+18)bs_4s_6-(36\alpha_1+36)a_1bs_5-\\
-125a_1^3s_2-100a_1^3s_3-75a_1^3s_4-50a_1^3s_5-360a_1s_2s_3s_4-240a_1s_2s_3s_5-120a_1s_2s_3s_6-\\
-216a_1s_2s_4s_5-108a_1s_2s_4s_6-96a_1s_2s_5s_6-204a_1s_3s_4s_5-102a_1s_3s_4s_6-90a_1s_3s_5s_6-\\
-54a_1bs_2s_4-84a_1s_4s_5s_6-180s_2s_3s_4s_5-90s_2s_3s_4s_6-78s_2s_3s_5s_6-72s_2s_4s_5s_6-66s_3s_4s_5s_6-\\
-(450\alpha_1-270)a_1s_4x-(240\alpha_1-156)s_3s_5x-(90\alpha_1+90)a_1bs_2-25s_2^4-16s_3^4-9s_4^4-\\
-216a_1^2s_2s_4-4s_5^4-s_6^4-39a_1s_5^2s_6-33a_1s_5s_6^2-162s_2^2s_3s_4-108s_2^2s_3s_5-54s_2^2s_3s_6-\\
-288a_1^2s_2s_3-96s_2^2s_4s_5-48s_2^2s_4s_6-42s_2^2s_5s_6-153s_2s_3^2s_4-102s_2s_3^2s_5-51s_2s_3^2s_6-\\
-72a_1bs_2s_3-21s_4s_5s_6^2-135s_2s_3s_4^2-78s_2s_3s_5^2-33s_2s_3s_6^2-84s_2s_4^2s_5-42s_2s_4^2s_6-72s_2s_4s_5^2-\\
-12bs_4s_5s_6-27s_4s_5^2s_6-30s_2s_4s_6^2-33s_2s_5^2s_6-27s_2s_5s_6^2-84s_3^2s_4s_5-42s_3^2s_4s_6-36s_3^2s_5s_6-\\
-30s_4^2s_5s_6-78s_3s_4^2s_5-39s_3s_4^2s_6-66s_3s_4s_5^2-27s_3s_4s_6^2-30s_3s_5^2s_6-24s_3s_5s_6^2-\\
-144a_1^2s_2s_5-72a_1^2s_2s_6-198a_1^2s_3s_4-132a_1^2s_3s_5-66a_1^2s_3s_6-120a_1^2s_4s_5-60a_1^2s_4s_6-\\
-54a_1^2s_5s_6-264a_1s_2^2s_3-198a_1s_2^2s_4-132a_1s_2^2s_5-66a_1s_2^2s_6-240a_1s_2s_3^2-162a_1s_2s_4^2-\\
-96a_1s_2s_5^2-42a_1s_2s_6^2-171a_1s_3^2s_4-114a_1s_3^2s_5-57a_1s_3^2s_6-153a_1s_3s_4^2-90a_1s_3s_5^2-\\
-39a_1s_3s_6^2-96a_1s_4^2s_5-48a_1s_4^2s_6-84a_1s_4s_5^2-36a_1s_4s_6^2-18a_1^2bs_5-9a_1^2bs_6-36a_1^2bs_3-\\
-45a_1^2bs_2-27a_1^2bs_4-(24\alpha_1^3-51\alpha_1^2+66\alpha_1-20)s_6-(21\alpha_1^2-21\alpha_1+18)s_6^2-\\
-(216\alpha_1^2-306\alpha_1+198)x^2-(216\alpha_1-132)s_4s_5x-(9\alpha_1-2)s_6^3-(300\alpha_1-180)a_1s_5x-\\
-342s_2s_3s_4x-(78\alpha_1-36)s_4s_5s_6-(120\alpha_1-78)s_3s_6x-63s_3^2s_4^2-36s_3^2s_5^2-15s_3^2s_6^2-39s_3s_4^3-\\
-20s_3s_5^3-7s_3s_6^3-24s_4^3s_5-12s_4^3s_6-30s_4^2s_5^2-12s_4^2s_6^2-18s_4s_5^3-6s_4s_6^3-8s_5^3s_6-9s_5^2s_6^2-\\
-12bs_3s_5s_6-5s_5s_6^3-180a_1^2s_2^2-132a_1^2s_3^2-90a_1^2s_4^2-54a_1^2s_5^2-24a_1^2s_6^2-110a_1s_2^3-76a_1s_3^3-\\
-12bs_3s_4s_6-48a_1s_4^3-26a_1s_5^3-10a_1s_6^3-80s_2^3s_3-60s_2^3s_4-40s_2^3s_5-20s_2^3s_6-108s_2^2s_3^2-72s_2^2s_4^2-\\
-24bs_3s_4s_5-42s_2^2s_5^2-18s_2^2s_6^2-68s_2s_3^3-42s_2s_4^3-22s_2s_5^3-8s_2s_6^3-48s_3^3s_4-32s_3^3s_5-16s_3^3s_6-\\
-(96\alpha_1-54)s_5s_6x-25a_1^3s_6-(150\alpha_1-90)a_1s_6x-(204\alpha_1-120)s_2s_4s_5-(102\alpha_1-60)s_2s_4s_6-\\
-(396\alpha_1-270)s_2s_4x-(90\alpha_1-48)s_2s_5s_6-(102\alpha_1-42)a_1s_5s_6-(264\alpha_1-180)s_2s_5x-\\
-12bs_2s_5s_6-(114\alpha_1-54)a_1s_4s_6-(144\alpha_1-102)x^3-(108\alpha_1-66)s_4s_6x-28a_1^4-8a_1^3b-15a_1^2b^2-
\end{multline*}

\begin{multline*}
-6b^2s_4s_6-(48\alpha_1+72)bs_2s_3-(36\alpha_1+54)bs_2s_4-(24\alpha_1+36)bs_2s_5-(12\alpha_1+18)bs_2s_6-\\
-(276\alpha_1-156)a_1s_2s_5-(60\alpha_1+90)bs_2x-(36\alpha_1+54)bs_3s_4-(24\alpha_1+36)bs_3s_5-\\
-(84\alpha_1-42)s_3s_5s_6-(360\alpha_1-234)s_3s_4x-(96\alpha_1-54)s_3s_4s_6-24bs_3x^2-18bs_4x^2-12bs_5x^2-\\
-6bs_6x^2-264s_2s_3x^2-198s_2s_4x^2-132s_2s_5x^2-66s_2s_6x^2-180s_3s_4x^2-120s_3s_5x^2-60s_3s_6x^2-\\
-108s_4s_5x^2-54s_4s_6x^2-48s_5s_6x^2-54a_1bx^2-375a_1s_2x^2-300a_1s_3x^2-225a_1s_4x^2-\\
-75a_1s_6x^2-30bs_2x^2-36x^4-234a_1^2x^2-150a_1x^3-18b^2x^2-12bx^3-120s_2x^3-96s_3x^3-72s_4x^3-\\
-48s_5x^3-24s_6x^3-165s_2^2x^2-120s_3^2x^2-81s_4^2x^2-48s_5^2x^2-21s_6^2x^2-12bs_4s_6x-12bs_5s_6x-\\
-252a_1s_3s_5x-228a_1s_4s_5x-114a_1s_4s_6x-102a_1s_5s_6x-48bs_2s_3x-36bs_2s_4x-24bs_2s_5x-12bs_2s_6x-\\
-138a_1s_2s_6x-36bs_3s_4x-24bs_3s_5x-12bs_3s_6x-24bs_4s_5x-228s_2s_3s_5x-114s_2s_3s_6x-204s_2s_4s_5x-\\
-90s_2s_5s_6x-192s_3s_4s_5x-96s_3s_4s_6x-84s_3s_5s_6x-78s_4s_5s_6x-90a_1bs_2x-72a_1bs_3x-54a_1bs_4x-\\
-36a_1bs_5x-18a_1bs_6x-105s_2^3x-72s_3^3x-45s_4^3x-24s_5^3x-9s_6^3x-150a_1^3x-54s_3^2s_6x-\\
-84s_3s_5^2x-36s_3s_6^2x-90s_4^2s_5x-45s_4^2s_6x-24b^2s_3x-18b^2s_4x-12b^2s_5x-6b^2s_6x-30bs_2^2x-\\
-378a_1s_3s_4x-24bs_3^2x-18bs_4^2x-12bs_5^2x-6bs_6^2x-252s_2^2s_3x-189s_2^2s_4x-126s_2^2s_5x-63s_2^2s_6x-\\
-153s_2s_4^2x-90s_2s_5^2x-39s_2s_6^2x-162s_3^2s_4x-108s_3^2s_5x-54a_1^2bx-390a_1^2s_2x-312a_1^2s_3x-\\
-234a_1^2s_4x-156a_1^2s_5x-78a_1^2s_6x-36a_1b^2x-345a_1s_2^2x-252a_1s_3^2x-171a_1s_4^2x-102a_1s_5^2x-\\
-45a_1s_6^2x-30b^2s_2x-78s_4s_5^2x-33s_4s_6^2x-36s_5^2s_6x-30s_5s_6^2x-(192\alpha_1-108)s_3s_4s_5-\\
-(528\alpha_1-360)s_2s_3x-12b^2s_4s_5-6b^2s_5s_6-30a_1b^2s_2-12b^2s_3s_5-12bs_2s_5^2-18a_1b^2s_4-\\
-24a_1b^2s_3-9a_1bs_6^2-6a_1b^2s_6-18b^2s_3s_4-36a_1bs_3^2-6bs_3s_6^2-12bs_3^2s_5-18bs_3s_4^2-\\
-12a_1b^2s_5-24b^2s_2s_3-24bs_2s_3^2-6bs_2s_6^2-12bs_2^2s_5-6bs_4s_6^2-27a_1bs_4^2-24bs_2^2s_3-\\
-378a_1s_3s_4x-18a_1bs_5^2-6bs_5s_6^2-(72\alpha_1^3-153\alpha_1^2+198\alpha_1-60)s_4-(228\alpha_1-108)a_1s_4s_5-\\
-414a_1s_2s_4x-15b^2s_2^2-12b^2s_3^2-9b^2s_4^2-6b^2s_5^2-3b^2s_6^2-10bs_2^3-8bs_3^3-6bs_4^3-4bs_5^3-2bs_6^3-\\
-12b^2s_2s_5-(114\alpha_1-72)s_2s_3s_6-(228\alpha_1-144)s_2s_3s_5-(342\alpha_1-216)s_2s_3s_4-18a_1bs_3s_6-\\
-18a_1bs_4s_6-18a_1bs_5s_6-36bs_2s_3s_4-24bs_2s_3s_5-12bs_2s_3s_6-24bs_2s_4s_5-(150\alpha_1-80)a_1^3-\\
-36a_1bs_2s_5-18a_1bs_2s_6-54a_1bs_3s_4-36a_1bs_3s_5-(180\alpha_1^2-234\alpha_1+162)s_3s_4-12bs_2s_4s_6-\\
-36a_1bs_4s_5-144s_3s_4^2x-(24\alpha_1-10)s_5^3-(132\alpha_1-90)s_2s_6x-102s_2s_4s_6x-150a_1s_5x^2-\\
-45a_1bs_2^2-(84\alpha_1-42)s_3s_5^2-(120\alpha_1^2-156\alpha_1+108)s_3s_5-(36\alpha_1-15)s_3s_6^2-228s_2s_3^2x-\\
-276a_1s_2s_5x-(60\alpha_1^2-78\alpha_1+54)s_3s_6-(288\alpha_1-204)s_3x^2-(288\alpha_1^2-408\alpha_1+264)s_3x-\\
-(345\alpha_1-195)a_1s_2^2-(375\alpha_1^2-450\alpha_1+315)a_1s_2-(252\alpha_1-132)a_1s_3^2-24b^2s_3\alpha_1-\\
-18bs_2s_4^2-(300\alpha_1^2-360\alpha_1+252)a_1s_3-(171\alpha_1-81)a_1s_4^2-(225\alpha_1^2-270\alpha_1+189)a_1s_4-\\
-126a_1s_3s_6x-(102\alpha_1-42)a_1s_5^2-(54\alpha_1+60)a_1^2b-(390\alpha_1-210)a_1^2s_2-(312\alpha_1-168)a_1^2s_3-\\
-12b^2s_5\alpha_1-6b^2s_6\alpha_1-36b^2x\alpha_1-36a_1b^2\alpha_1-(144\alpha_1^3-306\alpha_1^2+396\alpha_1-120)x-\\
-6b^2s_2s_6-(120\alpha_1^2-156\alpha_1+108)s_3^2-(96\alpha_1^3-204\alpha_1^2+264\alpha_1-80)s_3-(45\alpha_1-24)s_4^3-\\
-6bs_4^2s_6-(81\alpha_1^2-99\alpha_1+72)s_4^2-(120\alpha_1^3-255\alpha_1^2+330\alpha_1-100)s_2-(72\alpha_1-44)s_3^3-\\
-12bs_3s_5^2-(12\alpha_1^3+54\alpha_1^2-18\alpha_1-40)b-(105\alpha_1-70)s_2^3-(165\alpha_1^2-225\alpha_1+150)s_2^2-\\
-18bs_2^2s_4-(108\alpha_1^2-132\alpha_1+96)s_4s_5-(33\alpha_1-12)s_4s_6^2-(54\alpha_1^2-66\alpha_1+48)s_4s_6-\\
-18b^2s_2s_4-(90\alpha_1-48)s_4^2s_5-(45\alpha_1-24)s_4^2s_6-(162\alpha_1-99)s_4^2x-(78\alpha_1-36)s_4s_5^2-
\end{multline*}

\begin{multline*}
-(216\alpha_1-153)s_4x^2-(216\alpha_1^2-306\alpha_1+198)s_4x-(36\alpha_1-15)s_5^2s_6-(96\alpha_1-54)s_5^2x-\\
-(30\alpha_1-9)s_5s_6^2-(48\alpha_1^2-54\alpha_1+42)s_5s_6-(144\alpha_1-102)s_5x^2-(72\alpha_1-51)s_6x^2-\\
-6bs_2^2s_6-(144\alpha_1^2-204\alpha_1+132)s_5x-(42\alpha_1-21)s_6^2x-(36\alpha_1^2+108\alpha_1-18)bx-\\
-6bs_3^2s_6-(252\alpha_1-168)s_2^2s_3-(189\alpha_1-126)s_2^2s_4-(126\alpha_1-84)s_2^2s_5-(63\alpha_1-42)s_2^2s_6-\\
-(330\alpha_1-225)s_2^2x-(228\alpha_1-144)s_2s_3^2-(264\alpha_1^2-360\alpha_1+240)s_2s_3-30b^2s_2\alpha_1-\\
-(153\alpha_1-90)s_2s_4^2-(198\alpha_1^2-270\alpha_1+180)s_2s_4-(90\alpha_1-48)s_2s_5^2-18b^2s_4\alpha_1-\\
-18bs_3^2s_4-(132\alpha_1^2-180\alpha_1+120)s_2s_5-(39\alpha_1-18)s_2s_6^2-(66\alpha_1^2-90\alpha_1+60)s_2s_6-\\
-12bs_4s_5^2-(360\alpha_1-255)s_2x^2-(360\alpha_1^2-510\alpha_1+330)s_2x-(162\alpha_1-99)s_3^2s_4-(108\alpha_1-66)s_3^2s_5-\\
-(54\alpha_1-33)s_3^2s_6-(240\alpha_1-156)s_3^2x-(144\alpha_1-81)s_3s_4^2-(150\alpha_1^2-180\alpha_1+126)a_1s_5-\\
-(45\alpha_1-15)a_1s_6^2-(75\alpha_1^2-90\alpha_1+63)a_1s_6-(450\alpha_1-270)a_1x^2-(450\alpha_1^2-540\alpha_1+378)a_1x-\\
-(30\alpha_1+45)bs_2^2-(30\alpha_1^2+90\alpha_1-15)bs_2-(24\alpha_1+36)bs_3^2-(24\alpha_1^2+72\alpha_1-12)bs_3-\\
-6bs_5^2s_6-(6\alpha_1+9)bs_6^2-(6\alpha_1^2+18\alpha_1-3)bs_6-(36\alpha_1+54)bx^2-(234\alpha_1-126)a_1^2s_4-\\
-12bs_4^2s_5-(156\alpha_1-84)a_1^2s_5-(78\alpha_1-42)a_1^2s_6-(468\alpha_1-252)a_1^2x-(54\alpha_1^2+108\alpha_1-36)ba_1-\\
-6b^2s_3s_6-(18\alpha_1^2-15)b^2-(150\alpha_1^3-270\alpha_1^2+378\alpha_1-140)a_1-(234\alpha_1^2-252\alpha_1+189)a_1^2-\\
-(18\alpha_1+27)bs_4^2-(18\alpha_1^2+54\alpha_1-9)bs_4-(12\alpha_1+18)bs_5^2-(12\alpha_1^2+36\alpha_1-6)bs_5.
\end{multline*}
All coefficients of this polynomial are negative, so that $f<0$ in the case when $a_2-a_1>\alpha_1$.
In particular, one has $f<0$ in the case when $a_2-a_1\geqslant 0.9347$.
\end{proof}

\begin{lemma}
\label{lemma:Maple-P1-P1-d-5}
Suppose that $f$ is \eqref{equation:polynomial-P1xP1-d-5}.
If $a_2-a_1\geqslant 0.7452$, then $f<0$.
\end{lemma}

\begin{proof}
To show that $f(a_1,a_2,a_3,b)<0$ for $a_2-a_1\geqslant 0.7452$, let $g(a_1,s_1,s_2,b)=\widehat{f}$.
Let~$\delta$ the only positive root of the polynomial $g(0,x,0,0)=-6x^4-18x^3-3x^2+8x+5$.
Then~$\delta\approx 0.7451024$ and $g(a_1,x+\delta,s_2,b)$ can be expanded as
\begin{multline*}
-18a_1^2s_2^2-6a_1s_2^3-16a_1^3s_2-s_2^4-(18\delta^2+18\delta-3)bs_2-30a_1s_2^2x-12a_1b^2x-51a_1^2s_2x-\\
-(90\delta^2+132\delta)a_1x-(51\delta^2+72\delta-6)a_1^2-(90\delta+66)a_1x^2-(45\delta^2+66\delta)a_1s_2-\\
-(42\delta^2+36\delta-12)a_1b-(102\delta+72)a_1^2x-(30\delta+24)a_1s_2^2-(42\delta+12)a_1^2b-(51\delta+36)a_1^2s_2-\\
-(36\delta^2+54\delta+3)s_2x-(36\delta+27)s_2x^2-(36\delta^2+36\delta-6)bx-(24\delta+18)s_2^2x-3a_1^4-\\
-(36\delta+18)bx^2-45a_1s_2x^2-42a_1bx^2-18bs_2x^2-6x^4-12s_2x^3-12s_2^2x^2-51a_1^2x^2-21a_1^2bs_2-\\
-30a_1x^3-6b^2x^2-12bx^3-5s_2^3x-32a_1^3x-8a_1^3b-3a_1^2b^2-3b^2s_2^2-2bs_2^3-42a_1^2bx-12bs_2^2x-6b^2s_2x-\\
-(42\delta+18)a_1bs_2-(84\delta+36)a_1bx-(90\delta+66)a_1s_2x-(36\delta+18)bs_2x-42a_1bs_2x-15a_1bs_2^2-\\
-(12\delta+9)bs_2^2-6b^2s_2\delta-(24\delta^3+54\delta^2+6\delta-8)x-(36\delta^2+54\delta+3)x^2-(24\delta+18)x^3-\\
-12b^2x\delta-12a_1b^2\delta-(12\delta^3+27\delta^2+3\delta-4)s_2-(12\delta^2+18\delta+6)s_2^2-(5\delta+2)s_2^3-\\
-(32\delta+20)a_1^3-(12\delta^3+18\delta^2-6\delta-8)b-(6\delta^2-3)b^2-(30\delta^3+66\delta^2-12)a_1-6a_1b^2s_2.
\end{multline*}
All coefficients of here are negative, so that $g(a_1,s_1,s_2,b)<0$ whenever $s_1>\delta$.
\end{proof}

\begin{lemma}
\label{lemma:Maple-P1-P1-d-4-b}
Suppose that $f$ is \eqref{equation:polynomial-P1xP1-d-4-b}.
If $a_2-a_1\geqslant 0.848$ , then $f<0$.
\end{lemma}

\begin{proof}
Let $g(a_1,s_1,s_2,s_3,b)=\widehat{f}$. Then
$$
g(0,x,0,0,0)=7+21x-\frac{57}{2}x^3-\frac{9}{4}x^2-\frac{45}{4}x^4.
$$
Denote by $\delta$ the unique positive root of this polynomial.
Then $\delta\approx 0.84790543$ and
\begin{multline*}
4g(a_1,x+\delta,s_2,s_3,b)=-168a_1bs_2s_3-24s_2^4-3s_3^4-72bs_2^2s_3-\\
-48a_1b^2s_2-24a_1b^2s_3-24b^2s_2s_3-96a_1^2bs_3-192a_1^2bs_2-210a_1^2s_2s_3-\\
-120a_1s_2s_3^2-18s_2s_3^3-48bs_2^3-24b^2s_2^2-12bs_3^3-12b^2s_3^2-210a_1^2s_2^2-\\
-30a_1s_3^3-48s_2^3s_3-42s_2^2s_3^2-76a_1^3s_3-152a_1^3s_2-(84\delta^3+108\delta^2-72\delta-56)b-\\
-(222\delta^3+414\delta^2-36\delta-112)a_1-(180\delta^3+342\delta^2+18\delta-84)x-\\
-(246\delta+156)a_1^2s_3-(24\delta+18)s_3^3-(120\delta^3+228\delta^2+12\delta-56)s_2-\\
-(270\delta^2+342\delta+9)x^2-(180\delta+114)x^3-(60\delta^3+114\delta^2+6\delta-28)s_3-\\
-28a_1^4-68a_1^3b-24a_1^2b^2-(492\delta+312)a_1^2s_2-(288\delta+96)a_1^2b-\\
-(180\delta+114)s_3x^2-(96\delta+72)s_2^3-(150\delta^2+204\delta+18)s_2^2-(228\delta+160)a_1^3-\\
-(369\delta^2+468\delta-48)a_1^2-(360\delta^2+456\delta+12)s_2x-(150\delta^2+204\delta+18)s_3s_2-\\
-(96\delta+72)s_2s_3^2-(300\delta+204)s_2^2x-(144\delta+108)s_2^2s_3-(252\delta^2+216\delta-72)bx-\\
-(84\delta^2+72\delta-24)bs_3-(60\delta+36)bs_3^2-(144\delta+72)bs_2^2-120a_1s_2^3-\\
-72a_1bs_3^2-48bs_2s_3^2-(36\delta^2-24)b^2-(360\delta+228)s_2x^2-87a_1^2s_3^2-\\
-168a_1bs_2^2-180a_1s_2^2s_3-(288\delta^2+216\delta-108)a_1b-144bs_2s_3x-\\
-288a_1bx^2-(60\delta^2+90\delta+15)s_3^2-(120\delta+90)s_3^2x-(738\delta+468)a_1^2x-\\
-(168\delta^2+144\delta-48)bs_2-144bs_2^2x-(192\delta+72)a_1bs_3-369a_1^2x^2-\\
-(180\delta^2+228\delta+6)s_3x-222a_1s_3x^2-48b^2s_2x-(150\delta+114)a_1s_3^2-\\
-(252\delta+108)bx^2-444a_1s_2x^2-228a_1^3x-84bx^3-222a_1x^3-(336\delta+144)bs_2x-\\
-(666\delta^2+828\delta-36)a_1x-(666\delta+414)a_1x^2-(222\delta^2+276\delta-12)s_3a_1-\\
-(444\delta^2+552\delta-24)a_1s_2-(372\delta+252)a_1s_2^2-(168\delta+72)bs_3x-\\
-120s_2x^3-(444\delta+276)s_3a_1x-72a_1b^2\delta-48b^2s_2\delta-24b^2s_3\delta-72b^2x\delta-\\
-150s_2^2x^2-60s_3x^3-60s_3^2x^2-36b^2x^2-168bs_2x^2-84bs_3x^2-150s_2s_3x^2-\\
-45x^4-(888\delta+552)a_1s_2x-(576\delta+216)a_1bx-(372\delta+252)a_1s_2s_3-\\
-(300\delta+204)s_3s_2x-372a_1s_2s_3x-384a_1bs_2x-192a_1bs_3x-(384\delta+144)a_1bs_2-\\
-288a_1^2bx-492a_1^2s_2x-246a_1^2s_3x-144s_2^2s_3x-60bs_3^2x-72a_1b^2x-24b^2s_3x-\\
-150a_1s_3^2x-372a_1s_2^2x-96s_2^3x-96s_2s_3^2x-24s_3^3x-(144\delta+72)bs_2s_3.
\end{multline*}
All coefficients of this polynomial are negative.
Then $g(a_1,s_1,s_2,s_3,b)<0$ for $s_1>\delta$.
In~particular, one has $f(a_1,a_2,a_3,a_4,b)<0$ for $a_2-a_1\geqslant 0.848$, since $0.848>\delta$.
\end{proof}

\begin{lemma}
\label{lemma:Maple-P1-P1-d-4-a}
Suppose that $f$ is \eqref{equation:polynomial-P1xP1-d-4-a}, $a_3+a_4\geqslant 1+a_2$, $a_2-a_1\geqslant 0.848$.
Then $f<0$.
\end{lemma}

\begin{proof}
Let $g(a_1,s_1,s_2,s_3,b)=\widehat{f}$.
Then $g(a_1,x+\frac{21}{25},s_2,s_3,b)$ is a sum of the polynomial
\begin{multline*}
-12x^4-\frac{1558}{25}x^3-3b^2s_3^2-4bs_2^3-2bs_3^3-15s_3^2x^2-\frac{186}{5}s_3^2x-18s_2^3x-6s_3^3x-\\
-\frac{1112178}{15625}x-\frac{1812}{25}s_2^2x-\frac{387}{5}bx^2-\frac{9087}{125}bx-\frac{2766}{25}s_2x^2-\frac{92586}{625}s_2x-\\
-\frac{477}{25}bs_3^2-60a_1^3x-96a_1^2x^2-58a_1x^3-9b^2x^2-\frac{378}{25}b^2x-20bx^3-32s_2x^3-\\
-84a_1bs_2x-42a_1bs_3x-84a_1s_2s_3x-24bs_2s_3x-6b^2s_2s_3-15a_1bs_3^2-6bs_2s_3^2-\\
-\frac{172434}{625}a_1x-30a_1bs_2s_3-\frac{5754}{25}a_1x^2-\frac{1383}{25}s_3x^2-\frac{46293}{625}s_3x-21a_1^2s_3^2-\\
-36s_2^2x^2-16s_3x^3-6bs_2^2s_3-\frac{441}{25}s_2s_3^2-21s_2s_3^2x-\frac{6432}{25}a_1^2x-6a_1^2b^2-\\
-\frac{7204}{3125}b-\frac{954}{25}bs_2s_3-114a_1s_2x^2-\frac{1923}{25}a_1^2s_3-\frac{2364}{25}a_1s_2^2-\frac{2364}{25}a_1s_2s_3-\\
-\frac{15513}{625}bs_3-6a_1b^2s_3-12a_1b^2s_2-\frac{1332}{25}a_1bs_3-42a_1^2bs_2-21a_1^2bs_3-\frac{2664}{25}a_1bs_2-\\
-30a_1bs_2^2-69ba_1x^2-\frac{4248}{25}ba_1x-27s_2^2s_3x-\frac{2412}{25}bs_2x-18bs_3x^2-\frac{1206}{25}bs_3x-\\
-\frac{954}{25}bs_2^2-\frac{6888}{25}a_1s_2x-36bs_2x^2-\frac{252}{25}b^2s_2-12b^2s_2x-\frac{3444}{25}a_1s_3x-\\
-36a_1s_3^2x-24bs_2^2x-\frac{378}{25}a_1b^2-18a_1b^2x-\frac{92736}{625}a_1^2-12bs_3^2x-\frac{1812}{25}s_2s_3x-36s_2s_3x^2-84a_1s_2^2x-\\
-57a_1s_3x^2-63a_1^2s_3x-69a_1^2bx-\frac{3708}{125}s_3^2-126a_1^2s_2x-\frac{126}{25}b^2s_3-6b^2s_3x-\frac{73902}{625}x^2-6b^2s_2^2-\\
-4s_2^4-s_3^4-\frac{278}{25}s_2^3-\frac{54687}{625}a_1s_3-\frac{37176}{625}s_2s_3-40a_1^3s_2-20a_1^3s_3-48a_1^2s_2^2-20a_1s_2^3-\\
-\frac{151}{25}s_3^3-\frac{37176}{625}s_2^2-7a_1s_3^3-8s_2^3s_3-9s_2^2s_3^2-5s_2s_3^3-\frac{3846}{25}a_1^2s_2-\\
-48a_1^2s_2s_3-24a_1s_2s_3^2-\frac{1206}{25}a_1s_3^2-\frac{417}{25}s_2^2s_3-\frac{109374}{625}a_1s_2-30a_1s_2^2s_3-\\
-\frac{31026}{625}bs_2-8a_1^4-\frac{412}{5}a_1^3-16a_1^3b-\frac{2049}{25}a_1^2b-\frac{219}{625}b^2-\frac{43779}{625}ba_1.
\end{multline*}
and the polynomial
$$
\frac{265178}{390625}-\frac{1120738}{15625}a_1-\frac{385426}{15625}s_3-\frac{770852}{15625}s_2.
$$
All coefficients of the former polynomial are negative.
But we have $a_1+s_1+2s_2+s_3\geqslant 1$.
This follows from $a_3+a_4\geqslant 1+a_2$.
Thus, if $s_1\geqslant\frac{21}{25}$, then $a_1+2s_2+s_3\geqslant\frac{4}{25}$,
so that
$$
\frac{265178}{390625}-\frac{1120738}{15625}a_1-\frac{385426}{15625}s_3-\frac{770852}{15625}s_2<0.
$$
Hence, if $a_2-a_1\geqslant\frac{21}{25}=0.84$, then $f(a_1,a_2,a_3,a_4,b)<0$.
\end{proof}

\begin{lemma}
\label{lemma:Maple-P1-P1-d-3-a}
Suppose that $f$ is \eqref{equation:polynomial-P1xP1-d-3-a}, $a_3+a_4+a_5\leqslant 2+a_2$, $a_2-a_1\geqslant 0.8911$. Then~$f<0$.
\end{lemma}

\begin{proof}
Let $\delta$ be the unique positive root of the polynomial $-16x^4-48x^3+12x^2+32x+6$.
Then $0.8911>\delta\approx 0.8910180467$. If $a_3+a_4\geqslant a_2+a_5$, then
\begin{multline*}
4g(a_1,x+\delta,s_2,s_3,s_4,b)=-(192\delta+144)s_4x^2-(192\delta^2+288\delta-24)s_4x-\\
-60s_4^2x^2-(630\delta+450)a_1s_2^2-351a_1^2s_2^2-204a_1^2s_3^2-195a_1s_2^3-78bs_2^3-96s_2^3s_3-\\
-48s_2s_3^3-36b^2s_2^2-32bs_3^3-24b^2s_3^2-(320\delta+288)a_1^3-320a_1^3x-156s_2^3x-64s_3^3x-\\
-528a_1^2x^2-320a_1x^3-128bx^3-192s_2x^3-128s_3x^3-252s_2^2x^2-144s_3^2x^2-48b^2x^2-\\
-648a_1bs_2x-432a_1bs_3x-192a_1bs_2s_4-168a_1bs_3s_4-300a_1s_2s_3s_4-120bs_2s_3s_4-\\
-384a_1bs_2s_3-(320\delta^3+672\delta^2-192\delta-160)a_1-36s_2^4-8s_3^4-390a_1s_2^2s_3-\\
-156bs_2^2s_3-120bs_2s_3^2-72a_1b^2s_2-48a_1b^2s_3-48b^2s_2s_3-468a_1^2s_2s_3-324a_1^2bs_2-\\
-168a_1bs_3^2-(256\delta+192)x^3-72b^2s_2\delta-(64\delta^3+144\delta^2-24\delta-32)s_4-\\
-216a_1bs_4x-(60\delta^2+90\delta+15)s_4^2-87a_1^2s_4^2-25a_1s_4^3-48s_2^3s_4-30s_2^2s_4^2-\\
-10s_3s_4^3-10bs_4^3-12b^2s_4^2-(528\delta^2+768\delta-144)a_1^2-(192+264\delta)a_1^2s_4-\\
-(1056\delta+768)a_1^2x-(480\delta+336)s_4a_1x-(360\delta+264)a_1s_3^2-(960\delta+672)s_3a_1x-\\
-80a_1s_3^3-(128\delta^3+288\delta^2-48\delta-64)s_3-2s_4^4-48b^2s_3\delta-(432\delta^2+288\delta-192)a_1b-\\
-72a_1bs_4^2-(48\delta^2-36)b^2-(384\delta^2+576\delta-48)x^2-(256\delta^3+576\delta^2-96\delta-128)x-\\
-80a_1^3s_4-(128\delta^3+144\delta^2-144\delta-64)b-(156\delta+117)s_2^3-108a_1^2bs_4-234a_1^2s_2s_4-\\
-195a_1s_2^2s_4-78bs_2^2s_4-105a_1s_2s_4^2-120a_1s_3^2s_4-90a_1s_3s_4^2-102s_2^2s_3s_4-\\
-48s_2s_3s_4^2-(480\delta^2+672\delta-96)s_3a_1-(432\delta+144)a_1^2b-(792\delta+576)a_1^2s_2-\\
-24b^2s_4\delta-64x^4-(150\delta+114)a_1s_4^2-40a_1^4-(528\delta+384)a_1^2s_3-112a_1^3b-36a_1^2b^2-\\
-(360\delta+264)a_1s_3s_4-(720\delta^2+1008\delta-144)a_1s_2-168s_2s_4x^2-336s_2s_3x^2-192bs_3x^2-\\
-288bs_2x^2-480a_1s_3x^2-72s_3s_4^2x-432a_1^2bx-360a_1s_3^2x-96a_1b^2x-72b^2s_2x-\\
-432a_1bx^2-144s_3s_4x^2-240a_1s_4x^2-720a_1s_2x^2-792a_1^2s_2x-84s_2s_4^2x-252bs_2^2x-\\
-144bs_3^2x-150a_1s_4^2x-528a_1^2s_3x-60bs_4^2x-48b^2s_3x-312s_2^2s_3x-24b^2s_4x-264a_1^2s_4x-\\
-144bs_3s_4x-42bs_2s_4^2-96s_3^2s_4x-156s_2^2s_4x-240s_2s_3^2x-(144\delta^2+216\delta+12)s_3^2-\\
-(192\delta^3+432\delta^2-72\delta-96)s_2-(252\delta^2+378\delta-9)s_2^2-(648\delta+216)a_1bs_2-\\
-(1440\delta+1008)a_1s_2x-(420\delta+300)a_1s_2s_4-(840\delta+600)a_1s_2s_3-(864\delta+288)a_1bx-\\
-(168\delta+72)bs_2s_4-(336\delta+144)bs_2s_3-(216\delta+72)a_1bs_4-(432\delta+144)a_1bs_3-\\
-(252\delta+108)bs_2^2-(144\delta+72)bs_3^2-(288\delta^2+216\delta-108)bs_2-(288\delta+216)s_4s_3x-\\
-(192\delta^2+144\delta-72)bs_3-(60\delta+36)bs_4^2-(96\delta^2+72\delta-36)bs_4-\\
-(384\delta+144)bx^2-(240\delta+180)s_2s_3s_4-(192\delta+72)bs_4x-(144\delta+72)bs_3s_4-\\
-96\delta b^2x-96a_1b^2\delta-(672\delta+504)s_3s_2x-(384\delta^2+288\delta-144)bx-\\
-(240\delta^2+336\delta-48)s_4a_1-(960\delta+672)a_1x^2-(960\delta^2+1344\delta-192)a_1x-
\end{multline*}

\begin{multline*}
-48bs_3^2s_4-(384\delta+144)bs_3x-336bs_2s_3x-420a_1s_2s_4x-12s_2s_4^3-16s_3^3s_4-\\
-(336\delta+252)s_4s_2x-(576\delta+216)bs_2x-204a_1^2s_3s_4-160a_1^3s_3-18s_3^2s_4^2-\\
-(312\delta+234)s_2^2s_3-(156\delta+117)s_2^2s_4-(504\delta+378)s_2^2x-(240\delta+180)s_2s_3^2-24a_1b^2s_4-\\
-(384\delta+288)s_3x^2-(20\delta+15)s_4^3-(64\delta+48)s_3^3-630a_1s_2^2x-102s_2^2s_3^2-\\
-64s_4x^3-240s_2s_3s_4x-240a_1^3s_2-24b^2s_2s_4-840a_1s_2s_3x-300a_1s_2s_3^2-\\
-20s_4^3x-216a_1^2bs_3-288a_1bs_2^2-360a_1s_3s_4x-24b^2s_3s_4-168bs_2s_4x-96bs_4x^2-\\
-(336\delta^2+504\delta-12)s_3s_2-(84\delta+63)s_2s_4^2-(168\delta^2+252\delta-6)s_4s_2-(576\delta+432)s_2x^2-\\
-36bs_3s_4^2-(576\delta^2+864\delta-72)s_2x-(72+96\delta)s_3^2s_4-(288\delta+216)s_3^2x-(54+72\delta)s_3s_4^2-\\
-72s_2s_3^2s_4-(144\delta^2+216\delta+12)s_4s_3-(120\delta+90)s_4^2x-(384\delta^2+576\delta-48)s_3x.
\end{multline*}
If $a_3+a_4\leqslant a_2+a_5$, then $4g(a_1,x+\delta,s_2,s_3,b)$ is a sum of the latter polynomial and
\begin{multline*}
-5a_1s_2^3-2bs_2^3+(9+12\delta)s_2^2s_4-2s_2^3s_3+8s_2^3s_4-6s_2^2s_4^2+5a_1s_4^3+2s_3s_4^3+\\
+6bs_2^2s_4-(9+12\delta)s_2s_4^2+2bs_4^3-(3+4\delta)s_2^3-3s_2^4+s_4^4+(3+4\delta)s_4^3+15a_1s_2^2s_4+\\
-15a_1s_2s_4^2-6bs_2s_4^2+6s_2^2s_3s_4-6s_2s_3s_4^2+12s_2^2s_4x-12s_2s_4^2x+4s_4^3x-4s_2^3x.
\end{multline*}
Therefore, in both cases $g(a_1,x+\delta,s_2,s_3,b)$ is a polynomial with negative coefficients.
This implies that $f(a_1,a_2,a_3,a_4,a_5,b)<0$ provided that $a_2-a_1>\delta$.
\end{proof}

\begin{lemma}
\label{lemma:Maple-P1-P1-d-3-b}
Suppose that $f$ is \eqref{equation:polynomial-P1xP1-d-3-b}, $a_3+a_4+a_5\geqslant 2+a_2$, $a_2-a_1\geqslant 0.8911$.~Then~$f<0$.
\end{lemma}

\begin{proof}
Let $g(a_1,s_1,s_2,s_3,s_4,b)=\widehat{f}$. Then $g(a_1,x+\frac{22}{25},s_2,s_3,b)$ is a a sum of
\begin{multline*}
-\frac{21681}{125}x^2-24a_1^3b-9a_1^2b^2-\frac{63564}{625}a_1b-72a_1^3s_2-\frac{58698}{625}s_2^2-\frac{174888}{625}s_2a_1-234a_1^2s_2-\\
-54a_1^2s_3^2-24a_1^2s_4^2-\frac{50211}{625}bs_2-48a_1^3s_3-\frac{39078}{625}s_3^2-\frac{78264}{625}s_3s_2-\frac{116592}{625}s_3a_1-156a_1^2s_3-\\
-\frac{453}{25}s_3s_4^2-6bs_2^3-\frac{33474}{625}bs_3-24a_1^3s_4-\frac{19512}{625}s_4^2-\frac{39078}{625}s_4s_3-\frac{39132}{625}s_4s_2-\frac{58296}{625}s_4a_1-\\
-78a_1^2s_4-15a_1^4-\frac{138102}{625}a_1^2-\frac{183}{625}b^2-\frac{16737}{625}bs_4-12bs_2s_3s_4-126a_1bs_2x-84a_1bs_3x-\\
-42a_1bs_4x-\frac{4752}{25}a_1s_2s_3-216a_1s_2s_3x-\frac{2376}{25}a_1s_2s_4-108a_1s_2s_4x-96a_1s_3s_4x-48bs_2s_3x-\\
-24bs_2s_4x-24bs_3s_4x-66s_2s_3s_4x-60a_1bs_2s_3-30a_1bs_2s_4-30a_1bs_3s_4-72a_1s_2s_3s_4-\frac{3012}{25}a_1^2b-\\
-90a_1^2s_2^2-63a_1^2bs_2-21a_1^2bs_4-42a_1^2bs_3-\frac{3012}{25}a_1^3-\frac{154}{25}s_4^3-\frac{442}{5}x^3-20x^4-9s_2^4-\\
-42a_1s_2^3-22a_1s_3^3-4s_3^4-s_4^4-\frac{408}{25}s_2^3-\frac{58}{5}s_3^3-\frac{3564}{25}a_1s_2^2-162a_1s_2^2x-96a_1s_3^2x-42a_1s_4^2x-\\
-\frac{444}{25}s_2s_4^2-15a_1bs_4^2-\frac{978}{25}bs_2s_4-12a_1b^2s_3-\frac{2748}{25}a_1bs_3-120a_1^2s_2s_3-60a_1^2s_2s_4-20s_4x^3-
\end{multline*}

\begin{multline*}
-\frac{396}{25}b^2s_2-18b^2s_2x-\frac{264}{25}b^2s_3-12b^2s_3x-\frac{132}{25}b^2s_4-6b^2s_4x-54bs_2x^2-\frac{3726}{25}bs_2x-\\
-36bs_3x^2-\frac{2484}{25}bs_3x-18bs_4x^2-\frac{1242}{25}bs_4x-36bs_2^2x-24bs_3^2x-12bs_4^2x-96s_3s_2x^2-\frac{3624}{25}s_3s_2x-\\
-\frac{6972}{25}s_3a_1x-69s_4a_1x^2-\frac{3486}{25}s_4a_1x-27s_2s_3^2s_4-21s_2s_3s_4^2-45a_1bs_2^2-42a_1s_2^2s_4-\\
-48s_4s_2x^2-\frac{1812}{25}s_4s_2x-42s_4s_3x^2-\frac{1848}{25}s_4s_3x-78s_2^2s_3x-39s_2^2s_4x-66s_2s_3^2x-27s_2s_4^2x-\\
-30s_3^2s_4x-24s_3s_4^2x-\frac{1374}{25}a_1bs_4-\frac{852}{25}s_2s_3s_4-96a_1^2bx-225a_1^2s_2x-150a_1^2s_3x-75a_1^2s_4x-\\
-\frac{528}{25}a_1b^2-24a_1b^2x-96a_1bx^2-\frac{6024}{25}a_1bx-207s_2a_1x^2-\frac{10458}{25}s_2a_1x-138s_3a_1x^2-\\
-3b^2s_4^2-\frac{1956}{25}bs_2s_3-12bs_2^2s_3-6bs_2s_4^2-\frac{978}{25}bs_3s_4-18a_1b^2s_2-84a_1s_2^2s_3-27a_1s_3s_4^2-\\
-12bs_2s_3^2-12b^2s_2s_3-30s_2^2s_3s_4-6bs_2^2s_4-\frac{2412}{25}a_1s_3s_4-6b^2s_3s_4-30a_1s_2s_4^2-6a_1b^2s_4-\\
-6b^2s_2s_4-6bs_3^2s_4-30a_1bs_3^2-6bs_3s_4^2-54a_1^2s_3s_4-\frac{4122}{25}a_1bs_2-33a_1s_3^2s_4-72a_1s_2s_3^2-\\
-18s_4^2x^2-\frac{942}{25}s_4^2x-39s_2^3x-20s_3^3x-7s_4^3x-96a_1^3x-153a_1^2x^2-\frac{9132}{25}a_1^2x-94a_1x^3-\\
-\frac{408}{25}s_2^2s_4-\frac{852}{25}s_2s_3^2-6b^2s_3^2-\frac{279}{5}s_4x^2-\frac{9843}{125}s_4x-72s_2^2x^2-\frac{2718}{25}s_2^2x-42s_3^2x^2-\frac{1848}{25}s_3^2x-\\
-\frac{87}{5}s_3^2s_4-\frac{8154}{25}a_1x^2-\frac{252288}{625}a_1x-12b^2x^2-\frac{528}{25}b^2x-28bx^3-\frac{2748}{25}bx^2-\frac{65256}{625}bx-\\
-60s_2x^3-\frac{837}{5}s_2x^2-\frac{29529}{125}s_2x-40s_3x^3-\frac{558}{5}s_3x^2-\frac{19686}{125}s_3x-24s_2^3s_3-2bs_4^3-\\
-\frac{28744}{15625}b-\frac{332548}{3125}x-8a_1s_4^3-9b^2s_2^2-\frac{1224}{25}a_1s_4^2-\frac{1467}{25}bs_2^2-\frac{978}{25}bs_3^2-\frac{489}{25}bs_4^2-\frac{816}{25}s_2^2s_3-4bs_3^3-\\
-12s_2^3s_4-30s_2^2s_3^2-12s_2^2s_4^2-18s_2s_3^3-6s_2s_4^3-8s_3^3s_4-9s_3^2s_4^2-5s_3s_4^3-\frac{2412}{25}a_1s_3^2.
\end{multline*}
and the polynomial of degree one
$$
\frac{111281}{78125}-\frac{1667212}{15625}a_1-\frac{86602}{3125}s_4-\frac{173204}{3125}s_3-\frac{259806}{3125}s_2.
$$
The coefficients of the former polynomial are negative. But $2a_1+2s_1+3s_2+2s_3+s_4\geqslant 2$.
This follows from $a_3+a_4+a_5\geqslant 2+a_2$. Since $s_1\geqslant\frac{22}{25}$, we have
$$
2a_1+3s_2+2s_3+s_4\geqslant\frac{6}{25}.
$$
This inequality implies that
$$
\frac{111281}{78125}-\frac{1667212}{15625}a_1-\frac{86602}{3125}s_4-\frac{173204}{3125}s_3-\frac{259806}{3125}s_2<0.
$$
Thus, we see that $f(a_1,a_2,a_3,a_4,a_5,b)<0$ provided that $a_2-a_1>\frac{22}{25}=0.88$.
\end{proof}

\begin{lemma}
\label{lemma:Maple-P1-P1-d-2-a}
Suppose that $f$ is the polynomial \eqref{equation:polynomial-P1xP1-d-1-a}. 
If $a_2-a_1\geqslant 0.9305$, then $f<0$.
Similarly, if  $a_2-a_1\geqslant 0.915$, then $f(a_1,a_2,a_3,a_4,a_5,a_6,1,b)<0$.
\end{lemma}

\begin{proof}
Let $g_1(a_1,s_1,s_2,s_3,s_4,s_5,s_6,b)=\widehat{f}$. Then
$$
g_1(0,x,0,0,0,0,0,0)=-\frac{75}{4}x^4-\frac{1045}{8}x^3+\frac{837}{8}x^2+\frac{237}{8}x+\frac{9}{8}.
$$
Let us denote by $\delta_1$ the unique positive root of this polynomial. Then $\delta_1\approx 0.93040704$.
Moreover, the polynomial $8g_1(a_1,x+\delta_1,s_2,s_3,s_4,s_5,s_6,b)$ can be expanded as
\begin{multline*}
-(411\delta_1+219)s_4s_5^2-(708\delta_1^2+408\delta_1-324)bs_4-(1008\delta_1+840)s_4^2x-288a_1b^2\delta_1-\\
-(1506\delta_1+1518)s_5s_2x-(1560\delta_1+408)a_1bs_4-(474\delta_1^2+252\delta_1-198)bs_5-(180\delta_1+72)s_3s_6^2-\\
-(504\delta_1^2+840\delta_1-144)s_4^2-(1554\delta_1+486)ba_1^2-(780\delta_1+879)s_5x^2-192b^2s_3\delta_1-\\
-(144\delta_1^2-120)b^2-(1884\delta_1+564)bs_3x-(384\delta_1+96)bs_3s_6-96b^2s_5\delta_1-(1080\delta_1+336)bs_2^2-\\
-(432\delta_1+96)bs_2s_6-(345\delta_1^3+1443\delta_1^2-765\delta_1-111)s_4-(108\delta_1+53)s_5^3-\\
-(336\delta_1^2+399\delta_1-33)s_5^2-(260\delta_1^3+879\delta_1^2-462\delta_1-69)s_5-(150\delta_1^2+120\delta_1+18)s_6^2-\\
-(175\delta_1^3+315\delta_1^2-159\delta_1-27)s_6-(600\delta_1+1045)x^3-(900\delta_1^2+3135\delta_1-837)x^2-\\
-(2505\delta_1^2+6852\delta_1-2097)a_1x-(2076\delta_1+564)a_1bs_3-(1044\delta_1+252)a_1bs_5-(528\delta_1+96)a_1bs_6-\\
-(2700\delta_1+2748)a_1s_2s_4-(1176\delta_1+624)a_1s_2s_6-(4266\delta_1+5622)a_1s_2x-(2622\delta_1+2334)a_1s_3s_4-\\
-(1842\delta_1+1440)a_1s_3s_5-(1062\delta_1+546)a_1s_3s_6-(3522\delta_1+4392)a_1s_3x-(2778\delta_1+3162)s_4a_1x-\\
-(834\delta_1+390)a_1s_5s_6-(2034\delta_1+1932)s_5a_1x-(1290\delta_1+702)s_6a_1x-(768\delta_1+408)s_2s_3s_6-\\
-(2598\delta_1+3450)s_2s_3x-(2052\delta_1+2484)s_4s_2x-(960\delta_1+552)s_6s_2x-(2592\delta_1+720)a_1bs_2-\\
-(1329\delta_1+1275)s_2s_3^2-(756\delta_1+672)s_2^2s_5-(1080\delta_1+1104)s_2^2s_4-(1404\delta_1+1536)s_2^2s_3-\\
-(120\delta_1+48)s_4s_6^2-(432\delta_1+240)s_2^2s_6-(753\delta_1^2+1518\delta_1-363)s_5s_2-(1572\delta_1+2208)s_2^2x-\\
-(1026\delta_1^2+2484\delta_1-606)s_4s_2-(948\delta_1+732)s_2s_4^2-(1938\delta_1+1686)a_1s_2s_5-\\
-(1299\delta_1^2+3450\delta_1-849)s_2s_3-(585\delta_1+339)s_2s_5^2-(1545\delta_1+2571)s_2x^2-(240\delta_1+96)s_2s_6^2-\\
-(480\delta_1^2+552\delta_1-120)s_6s_2-(1545\delta_1^2+5142\delta_1-1371)s_2x-(525\delta_1+315)s_6x^2-\\
-(789\delta_1+1500)a_1^3-(1395\delta_1^2+3849\delta_1-1278)a_1^2-(1284\delta_1+900)s_2s_4s_5-\\
-(672\delta_1+336)s_2s_4s_6-(60\delta_1+24)s_5s_6^2-(780\delta_1^2+1758\delta_1-462)s_5x-(780\delta_1+228)bs_3s_5-\\
-(576\delta_1+264)s_2s_5s_6-(3108\delta_1+876)ba_1x-(498\delta_1+279)s_3s_5^2-(1290\delta_1+2007)s_3x^2-\\
-(696\delta_1+216)bs_4s_5-(228\delta_1+108)s_4^2s_6-(1746\delta_1+1194)a_1s_4s_5-(1035\delta_1^2+2886\delta_1-765)s_4x-\\
-(1035\delta_1+1443)s_4x^2-(786\delta_1^2+2208\delta_1-546)s_2^2-(576\delta_1+656)s_2^3-(1398\delta_1+1122)s_2s_3s_5-\\
-(600\delta_1^3+3135\delta_1^2-1674\delta_1-237)x-(430\delta_1^3+2007\delta_1^2-1068\delta_1-153)s_3-\\
-672s_2s_4s_6x-2700a_1s_2s_4x-(276\delta_1+172)s_4^3-(515\delta_1^3+2571\delta_1^2-1371\delta_1-195)s_2-(450\delta_1+369)s_3^3-\\
-48b^2s_6x-(654\delta_1^2+1443\delta_1-315)s_3^2-(528\delta_1+312)s_4^2s_5-(948\delta_1+252)bs_5x-\\
-1044a_1bs_5x-117s_5^2s_6x-585s_2s_5^2x-528bs_4^2x-240s_2s_6^2x-672s_3^2s_5x-825a_1s_5^2x-
\end{multline*}

\begin{multline*}
-(2028\delta_1+1836)s_2s_3s_4-(288\delta_1+96)bs_5s_6-(1296\delta_1+384)bs_2s_4-(1416\delta_1+408)bs_4x-\\
-(864\delta_1+240)bs_2s_5-240b^2s_2\delta_1-(1176\delta_1^2+720\delta_1-576)bs_2-(786\delta_1+246)bs_3^2-\\
-(1017\delta_1^2+2082\delta_1-447)s_4s_3-(3462\delta_1+3810)a_1s_2s_3-48b^2s_6\delta_1-288\delta_1b^2x-\\
-1956a_1^2s_3x-1080s_2^2s_4x-1011s_3^2s_4x-2373a_1^2s_2x-948s_2s_4^2x-1080bs_2^2x-\\
-120s_4s_6^2x-1329s_2s_3^2x-228s_4^2s_6x-306bs_5^2x-192b^2s_3x-756s_2^2s_5x-1122a_1^2s_5x-\\
-333s_3^2s_6x-432s_2^2s_6x-1404s_2^2s_3x-1272a_1s_4^2x-240b^2s_2x-360a_1s_6^2x-\\
-(390\delta_1^2+396\delta_1-42)s_6s_4-(300\delta_1+120)s_6^2x-708bs_4x^2-726s_3s_5x^2-390s_4s_6x^2-\\
-1176bs_2x^2-435s_3s_6x^2-1299s_2s_3x^2-1017s_3s_4x^2-753s_2s_5x^2-1761a_1s_3x^2-1017a_1s_5x^2-\\
-1026s_2s_4x^2-1389a_1s_4x^2-345s_5s_6x^2-480s_2s_6x^2-474bs_5x^2-645a_1s_6x^2-1539a_1^2s_4x-\\
-786bs_3^2x-942bs_3x^2-240bs_6x^2-1554a_1bx^2-2133a_1s_2x^2-699s_4s_5x^2-(888\delta_1+624)s_3s_4^2-\\
-(1308\delta_1+1443)s_3^2x-(1017\delta_1^2+1932\delta_1-585)s_5a_1-(336\delta_1+96)bs_4s_6-\\
-(2133\delta_1^2+5622\delta_1-1719)a_1s_2-834a_1s_5s_6x-1398s_2s_3s_5x-768s_2s_3s_6x-1284s_2s_4s_5x-\\
-576s_2s_5s_6x-1170s_3s_4s_5x-3462a_1s_2s_3x-2592a_1bs_2x-2076a_1bs_3x-1728bs_2s_3x-\\
-1560a_1bs_4x-1296bs_2s_4x-2622a_1s_3s_4x-1176bs_3s_4x-2028s_2s_3s_4x-1938a_1s_2s_5x-\\
-864bs_2s_5x-1842a_1s_3s_5x-780bs_3s_5x-1176a_1s_2s_6x-1062a_1s_3s_6x-1746a_1s_4s_5x-\\
-948a_1s_4s_6x-564s_3s_4s_6x-462s_3s_5s_6x-348s_4s_5s_6x-696bs_4s_5x-528a_1bs_6x-432bs_2s_6x-\\
-384bs_3s_6x-336bs_4s_6x-288bs_5s_6x-705a_1^2s_6x-60s_5s_6^2x-1701a_1s_3^2x-96b^2s_5x-\\
-1554ba_1^2x-2112a_1s_2^2x-888s_3s_4^2x-498s_3s_5^2x-288a_1b^2x-120bs_6^2x-144b^2s_4x-\\
-528s_4^2s_5x-411s_4s_5^2x-180s_3s_6^2x-(120\delta_1+48)bs_6^2-150x^4-(360\delta_1+156)a_1s_6^2-\\
-260s_5x^3-175s_6x^3-504s_4^2x^2-336s_5^2x^2-150s_6^2x^2-144b^2x^2-1395a_1^2x^2-835a_1x^3-\\
-470bx^3-515s_2x^3-430s_3x^3-345s_4x^3-786s_2^2x^2-654s_3^2x^2-789a_1^3x-576s_2^3x-450s_3^3x-\\
-276s_4^3x-108s_5^3x-(942\delta_1^2+564\delta_1-450)bs_3-63a_1^4-438a_1^3b-120a_1^2b^2-\\
-(333\delta_1+171)s_3^2s_6-(825\delta_1+474)a_1s_5^2-219a_1^3s_6-333a_1^3s_5-447a_1^3s_4-\\
-561a_1^3s_3-675a_1^3s_2-(1170\delta_1+762)s_3s_4s_5-(564\delta_1+276)s_3s_4s_6-(2034\delta_1+2082)s_4s_3x-\\
-(462\delta_1+210)s_3s_5s_6-(1452\delta_1+1278)s_5s_3x-(870\delta_1+474)s_6s_3x-(348\delta_1+156)s_4s_5s_6-\\
-(1398\delta_1+1038)s_5s_4x-(780\delta_1+396)s_6s_4x-(690\delta_1+318)s_6s_5x-(1701\delta_1+1614)a_1s_3^2-\\
-(1410\delta_1^2+876\delta_1-702)bx-(1761\delta_1^2+4392\delta_1-1341)a_1s_3-(1272\delta_1+960)a_1s_4^2-\\
-(117\delta_1+51)s_5^2s_6-(528\delta_1+168)bs_4^2-(2352\delta_1+720)bs_2x-(470\delta_1^3+438\delta_1^2-702\delta_1-78)b-\\
-(306\delta_1+102)bs_5^2-(1728\delta_1+528)bs_2s_3-(672\delta_1+483)s_3^2s_5-(345\delta_1^2+318\delta_1-3)s_6s_5-\\
-(1011\delta_1+795)s_3^2s_4-(1410\delta_1+438)bx^2-(2790\delta_1+3849)a_1^2x-(1122\delta_1+1089)a_1^2s_5-\\
-(1956\delta_1+2469)a_1^2s_3-(1554\delta_1^2+876\delta_1-846)ba_1-(2373\delta_1+3159)a_1^2s_2-\\
-(835\delta_1^3+3426\delta_1^2-2097\delta_1-276)a_1-(705\delta_1+399)a_1^2s_6-(2112\delta_1+2436)a_1s_2^2-\\
-(1539\delta_1+1779)a_1^2s_4-(1389\delta_1^2+3162\delta_1-963)s_4a_1-144b^2s_4\delta_1-(2505\delta_1+3426)a_1x^2-\\
-(240\delta_1^2+96\delta_1-72)bs_6-(948\delta_1+468)a_1s_4s_6-(435\delta_1^2+474\delta_1-81)s_6s_3-\\
-(525\delta_1^2+630\delta_1-159)s_6x-(645\delta_1^2+702\delta_1-207)s_6a_1-(480\delta_1+96)bs_6x-\\
\end{multline*}

\begin{multline*}
-957a_1^2s_3^2-752a_1s_2^3-352bs_2^3-720a_1^2s_4^2-471a_1^2s_5^2-210a_1^2s_6^2-567a_1s_3^3-\\
-340a_1s_4^3-512s_2^3s_3-384s_2^3s_4-256s_2^3s_5-128s_2^3s_6-690s_2^2s_3^2-456s_2^2s_4^2-\\
-258s_2^2s_5^2-96s_2^2s_6^2-441s_2s_3^3-248s_2s_4^3-315s_3^3s_4-198s_3^3s_5-81s_3^3s_6-\\
-384s_3^2s_4^2-186s_3^2s_5^2-54s_3^2s_6^2-220s_3s_4^3-116s_4^3s_5-40s_4^3s_6-120s_4^2s_5^2-\\
-24s_4^2s_6^2-131a_1s_5^3-91s_2s_5^3-74s_3s_5^3-57s_4s_5^3-11s_5^3s_6-6s_5^2s_6^2-120b^2s_2^2-\\
-96b^2s_3^2-72b^2s_4^2-48b^2s_5^2-24b^2s_6^2-234bs_3^3-128bs_4^3-46bs_5^3-1182a_1^2s_2^2-\\
-(1290\delta_1^2+4014\delta_1-1068)s_3x-(699\delta_1^2+1038\delta_1-165)s_5s_4-160s_2^4-\\
-108s_3^4-48s_4^4-10s_5^4-1758a_1s_2s_3s_5-936a_1s_2s_3s_6-1596a_1s_2s_4s_5-816a_1s_2s_4s_6-\\
-696a_1s_2s_5s_6-1446a_1s_3s_4s_5-684a_1s_3s_4s_6-558a_1s_3s_5s_6-420a_1s_4s_5s_6-\\
-1068s_2s_3s_4s_5-480s_2s_3s_4s_6-384s_2s_3s_5s_6-288s_2s_4s_5s_6-228s_3s_4s_5s_6-\\
-792a_1bs_4s_5-720bs_2s_3s_5-624bs_2s_4s_5-552bs_3s_4s_5-480a_1bs_2s_6-432a_1bs_3s_6-\\
-384a_1bs_4s_6-336a_1bs_5s_6-336bs_2s_3s_6-288bs_2s_4s_6-240bs_2s_5s_6-240bs_3s_4s_6-\\
-192bs_3s_5s_6-144bs_4s_5s_6-1920a_1bs_2s_3-1440a_1bs_2s_4-1320a_1bs_3s_4-2580a_1s_2s_3s_4-\\
-1104bs_2s_3s_4-960a_1bs_2s_5-876a_1bs_3s_5-144a_1bs_6^2-192bs_2^2s_6-96bs_2s_6^2-\\
-144bs_3^2s_6-72bs_3s_6^2-96bs_4^2s_6-48bs_4s_6^2-48bs_5^2s_6-24bs_5s_6^2-1032s_2^2s_3s_4-\\
-684s_2^2s_3s_5-336s_2^2s_3s_6-600s_2^2s_4s_5-288s_2^2s_4s_6-240s_2^2s_5s_6-978s_2s_3^2s_4-\\
-633s_2s_3^2s_5-288s_2s_3^2s_6-828s_2s_3s_4^2-435s_2s_3s_5^2-144s_2s_3s_6^2-468s_2s_4^2s_5-\\
-192s_2s_4^2s_6-354s_2s_4s_5^2-96s_2s_4s_6^2-483s_3^2s_4s_5-198s_3^2s_4s_6-153s_3^2s_5s_6-\\
-408s_3s_4^2s_5-156s_3s_4^2s_6-297s_3s_4s_5^2-72s_3s_4s_6^2-84s_4^2s_5s_6-141a_1s_5^2s_6-\\
-96s_2s_5^2s_6-75s_3s_5^2s_6-54s_4s_5^2s_6-72a_1s_5s_6^2-48s_2s_5s_6^2-36s_3s_5s_6^2-\\
-24s_4s_5s_6^2-240a_1b^2s_2-192a_1b^2s_3-144a_1b^2s_4-96a_1b^2s_5-192b^2s_2s_3-144b^2s_2s_4-\\
-96b^2s_2s_5-144b^2s_3s_4-96b^2s_3s_5-96b^2s_4s_5-48a_1b^2s_6-48b^2s_2s_6-48b^2s_3s_6-\\
-48b^2s_4s_6-48b^2s_5s_6-600a_1bs_4^2-480bs_2s_4^2-516bs_3^2s_4-432bs_3s_4^2-354a_1bs_5^2-\\
-408bs_2^2s_5-264bs_2s_5^2-330bs_3^2s_5-222bs_3s_5^2-240bs_4^2s_5-180bs_4s_5^2-648a_1^2s_2s_6-\\
-591a_1^2s_3s_6-987a_1^2s_4s_5-534a_1^2s_4s_6-477a_1^2s_5s_6-960a_1s_2^2s_5-528a_1s_2^2s_6-\\
-1188a_1s_2s_4^2-717a_1s_2s_5^2-288a_1s_2s_6^2-1269a_1s_3^2s_4-837a_1s_3^2s_5-405a_1s_3^2s_6-\\
-1104a_1s_3s_4^2-609a_1s_3s_5^2-216a_1s_3s_6^2-648a_1s_4^2s_5-276a_1s_4^2s_6-501a_1s_4s_5^2-\\
-144a_1s_4s_6^2-1935a_1^2s_2s_3-1200a_1bs_2^2-1506a_1^2s_2s_4-1473a_1^2s_3s_4-882a_1bs_3^2-\\
-1824a_1s_2^2s_3-1392a_1s_2^2s_4-1701a_1s_2s_3^2-840bs_2^2s_3-624bs_2^2s_4-744bs_2s_3^2-\\
-1077a_1^2s_2s_5-1032a_1^2s_3s_5-264ba_1^2s_6-522ba_1^2s_5-780ba_1^2s_4-1038ba_1^2s_3-\\
-1296ba_1^2s_2-(1176\delta_1+360)bs_3s_4-(672\delta_1+399)s_5^2x-(726\delta_1^2+1278\delta_1-264)s_5s_3.
\end{multline*}
All coefficients of this polynomial are negative, so that $f<0$ for $a_2-a_1>\delta_1$.

Now we let $f_2=f(a_1,a_2,a_3,a_4,a_5,a_6,1,b)$, and we let $g_2(a_1,s_1,s_2,s_3,s_4,s_5,b)=\widehat{f}_2$.
Then~the polynomial $g_2(0,x,0,0,0,0,0,0,0)$ expands as $-\frac{125}{8}x^4-\frac{725}{8}x^3+\frac{435}{8}x^2+\frac{285}{8}x+\frac{9}{4}$.
Let us denote by $\delta_2$ the unique positive root of this polynomial. Then $\delta_2\approx 0.9149603$.
Moreover, the polynomial $8g_2(a_1,x+\delta_2,s_2,s_3,s_4,s_5,s_6,b)$ can be expanded as
\begin{multline*}
-(900\delta_2+927)s_3^2x-(448\delta_2+448)s_2^3-(1248\delta_2+1104)s_2s_3s_4-(645\delta_2+1056)a_1^3-\\
-192s_4s_5^2x-(420\delta_2+156)bs_5x-(840\delta_2+312)bs_4x-(312\delta_2+120)bs_4s_5-(1260\delta_2+468)bs_3x-\\
-(350\delta_2^3+390\delta_2^2-510\delta_2-102)b-(1872\delta_2+624)a_1bs_2-(324\delta_2+243)s_3^2s_5-\\
-(2340\delta_2+780)ba_1x-(1800\delta_2+2232)s_2s_3x-(756\delta_2+576)a_1s_4s_5-(1548\delta_2+1416)a_1s_3s_4-\\
-240a_1b^2\delta_2-(576\delta_2+432)s_2s_4s_5-(1200\delta_2+1488)s_4s_2x-(600\delta_2+744)s_5s_2x-\\
-264s_4^2s_5x-(552\delta_2+372)s_3s_4s_5-(1200\delta_2+1236)s_4s_3x-(3240\delta_2+3888)a_1s_2x-240\delta_2b^2x-\\
-1872a_1bs_2x-(624\delta_2+552)s_2s_3s_5-(600\delta_2^2+1488\delta_2-264)s_2^2-(1680\delta_2+624)bs_2x-\\
-(1620\delta_2+1944)s_4a_1x-(450\delta_2+549)a_1^2s_5-(1350\delta_2+1647)a_1^2s_3-(1170\delta_2+414)ba_1^2-\\
-(1800\delta_2+2196)a_1^2s_2-(900\delta_2+1098)a_1^2s_4-(2250\delta_2+2745)a_1^2x-(2376\delta_2+2520)a_1s_2s_3-\\
-(324\delta_2+243)s_3^3-(450\delta_2^2+927\delta_2-135)s_3^2-(468\delta_2+156)a_1bs_5-(120\delta_2^2-96)b^2-\\
-(300\delta_2^2+492\delta_2-48)s_4^2-(2430\delta_2+2916)a_1s_3x-(900\delta_2+1305)s_3x^2-96b^2s_4\delta_2-\\
-96b^2s_4x-(1584\delta_2+1680)a_1s_2s_4-144b^2s_3\delta_2-(400\delta_2^3+1740\delta_2^2-696\delta_2-228)s_2-\\
-240a_1b^2x-(500\delta_2^3+2175\delta_2^2-870\delta_2-285)x-(810\delta_2+972)s_5a_1x-(52\delta_2+25)s_5^3-\\
-1404a_1bs_3x-(900\delta_2^2+2610\delta_2-522)s_3x-(600\delta_2+492)s_4^2x-(264\delta_2+156)s_4^2s_5-\\
-216a_1s_4^3-(300\delta_2^2+870\delta_2-174)s_5x-(192\delta_2+102)s_4s_5^2-(300\delta_2^2+492\delta_2-48)s_5s_4-\\
-(600\delta_2+870)s_4x^2-(600\delta_2^2+1740\delta_2-348)s_4x-(300\delta_2+183)s_5^2x-(300\delta_2+435)s_5x^2-\\
-(1584\delta_2+1680)a_1s_2^2-(1170\delta_2^2+780\delta_2-630)ba_1-(1620\delta_2^2+3888\delta_2-948)a_1s_2-\\
-228s_3s_5^2x-(1161\delta_2+1062)a_1s_3^2-(1215\delta_2^2+2916\delta_2-711)a_1s_3-(810\delta_2^2+1944\delta_2-474)s_4a_1-\\
-(756\delta_2+576)a_1s_4^2-(1200\delta_2+1740)s_2x^2-(369\delta_2+222)a_1s_5^2-(405\delta_2^2+972\delta_2-237)s_5a_1-\\
-2376a_1s_2s_3x-(2025\delta_2+2430)a_1x^2-(522\delta_2+198)bs_3^2-264s_2s_5^2x-312bs_4^2x-324s_3^2s_5x-\\
-552s_3s_4s_5x-369a_1s_5^2x-1350a_1^2s_3x-672s_2^2s_4x-648s_3^2s_4x-1800a_1^2s_2x-576s_2s_4^2x-\\
-768bs_2^2x-936s_2s_3^2x-138bs_5^2x-144b^2s_3x-(264\delta_2+156)s_2s_5^2-(300\delta_2^2+744\delta_2-132)s_5s_2-\\
-336s_2^2s_5x-450a_1^2s_5x-1008s_2^2s_3x-756a_1s_4^2x-192b^2s_2x-420bs_4x^2-300s_3s_5x^2-\\
-840bs_2x^2-900s_2s_3x^2-600s_3s_4x^2-300s_2s_5x^2-1215a_1s_3x^2-405a_1s_5x^2-600s_2s_4x^2-\\
-576s_2s_4s_5x-810a_1s_4x^2-210bs_5x^2-900a_1^2s_4x-522bs_3^2x-630bs_3x^2-1170a_1bx^2-1620a_1s_2x^2-\\
-552s_3s_4^2x-300s_4s_5x^2-(774\delta_2+708)a_1s_3s_5-(176\delta_2+104)s_4^3-(576\delta_2+432)s_2s_4^2-\\
-348bs_3s_5x-(200\delta_2^3+870\delta_2^2-348\delta_2-114)s_4-125x^4-(1200\delta_2^2+3480\delta_2-696)s_2x-\\
-100s_5x^3-300s_4^2x^2-150s_5^2x^2-120b^2x^2-1125a_1^2x^2-675a_1x^3-350bx^3-400s_2x^3-\\
-1248s_2s_3s_4x-(150\delta_2^2+183\delta_2-3)s_5^2-(300\delta_2^3+1305\delta_2^2-522\delta_2-171)s_3-(500\delta_2+725)x^3-
\end{multline*}

\begin{multline*}
-624s_2s_3s_5x-1152bs_2s_3x-1584a_1s_2s_4x-936a_1bs_4x-768bs_2s_4x-1548a_1s_3s_4x-696bs_3s_4x-\\
-108s_3^3s_5-300s_3x^3-200s_4x^3-600s_2^2x^2-450s_3^2x^2-645a_1^3x-448s_2^3x-324s_3^3x-213a_1^2s_5^2-\\
-792a_1s_2s_5x-176s_4^3x-52s_5^3x-(768\delta_2+288)bs_2s_4-(1152\delta_2+432)bs_2s_3-54a_1^4-\\
-318a_1^3b-96a_1^2b^2-(840\delta_2^2+624\delta_2-408)bs_2-(2025\delta_2^2+4860\delta_2-1185)a_1x-\\
-504bs_2s_3^2-(768\delta_2+288)bs_2^2-(384\delta_2+144)bs_2s_5-129a_1^3s_5-258a_1^3s_4-387a_1^3s_3-\\
-516a_1^3s_2-(696\delta_2+264)bs_3s_4-(648\delta_2+486)s_3^2s_4-(600\delta_2^2+1488\delta_2-264)s_4s_2-\\
-128s_2^3s_5-756a_1s_4s_5x-(552\delta_2+372)s_3s_4^2-(900\delta_2^2+2232\delta_2-396)s_2s_3-48b^2s_5\delta_2-\\
-312bs_4s_5x-(600\delta_2^2+1236\delta_2-180)s_4s_3-(936\delta_2+312)a_1bs_4-(936\delta_2+828)s_2s_3^2-\\
-160s_2s_4^3-468a_1bs_5x-(675\delta_2^3+2430\delta_2^2-1185\delta_2-336)a_1-(1200\delta_2+1488)s_2^2x-\\
-216s_3^3s_4-384bs_2s_5x-(792\delta_2+840)a_1s_2s_5-(600\delta_2+618)s_5s_3x-(600\delta_2+492)s_5s_4x-\\
-384bs_2^2s_4-(672\delta_2+672)s_2^2s_4-(336\delta_2+336)s_2^2s_5-(1050\delta_2^2+780\delta_2-510)bx-\\
-(1050\delta_2+390)bx^2-(1008\delta_2+1008)s_2^2s_3-(1404\delta_2+468)a_1bs_3-(210\delta_2^2+156\delta_2-102)bs_5-\\
-504s_2^2s_3^2-(138\delta_2+54)bs_5^2-(420\delta_2^2+312\delta_2-204)bs_4-(312\delta_2+120)bs_4^2-\\
-324s_2s_3^3-288s_2^2s_4^2-(630\delta_2^2+468\delta_2-306)bs_3-(300\delta_2^2+618\delta_2-90)s_5s_3-432a_1^2s_4^2-\\
-384s_2^3s_3-(750\delta_2^2+2175\delta_2-435)x^2-657a_1^2s_3^2-576a_1s_2^3-256bs_2^3-405a_1s_3^3-\\
-444a_1^2s_2s_5-90s_3^2s_5^2-64s_4^3s_5-60s_4^2s_5^2-63a_1s_5^3-44s_2s_5^3-36s_3s_5^3-28s_4s_5^3-96b^2s_2^2-\\
-234ba_1^2s_5-256s_2^3s_4-72b^2s_3^2-48b^2s_4^2-24b^2s_5^2-162bs_3^3-80bs_4^3-22bs_5^3-888a_1^2s_2^2-\\
-774a_1s_3s_5x-(100\delta_2^3+435\delta_2^2-174\delta_2-57)s_5-128s_2^4-81s_3^4-32s_4^4-936ba_1^2s_2-\\
-1584a_1s_2^2x-5s_5^4-792a_1s_2s_3s_5-720a_1s_2s_4s_5-684a_1s_3s_4s_5-528s_2s_3s_4s_5-360a_1bs_4s_5-\\
-252s_3^2s_4^2-336bs_2s_3s_5-288bs_2s_4s_5-264bs_3s_4s_5-1296a_1bs_2s_3-864a_1bs_2s_4-792a_1bs_3s_4-\\
-1584a_1s_2s_3s_4-672bs_2s_3s_4-432a_1bs_2s_5-396a_1bs_3s_5-672s_2^2s_3s_4-336s_2^2s_3s_5-\\
-288s_2^2s_4s_5-648s_2s_3^2s_4-324s_2s_3^2s_5-528s_2s_3s_4^2-204s_2s_3s_5^2-702ba_1^2s_3-\\
-438a_1^2s_3s_5-240s_2s_4^2s_5-168s_2s_4s_5^2-252s_3^2s_4s_5-216s_3s_4^2s_5-144s_3s_4s_5^2-192a_1b^2s_2-\\
-1170ba_1^2x-144a_1b^2s_3-96a_1b^2s_4-48a_1b^2s_5-144b^2s_2s_3-96b^2s_2s_4-48b^2s_2s_5-96b^2s_3s_4-\\
-144s_3s_4^3-48b^2s_3s_5-48b^2s_4s_5-360a_1bs_4^2-288bs_2s_4^2-324bs_3^2s_4-264bs_3s_4^2-162a_1bs_5^2-\\
-192bs_2^2s_5-120bs_2s_5^2-162bs_3^2s_5-102bs_3s_5^2-120bs_4^2s_5-84bs_4s_5^2-432a_1^2s_4s_5-\\
-48b^2s_5x-432a_1s_2^2s_5-720a_1s_2s_4^2-324a_1s_2s_5^2-810a_1s_3^2s_4-405a_1s_3^2s_5-684a_1s_3s_4^2-\\
-1161a_1s_3^2x-279a_1s_3s_5^2-324a_1s_4^2s_5-234a_1s_4s_5^2-1332a_1^2s_2s_3-864a_1bs_2^2-888a_1^2s_2s_4-\\
-468ba_1^2s_4-876a_1^2s_3s_4-594a_1bs_3^2-1296a_1s_2^2s_3-864a_1s_2^2s_4-1188a_1s_2s_3^2-576bs_2^2s_3-\\
-120s_2^2s_5^2-(1125\delta_2^2+2745\delta_2-756)a_1^2-(348\delta_2+132)bs_3s_5-(228\delta_2+129)s_3s_5^2-192b^2s_2\delta_2
\end{multline*}
All coefficients of this polynomial are negative, so that $f_2<0$ for $a_2-a_1>\delta_2$.
\end{proof}

\begin{lemma}
\label{lemma:Maple-P1-P1-d-2-b}
Suppose that $f$ is the polynomial \eqref{equation:polynomial-P1xP1-d-1-b} and $1+a_2+a_3\leqslant a_4+a_5+a_6$.
If~$a_2-a_1\geqslant \frac{23}{25}$, then $f<0$.
Similarly, if  $a_2-a_1\geqslant \frac{9}{10}$, then $f(a_1,a_2,a_3,a_4,a_5,a_6,1,b)<0$.
\end{lemma}

\begin{proof}
Let $g_1(a_1,s_1,s_2,s_3,s_4,s_5,s_6,b)=\widehat{f}$. Then $g_1(a_1,x+\frac{23}{25},s_2,s_3,s_4,s_5,s_6,b)$ equals
$$
\frac{697537}{390625}-\frac{525238}{3125}a_1-\frac{2188576}{15625}s_2-\frac{1753178}{15625}s_3-\frac{1318269}{15625}s_4-\frac{176672}{3125}s_5+\bigstar,
$$
where $\bigstar$ is the polynomial
\begin{multline*}
-\frac{25132}{15625}b-\frac{2623974}{15625}x-231a_1bs_2s_3-\frac{423}{2}a_1s_2s_3s_4-174a_1bs_2s_4-\\
-\frac{109}{2}a_1^3b-15a_1^2b^2-\frac{27763}{100}a_1^3-\frac{1076637}{2500}a_1^2-\frac{5988}{25}ba_1^2-\frac{147}{625}b^2-\frac{99624}{625}ba_1-\\
-\frac{225}{2}a_1s_2s_3s_6-99a_1s_2s_4s_6-\frac{171}{2}a_1s_2s_5s_6-81a_1s_3s_4s_6-\frac{135}{2}a_1s_3s_5s_6-\\
-54a_1bs_3s_6-48a_1bs_4s_6-42a_1bs_5s_6-42bs_2s_3s_6-36bs_2s_4s_6-30bs_2s_5s_6-30bs_3s_4s_6-24bs_3s_5s_6-\\
-\frac{247}{4}a_1^3s_3-\frac{157}{2}a_1^3s_2-7a_1^4-\frac{153867}{1250}a_1s_6-\frac{3429}{125}bs_6-\frac{62037}{625}s_6s_2-\frac{113001}{1250}s_6s_3-\\
-18bs_4s_5s_6-\frac{828}{25}a_1b^2-\frac{109}{4}a_1^3s_6-\frac{155}{4}a_1^3s_5-\frac{201}{4}a_1^3s_4-\frac{90867}{1250}s_6s_5-\\
-\frac{193749}{2500}s_5^2-\frac{216831}{1250}s_5s_4-\frac{239913}{1250}s_5s_3-\frac{262029}{1250}s_5s_2-\frac{321279}{1250}s_5a_1-\frac{50967}{625}s_6s_4-\\
-\frac{82932}{625}s_4^2-\frac{26319}{100}a_1^2s_5-\frac{1977}{5}a_1^2s_4-\frac{52761}{100}a_1^2s_3-\frac{66387}{100}a_1^2s_2-\frac{33984}{625}bs_5-\\
-105a_1bs_3s_5-96a_1bs_4s_5-120bs_2s_3s_4-81bs_2s_3s_5-72bs_2s_4s_5-60bs_3s_4s_5-\frac{135}{2}s_2s_3s_4s_5-\\
-\frac{50823}{625}bs_4-\frac{67662}{625}bs_3-\frac{84507}{625}bs_2-\frac{6549}{50}a_1^2s_6-\frac{488691}{1250}a_1s_4-\frac{14673}{50}s_3s_4-\\
-51a_1s_4s_5s_6-51s_2s_3s_4s_6-\frac{87}{2}s_2s_3s_5s_6-33s_2s_4s_5s_6-24s_3s_4s_5s_6-60a_1bs_2s_6-\frac{798}{25}s_6^2-\\
-\frac{493737}{2500}s_3^2-\frac{537939}{1250}s_3s_2-\frac{656103}{1250}s_3a_1-\frac{199992}{625}s_2s_4-\frac{168492}{625}s_2^2-\frac{822579}{1250}a_1s_2-\\
-\frac{237}{4}s_2x^3-\frac{49503}{100}s_2x^2-\frac{93}{2}s_3x^3-\frac{39309}{100}s_3x^2-\frac{153}{4}s_4x^3-\frac{7383}{25}s_4x^2-\frac{171}{2}s_2^2x^2-\\
-117a_1bs_2s_5-162a_1s_2s_3s_5-\frac{321}{2}a_1s_2s_4s_5-\frac{255}{2}a_1s_3s_4s_5-156a_1bs_3s_4-\\
-\frac{11208}{25}s_2^2x-57s_3^2x^2-\frac{31113}{100}s_3^2x-\frac{117}{2}s_2^3x-30s_5x^3-\frac{3951}{20}s_5x^2-\frac{207}{4}s_4^2x^2-\frac{10611}{50}s_4^2x-\\
-\frac{5424}{25}bx^2-\frac{4989}{50}s_6x^2-\frac{75}{4}s_6^2x^2-\frac{99}{2}s_6^2x-12s_5^3x-\frac{339}{2}a_1^2x^2-\frac{405}{4}a_1x^3-\frac{14289}{20}a_1x^2-\\
-\frac{56829}{250}s_5x-\frac{80013}{100}a_1^2x-\frac{101352}{625}bx-\frac{433143}{1250}s_4x-\frac{582141}{1250}s_3x-\frac{197811}{250}a_1x-\frac{730197}{1250}s_2x-\\
-\frac{45}{2}s_3^3x-\frac{87}{4}s_6x^3-39s_5^2x^2-\frac{12501}{100}s_5^2x-24s_4^3x-18b^2x^2-\frac{828}{25}b^2x-\frac{381}{4}a_1^3x-\frac{135147}{1250}s_6x-
\end{multline*}

\begin{multline*}
-\frac{19899}{100}x^3-18x^4-\frac{291}{2}a_1s_2s_6x-\frac{261}{2}a_1s_3s_6x-\frac{183}{2}s_2s_3s_6x-117a_1s_4s_6x-81s_2s_4s_6x-\\
-\frac{207}{2}a_1s_5s_6x-\frac{141}{2}s_2s_5s_6x-66s_3s_4s_6x-\frac{111}{2}s_3s_5s_6x-42s_4s_5s_6x-66a_1bs_6x-\\
-48bs_3s_6x-42bs_4s_6x-36bs_5s_6x-357a_1s_2s_3x-321a_1bs_2x-255a_1bs_3x-207bs_2s_3x-\\
-\frac{117}{2}bx^3-\frac{573}{2}a_1s_2s_4x-\frac{513}{2}a_1s_3s_4x-\frac{303}{2}s_2s_3s_4x-192a_1bs_4x-156bs_2s_4x-\\
-129a_1bs_5x-105bs_2s_5x-216a_1s_2s_5x-\frac{387}{2}a_1s_3s_5x-\frac{243}{2}s_2s_3s_5x-\frac{387}{2}a_1s_4s_5x-\\
-138bs_3s_4x-93bs_3s_5x-84bs_4s_5x-\frac{249}{2}s_2s_4s_5x-\frac{195}{2}s_3s_4s_5x-\frac{237}{4}s_6s_2x^2-\\
-54bs_2s_6x-15bs_6^2x-15s_4s_6^2x-\frac{303}{4}s_5s_4x^2-\frac{441}{2}a_1^2s_3x-\frac{231}{4}s_3^2s_4x-\frac{552}{25}b^2s_3-\\
-24b^2s_3x-\frac{105}{2}s_2^2s_6x-\frac{153}{4}s_3^2s_6x-48s_3^2s_5x-\frac{387}{4}a_1s_5^2x-\frac{17667}{25}a_1s_4x-\frac{15}{2}s_5s_6^2x-\\
-\frac{633}{4}a_1s_4x^2-45a_1s_6^2x-\frac{11676}{25}ba_1x-\frac{414}{25}b^2s_4-18b^2s_4x-\frac{99}{2}s_3s_5^2x-\frac{321}{4}a_1s_6x^2-\\
-\frac{5904}{25}a_1s_6x-30s_2s_6^2x-237a_1s_2^2x-\frac{213}{4}s_6s_3x^2-\frac{4488}{25}s_6s_2x-\frac{57}{4}s_5^2s_6x-\\
-99a_1s_2^2s_5-36a_1b^2x-\frac{336}{5}bs_6x-\frac{1119}{4}a_1^2s_2x-\frac{255}{4}s_2s_5^2x-\frac{99}{2}s_4^2s_5x-\frac{129}{2}s_3s_4^2x-\\
-\frac{138}{25}b^2s_6-6b^2s_6x-\frac{276}{25}b^2s_5-12b^2s_5x-\frac{387}{2}ba_1x^2-\frac{351}{4}a_1^2s_6x-30bs_6x^2-\\
-\frac{705}{4}a_1^2s_4x-\frac{23571}{50}s_5a_1x-\frac{507}{4}s_3s_2x^2-\frac{35211}{50}s_3s_2x-\frac{417}{4}s_2s_4x^2-\frac{13233}{25}s_2s_4x-\\
-\frac{375}{4}s_3s_4x^2-468s_3s_4x-120s_2^2s_3x-\frac{195}{2}s_2^2s_4x-\frac{363}{4}s_2s_3^2x-87bs_4x^2-\frac{5352}{25}bs_4x-\frac{117}{2}bs_5x^2-\\
-\frac{3516}{25}bs_5x-132a_1^2s_5x-\frac{477}{4}s_5a_1x^2-\frac{327}{4}s_5s_2x^2-\frac{17721}{50}s_5s_2x-\frac{147}{2}s_5s_3x^2-\frac{15687}{50}s_5s_3x-\\
-75s_2^2s_5x-135a_1s_4^2x-\frac{183}{2}bs_3^2x-63bs_4^2x-\frac{387}{2}ba_1^2x-\frac{231}{2}bs_3x^2-\frac{171}{4}s_6s_5x^2-\frac{2979}{25}s_6s_5x-\\
-\frac{7047}{25}s_5s_4x-\frac{171}{4}s_4s_5^2x-\frac{138}{5}b^2s_2-30b^2s_2x-\frac{75}{2}bs_5^2x-\frac{639}{4}a_1s_3^2x-\frac{47097}{50}s_3a_1x-\\
-132bs_2^2x-\frac{45}{2}s_3s_6^2x-27s_4^2s_6x-\frac{291}{2}bs_2x^2-\frac{9018}{25}bs_2x-\frac{789}{4}s_3a_1x^2-84s_2s_4^2x-\\
-\frac{501}{2}a_1s_2x^2-\frac{59271}{50}a_1s_2x-48s_6s_4x^2-\frac{3483}{25}s_6s_4x-\frac{3987}{25}s_6s_3x-\frac{7188}{25}bs_3x-\frac{878253}{2500}x^2-\\
-6b^2s_5^2-3b^2s_6^2-\frac{609}{10}a_1s_6^2-\frac{4161}{25}bs_2^2-\frac{3042}{25}bs_3^2-\frac{2049}{25}bs_4^2-48bs_5^2-\frac{99}{5}bs_6^2-\\
-\frac{1632}{5}s_2^2s_3-\frac{1221}{5}s_2^2s_4-162s_2^2s_5-\frac{399}{5}s_2^2s_6-\frac{27099}{100}s_2s_3^2-\frac{4557}{25}s_2s_4^2-\\
-21bs_4s_5^2-18a_1bs_6^2-24bs_2^2s_6-12bs_2s_6^2-18bs_3^2s_6-9bs_3s_6^2-12bs_4^2s_6-6bs_4s_6^2-\\
-\frac{4347}{25}s_3^2s_4-\frac{11691}{100}s_3^2s_5-\frac{2997}{50}s_3^2s_6-\frac{7617}{50}s_3s_4^2-\frac{4527}{50}s_4^2s_5-\frac{996}{25}s_4^2s_6-\\
-\frac{5766}{25}a_1bs_4-\frac{105}{4}a_1^2s_6^2-12s_2^2s_6^2-\frac{27}{4}s_3^2s_6^2-8s_4^3s_5-4s_4^3s_6-\frac{39}{4}s_4^2s_5^2-\\
-\frac{14187}{25}a_1s_3s_4-\frac{18951}{50}a_1s_3s_5-\frac{8538}{25}a_1s_4s_5-\frac{807}{4}a_1^2s_2s_3-147a_1bs_2^2-\frac{645}{4}a_1^2s_2s_4-\\
\end{multline*}

\begin{multline*}
-168a_1s_2^2s_3-3s_4^2s_6^2-\frac{2103}{20}s_2s_5^2-\frac{198}{5}s_2s_6^2-\frac{8529}{100}s_3s_5^2-\frac{297}{10}s_3s_6^2-\frac{1752}{25}s_4s_5^2-\\
-\frac{267}{2}a_1s_2^2s_4-\frac{261}{4}a_1s_3^2s_5-\frac{99}{5}s_4s_6^2-\frac{59}{4}a_1s_5^3-\frac{37}{4}s_2s_5^3-\frac{13}{2}s_3s_5^3-\frac{21}{4}s_4s_5^3-\frac{5}{4}s_5^3s_6-\\
-18b^2s_2s_4-12b^2s_2s_5-6b^2s_2s_6-18b^2s_3s_4-12b^2s_3s_5-6b^2s_3s_6-12b^2s_4s_5-6b^2s_4s_6-6b^2s_5s_6-\\
-\frac{483}{4}a_1^2s_2s_5-111a_1s_2s_4^2-\frac{159}{2}a_1s_2^3-\frac{309}{4}a_1^2s_4^2-\frac{135}{4}a_1s_3^3-42bs_2^3-\frac{45}{2}bs_3^3-\frac{75}{2}s_2^3s_3-30s_2^3s_4-\\
-\frac{585}{4}a_1^2s_3s_4-12b^2s_3^2-9b^2s_4^2-15s_2^4-\frac{3558}{25}s_2^3-\frac{1539}{20}s_3^3-\frac{1177}{25}s_4^3-3s_4^4-s_5^4-\frac{1829}{100}s_5^3-\\
-\frac{219}{2}a_1^2s_3s_5-\frac{45}{2}s_2^3s_5-\frac{69}{2}s_2^2s_3^2-33s_2^2s_4^2-\frac{45}{4}s_2s_3^3-\frac{111}{2}a_1^2s_5^2-31a_1s_4^3-15s_2^3s_6-\\
-42s_2s_3s_4^2-\frac{51}{2}s_2^2s_5^2-17s_2s_4^3-\frac{9}{4}s_3^3s_4-\frac{9}{2}s_3^3s_5-\frac{27}{4}s_3^3s_6-\frac{39}{4}s_3^2s_4^2-12s_3^2s_5^2-\\
-\frac{321}{4}a_1^2s_2s_6-10s_3s_4^3-\frac{13476}{25}a_1s_2^2-\frac{37797}{100}a_1s_3^2-\frac{2577}{10}a_1s_4^2-\frac{15201}{100}a_1s_5^2-15b^2s_2^2-\\
-\frac{291}{4}a_1^2s_3s_6-\frac{45}{4}s_2s_5^2s_6-\frac{33}{4}s_3s_5^2s_6-6s_4s_5^2s_6-9a_1s_5s_6^2-6s_2s_5s_6^2-\frac{9}{2}s_3s_5s_6^2-\\
-\frac{129}{2}a_1s_2^2s_6-\frac{2493}{25}s_3s_4s_6-3s_4s_5s_6^2-\frac{87}{2}a_1bs_5^2-\frac{63}{2}bs_2s_5^2-\frac{51}{2}bs_3s_5^2-27bs_4^2s_5-\\
-\frac{261}{2}a_1s_2s_3^2-\frac{993}{50}s_5^2s_6-\frac{3}{4}s_5^2s_6^2-\frac{99}{10}s_5s_6^2-14bs_4^3-\frac{11}{2}bs_5^3-\frac{267}{2}a_1^2s_2^2-\frac{183}{2}a_1^2s_3^2-\\
-66a_1^2s_4s_6-6bs_5^2s_6-3bs_5s_6^2-30a_1b^2s_2-24a_1b^2s_3-18a_1b^2s_4-12a_1b^2s_5-6a_1b^2s_6-24b^2s_2s_3-\\
-\frac{5334}{25}a_1s_2s_6-\frac{4764}{25}a_1s_3s_6-\frac{4191}{25}a_1s_4s_6-\frac{3618}{25}a_1s_5s_6-\frac{9708}{25}a_1bs_2-\frac{1548}{5}a_1bs_3-\\
-\frac{237}{4}a_1^2s_5s_6-\frac{3792}{25}a_1bs_5-\frac{1818}{25}a_1bs_6-\frac{6636}{25}bs_2s_3-\frac{4938}{25}bs_2s_4-\frac{648}{5}bs_2s_5-\frac{1542}{25}bs_2s_6-\\
-\frac{4524}{25}bs_3s_4-\frac{2964}{25}bs_3s_5-\frac{1404}{25}bs_3s_6-\frac{2682}{25}bs_4s_5-\frac{1266}{25}bs_4s_6-\frac{1128}{25}bs_5s_6-\\
-\frac{333}{4}a_1s_3^2s_4-36s_2s_4^2s_5-21s_2s_4^2s_6-\frac{129}{4}s_2s_4s_5^2-\frac{75}{4}s_3^2s_4s_5-18s_3^2s_4s_6-\frac{63}{4}s_3^2s_5s_6-\\
-\frac{10197}{25}s_2s_3s_4-\frac{13689}{50}s_2s_3s_5-\frac{3492}{25}s_2s_3s_6-\frac{6051}{25}s_2s_4s_5-\frac{2988}{25}s_2s_4s_6-\frac{1011}{5}s_3s_4s_5-\\
-36a_1s_2s_6^2-27a_1s_3s_6^2-18a_1s_4s_6^2-18s_2s_3s_6^2-12s_2s_4s_6^2-9s_3s_4s_6^2-9s_4^2s_5s_6-\frac{2484}{25}s_2s_5s_6-\\
-\frac{1989}{25}s_3s_5s_6-\frac{1491}{25}s_4s_5s_6-\frac{69}{4}a_1s_5^2s_6-72bs_2^2s_4-48bs_2^2s_5-\frac{159}{2}bs_2s_3^2-54bs_2s_4^2-51bs_3^2s_4-\\
-\frac{441}{4}a_1^2s_4s_5-\frac{69}{2}bs_3^2s_5-45bs_3s_4^2-\frac{117}{2}s_2^2s_3s_4-48s_2^2s_3s_5-\frac{99}{2}s_2^2s_4s_5-\frac{129}{4}s_2s_3^2s_4-\frac{123}{4}s_2s_3^2s_5-\\
-\frac{159}{2}a_1s_2s_5^2-\frac{189}{4}a_1s_3^2s_6-\frac{249}{4}a_1s_3s_5^2-63a_1s_4^2s_5-33a_1s_4^2s_6-\frac{213}{4}a_1s_4s_5^2-\\
-87a_1s_3s_4^2-\frac{207}{2}a_1bs_3^2--72a_1bs_4^2-\frac{75}{2}s_2^2s_3s_6-33s_2^2s_4s_6-\frac{57}{2}s_2^2s_5s_6-\frac{117}{4}s_2s_3^2s_6-\frac{147}{4}s_2s_3s_5^2-\\
-\frac{45}{2}s_3s_4^2s_5-15s_3s_4^2s_6-\frac{87}{4}s_3s_4s_5^2-\frac{42447}{50}a_1s_2s_3-\frac{15927}{25}a_1s_2s_4-\frac{21261}{50}a_1s_2s_5-\\
-96bs_2^2s_3-33ba_1^2s_6-\frac{129}{2}ba_1^2s_5-96ba_1^2s_4-\frac{255}{2}ba_1^2s_3-\frac{321}{2}ba_1^2s_2-\frac{448451}{15625}s_6.
\end{multline*}
All coefficients of the this polynomial are negative. But $2s_4+3s_3+2s_2+s_5+s_1+a_1\geqslant 1$.
This follows from $1+a_2+a_3\leqslant a_4+a_5+a_6$.
So, if $s_1\geqslant\frac{23}{25}$, then $2s_4+3s_3+2s_2+s_5+a_1\geqslant\frac{2}{25}$.
This inequality implies that $\frac{697537}{390625}-\frac{525238}{3125}a_1-\frac{2188576}{15625}s_2-\frac{1753178}{15625}s_3-\frac{1318269}{15625}s_4-\frac{176672}{3125}s_5<0$.
Thus, we see that $f(a_1,a_2,a_3,a_4,a_5,a_6,a_7,b)<0$ provided that $a_2-a_1>\frac{23}{25}=0.92$.

As above, we let $f_2=f(a_1,a_2,a_3,a_4,a_5,a_6,1,b)$ and we let $g_2(a_1,s_1,s_2,s_3,s_4,s_5,b)=\widehat{f}_2$.
Then $g_2(a_1,x+\frac{9}{10},s_2,s_3,s_4,s_5,b)=\frac{8153}{4000}-\frac{528843}{4000}a_1-\frac{423107}{4000}s_2-\frac{9912}{125}s_3-\frac{6608}{125}s_4-\frac{3304}{125}s_5+\blacklozenge$~for
\begin{multline*}
\blacklozenge=-108a_1s_2s_3s_4-102a_1bs_2s_4-\frac{2513}{2000}b-\frac{52903}{400}x-153a_1bs_2s_3-51a_1bs_2s_5-54a_1s_2s_3s_5-\\
-\frac{117}{2}bs_3^2x-30bs_2s_4s_5-24bs_3s_4s_5-21s_2s_3s_4s_5-27a_1b^2-\frac{55}{4}a_1^3s_5-\frac{55}{2}a_1^3s_4-\frac{165}{4}a_1^3s_3-\\
-75a_1s_4^2x-\frac{119}{2}a_1^3s_2-6a_1^4-\frac{891}{25}s_5^2-\frac{16137}{200}s_5s_4-\frac{9009}{100}s_5s_3-\frac{39411}{400}s_5s_2-\frac{48477}{400}s_5a_1-\\
-\frac{165}{4}a_1s_5^2x-\frac{16137}{200}s_4^2-\frac{1173}{10}a_1^2s_5-\frac{1173}{5}a_1^2s_4-\frac{3519}{10}a_1^2s_3-\frac{18927}{40}a_1^2s_2-\frac{5211}{200}bs_5-\\
-\frac{135}{2}a_1s_3s_5x-\frac{5211}{100}bs_4-\frac{15633}{200}bs_3-\frac{20847}{200}bs_2-\frac{48477}{200}a_1s_4-\frac{9009}{50}s_3s_4-\frac{27027}{200}s_3^2-\\
-60a_1s_2s_4s_5-45a_1s_3s_4s_5-90a_1bs_3s_4-45a_1bs_3s_5-42a_1bs_4s_5-66bs_2s_3s_4-33bs_2s_3s_5-\\
-27s_2^2s_5x-\frac{118233}{400}s_3s_2-\frac{145431}{400}s_3a_1-\frac{39411}{200}s_2s_4-\frac{39219}{200}s_2^2-\frac{12096}{25}a_1s_2-\frac{79}{2}a_1^3b-\\
-171a_1bs_3x-12a_1^2b^2-\frac{8159}{40}a_1^3-\frac{8211}{25}a_1^2-\frac{3669}{20}ba_1^2-\frac{3}{20}b^2-\frac{25491}{200}ba_1-\frac{183}{4}s_2x^3-\\
-36bs_4^2x-\frac{13941}{40}s_2x^2-\frac{63}{2}s_3x^3-\frac{2583}{10}s_3x^2-21s_4x^3-\frac{861}{5}s_4x^2-\frac{129}{2}s_2^2x^2-\frac{1563}{5}s_2^2x-\\
-\frac{291}{2}ba_1^2x-36s_3^2x^2-\frac{3951}{20}s_3^2x-45s_2^3x-\frac{21}{2}s_5x^3-\frac{861}{10}s_5x^2-\frac{57}{2}s_4^2x^2-\frac{1203}{10}s_4^2x-\frac{27}{2}s_3^3x-\\
-\frac{33}{2}s_5^2x^2-\frac{1089}{20}s_5^2x-14s_4^3x-15b^2x^2-27b^2x-\frac{311}{4}a_1^3x-\frac{21393}{200}s_5x-\frac{11889}{20}a_1^2x-\frac{26061}{200}bx-\\
-\frac{21393}{100}s_4x-\frac{64179}{200}s_3x-\frac{241641}{400}a_1x-\frac{170769}{400}s_2x-\frac{3339}{20}bx^2-\frac{11}{2}s_5^3x-\frac{273}{2}a_1^2x^2-\\
-\frac{153}{2}bs_3x^2-135a_1s_3s_4x-75s_2s_3s_4x-114a_1bs_4x-90bs_2s_4x-57a_1bs_5x-45bs_2s_5x-78a_1s_2s_5x-\\
-\frac{1317}{10}s_5s_3x-\frac{327}{4}a_1x^3-\frac{21189}{40}a_1x^2-\frac{87}{2}bx^3-\frac{585}{4}x^3-15x^4-234a_1s_2s_3x-231a_1bs_2x-\\
-156a_1s_2s_4x-\frac{75}{2}s_2s_3s_5x-75a_1s_4s_5x-78bs_3s_4x-39bs_3s_5x-36bs_4s_5x-45s_2s_4s_5x-\\
-\frac{351}{2}a_1s_2^2x-30a_1b^2x-\frac{843}{4}a_1^2s_2x-\frac{105}{4}s_2s_5^2x-21s_4^2s_5x-33s_3s_4^2x-\frac{27}{5}b^2s_5-\\
-33s_3s_4s_5x-135bs_2s_3x-\frac{4173}{10}a_1s_4x-\frac{177}{2}a_1s_4x^2-\frac{3609}{10}ba_1x-\frac{54}{5}b^2s_4-12b^2s_4x-\frac{39}{2}s_3s_5^2x-\\
-18a_1b^2s_3-12a_1b^2s_4-6a_1b^2s_5-18b^2s_2s_3-12b^2s_2s_4-6b^2s_2s_5-12b^2s_3s_4-6b^2s_3s_5-6b^2s_4s_5-\\
-\frac{3039}{10}s_2s_4x-48s_3s_4x^2-\frac{1317}{5}s_3s_4x-81s_2^2s_3x-54s_2^2s_4x-\frac{225}{4}s_2s_3^2x-51bs_4x^2-\frac{669}{5}bs_4x.
\end{multline*}

\begin{multline*}
-\frac{57}{2}s_5s_4x^2-\frac{51}{2}bs_5x^2-\frac{669}{10}bs_5x-\frac{99}{2}a_1^2s_5x-\frac{177}{4}s_5a_1x^2-\frac{111}{4}s_5s_2x^2-\frac{3039}{20}s_5s_2x-\\
-\frac{1857}{5}a_1s_2s_4-6b^2s_5x-\frac{291}{2}ba_1x^2-99a_1^2s_4x-\frac{4173}{20}s_5a_1x-\frac{333}{4}s_3s_2x^2-\frac{9117}{20}s_3s_2x-\frac{111}{2}s_2s_4x^2-\\
-\frac{27}{2}s_2s_4s_5^2-24s_5s_3x^2-\frac{1203}{10}s_5s_4x-18s_4s_5^2x-\frac{108}{5}b^2s_2-24b^2s_2x-\frac{33}{2}bs_5^2x-\frac{405}{4}a_1s_3^2x-\\
-93bs_2^2x-\frac{207}{2}bs_2x^2-\frac{2673}{10}bs_2x-\frac{531}{4}s_3a_1x^2-45s_2s_4^2x-189a_1s_2x^2-\frac{8427}{10}a_1s_2x-\frac{2007}{10}bs_3x-\\
-\frac{87}{2}a_1^2s_4s_5-\frac{10659}{40}x^2-3b^2s_5^2-\frac{1227}{10}bs_2^2-\frac{1683}{20}bs_3^2-\frac{252}{5}bs_4^2-\frac{447}{20}bs_5^2-\frac{1062}{5}s_2^2s_3-\\
-\frac{297}{2}a_1^2s_3x-42a_1bs_4^2-\frac{708}{5}s_2^2s_4-\frac{354}{5}s_2^2s_5-\frac{1341}{8}s_2s_3^2-\frac{201}{2}s_2s_4^2-\frac{459}{5}s_3^2s_4-\frac{459}{10}s_3^2s_5-\\
-63bs_2^2s_3-\frac{807}{10}s_3s_4^2-\frac{399}{10}s_4^2s_5-4s_4^3s_5-\frac{9}{2}s_4^2s_5^2-\frac{357}{8}s_2s_5^2-\frac{174}{5}s_3s_5^2-\frac{297}{10}s_4s_5^2-\\
-\frac{9}{2}s_3^2s_4s_5-\frac{57}{2}ba_1^2s_5-\frac{27}{4}a_1s_5^3-\frac{17}{4}s_2s_5^3-3s_3s_5^3-\frac{5}{2}s_4s_5^3-8bs_4^3-\frac{5}{2}bs_5^3-99a_1^2s_2^2-\frac{117}{2}a_1^2s_3^2-\\
-15s_2s_4^2s_5-57ba_1^2s_4-60a_1s_2^3-\frac{87}{2}a_1^2s_4^2-\frac{81}{4}a_1s_3^3-30bs_2^3-\frac{27}{2}bs_3^3-27s_2^3s_3-18s_2^3s_4-9s_2^3s_5-\\
-\frac{171}{2}ba_1^2s_3-\frac{45}{2}s_2^2s_3^2-18s_2^2s_4^2-\frac{27}{4}s_2s_3^3-24a_1^2s_5^2-18a_1s_4^3-\frac{21}{2}s_2^2s_5^2-10s_2s_4^3-\frac{9}{2}s_3^2s_4^2-\\
-27s_3^2s_4x-\frac{231}{2}ba_1^2s_2-\frac{9}{2}s_3^2s_5^2-6s_3s_4^3-\frac{7599}{20}a_1s_2^2-\frac{1953}{8}a_1s_3^2-\frac{297}{2}a_1s_4^2-\frac{537}{8}a_1s_5^2-\\
-\frac{75}{2}a_1s_2^2s_5-\frac{5571}{10}a_1s_2s_3-12b^2s_2^2-9b^2s_3^2-6b^2s_4^2-12s_2^4-\frac{201}{2}s_2^3-\frac{459}{10}s_3^3-\frac{133}{5}s_4^3-2s_4^4-\frac{1}{2}s_5^4-\\
-\frac{81}{2}a_1s_3^2s_4-39a_1^2s_3s_5-\frac{41}{5}s_5^3-\frac{39}{2}a_1bs_5^2-\frac{27}{2}bs_2s_5^2-\frac{21}{2}bs_3s_5^2-12bs_4^2s_5-9bs_4s_5^2-24a_1b^2s_2-\\
-9s_3s_4^2s_5-9s_3s_4s_5^2-\frac{2889}{10}a_1bs_2-\frac{2169}{10}a_1bs_3-\frac{723}{5}a_1bs_4-\frac{723}{10}a_1bs_5-\frac{369}{2}bs_2s_3-123bs_2s_4-\\
-\frac{177}{4}a_1^2s_2s_5-\frac{561}{5}bs_3s_4-\frac{561}{10}bs_3s_5-\frac{252}{5}bs_4s_5-\frac{447}{2}s_2s_3s_4-\frac{447}{4}s_2s_3s_5-\frac{201}{2}s_2s_4s_5-\\
-75a_1s_2^2s_4-45a_1s_3s_4^2-\frac{807}{10}s_3s_4s_5-42bs_2^2s_4-21bs_2^2s_5-\frac{99}{2}bs_2s_3^2-30bs_2s_4^2-27bs_3^2s_4-\frac{27}{2}bs_3^2s_5-\\
-\frac{81}{5}b^2s_3-\frac{81}{4}a_1s_3^2s_5-24bs_3s_4^2-30s_2^2s_3s_4-15s_2^2s_3s_5-18s_2^2s_4s_5-\frac{27}{2}s_2s_3^2s_4-\frac{27}{4}s_2s_3^2s_5-\\
-\frac{27}{2}s_3^2s_5x-18b^2s_3x-81a_1s_2s_3^2-21s_2s_3s_4^2-33a_1s_2s_5^2-\frac{99}{4}a_1s_3s_5^2-27a_1s_4^2s_5-\frac{45}{2}a_1s_4s_5^2-\frac{57}{4}s_2s_3s_5^2-\\
-\frac{1857}{10}a_1s_2s_5-\frac{651}{2}a_1s_3s_4-\frac{651}{4}a_1s_3s_5-\frac{297}{2}a_1s_4s_5-\frac{12519}{20}s_3a_1x-\frac{123}{2}bs_2s_5-\\
-\frac{135}{2}a_1bs_3^2-60a_1s_2s_4^2-\frac{531}{4}a_1^2s_2s_3-105a_1bs_2^2-\frac{177}{2}a_1^2s_2s_4-78a_1^2s_3s_4-\frac{225}{2}a_1s_2^2s_3.
\end{multline*}
If $s_1\geqslant\frac{9}{10}$, then $2s_4+3s_3+2s_2+s_5+a_1\geqslant\frac{1}{10}$, which gives $f(a_1,a_2,a_3,a_4,a_5,a_6,1,b)<0$,
because $\blacklozenge\leqslant 0$ and $\frac{8153}{4000}-\frac{528843}{4000}a_1-\frac{423107}{4000}s_2-\frac{9912}{125}s_3-\frac{6608}{125}s_4-\frac{3304}{125}s_5<0$.
\end{proof}

\begin{lemma}
\label{lemma:Maple-P1-P1-d-2-c}
Suppose that $f$ is the polynomial~\eqref{equation:polynomial-P1xP1-d-1-c} and $a_3+a_5+a_6\geqslant 1+a_2+a_4$.
If~$a_2-a_1\geqslant \frac{23}{25}$, then $f<0$.
Similarly, if  $a_2-a_1\geqslant \frac{9}{10}$, then $f(a_1,a_2,a_3,a_4,a_5,a_6,1,b)<0$.
\end{lemma}

\begin{proof}
Let $g_1(a_1,s_1,s_2,s_3,s_4,s_5,s_6,b)=\widehat{f}$. Then
\begin{multline*}
g_1(a_1,x+\frac{23}{25},s_2,s_3,s_4,s_5,s_6,b)=\frac{697374}{390625}+\frac{1}{2}s_3^4+\frac{9}{8}s_3^3s_4-\\
-\frac{1050381}{6250}a_1-\frac{2188092}{15625}s_2-\frac{28047}{250}s_3-\frac{1317783}{15625}s_4-\frac{353247}{6250}s_5+\bigstar,
\end{multline*}
where $\bigstar$ is the following polynomial with negative coefficients:
\begin{multline*}
-\frac{189}{2}ba_1^2s_4-\frac{507}{4}ba_1^2s_3-159ba_1^2s_2-\frac{255}{4}ba_1^2s_5-33ba_1^2s_6-144a_1bs_2^2-\frac{411}{4}a_1bs_3^2-\\
-93bs_2^2s_3-66bs_2^2s_4-45bs_2^2s_5-78bs_2s_3^2-48bs_2s_4^2-\frac{99}{2}bs_3^2s_4-\frac{135}{4}bs_3^2s_5-42bs_3s_4^2-\\
-24bs_4^2s_5-\frac{171}{4}a_1bs_5^2-30bs_2s_5^2-\frac{99}{4}bs_3s_5^2-\frac{39}{2}bs_4s_5^2-\frac{1491}{8}a_1^2s_2s_3-\frac{537}{4}a_1^2s_2s_4-\\
-\frac{855}{8}a_1^2s_2s_5-\frac{1053}{8}a_1^2s_3s_4-\frac{777}{8}a_1^2s_4s_5-144a_1s_2^2s_3-93a_1s_2^2s_4-78a_1s_2^2s_5-\\
-\frac{939}{8}a_1s_2s_3^2-\frac{147}{2}a_1s_2s_4^2-\frac{567}{8}a_1s_3^2s_4-\frac{471}{8}a_1s_3^2s_5-66a_1s_3s_4^2-45a_1s_4^2s_5-\\
-18bs_3^2s_6-\frac{555}{8}a_1s_2s_5^2-\frac{453}{8}a_1s_3s_5^2-\frac{351}{8}a_1s_4s_5^2-\frac{9714}{25}a_1bs_2-\frac{7743}{25}a_1bs_3-\\
-\frac{5772}{25}a_1bs_4-\frac{759}{5}a_1bs_5-\frac{6642}{25}bs_2s_3-198bs_2s_4-\frac{3246}{25}bs_2s_5-\frac{906}{5}bs_3s_4-\\
-\frac{2967}{25}bs_3s_5-\frac{2688}{25}bs_4s_5-\frac{42039}{50}a_1s_2s_3-\frac{15504}{25}a_1s_2s_4-\frac{20841}{50}a_1s_2s_5-\\
-\frac{2796}{5}a_1s_3s_4-\frac{37491}{100}a_1s_3s_5-333a_1s_4s_5-27s_2^2s_3s_4-\frac{63}{2}s_2^2s_3s_5-24s_2^2s_4s_5-\\
-102a_1^2s_3s_5-\frac{57}{4}s_2s_3^2s_4-\frac{171}{8}s_2s_3^2s_5-\frac{27}{2}s_2s_3s_4^2-\frac{231}{8}s_2s_3s_5^2-\frac{27}{2}s_2s_4^2s_5-\\
-\frac{81}{4}s_2s_4s_5^2-\frac{81}{8}s_3^2s_4s_5-9s_3s_4^2s_5-\frac{117}{8}s_3s_4s_5^2-\frac{9756}{25}s_2s_3s_4-\frac{13251}{50}s_2s_3s_5-\\
-\frac{5598}{25}s_2s_4s_5-\frac{4833}{25}s_3s_4s_5-\frac{5337}{25}a_1s_2s_6-\frac{9531}{50}a_1s_3s_6-\frac{4194}{25}a_1s_4s_6-\\
-\frac{159}{2}a_1^2s_2s_6-\frac{579}{8}a_1^2s_3s_6-\frac{261}{4}a_1^2s_4s_6-\frac{471}{8}a_1^2s_5s_6-63a_1s_2^2s_6-36a_1s_2s_6^2-\\
-\frac{375}{8}a_1s_3^2s_6-27a_1s_3s_6^2-\frac{63}{2}a_1s_4^2s_6-18a_1s_4s_6^2-18a_1bs_6^2-24bs_2^2s_6-12bs_2s_6^2-\\
-9bs_3s_6^2-12bs_4^2s_6-6bs_4s_6^2-36s_2^2s_3s_6-30s_2^2s_4s_6-27s_2^2s_5s_6-\frac{57}{2}s_2s_3^2s_6-\\
-3s_4s_5s_6^2-6bs_5^2s_6-3bs_5s_6^2-30a_1b^2s_2-24a_1b^2s_3-18a_1b^2s_4-24b^2s_2s_3-18b^2s_2s_4-\\
-69a_1bs_4^2-18s_2s_4^2s_6-\frac{69}{4}s_3^2s_4s_6-\frac{123}{8}s_3^2s_5s_6-\frac{27}{2}s_3s_4^2s_6-\frac{15}{2}s_4^2s_5s_6-
\end{multline*}

\begin{multline*}
-18b^2s_3s_4-12a_1b^2s_5-12b^2s_3s_5-12b^2s_4s_5-6a_1b^2s_6-6b^2s_2s_6-6b^2s_3s_6-\\
-\frac{699}{5}s_2s_3s_6-\frac{2994}{25}s_2s_4s_6-\frac{2487}{25}s_2s_5s_6-\frac{2496}{25}s_3s_4s_6-\frac{3981}{50}s_3s_5s_6-\\
-\frac{135}{8}a_1s_5^2s_6-\frac{21}{2}s_2s_5^2s_6-\frac{63}{8}s_3s_5^2s_6-\frac{21}{4}s_4s_5^2s_6-\frac{7239}{50}a_1s_5s_6-\\
-\frac{1494}{25}s_4s_5s_6-9a_1s_5s_6^2-18s_2s_3s_6^2-12s_2s_4s_6^2-6s_2s_5s_6^2-9s_3s_4s_6^2-\frac{9}{2}s_3s_5s_6^2-\\
-6b^2s_4s_6-6b^2s_5s_6-\frac{1818}{25}a_1bs_6-\frac{1542}{25}bs_2s_6-\frac{1404}{25}bs_3s_6-\frac{1266}{25}bs_4s_6-\\
-66a_1bs_6x-\frac{1128}{25}bs_5s_6-\frac{25134}{15625}b-\frac{5247471}{31250}x-\frac{798}{25}s_6^2-\frac{45432}{625}s_6s_5-\\
-204bs_2s_3x-\frac{50964}{625}s_6s_4-\frac{56499}{625}s_6s_3-\frac{62034}{625}s_6s_2-\frac{76932}{625}s_6a_1-\frac{448452}{15625}s_6-\\
-54bs_2s_6x-\frac{200946}{625}s_4s_2-\frac{244809}{625}s_4a_1-\frac{247101}{1250}s_3^2-\frac{269439}{625}s_3s_2-\frac{656559}{1250}s_3a_1-\\
-42bs_2s_3s_6-\frac{168966}{625}s_2^2-\frac{3294}{5}s_2a_1-\frac{97113}{1250}s_5^2-\frac{108894}{625}s_5s_4-\frac{120192}{625}s_5s_3-\\
-\frac{26298}{125}s_5s_2-\frac{321741}{1250}s_5a_1-\frac{83412}{625}s_4^2-\frac{36777}{125}s_4s_3-\frac{231}{2}a_1s_4s_6x-\\
-\frac{411}{4}a_1s_5s_6x-48bs_3s_6x-42bs_4s_6x-36bs_5s_6x-90s_2s_3s_6x-78s_2s_4s_6x-69s_2s_5s_6x-\\
-\frac{129}{2}s_3s_4s_6x-\frac{219}{4}s_3s_5s_6x-\frac{81}{2}s_4s_5s_6x-318a_1bs_2x-\frac{507}{2}a_1bs_3x-189a_1bs_4x-\\
-150bs_2s_4x-102bs_2s_5x-135bs_3s_4x-\frac{183}{2}bs_3s_5x-81bs_4s_5x-\frac{1311}{4}a_1s_2s_3x-\\
-\frac{471}{2}a_1s_2s_4x-\frac{759}{4}a_1s_2s_5x-\frac{915}{4}a_1s_3s_4x-\frac{717}{4}a_1s_3s_5x-\frac{675}{4}a_1s_4s_5x-\\
-\frac{219}{2}s_2s_3s_4x-\frac{399}{4}s_2s_3s_5x-\frac{177}{2}s_2s_4s_5x-\frac{309}{4}s_3s_4s_5x-144a_1s_2s_6x-\frac{519}{4}a_1s_3s_6x-\\
-30bs_2s_5s_6-30bs_3s_4s_6-24bs_3s_5s_6-18bs_4s_5s_6-48s_2s_3s_4s_6-42s_2s_3s_5s_6-\\
-30s_2s_4s_5s_6-\frac{45}{2}s_3s_4s_5s_6-228a_1bs_2s_3-168a_1bs_2s_4-153a_1bs_3s_4-114a_1bs_2s_5-\\
-\frac{207}{2}a_1bs_3s_5-93a_1bs_4s_5-114bs_2s_3s_4-78bs_2s_3s_5-66bs_2s_4s_5-57bs_3s_4s_5-\\
-\frac{255}{2}a_1bs_5x-\frac{333}{2}a_1s_2s_3s_4-\frac{555}{4}a_1s_2s_3s_5-\frac{243}{2}a_1s_2s_4s_5-\frac{423}{4}a_1s_3s_4s_5-\\
-\frac{75}{2}s_2s_3s_4s_5-111a_1s_2s_3s_6-96a_1s_2s_4s_6-84a_1s_2s_5s_6-\frac{159}{2}a_1s_3s_4s_6-\\
-\frac{267}{4}a_1s_3s_5s_6-\frac{99}{2}a_1s_4s_5s_6-60a_1bs_2s_6-54a_1bs_3s_6-48a_1bs_4s_6-42a_1bs_5s_6-\\
-36bs_2s_4s_6-\frac{9786}{25}a_1^2s_4-\frac{105141}{200}a_1^2s_3-\frac{65997}{100}a_1^2s_2-\frac{50817}{625}bs_4-\frac{67659}{625}bs_3-\\
-12b^2s_2s_5-\frac{84501}{625}bs_2-\frac{357}{8}a_1^3s_4-\frac{469}{8}a_1^3s_3-\frac{581}{8}a_1^3s_2-\frac{49}{8}a_1^4-\frac{217}{4}a_1^3b-\\
\end{multline*}

\begin{multline*}
-15a_1^2b^2-\frac{11081}{40}a_1^3-\frac{269271}{625}a_1^2-\frac{11979}{50}ba_1^2-\frac{147}{625}b^2-\frac{99621}{625}ba_1-\\
-\frac{9}{8}s_5^3s_6-10s_2^4-\frac{682}{5}s_2^3-\frac{15251}{200}s_3^3-\frac{1023}{25}s_4^3-\frac{3}{4}s_5^4-\frac{3507}{200}s_5^3-\\
-\frac{39671}{200}x^3-\frac{69}{4}x^4-111a_1s_4^2x-\frac{273}{8}s_4s_5^2x-30bs_6x^2-\frac{2103}{8}a_1^2s_2x-\\
-\frac{99}{5}s_4s_6^2-\frac{771}{4}ba_1x^2-\frac{423}{2}a_1^2s_3x-\frac{891}{8}s_5a_1x^2-\frac{12804}{25}s_4s_2x-\frac{459}{4}bs_3x^2-\\
-\frac{7191}{25}bs_3x-\frac{363}{4}bs_3^2x-\frac{771}{4}ba_1^2x-42s_3^2s_5x-\frac{903}{8}s_3s_2x^2-\frac{513}{8}s_5s_4x^2-\\
-\frac{177}{4}s_3s_5^2x-60bs_4^2x-\frac{93807}{100}s_3a_1x-\frac{1509}{8}s_3a_1x^2-\frac{1215}{8}a_1s_3^2x-\frac{699}{8}a_1^2s_6x-\\
-\frac{555}{8}s_5s_2x^2-\frac{267}{4}s_5s_3x^2-\frac{3519}{25}bs_5x-\frac{231}{4}bs_5x^2-\frac{1143}{8}s_4a_1x^2-\frac{495}{4}a_1^2s_5x-\\
-\frac{369}{8}s_3^2s_4x-\frac{321}{4}s_4s_2x^2-\frac{639}{8}s_6a_1x^2-45a_1s_6^2x-\frac{1875}{8}s_2a_1x^2-\frac{2355}{2}s_2a_1x-\\
-15s_4s_6^2x-\frac{15}{2}s_5s_6^2x-\frac{828}{25}a_1b^2-36a_1b^2x-\frac{138}{5}b^2s_2-30b^2s_2x-\frac{552}{25}b^2s_3-\\
-24b^2s_3x-\frac{414}{25}b^2s_4-18b^2s_4x-\frac{276}{25}b^2s_5-12b^2s_5x-\frac{138}{25}b^2s_6-6b^2s_6x-\frac{336}{5}bs_6x-\\
-\frac{99}{2}s_2s_4^2x-\frac{171}{2}bs_4x^2-\frac{5358}{25}bs_4x-129bs_2^2x-210a_1s_2^2x-144bs_2x^2-\frac{9024}{25}bs_2x-\\
-\frac{17466}{25}s_4a_1x-\frac{11679}{25}ba_1x-15bs_6^2x-45s_3s_4^2x-\frac{435}{8}s_2s_5^2x-\frac{117}{2}s_6s_2x^2-\\
-\frac{1989}{100}s_5^2s_6-\frac{4491}{25}s_6s_2x-\frac{1281}{8}a_1^2s_4x-\frac{11811}{50}s_6a_1x-\frac{195}{2}s_2^2s_3x-\frac{3459}{10}s_5s_2x-\\
-\frac{3}{4}s_5^2s_6^2-\frac{627}{8}s_2s_3^2x-\frac{34797}{50}s_3s_2x-\frac{147}{4}bs_5^2x-\frac{723}{8}a_1s_5^2x-\frac{30957}{100}s_5s_3x-\\
-\frac{423}{8}s_6s_3x^2-\frac{7977}{50}s_6s_3x-\frac{46743}{100}s_5a_1x-\frac{6831}{25}s_5s_4x-33s_4^2s_5x-\frac{111}{2}s_2^2s_5x-\\
-\frac{645}{8}s_4s_3x^2-\frac{2298}{5}s_4s_3x-60s_2^2s_4x-\frac{189}{4}s_6s_4x^2-\frac{3486}{25}s_6s_4x-\frac{339}{8}s_6s_5x^2-\\
-51s_2^2s_6x-\frac{303}{8}s_3^2s_6x-\frac{51}{2}s_4^2s_6x-\frac{111}{8}s_5^2s_6x-30s_2s_6^2x-\frac{175}{4}s_3x^3-\\
-\frac{3129}{8}s_3x^2-\frac{267}{8}s_4x^3-\frac{7281}{25}s_4x^2-\frac{55}{2}s_5x^3-\frac{7821}{40}s_5x^2-\frac{291}{4}s_2^2x^2-\\
-\frac{439353}{1250}x^2-\frac{21993}{50}s_2^2x-\frac{213}{4}s_3^2x^2-\frac{61821}{200}s_3^2x-\frac{81}{2}s_4^2x^2-\frac{5088}{25}s_4^2x-36s_5^2x^2-\\
-\frac{24573}{200}s_5^2x-45s_2^3x-\frac{81}{4}s_3^3x-\frac{27}{2}s_4^3x-\frac{9981}{100}s_6x^2-\frac{173}{8}s_6x^3-\frac{21}{2}s_5^3x-\\
-\frac{75}{4}s_6^2x^2-\frac{99}{2}s_6^2x-18b^2x^2-\frac{828}{25}b^2x-\frac{67572}{625}s_6x-\frac{217038}{625}s_4x-\frac{11652}{25}s_3x-\\
-\frac{49101}{100}s_2x^2-\frac{1317}{8}a_1^2x^2-\frac{785}{8}a_1x^3-\frac{28503}{40}a_1x^2-\frac{233}{4}bx^3-\frac{10851}{50}bx^2-\frac{433}{8}s_2x^3-
\end{multline*}

\begin{multline*}
-\frac{28461}{125}s_5x-\frac{197901}{250}a_1x-\frac{159657}{200}a_1^2x-\frac{101349}{625}bx-\frac{735}{8}a_1^3x-\\
-\frac{45}{2}s_3s_6^2x-\frac{129}{2}a_1^2s_4^2-65a_1s_2^3-\frac{251}{8}a_1s_3^3-\frac{39}{2}a_1s_4^3-\frac{417}{8}a_1^2s_5^2-\frac{4167}{25}bs_2^2-\\
-\frac{99}{10}s_5s_6^2-3b^2s_6^2-\frac{99}{5}bs_6^2-40bs_2^3-\frac{89}{4}bs_3^3-12bs_4^3-\frac{477}{4}a_1^2s_2^2-\frac{699}{8}a_1^2s_3^2-\\
-\frac{10449}{40}a_1^2s_5-\frac{6087}{50}bs_3^2-\frac{411}{5}bs_4^2-\frac{5307}{10}a_1s_2^2-\frac{15039}{40}a_1s_3^2-\frac{6228}{25}a_1s_4^2-\\
-\frac{287}{8}a_1^3s_5-26s_2^3s_3-12s_2^3s_4-13s_2^3s_5-\frac{99}{4}s_2^2s_3^2-9s_2^2s_4^2-\frac{75}{4}s_2^2s_5^2-\frac{61}{8}s_2s_3^3-\\
-\frac{365562}{625}s_2x-3s_2s_4^3-\frac{11}{4}s_3^3s_5-\frac{3}{2}s_3^2s_4^2-\frac{39}{4}s_3^2s_5^2-\frac{3}{2}s_3s_4^3-\frac{3177}{10}s_2^2s_3-\\
-\frac{1131}{5}s_2^2s_4-\frac{7653}{50}s_2^2s_5-\frac{26673}{100}s_2s_3^2-\frac{4101}{25}s_2s_4^2-\frac{4239}{25}s_3^2s_4-\\
-\frac{33981}{625}bs_5-\frac{22953}{200}s_3^2s_5-\frac{717}{5}s_3s_4^2-\frac{105}{4}a_1^2s_6^2-14s_2^3s_6-\frac{53}{8}s_3^3s_6-\frac{3}{2}s_4^3s_5-\\
-3s_4^3s_6-\frac{9}{2}s_4^2s_5^2-\frac{105}{8}a_1s_5^3-\frac{57}{8}s_2s_5^3-\frac{21}{4}s_3s_5^3-\frac{27}{8}s_4s_5^3-\\
-\frac{29979}{200}a_1s_5^2-\frac{1998}{25}s_2^2s_6-\frac{2013}{20}s_2s_5^2-\frac{5997}{100}s_3^2s_6-\frac{16617}{200}s_3s_5^2-\\
-\frac{2034}{25}s_4^2s_5-\frac{999}{25}s_4^2s_6-\frac{1638}{25}s_4s_5^2-12s_2^2s_6^2-\frac{27}{4}s_3^2s_6^2-3s_4^2s_6^2-\\
-\frac{21}{4}bs_5^3-15b^2s_2^2-12b^2s_3^2-9b^2s_4^2-6b^2s_5^2-\frac{2403}{50}bs_5^2-\frac{609}{10}a_1s_6^2-\\
-\frac{5961}{50}s_6s_5x-\frac{198}{5}s_2s_6^2-\frac{297}{10}s_3s_6^2-\frac{217}{8}a_1^3s_6-\frac{13101}{100}a_1^2s_6-\frac{3429}{125}bs_6.
\end{multline*}
Therefore, if $s_1\geqslant\frac{23}{25}$, then
\begin{multline*}
g_1(a_1,s_1,s_2,s_3,s_4,s_5,s_6,b)\leqslant\frac{697374}{390625}+\frac{1}{2}s_3^4+\frac{9}{8}s_3^3s_4-\\
-\frac{1050381}{6250}a_1-\frac{2188092}{15625}s_2-\frac{28047}{250}s_3-\frac{1317783}{15625}s_4-\frac{353247}{6250}s_5\leqslant\\
\leqslant\frac{697374}{390625}+\frac{1}{2}s_3+\frac{9}{8}s_3-\frac{1050381}{6250}a_1-\frac{2188092}{15625}s_2-\frac{28047}{250}s_3-\frac{1317783}{15625}s_4-\frac{353247}{6250}s_5=\\
=\frac{697374}{390625}-\frac{110563}{1000}s_3-\frac{1050381}{6250}a_1-\frac{2188092}{15625}s_2-\frac{1317783}{15625}s_4-\frac{353247}{6250}s_5<0.
\end{multline*}
Moreover, we have $2s_4+2s_2+s_5+s_3+s_1+a1\geqslant 1$, because $a_3+a_5+a_6\geqslant 1+a_2+a_4$.
Hence, if $s_1\geqslant\frac{23}{25}$, then $2s_4+2s_2+s_5+s_3+a1\geqslant\frac{2}{25}$, which implies that
$$
\frac{697374}{390625}-\frac{110563}{1000}s_3-\frac{1050381}{6250}a_1-\frac{2188092}{15625}s_2-\frac{1317783}{15625}s_4-\frac{353247}{6250}s_5<0.
$$
Thus, we see that $f(a_1,a_2,a_3,a_4,a_5,a_6,a_7,b)<0$ provided that $a_2-a_1>\frac{23}{25}=0.92$.

Let $f_2=f(a_1,a_2,a_3,a_4,a_5,a_6,1,b)$ and $g_2(a_1,s_1,s_2,s_3,s_4,s_5,b)=\widehat{f}_2$. Then
\begin{multline*}
g_2(a_1,x+\frac{9}{10},s_2,s_3,s_4,s_5,b)=\frac{32599}{16000}+\frac{3}{2}s_3^2s_4s_5+\frac{3}{8}s_3^4+\\
+\frac{5}{2}s_3^3s_4+\frac{5}{4}s_3^3s_5+\frac{3}{2}s_3^2s_4^2-\frac{1057497}{8000}a_1-\frac{211457}{2000}s_2-\\
-\frac{9909}{125}s_3-\frac{105631}{2000}s_4-\frac{105631}{4000}s_5+\blacklozenge,
\end{multline*}
where $\blacklozenge$ is the following polynomial with negative coefficients:
\begin{multline*}
-\frac{111}{2}ba_1^2s_4-\frac{339}{4}ba_1^2s_3-114ba_1^2s_2-\frac{111}{4}ba_1^2s_5-102a_1bs_2^2-\frac{267}{4}a_1bs_3^2-\\
-39a_1bs_4^2-60bs_2^2s_3-36bs_2^2s_4-18bs_2^2s_5-48bs_2s_3^2-24bs_2s_4^2-\frac{51}{2}bs_3^2s_4-\frac{51}{4}bs_3^2s_5-\\
-21bs_3s_4^2-9bs_4^2s_5-\frac{75}{4}a_1bs_5^2-12bs_2s_5^2-\frac{39}{4}bs_3s_5^2-\frac{15}{2}bs_4s_5^2-120a_1^2s_2s_3-\\
-66a_1^2s_2s_4-33a_1^2s_2s_5-66a_1^2s_3s_4-33a_1^2s_3s_5-33a_1^2s_4s_5-93a_1s_2^2s_3-42a_1s_2^2s_4-\\
-21a_1s_2^2s_5-\frac{141}{2}a_1s_2s_3^2-30a_1s_2s_4^2-\frac{123}{4}a_1s_3^2s_4-\frac{123}{8}a_1s_3^2s_5-\frac{57}{2}a_1s_3s_4^2-\\
-\frac{27}{2}a_1s_4^2s_5-\frac{51}{2}a_1s_2s_5^2-\frac{165}{8}a_1s_3s_5^2-\frac{63}{4}a_1s_4s_5^2-\frac{1446}{5}a_1bs_2-\\
-\frac{4341}{20}a_1bs_3-\frac{1449}{10}a_1bs_4-\frac{1449}{20}a_1bs_5-\frac{924}{5}bs_2s_3-\frac{618}{5}bs_2s_4-\\
-\frac{309}{5}bs_2s_5-\frac{225}{2}bs_3s_4-\frac{225}{4}bs_3s_5-\frac{507}{10}bs_4s_5-549a_1s_2s_3-\frac{1773}{5}a_1s_2s_4-\\
-\frac{1773}{10}a_1s_2s_5-\frac{1269}{4}a_1s_3s_4-\frac{1269}{8}a_1s_3s_5-\frac{2799}{20}a_1s_4s_5-6s_2^2s_3s_4-\\
-3s_2^2s_3s_5-9s_2s_3s_5^2-6s_2s_4s_5^2-\frac{9}{2}s_3s_4s_5^2-\frac{1029}{5}s_2s_3s_4-\frac{1029}{10}s_2s_3s_5-\\
-\frac{411}{5}s_2s_4s_5-\frac{717}{10}s_3s_4s_5-24a_1b^2s_2-18a_1b^2s_3-12a_1b^2s_4-18b^2s_2s_3-12b^2s_2s_4-\\
-6a_1b^2s_5-6b^2s_2s_5-6b^2s_3s_5-6b^2s_4s_5-\frac{5027}{4000}b-\frac{105787}{800}x-\frac{4974}{25}s_4s_2-\\
-12b^2s_3s_4-\frac{97323}{400}s_4a_1-\frac{54147}{400}s_3^2-\frac{7413}{25}s_3s_2-\frac{11649}{32}s_3a_1-\frac{4926}{25}s_2^2-\\
-\frac{96951}{200}s_2a_1-\frac{897}{25}s_5^2-\frac{16329}{200}s_5s_4-\frac{1449}{16}s_5s_3-33bs_4s_5x-\frac{2487}{25}s_5s_2-\\
-48a_1bs_2s_5-\frac{87}{2}a_1bs_3s_5-39a_1bs_4s_5-60bs_2s_3s_4-30bs_2s_3s_5-24bs_2s_4s_5-\\
-\frac{97323}{800}s_5a_1-\frac{16329}{200}s_4^2-\frac{1449}{8}s_4s_3-228a_1bs_2x-\frac{339}{2}a_1bs_3x-111a_1bs_4x-\\
-\frac{111}{2}a_1bs_5x-132bs_2s_3x-84bs_2s_4x-42bs_2s_5x-75bs_3s_4x-\frac{75}{2}bs_3s_5x-\\
\end{multline*}

\begin{multline*}
-21bs_3s_4s_5-72a_1s_2s_3s_4-36a_1s_2s_3s_5-30a_1s_2s_4s_5-\frac{57}{2}a_1s_3s_4s_5-\frac{2307}{10}a_1^2s_4-\\
-210a_1s_2s_3x-114a_1s_2s_4x-57a_1s_2s_5x-\frac{225}{2}a_1s_3s_4x-\frac{225}{4}a_1s_3s_5x-\frac{111}{2}a_1s_4s_5x-\\
-\frac{3369}{40}bs_3^2-\frac{14001}{40}a_1^2s_3-\frac{9387}{20}a_1^2s_2-\frac{10419}{200}bs_4-\frac{31263}{400}bs_3-\frac{5211}{50}bs_2-\\
-21s_2s_3s_5x-18s_2s_4s_5x-18s_3s_4s_5x-150a_1bs_2s_3-96a_1bs_2s_4-87a_1bs_3s_4-\\
-\frac{45}{8}a_1s_5^3-\frac{303}{4}s_4a_1x^2-\frac{91}{4}a_1^3s_4-\frac{309}{8}a_1^3s_3-\frac{109}{2}a_1^3s_2-\frac{21}{4}a_1^4-\frac{157}{4}a_1^3b-12a_1^2b^2-\\
-\frac{3}{2}s_4^2s_5^2-36s_4s_2x^2-\frac{16271}{80}a_1^3-\frac{262929}{800}a_1^2-\frac{7341}{40}ba_1^2-\frac{3}{20}b^2-\frac{50979}{400}ba_1-8s_2^4-\\
-42s_2s_3s_4x-18s_3^2s_4x-\frac{473}{5}s_2^3-\frac{226}{5}s_3^3-\frac{102}{5}s_4^3-\frac{3}{8}s_5^4-\frac{297}{40}s_5^3-\frac{1165}{8}x^3-\\
-\frac{171}{4}a_1^2s_5x-\frac{115}{8}x^4-\frac{111}{2}a_1s_4^2x-12s_4s_5^2x-\frac{393}{2}a_1^2s_2x-\frac{579}{4}ba_1x^2-141a_1^2s_3x-\\
-\frac{351}{2}s_2a_1x^2-\frac{303}{8}s_5a_1x^2-\frac{1434}{5}s_4s_2x-\frac{303}{4}bs_3x^2-\frac{4017}{20}bs_3x-\frac{231}{4}bs_3^2x-\\
-\frac{579}{4}ba_1^2x-9s_3^2s_5x-72s_3s_2x^2-\frac{39}{2}s_5s_4x^2-\frac{63}{4}s_3s_5^2x-33bs_4^2x-\frac{4977}{8}s_3a_1x-\\
-\frac{1005}{8}s_3a_1x^2-\frac{759}{8}a_1s_3^2x-18s_5s_2x^2-\frac{75}{4}s_5s_3x^2-\frac{1341}{20}bs_5x-\frac{99}{4}bs_5x^2-\\
-30a_1b^2x-\frac{108}{5}b^2s_2-24b^2s_2x-\frac{81}{5}b^2s_3-18b^2s_3x-\frac{54}{5}b^2s_4-12b^2s_4x-\frac{27}{5}b^2s_5-\\
-6b^2s_5x-18s_2s_4^2x-\frac{99}{2}bs_4x^2-\frac{1341}{10}bs_4x-90bs_2^2x-153a_1s_2^2x-102bs_2x^2-\\
-\frac{1338}{5}bs_2x-\frac{8187}{20}s_4a_1x-\frac{7221}{20}ba_1x-18s_3s_4^2x-\frac{39}{2}s_2s_5^2x-\frac{171}{2}a_1^2s_4x-\\
-\frac{19371}{80}a_1s_3^2-63s_2^2s_3x-\frac{717}{5}s_5s_2x-\frac{93}{2}s_2s_3^2x-\frac{2238}{5}s_3s_2x-\frac{63}{4}bs_5^2x-\frac{291}{8}a_1s_5^2x-\\
-\frac{255}{2}s_5s_3x-\frac{8187}{40}s_5a_1x-\frac{558}{5}s_5s_4x-9s_4^2s_5x-12s_2^2s_5x-\frac{75}{2}s_4s_3x^2-255s_4s_3x-\\
-24s_2^2s_4x-\frac{117}{4}s_3x^3-\frac{5127}{20}s_3x^2-17s_4x^3-\frac{3363}{20}s_4x^2-\frac{17}{2}s_5x^3-\frac{3363}{40}s_5x^2-\\
-\frac{3717}{10}a_1s_2^2-54s_2^2x^2-\frac{1521}{5}s_2^2x-33s_3^2x^2-\frac{7821}{40}s_3^2x-\frac{39}{2}s_4^2x^2-\frac{558}{5}s_4^2x-\frac{57}{4}s_5^2x^2-\\
-6s_2^2s_5^2-\frac{2091}{40}s_5^2x-34s_2^3x-\frac{47}{4}s_3^3x-6s_4^3x-\frac{9}{2}s_5^3x-15b^2x^2-27b^2x-\frac{10743}{50}s_4x-\\
-\frac{8349}{10}s_2a_1x-\frac{85569}{200}s_2x-\frac{10743}{100}s_5x-\frac{483639}{800}a_1x-\frac{11853}{20}a_1^2x-\frac{52119}{400}bx-\\
-\frac{507}{10}bs_4^2-\frac{599}{8}a_1^3x-\frac{1059}{8}a_1^2x^2-\frac{633}{8}a_1x^3-\frac{42231}{80}a_1x^2-\frac{173}{4}bx^3-\frac{6681}{40}bx^2-\\
-27a_1b^2-\frac{83}{2}s_2x^3-\frac{6891}{20}s_2x^2-\frac{2667}{10}x^2-28bs_2^3-\frac{53}{4}bs_3^3-6bs_4^3-87a_1^2s_2^2-\\
-3s_3^2s_5^2-\frac{441}{8}a_1^2s_3^2-33a_1^2s_4^2-48a_1s_2^3-\frac{147}{8}a_1s_3^3-9a_1s_4^3-\frac{171}{8}a_1^2s_5^2-123bs_2^2-
\end{multline*}

\begin{multline*}
-4s_2s_3^3-\frac{2799}{20}a_1s_4^2-18s_2^3s_3-\frac{2037}{10}s_2^2s_3-\frac{618}{5}s_2^2s_4-\frac{309}{5}s_2^2s_5-\frac{3267}{20}s_2s_3^2-\frac{411}{5}s_2s_4^2-\\
-\frac{1749}{20}s_3^2s_4-4s_2^3s_4-3s_2s_5^3-\frac{9}{4}s_3s_5^3-\frac{3}{2}s_4s_5^3-\frac{5199}{80}a_1s_5^2-\frac{801}{20}s_2s_5^2-\frac{651}{20}s_3s_5^2-\\
-\frac{1749}{40}s_3^2s_5-2s_2^3s_5-\frac{501}{20}s_4s_5^2-\frac{9}{4}bs_5^3-12b^2s_2^2-9b^2s_3^2-6b^2s_4^2-3b^2s_5^2-\frac{897}{40}bs_5^2-\\
-\frac{717}{10}s_3s_4^2-15s_2^2s_3^2-\frac{128541}{400}s_3x-\frac{153}{5}s_4^2s_5-\frac{2307}{20}a_1^2s_5-\frac{10419}{400}bs_5-\frac{91}{8}a_1^3s_5.
\end{multline*}
Thus, if $s_1\geqslant\frac{9}{10}$, then
\begin{multline*}
g_2(a_1,s_1,s_2,s_3,s_4,s_5,b)\leqslant\frac{32599}{16000}+\frac{3}{2}s_3^2s_4s_5+\frac{3}{8}s_3^4+\frac{5}{2}s_3^3s_4+\frac{5}{4}s_3^3s_5+\frac{3}{2}s_3^2s_4^2-\\
-\frac{1057497}{8000}a_1-\frac{211457}{2000}s_2-\frac{9909}{125}s_3-\frac{105631}{2000}s_4-\frac{105631}{4000}s_5\leqslant\\
\leqslant\frac{32599}{16000}+\frac{3}{2}s_5+\frac{3}{8}s_3+\frac{5}{2}s_4+\frac{5}{4}s_5+\frac{3}{2}s_4-\\
-\frac{1057497}{8000}a_1-\frac{211457}{2000}s_2-\frac{9909}{125}s_3-\frac{105631}{2000}s_4-\frac{105631}{4000}s_5=\\
=\frac{32599}{16000}-\frac{94631}{4000}s_5-\frac{78897}{1000}s_3-\frac{97631}{2000}s_4-\frac{1057497}{8000}a_1-\frac{211457}{2000}s_2.
\end{multline*}
If $s_1\geqslant\frac{9}{10}$, then $2s_4+2s_2+s_5+s_3+a_1\geqslant\frac{1}{10}$.
This gives
$$
\frac{32599}{16000}-\frac{94631}{4000}s_5-\frac{78897}{1000}s_3-\frac{97631}{2000}s_4-\frac{1057497}{8000}a_1-\frac{211457}{2000}s_2<0.
$$
Thus, we see that $f(a_1,a_2,a_3,a_4,a_5,a_6,1,b)<0$ provided that $a_2-a_1>\frac{9}{10}$.
\end{proof}

\begin{lemma}
\label{lemma:Maple-P1-P1-d-2-d}
Suppose that $f$ is the polynomial~\eqref{equation:polynomial-P1xP1-d-1-d} and $a_3+a_4+a_6\geqslant 1+a_2+a_5$.
If~$a_2-a_1\geqslant \frac{23}{25}$, then $f<0$.
Similarly, if  $a_2-a_1\geqslant \frac{9}{10}$, then $f(a_1,a_2,a_3,a_4,a_5,a_6,1,b)<0$.
\end{lemma}

\begin{proof}
Let $g_1(a_1,s_1,s_2,s_3,s_4,s_5,s_6,b)=\widehat{f}$. Then
\begin{multline*}
g_1(a_1,x+\frac{23}{25},s_2,s_3,s_4,s_5,s_6,b)=\frac{697211}{390625}+s_3^4+\frac{3}{2}s_3^3s_4-\\
-\frac{7063}{125}s_5-\frac{525143}{3125}a_1-\frac{1317786}{15625}s_4-\frac{1752697}{15625}s_3-\frac{2187608}{15625}s_2+\bigstar,
\end{multline*}
where $\bigstar$ is the following polynomial with negative coefficients:
\begin{multline*}
-126a_1bs_5x-\frac{25136}{15625}b-\frac{2623497}{15625}x-315a_1bs_2x-252a_1bs_3x-201bs_2s_3x-231a_1s_2s_4x-\\
-54s_3s_5s_6x-189a_1bs_4x-150bs_2s_4x-\frac{453}{2}a_1s_3s_4x-135bs_3s_4x-105s_2s_3s_4x-\frac{327}{2}a_1s_2s_5x-\\
-\frac{135}{2}s_2s_5s_6x-129a_1s_3s_6x-\frac{177}{2}s_2s_3s_6x-\frac{231}{2}a_1s_4s_6x-78s_2s_4s_6x-\frac{129}{2}s_3s_4s_6x-\\
-102a_1s_5s_6x-165a_1s_3s_5x-78s_2s_3s_5x-\frac{333}{2}a_1s_4s_5x-84s_2s_4s_5x-75s_3s_4s_5x-\frac{285}{2}a_1s_2s_6x-
\end{multline*}

\begin{multline*}
-\frac{81}{2}s_4s_5s_6x-66a_1bs_6x-54bs_2s_6x-48bs_3s_6x-42bs_4s_6x-36bs_5s_6x-\frac{597}{2}a_1s_2s_3x-\\
-225a_1bs_2s_3-168a_1bs_2s_4-153a_1bs_3s_4-162a_1s_2s_3s_4-114bs_2s_3s_4-\frac{231}{2}a_1s_2s_3s_5-\\
-99bs_2s_5x-117a_1s_2s_4s_5-\frac{207}{2}a_1s_3s_4s_5-33s_2s_3s_4s_5-\frac{219}{2}a_1s_2s_3s_6-96a_1s_2s_4s_6-\\
-90bs_3s_5x-\frac{159}{2}a_1s_3s_4s_6-48s_2s_3s_4s_6-\frac{165}{2}a_1s_2s_5s_6-66a_1s_3s_5s_6-\frac{81}{2}s_2s_3s_5s_6-\\
-111a_1bs_2s_5-102a_1bs_3s_5-93a_1bs_4s_5-75bs_2s_3s_5-66bs_2s_4s_5-57bs_3s_4s_5-\frac{99}{2}a_1s_4s_5s_6-\\
-81bs_4s_5x-30s_2s_4s_5s_6-\frac{45}{2}s_3s_4s_5s_6-60a_1bs_2s_6-54a_1bs_3s_6-48a_1bs_4s_6-42a_1bs_5s_6-\\
-\frac{3262}{25}s_2^3-42bs_2s_3s_6-\frac{39153}{100}a_1^2s_4-\frac{2619}{5}a_1^2s_3-\frac{65607}{100}a_1^2s_2-\frac{33978}{625}bs_5-\frac{50817}{625}bs_4-\\
-\frac{27}{2}s_4^3x-5s_2^4-36bs_2s_4s_6-\frac{67656}{625}bs_3-\frac{16899}{125}bs_2-\frac{798}{25}s_6^2-\frac{90861}{1250}s_6s_5-\frac{50964}{625}s_6s_4-\\
-45a_1s_4^2s_5-\frac{267}{4}a_1^3s_2-\frac{62031}{625}s_6s_2-\frac{153861}{1250}s_6a_1-33a_1^3s_5-\frac{177}{4}a_1^3s_4-\frac{111}{2}a_1^3s_3-\\
-\frac{322203}{1250}a_1s_5-24bs_3s_5s_6-\frac{489609}{1250}a_1s_4-\frac{194703}{2500}s_5^2-\frac{217779}{1250}s_5s_4-\frac{48171}{250}s_5s_3-\\
-\frac{263931}{1250}s_5s_2-\frac{83412}{625}s_4^2-\frac{367761}{1250}s_4s_3-\frac{200937}{625}s_4s_2-30bs_2s_5s_6-\frac{494667}{2500}s_3^2-\\
-\frac{824421}{1250}a_1s_2-30bs_3s_4s_6-18bs_4s_5s_6-\frac{131403}{250}a_1s_3-\frac{539817}{1250}s_2s_3-\frac{33888}{125}s_2^2-33bs_3^2s_5-\\
-\frac{12963}{50}a_1^2s_5-24bs_4^2s_5-\frac{39}{2}bs_4s_5^2-\frac{2496}{25}s_3s_4s_6-\frac{1992}{25}s_3s_5s_6-\frac{63}{2}a_1s_4^2s_6-18s_2s_4^2s_6-\\
-\frac{1889}{25}s_3^3-\frac{27}{2}s_3s_4^2s_6-18a_1s_4s_6^2-12s_2s_4s_6^2-9s_3s_4s_6^2-\frac{15}{2}s_4^2s_5s_6-\frac{1494}{25}s_4s_5s_6-\\
-\frac{387}{2}s_3s_4s_5-\frac{22599}{250}s_6s_3-\frac{33}{2}a_1s_5^2s_6-\frac{39}{4}s_2s_5^2s_6-\frac{15}{2}s_3s_5^2s_6-\frac{21}{4}s_4s_5^2s_6-9a_1s_5s_6^2-\\
-\frac{1023}{25}s_4^3-6s_2s_5s_6^2-\frac{9}{2}s_3s_5s_6^2-3s_4s_5s_6^2-18a_1bs_6^2-24bs_2^2s_6-12bs_2s_6^2-18bs_3^2s_6-\\
-\frac{839}{50}s_5^3-9bs_3s_6^2-12bs_4^2s_6-6bs_4s_6^2-6bs_5^2s_6-3bs_5s_6^2-\frac{5607}{25}s_2s_4s_5-\frac{2994}{25}s_2s_4s_6-\\
-42bs_3s_4^2-\frac{1}{2}s_5^4-\frac{498}{5}s_2s_5s_6-51a_1s_3s_5^2-27a_1s_3s_6^2-21s_2s_3s_5^2-18s_2s_3s_6^2-9s_3^2s_4s_5-\\
-\frac{27}{2}s_2s_4^2s_5-\frac{69}{4}s_3^2s_4s_6-15s_3^2s_5s_6-30a_1b^2s_2-24a_1b^2s_3-18a_1b^2s_4-12a_1b^2s_5-6a_1b^2s_6-\\
-24b^2s_2s_3-18b^2s_2s_4-12b^2s_2s_5-6b^2s_2s_6-18b^2s_3s_4-12b^2s_3s_5-6b^2s_3s_6-12b^2s_4s_5-\\
-24bs_3s_5^2-\frac{99}{2}bs_3^2s_4-\frac{594}{5}bs_3s_5-\frac{1404}{25}bs_3s_6-\frac{2688}{25}bs_4s_5-\frac{1266}{25}bs_4s_6-\frac{1128}{25}bs_5s_6-\\
-48bs_2s_4^2-6b^2s_4s_6-6b^2s_5s_6-\frac{1944}{5}a_1bs_2-\frac{7746}{25}a_1bs_3-\frac{5772}{25}a_1bs_4-\frac{3798}{25}a_1bs_5-\\
-69a_1bs_4^2-\frac{1818}{25}a_1bs_6-\frac{6648}{25}bs_2s_3-198bs_2s_4-\frac{3252}{25}bs_2s_5-\frac{1542}{25}bs_2s_6-\frac{906}{5}bs_3s_4-
\end{multline*}

\begin{multline*}
-9s_3s_4^2s_5-\frac{171}{4}a_1s_4s_5^2-18s_2s_4s_5^2-\frac{27}{2}s_3s_4s_5^2-42a_1bs_5^2-42bs_2^2s_5-\frac{57}{2}bs_2s_5^2-\\
-57a_1s_2^2s_5-\frac{45}{2}s_2^2s_3s_4-15s_2^2s_3s_5-\frac{39}{2}s_2^2s_4s_5-12s_2s_3^2s_4-12s_2s_3^2s_5-\frac{27}{2}s_2s_3s_4^2-\\
-\frac{732051}{1250}s_2x-96a_1^2s_4s_5-\frac{315}{4}a_1^2s_2s_6-72a_1^2s_3s_6-\frac{123}{2}a_1s_2^2s_6-\frac{261}{4}a_1^2s_4s_6-\frac{93}{2}a_1s_3^2s_6-\\
-\frac{583059}{1250}s_3x-\frac{147}{2}a_1s_2s_4^2-\frac{69}{2}s_2^2s_3s_6-30s_2^2s_4s_6-\frac{111}{4}s_2s_3^2s_6-\frac{117}{2}a_1^2s_5s_6-\frac{41631}{50}a_1s_2s_3-\\
-\frac{279}{4}a_1s_3^2s_4-\frac{15513}{25}a_1s_2s_4-\frac{20421}{50}a_1s_2s_5-\frac{1068}{5}a_1s_2s_6-\frac{27969}{50}a_1s_3s_4-\frac{1854}{5}a_1s_3s_5-\\
-\frac{105}{2}a_1s_3^2s_5-\frac{4767}{25}a_1s_3s_6-\frac{16659}{50}a_1s_4s_5-\frac{4194}{25}a_1s_4s_6-\frac{3621}{25}a_1s_5s_6-\frac{237}{4}a_1s_2s_5^2-\\
-49s_2x^3-36a_1s_2s_6^2-\frac{51}{2}s_2^2s_5s_6-\frac{1953}{5}s_2s_3s_4-\frac{12813}{50}s_2s_3s_5-\frac{3498}{25}s_2s_3s_6-171a_1^2s_2s_3-\\
-66a_1s_3s_4^2-141a_1bs_2^2-132a_1^2s_2s_4-\frac{261}{2}a_1^2s_3s_4-102a_1bs_3^2-120a_1s_2^2s_3-\frac{177}{2}a_1s_2^2s_4-\\
-\frac{48699}{100}s_2x^2-36a_1b^2x-\frac{417}{4}a_1s_2s_3^2-90bs_2^2s_3-66bs_2^2s_4-\frac{153}{2}bs_2s_3^2-93a_1^2s_2s_5-\frac{189}{2}a_1^2s_3s_5-\\
-\frac{197991}{250}a_1x-33ba_1^2s_6-63ba_1^2s_5-\frac{189}{2}ba_1^2s_4-126ba_1^2s_3-\frac{315}{2}ba_1^2s_2-\frac{448453}{15625}s_6-\frac{21}{4}a_1^4-\\
-\frac{434067}{1250}s_4x-54a_1^3b-15a_1^2b^2-\frac{13821}{50}a_1^3-\frac{1077531}{2500}a_1^2-\frac{5991}{25}ba_1^2-\frac{147}{625}b^2-\frac{99618}{625}ba_1-\\
-192ba_1x^2-\frac{828}{25}a_1b^2-\frac{15}{2}s_5s_6^2x-36bs_5^2x-\frac{51}{2}s_4^2s_6x-45s_3^2s_4x-\frac{105}{2}s_6s_3x^2-\frac{99}{2}s_2^2s_6x-\\
-36s_2^2s_5x-\frac{231}{4}s_6s_2x^2-39s_3s_5^2x-\frac{552}{25}b^2s_3-24b^2s_3x-30bs_6x^2-\frac{336}{5}bs_6x-\frac{138}{25}b^2s_6-\\
-\frac{33}{2}x^4-6b^2s_6x-\frac{189}{4}s_6s_4x^2-33s_4s_5^2x-\frac{5907}{25}s_6a_1x-\frac{45}{2}s_3s_6^2x-57bs_5x^2-\frac{3522}{25}bs_5x-\\
-15s_4s_6^2x-\frac{138}{5}b^2s_2-30b^2s_2x-\frac{75}{2}s_3^2s_6x-\frac{414}{25}b^2s_4-18b^2s_4x-36s_3^2s_5x-\frac{27}{2}s_5^2s_6x-\\
-\frac{4943}{25}x^3-\frac{171}{2}bs_4x^2-\frac{5358}{25}bs_4x-144a_1s_3^2x-90bs_3^2x-99s_2s_3x^2-\frac{34383}{50}s_2s_3x-78s_4s_2x^2-\\
-15bs_6^2x-\frac{12813}{25}s_4s_2x-\frac{159}{2}s_4s_3x^2-75s_2^2s_3x-\frac{111}{2}s_2^2s_4x-66s_2s_3^2x-\frac{231}{2}a_1^2s_5x-\\
-\frac{798}{5}s_6s_3x-\frac{207}{2}a_1s_5x^2-\frac{11586}{25}a_1s_5x-192ba_1^2x-60bs_4^2x-33s_4^2s_5x-57s_5s_2x^2-\frac{16869}{50}s_5s_2x-\\
-45a_1s_6^2x-\frac{285}{2}bs_2x^2-84a_1s_5^2x-\frac{99}{2}s_2s_4^2x-60s_5s_3x^2-\frac{1527}{5}s_5s_3x-111a_1s_4^2x-42s_6s_5x^2-63s_5s_4x^2-\\
-\frac{11682}{25}ba_1x-45s_3s_4^2x-\frac{276}{25}b^2s_5-12b^2s_5x-87a_1^2s_6x-\frac{22989}{50}s_4s_3x-\frac{4494}{25}s_6s_2x-\\
-183a_1s_2^2x-\frac{405}{2}a_1^2s_3x-\frac{873}{4}a_1s_2x^2-\frac{58479}{50}a_1s_2x-180a_1s_3x^2-18b^2x^2-\frac{828}{25}b^2x-\\
-\frac{4671}{5}a_1s_3x-\frac{13671}{50}s_5s_4x-\frac{3486}{25}s_6s_4x-\frac{2982}{25}s_6s_5x-45s_2s_5^2x-30s_2s_6^2x-246a_1^2s_2x-\\
-\frac{159}{2}s_6a_1x^2-\frac{1806}{5}bs_2x-114bs_3x^2-\frac{7194}{25}bs_3x-126bs_2^2x-159a_1^2s_4x-\frac{567}{4}a_1s_4x^2-\frac{34941}{50}a_1s_4x-
\end{multline*}

\begin{multline*}
-\frac{7677}{25}s_3^2x-41s_3x^3-\frac{9729}{25}s_3x^2-33s_4x^3-\frac{29133}{100}s_4x^2-60s_2^2x^2-\frac{2157}{5}s_2^2x-\frac{99}{2}s_3^2x^2-\\
-\frac{63}{2}s_2^3x-25s_5x^3-\frac{19911}{25}a_1^2x-\frac{101346}{625}bx-\frac{135141}{1250}s_6x-\frac{177}{2}a_1^3x-\frac{11403}{50}s_5x-\\
-\frac{5427}{25}bx^2-\frac{43}{2}s_6x^3-\frac{387}{2}s_5x^2-\frac{879159}{2500}x^2-\frac{1203}{25}bs_5^2-\frac{99}{5}bs_6^2-\frac{717}{5}s_3s_4^2-\frac{39}{2}a_1s_4^3-\\
-58bx^3-\frac{2496}{25}s_6x^2-3s_2s_4^3-\frac{3}{2}s_3s_4^3-22bs_3^3-12bs_4^3-\frac{2022}{25}s_3s_5^2-\frac{3}{2}s_4^3s_5-\frac{9}{2}s_4^2s_5^2-\\
-33s_5^2x^2-\frac{609}{10}a_1s_6^2-\frac{23}{2}a_1s_5^3-5s_2s_5^3-4s_3s_5^3-3s_4s_5^3-\frac{2034}{25}s_4^2s_5-\frac{6561}{100}s_4s_5^2-5bs_5^3-\\
-95a_1x^3-\frac{3018}{25}s_5^2x-\frac{297}{10}s_3s_6^2-3s_4^3s_6-\frac{999}{25}s_4^2s_6-3s_4^2s_6^2-\frac{99}{5}s_4s_6^2-s_5^3s_6-\frac{498}{25}s_5^2s_6-\\
-60s_3^2s_6-\frac{81}{2}s_4^2x^2-\frac{3}{4}s_5^2s_6^2-\frac{99}{10}s_5s_6^2-105a_1^2s_2^2-\frac{333}{4}a_1^2s_3^2-\frac{101}{2}a_1s_2^3-38bs_2^3-\\
-\frac{75}{4}s_6^2x^2-18s_3^3x-\frac{411}{5}bs_4^2-\frac{129}{2}a_1^2s_4^2-29a_1s_3^3-\frac{29}{2}s_2^3s_3-9s_2^3s_4-\frac{7}{2}s_2^3s_5-15s_2^2s_3^2-\\
-\frac{99}{2}s_6^2x-9s_2^2s_4^2-4s_2s_3^3-27a_1^3s_6-13s_2^3s_6-\frac{195}{4}a_1^2s_5^2-\frac{105}{4}a_1^2s_6^2-\frac{3276}{25}a_1^2s_6-\\
-\frac{5631}{50}s_3^2s_5-\frac{3393}{20}s_3^2s_4-\frac{5088}{25}s_4^2x-\frac{13059}{25}a_1s_2^2-\frac{18699}{50}a_1s_3^2-\frac{6228}{25}a_1s_4^2-\frac{7389}{50}a_1s_5^2-\\
-\frac{7107}{10}a_1x^2-\frac{609}{5}bs_3^2-309s_2^2s_3-\frac{5664}{25}s_2^2s_4-\frac{3603}{25}s_2^2s_5-\frac{2001}{25}s_2^2s_6-12s_2^2s_5^2-12s_2^2s_6^2-\\
-\frac{639}{4}a_1^2x^2-\frac{13}{2}s_3^3s_6-\frac{4173}{25}bs_2^2-\frac{26247}{100}s_2s_3^2-\frac{4101}{25}s_2s_4^2-\frac{1923}{20}s_2s_5^2-\frac{198}{5}s_2s_6^2-s_3^3s_5-\\
-9s_5^3x-\frac{3}{2}s_3^2s_4^2-3b^2s_6^2-6b^2s_5^2-\frac{15}{2}s_3^2s_5^2-\frac{27}{4}s_3^2s_6^2-\frac{3429}{125}bs_6-15b^2s_2^2-12b^2s_3^2-9b^2s_4^2.
\end{multline*}
In particular, if $s_1\geqslant\frac{23}{25}$, then
\begin{multline*}
g_1(a_1,s_1,s_2,s_3,s_4,s_5,s_6,b)\leqslant\frac{697211}{390625}+s_3^4+\frac{3}{2}s_3^3s_4-\\
-\frac{7063}{125}s_5-\frac{525143}{3125}a_1-\frac{1317786}{15625}s_4-\frac{1752697}{15625}s_3-\frac{2187608}{15625}s_2\leqslant\\
\leqslant\frac{697211}{390625}+s_3+\frac{3}{2}s_4-\frac{7063}{125}s_5-\frac{525143}{3125}a_1-\frac{1317786}{15625}s_4-\frac{1752697}{15625}s_3-\frac{2187608}{15625}s_2=\\
=\frac{697211}{390625}-\frac{1737072}{15625}s_3-\frac{2588697}{31250}s_4-\frac{7063}{125}s_5-\frac{525143}{3125}a_1-\frac{2187608}{15625}s_2.
\end{multline*}
On the other hand, we have $a_3+a_4+a_6\geqslant 1+a_2+a_5$. This gives $2s_2+s_5+s_3+s_1+a_1\geqslant 1$.
Thus, if $s_1\geqslant\frac{23}{25}$, then $2s_2+s_5+s_3+a_1\geqslant\frac{2}{25}$, which implies that
$$
\frac{697211}{390625}-\frac{1737072}{15625}s_3-\frac{2588697}{31250}s_4-\frac{7063}{125}s_5-\frac{525143}{3125}a_1-\frac{2187608}{15625}s_2<0.
$$
This show that $f(a_1,a_2,a_3,a_4,a_5,a_6,a_7,b)<0$ in the case when $a_2-a_1\geqslant \frac{23}{25}$.

Let $f_2=f(a_1,a_2,a_3,a_4,a_5,a_6,1,b)$ and $g_2(a_1,s_1,s_2,s_3,s_4,s_5,b)=\widehat{f}_2$. Then
\begin{multline*}
g_2(a_1,x+\frac{9}{10},s_2,s_3,s_4,s_5,b)=\frac{16293}{8000}+3s_2s_3s_4s_5+3s_2^2s_5x+\frac{3}{4}s_3^4+5s_2^3s_5+\\
+\frac{11}{4}s_3^3s_4+\frac{3}{2}s_3^2s_4^2+\frac{5}{2}s_3^3s_5+\frac{9}{4}s_3^2s_4s_5+9s_2^2s_3s_5+3s_2^2s_4s_5+\frac{3}{2}s_2s_3^2s_4+\frac{27}{4}s_2s_3^2s_5-\\
-\frac{52767}{2000}s_5-\frac{211263}{4000}s_4-\frac{9906}{125}s_3-\frac{422721}{4000}s_2-\frac{264327}{2000}a_1-\frac{13221}{100}x+\blacklozenge,
\end{multline*}
where $\blacklozenge$ is the following polynomial with negative coefficients:
\begin{multline*}
-111a_1s_3s_4x-\frac{264327}{2000}a_1-\frac{1257}{1000}b-225a_1bs_2x-168a_1bs_3x-\frac{282}{5}bs_3s_5-129bs_2s_3x-\\
-75bs_3s_4x-39s_2s_3s_4x-36a_1s_2s_5x-45a_1s_3s_5x-\frac{9}{2}s_2s_3s_5x-54a_1s_4s_5x-15s_2s_4s_5x-\\
-\frac{33}{2}s_3s_4s_5x-54a_1bs_5x-39bs_2s_5x-36bs_3s_5x-33bs_4s_5x-186a_1s_2s_3x-147a_1bs_2s_3-\\
-18a_1s_2s_3s_5-27a_1s_2s_4s_5-27a_1s_3s_4s_5-45a_1bs_2s_5-42a_1bs_3s_5-39a_1bs_4s_5-\\
-27bs_2s_3s_5-24bs_2s_4s_5-21bs_3s_4s_5-\frac{567}{5}a_1^2s_5-\frac{9231}{40}a_1^2s_4-\frac{6963}{20}a_1^2s_3-\frac{18621}{40}a_1^2s_2-\\
-\frac{651}{25}bs_5-\frac{10419}{200}bs_4-\frac{1563}{20}bs_3-\frac{20841}{200}bs_2-9a_1^3s_5-\frac{45}{2}a_1^3s_4-36a_1^3s_3-\frac{99}{2}a_1^3s_2-\\
-69a_1s_2s_3s_4-4s_2^4-\frac{887}{10}s_2^3-\frac{89}{2}s_3^3-\frac{102}{5}s_4^3-\frac{1}{4}s_5^4-\frac{133}{20}s_5^3-\frac{24423}{200}a_1s_5-\frac{2433}{10}a_1s_4-\\
-87a_1bs_3s_4-\frac{903}{25}s_5^2-\frac{6531}{80}s_5s_4-\frac{18207}{200}s_5s_3-\frac{40173}{400}s_5s_2-\frac{16329}{200}s_4^2-\frac{72447}{400}s_4s_3-\\
-111a_1s_2s_4x-\frac{39789}{200}s_4s_2-\frac{678}{5}s_3^2-\frac{48567}{100}a_1s_2-\frac{72897}{200}a_1s_3-\frac{118983}{400}s_2s_3-\frac{39597}{200}s_2^2-\\
-12bs_3^2s_5-9bs_3s_5^2-9bs_4^2s_5-\frac{15}{2}bs_4s_5^2-\frac{165}{2}s_2s_4s_5-\frac{33}{2}a_1s_3s_5^2-\frac{15}{4}s_2s_3s_5^2-24a_1b^2s_2-\\
-\frac{1086}{5}a_1bs_3-\frac{1449}{10}a_1bs_4-\frac{363}{5}a_1bs_5-\frac{1851}{10}bs_2s_3-\frac{618}{5}bs_2s_4-\frac{621}{10}bs_2s_5-\frac{225}{2}bs_3s_4-\\
-\frac{507}{10}bs_4s_5-39a_1bs_4^2-24bs_2s_4^2-\frac{51}{2}bs_3^2s_4-21bs_3s_4^2-\frac{1437}{20}s_3s_4s_5-\frac{27}{2}a_1s_4^2s_5-\\
-60bs_2s_3s_4-15a_1s_4s_5^2-\frac{9}{2}s_2s_4s_5^2-\frac{15}{4}s_3s_4s_5^2-18a_1bs_5^2-15bs_2^2s_5-\frac{21}{2}bs_2s_5^2-\frac{9}{2}a_1s_2^2s_5-\\
-\frac{129}{4}a_1^2s_4s_5-30a_1s_2s_4^2-30a_1s_3^2s_4-\frac{21}{2}a_1s_3^2s_5-\frac{57}{2}a_1s_3s_4^2-3s_2^2s_3s_4-\frac{5409}{10}a_1s_2s_3-\\
-96a_1bs_2s_4-\frac{3549}{10}a_1s_2s_4-\frac{1689}{10}a_1s_2s_5-\frac{1587}{5}a_1s_3s_4-\frac{309}{2}a_1s_3s_5-\frac{1401}{10}a_1s_4s_5-18a_1s_2s_5^2-\\
-\frac{2061}{10}s_2s_3s_4-6a_1b^2s_5-\frac{1881}{20}s_2s_3s_5-\frac{429}{4}a_1^2s_2s_3-99a_1bs_2^2-\frac{129}{2}a_1^2s_2s_4-\frac{261}{4}a_1^2s_3s_4-\\
-12a_1b^2s_4-18b^2s_2s_3-12b^2s_2s_4-6b^2s_2s_5-12b^2s_3s_4-6b^2s_3s_5-6b^2s_4s_5-\frac{579}{2}a_1bs_2-
\end{multline*}

\begin{multline*}
-\frac{147}{2}a_1s_2^2s_3-39a_1s_2^2s_4-60a_1s_2s_3^2-57bs_2^2s_3-36bs_2^2s_4-\frac{93}{2}bs_2s_3^2-\frac{87}{4}a_1^2s_2s_5-\\
-27a_1^2s_3s_5-27ba_1^2s_5-\frac{111}{2}ba_1^2s_4-84ba_1^2s_3-\frac{225}{2}ba_1^2s_2-\frac{52767}{2000}s_5-\frac{211263}{4000}s_4-\\
-\frac{9906}{125}s_3-\frac{422721}{4000}s_2-\frac{9}{2}a_1^4-39a_1^3b-12a_1^2b^2-\frac{1014}{5}a_1^3-\frac{131553}{400}a_1^2-\frac{918}{5}ba_1^2-\\
-\frac{3186}{25}ba_1-15bs_5^2x-\frac{69}{4}s_3^2s_4x-12s_3s_5^2x-\frac{81}{5}b^2s_3-18b^2s_3x-\frac{45}{4}s_4s_5^2x-24bs_5x^2-\\
-\frac{336}{5}bs_5x-\frac{108}{5}b^2s_2-24b^2s_2x-\frac{54}{5}b^2s_4-12b^2s_4x-\frac{9}{2}s_3^2s_5x-\frac{63}{2}a_1s_5^2x-18s_2s_4^2x-\\
-\frac{27}{2}s_5s_3x^2-\frac{1233}{10}s_5s_3x-\frac{111}{2}a_1s_4^2x-\frac{75}{4}s_5s_4x^2-\frac{6183}{10}a_1s_3x-\frac{261}{2}a_1s_2^2x-\\
-\frac{1806}{5}ba_1x-\frac{201}{2}bs_2x^2-\frac{2679}{10}bs_2x-75bs_3x^2-201bs_3x-87bs_2^2x-\frac{339}{4}a_1^2s_4x-75a_1s_4x^2-\\
-\frac{819}{2}a_1s_4x-\frac{99}{2}bs_4x^2-\frac{1341}{10}bs_4x-\frac{177}{2}a_1s_3^2x-57bs_3^2x-\frac{243}{4}s_2s_3x^2-\frac{8787}{20}s_2s_3x-\\
-\frac{69}{2}s_4s_2x^2-\frac{2871}{10}s_4s_2x-\frac{147}{4}s_4s_3x^2-45s_2^2s_3x-21s_2^2s_4x-\frac{147}{4}s_2s_3^2x-36a_1^2s_5x-\\
-\frac{63}{2}a_1s_5x^2-\frac{2007}{10}a_1s_5x-144ba_1^2x-33bs_4^2x-9s_4^2s_5x-\frac{33}{4}s_5s_2x^2-\frac{2697}{20}s_5s_2x-145x^3-\\
-\frac{55}{4}x^4-27a_1b^2-30a_1b^2x-18s_3s_4^2x-\frac{27}{5}b^2s_5-6b^2s_5x-\frac{5103}{20}s_4s_3x-\frac{447}{4}s_5s_4x-\frac{51}{4}s_2s_5^2x-\\
-\frac{267}{2}a_1^2s_3x-162a_1s_2x^2-\frac{8271}{10}a_1s_2x-\frac{237}{2}a_1s_3x^2-15b^2x^2-27b^2x-\frac{85941}{400}s_4x-\\
-\frac{120999}{200}a_1x-\frac{171507}{400}s_2x-\frac{32181}{100}s_3x-6s_4^3x-\frac{7}{2}s_5^3x-\frac{513}{4}a_1^2x^2-\frac{153}{2}a_1x^3-\\
-43bx^3-\frac{1671}{10}bx^2-\frac{149}{4}s_2x^3-\frac{13623}{40}s_2x^2-27s_3x^3-\frac{1272}{5}s_3x^2-\frac{67}{4}s_4x^3-\\
-\frac{6729}{40}s_4x^2-\frac{87}{2}s_2^2x^2-\frac{1479}{5}s_2^2x-30s_3^2x^2-\frac{387}{2}s_3^2x-23s_2^3x-\frac{13}{2}s_5x^3-\frac{1641}{20}s_5x^2-\\
-\frac{39}{2}s_4^2x^2-\frac{558}{5}s_4^2x-10s_3^3x-12s_5^2x^2-\frac{501}{10}s_5^2x-\frac{11817}{20}a_1^2x-\frac{13029}{100}bx-72a_1^3x-\\
-\frac{21579}{200}s_5x-\frac{10677}{40}x^2-\frac{45}{2}bs_5^2-\frac{717}{10}s_3s_4^2-9a_1s_4^3-13bs_3^3-6bs_4^3-\frac{303}{10}s_3s_5^2-\\
-\frac{10521}{20}a_1x^2-\frac{3}{2}s_4^2s_5^2-\frac{9}{2}a_1s_5^3-\frac{7}{4}s_2s_5^3-\frac{3}{2}s_3s_5^3-\frac{5}{4}s_4s_5^3-\frac{153}{5}s_4^2s_5-\frac{201}{8}s_4s_5^2-\\
-2bs_5^3-75a_1^2s_2^2-\frac{207}{4}a_1^2s_3^2-36a_1s_2^3-26bs_2^3-33a_1^2s_4^2-\frac{33}{2}a_1s_3^3-9s_2^3s_3-2s_2^3s_4-\frac{15}{2}s_2^2s_3^2-\\
-66a_1bs_3^2-\frac{5}{4}s_2s_3^3-\frac{75}{4}a_1^2s_5^2-\frac{7269}{20}a_1s_2^2-\frac{4803}{20}a_1s_3^2-\frac{2799}{20}a_1s_4^2-\frac{1257}{20}a_1s_5^2-\\
-144ba_1x^2-195s_2^2s_3-\frac{1239}{10}s_2^2s_4-\frac{264}{5}s_2^2s_5-\frac{3}{2}s_2^2s_5^2-\frac{6363}{40}s_2s_3^2-\frac{411}{5}s_2s_4^2-\\
-\frac{3}{20}b^2-\frac{1419}{40}s_2s_5^2-\frac{3501}{40}s_3^2s_4-\frac{831}{20}s_3^2s_5-\frac{3}{2}s_3^2s_5^2-12b^2s_2^2-9b^2s_3^2-6b^2s_4^2-\\
-\frac{729}{4}a_1^2s_2x-18a_1b^2s_3-111a_1bs_4x-84bs_2s_4x-3b^2s_5^2-\frac{1233}{10}bs_2^2-\frac{843}{10}bs_3^2-\frac{507}{10}bs_4^2.
\end{multline*}
Thus, if $x\geqslant 0$, then
\begin{multline*}
g_2(a_1,x+\frac{9}{10},s_2,s_3,s_4,s_5,b)\leqslant\frac{16293}{8000}+3s_2s_3s_4s_5+3s_2^2s_5x+\frac{3}{4}s_3^4+5s_2^3s_5+\\
+\frac{11}{4}s_3^3s_4+\frac{3}{2}s_3^2s_4^2+\frac{5}{2}s_3^3s_5+\frac{9}{4}s_3^2s_4s_5+9s_2^2s_3s_5+3s_2^2s_4s_5+\frac{3}{2}s_2s_3^2s_4+\frac{27}{4}s_2s_3^2s_5-\\
-\frac{52767}{2000}s_5-\frac{211263}{4000}s_4-\frac{9906}{125}s_3-\frac{422721}{4000}s_2-\frac{264327}{2000}a_1-\frac{13221}{100}x\leqslant\\
\leqslant\frac{16293}{8000}+3s_2s_3s_4s_5+3x+\frac{3}{4}s_3^4+5s_2^3s_5+\\
+\frac{11}{4}s_3^3s_4+\frac{3}{2}s_3^2s_4^2+\frac{5}{2}s_3^3s_5+\frac{9}{4}s_3^2s_4s_5+9s_2^2s_3s_5+3s_2^2s_4s_5+\frac{3}{2}s_2s_3^2s_4+\frac{27}{4}s_2s_3^2s_5-\\
-\frac{52767}{2000}s_5-\frac{211263}{4000}s_4-\frac{9906}{125}s_3-\frac{422721}{4000}s_2-\frac{264327}{2000}a_1-\frac{13221}{100}x\leqslant\\
\leqslant\frac{16293}{8000}+3s_2s_3s_4s_5+\frac{3}{4}s_3^4+5s_2^3s_5+\\
+\frac{11}{4}s_3^3s_4+\frac{3}{2}s_3^2s_4^2+\frac{5}{2}s_3^3s_5+\frac{9}{4}s_3^2s_4s_5+9s_2^2s_3s_5+3s_2^2s_4s_5+\frac{3}{2}s_2s_3^2s_4+\frac{27}{4}s_2s_3^2s_5-\\
-\frac{52767}{2000}s_5-\frac{211263}{4000}s_4-\frac{9906}{125}s_3-\frac{422721}{4000}s_2-\frac{264327}{2000}a_1\leqslant\\
\leqslant\frac{16293}{8000}+3s_2+\frac{3}{4}s_3+5s_2+\frac{11}{4}s_3+\frac{3}{2}s_3+\frac{5}{2}s_3+\frac{9}{4}s_3+9s_2+3s_2+\frac{3}{2}s_2+\frac{27}{4}s_2-\\
-\frac{52767}{2000}s_5-\frac{211263}{4000}s_4-\frac{9906}{125}s_3-\frac{422721}{4000}s_2-\frac{264327}{2000}a_1=\\
=\frac{16293}{8000}-\frac{309721}{4000}s_2-\frac{34749}{500}s_3-\frac{52767}{2000}s_5-\frac{211263}{4000}s_4-\frac{264327}{2000}a_1.
\end{multline*}
Hence, if $s_1\geqslant\frac{9}{10}$, then
\begin{multline*}
g_2(a_1,s_1,s_2,s_3,s_4,s_5,b)\leqslant\frac{16293}{8000}-\frac{309721}{4000}s_2-\frac{34749}{500}s_3-\frac{52767}{2000}s_5-\frac{211263}{4000}s_4-\frac{264327}{2000}a_1.
\end{multline*}
Moreover, if $s_1\geqslant\frac{9}{10}$, then $2s_4+2s_2+s_5+s_3+a_1\geqslant\frac{1}{10}$, since $a_3+a_4+a_6\geqslant 1+a_2+a_5$.
The latter inequality gives $\frac{16293}{8000}-\frac{309721}{4000}s_2-\frac{34749}{500}s_3-\frac{52767}{2000}s_5-\frac{211263}{4000}s_4-\frac{264327}{2000}a_1<0$.
This shows that $f(a_1,a_2,a_3,a_4,a_5,a_6,1,b)<0$ provided that $a_2-a_1>\frac{9}{10}$.
\end{proof}

\begin{lemma}
\label{lemma:Maple-P1-P1-d-2-e}
Suppose that $f$ is the polynomial~\eqref{equation:polynomial-P1xP1-d-1-e} and $1+a_2+a_6\leqslant a_3+a_4+a_5$.
If~$a_2-a_1\geqslant \frac{23}{25}$, then $f<0$.
Similarly, if  $a_2-a_1\geqslant \frac{9}{10}$, then $f(a_1,a_2,a_3,a_4,a_5,a_6,1,b)<0$.
\end{lemma}

\begin{proof}
Let $g_1(a_1,s_1,s_2,s_3,s_4,s_5,s_6,b)=\widehat{f}$. Then
\begin{multline*}
g_1\Big(a_1,x+\frac{23}{25},s_2,s_3,s_4,s_5,s_6,b\Big)=\frac{735066}{390625}+2s_2s_5^3+4s_3s_5^3+3s_4^3s_5+6s_4^2s_5^2+6s_4s_5^3+2s_5^4-\\
-\frac{2662308}{15625}a_1-\frac{870744}{15625}s_5-\frac{1330266}{15625}s_4-\frac{1789788}{15625}s_3-\frac{449862}{3125}s_2+\bigstar,
\end{multline*}
where $\bigstar$ is the following polynomial with negative coefficients:
\begin{multline*}
-258a_1s_2s_3s_4-\frac{2662308}{15625}a_1-\frac{24252}{15625}b-\frac{2660532}{15625}x-\frac{261216}{625}a_1^2-40a_1^3s_5-60a_1^3s_4-80a_1^3s_3-\\
-\frac{258}{5}a_1s_6^2-246a_1s_2s_5x-\frac{3609}{25}bs_2^2-\frac{2832}{25}bs_3^2-\frac{411}{5}bs_4^2-\frac{1278}{25}bs_5^2-\frac{501}{25}bs_6^2-\\
-90a_1s_2s_6x-306a_1s_3s_4x-216a_1s_3s_5x-90a_1s_3s_6x-186a_1s_4s_5x-90a_1s_4s_6x-\\
-186a_1bx^2-\frac{6024}{25}s_2^2s_3-\frac{5022}{25}s_2^2s_4-180bs_2s_3x-\frac{804}{5}s_2^2s_5-\frac{1002}{25}s_2^2s_6-\frac{5523}{25}s_2s_3^2-\\
-90s_2s_4^2x-54s_2s_5^2x-12s_2s_6^2x-84s_3^2s_4x-60s_3^2s_5x-24s_3^2s_6x-72s_3s_4^2x-36s_3s_5^2x-\\
-162a_1bs_2s_4-100a_1^3s_2-\frac{6774}{25}a_1^2s_5-\frac{9492}{25}a_1^2s_4-\frac{2442}{5}a_1^2s_3-\frac{14928}{25}a_1^2s_2-204a_1bs_2s_3-\\
-60bs_5x^2-\frac{804}{5}s_2s_4^2-\frac{2517}{25}s_2s_5^2-36bs_5s_6x-\frac{501}{25}s_2s_6^2-\frac{4032}{25}s_3^2s_4-\frac{606}{5}s_3^2s_5-\\
-144bs_2s_4x-108bs_2s_5x-36bs_2s_6x-132bs_3s_4x-96bs_3s_5x-36bs_3s_6x-84bs_4s_5x-\\
-\frac{1002}{25}s_3^2s_6-\frac{3531}{25}s_3s_4^2-14a_1^4-36bs_4s_6x-\frac{2028}{25}s_3s_5^2-\frac{501}{25}s_3s_6^2-\frac{408}{5}s_4^2s_5-\\
-96s_3s_4s_5x-14s_5^3-\frac{1002}{25}s_4^2s_6-\frac{1539}{25}s_4s_5^2-\frac{501}{25}s_4s_6^2-\frac{1002}{25}s_5^2s_6-\frac{501}{25}s_5s_6^2-\\
-90a_1s_5s_6x-48s_3s_4s_6x-48s_3s_5s_6x-48s_4s_5s_6x-\frac{6738}{125}bs_5-\frac{51111}{625}bs_4-\frac{68532}{625}bs_3-\\
-60a_1s_3s_4s_6-\frac{85953}{625}bs_2-\frac{17421}{625}bs_6-20a_1^3s_6-\frac{10389}{125}s_5^2-\frac{114963}{625}s_5s_4-\frac{126036}{625}s_5s_3-\\
-204s_2s_3s_4x-156s_2s_3s_5x-48s_2s_3s_6x-132s_2s_4s_5x-48s_2s_4s_6x-48s_2s_5s_6x-\\
-24bs_2s_5s_6-\frac{137109}{625}s_5s_2-\frac{166554}{625}s_5a_1-\frac{20448}{625}s_6^2-\frac{51969}{625}s_6s_5-\frac{51969}{625}s_6s_4-\frac{51969}{625}s_6s_3-\\
-306s_2a_1x^2-30s_2^3s_4-25s_2^3s_5-5s_2^3s_6-48s_2^2s_3^2-36s_2^2s_4^2-24s_2^2s_5^2-3s_2^2s_6^2-26s_2s_3^3-\\
-\frac{1014}{25}s_4^3-24bs_6x^2-108bs_2^2x-84bs_3^2x-60bs_4^2x-36bs_5^2x-12bs_6^2x-168s_3s_2x^2-\\
-132s_4s_2x^2-96s_5s_2x^2-36s_6s_2x^2-120s_4s_3x^2-84s_5s_3x^2-36s_6s_3x^2-72s_5s_4x^2-186a_1^2bx-\\
-\frac{51969}{625}s_6s_2-60bs_3s_4s_5-\frac{14508}{125}s_6a_1-\frac{459522}{15625}s_6-\frac{870744}{15625}s_5-\frac{1330266}{15625}s_4-\frac{1789788}{15625}s_3-\\
-336a_1^2s_2x-270a_1^2s_3x-204a_1^2s_4x-138a_1^2s_5x-66a_1^2s_6x-\frac{828}{25}a_1b^2-36a_1b^2x-\\
-138a_1s_4^2x-78a_1s_5^2x-30a_1s_6^2x-\frac{138}{5}b^2s_2-30b^2s_2x-\frac{552}{25}b^2s_3-24b^2s_3x-\frac{414}{25}b^2s_4-\\
-\frac{449862}{3125}s_2-\frac{83466}{625}s_4^2-\frac{35601}{125}s_4s_3-336a_1s_2s_4x-\frac{189078}{625}s_4s_2-\frac{239094}{625}s_4a_1-\frac{114987}{625}s_3^2-\\
-2a_1s_5^3-15b^2s_2^2-12b^2s_3^2-9b^2s_4^2-6b^2s_5^2-3b^2s_6^2-28bs_2^3-20bs_3^3-12bs_4^3-4bs_5^3-35s_2^3s_3-\\
-18b^2s_4x-\frac{276}{25}b^2s_5-12b^2s_5x-\frac{138}{25}b^2s_6-6b^2s_6x-132bs_2x^2-108bs_3x^2-84bs_4x^2-\\
-\frac{5778}{25}a_1^2b-\frac{100494}{625}a_1b-\frac{147}{625}b^2-\frac{4708}{25}x^3-24x^4-10s_2^4-4s_3^4-\frac{2342}{25}s_2^3-\frac{1678}{25}s_3^3-\\
-6s_2s_4s_6^2-246s_3a_1x^2-186s_4a_1x^2-126s_5a_1x^2-60s_6a_1x^2-258a_1s_2^2x-198a_1s_3^2x-
\end{multline*}

\begin{multline*}
-15s_2s_4^2s_6-12s_2s_4s_5^2-15s_2s_5^2s_6-6s_2s_5s_6^2-18s_3^2s_4s_5-12s_3^2s_4s_6-12s_3^2s_5s_6-\\
-\frac{6768}{25}bs_3x-\frac{5214}{25}bs_4x-\frac{732}{5}bs_5x-\frac{1554}{25}bs_6x-\frac{15078}{25}s_3s_2x-\frac{12072}{25}s_4s_2x-\\
-12s_3s_6^2x-30s_4^2s_5x-24s_4^2s_6x-18s_4s_5^2x-12s_4s_6^2x-24s_5^2s_6x-12s_5s_6^2x-\frac{11256}{25}a_1bx-\\
-\frac{9066}{25}s_5s_2x-\frac{3006}{25}s_6s_2x-\frac{2214}{5}s_4s_3x-\frac{8064}{25}s_5s_3x-\frac{3006}{25}s_6s_3x-\frac{7062}{25}s_5s_4x-\\
-36s_6s_4x^2-\frac{3006}{25}s_6s_4x-36s_6s_5x^2-\frac{3006}{25}s_6s_5x-138s_2^2s_3x-114s_2^2s_4x-90s_2^2s_5x-24s_2^2s_6x-\\
-126s_2s_3^2x-\frac{1188}{5}bs_2s_3-\frac{4662}{25}bs_2s_4-\frac{3384}{25}bs_2s_5-\frac{1278}{25}bs_2s_6-\frac{4386}{25}bs_3s_4-\frac{3108}{25}bs_3s_5-\\
-\frac{1278}{25}bs_3s_6-\frac{2832}{25}bs_4s_5-\frac{1278}{25}bs_4s_6-\frac{1278}{25}bs_5s_6-\frac{9042}{25}s_2s_3s_4-\frac{7038}{25}s_2s_3s_5-\\
-60bs_2^2s_4-48bs_2^2s_5-12bs_2^2s_6-66bs_2s_3^2-48bs_2s_4^2-30bs_2s_5^2-6bs_2s_6^2-48bs_3^2s_4-36bs_3^2s_5-\\
-\frac{2004}{25}s_2s_3s_6-\frac{6036}{25}s_2s_4s_5-\frac{2004}{25}s_2s_4s_6-\frac{2004}{25}s_2s_5s_6-\frac{5058}{25}s_3s_4s_5-\frac{2004}{25}s_3s_4s_6-\\
-15s_2^2s_5s_6-63s_2s_3^2s_4-48s_2s_3^2s_5-15s_2s_3^2s_6-54s_2s_3s_4^2-30s_2s_3s_5^2-6s_2s_3s_6^2-21s_2s_4^2s_5-\\
-\frac{2004}{25}s_3s_5s_6-\frac{2004}{25}s_4s_5s_6-81s_2^2s_3s_4-66s_2^2s_3s_5-15s_2^2s_3s_6-57s_2^2s_4s_5-15s_2^2s_4s_6-\\
-\frac{9012}{25}a_1bs_2-12b^2s_2s_5-\frac{1464}{5}a_1bs_3-\frac{5628}{25}a_1bs_4-\frac{3936}{25}a_1bs_5-\frac{1692}{25}a_1bs_6-\frac{18348}{25}a_1s_2s_3-\\
-51a_1^2s_5s_6-114a_1s_2^2s_5-120a_1^2s_3s_5-30a_1b^2s_2-12b^2s_3s_5-93a_1^2bs_4-174a_1s_2^2s_3-\\
-\frac{14628}{25}a_1s_2s_4-\frac{10908}{25}a_1s_2s_5-69a_1bs_4^2-\frac{744}{5}a_1s_2s_6-\frac{13488}{25}a_1s_3s_4-\frac{9768}{25}a_1s_3s_5-\\
-30a_1s_3^2s_6-18b^2s_2s_4-51a_1^2s_2s_6-6b^2s_3s_6-12b^2s_4s_5-6b^2s_4s_6-6b^2s_5s_6-72bs_2^2s_3-\\
-\frac{744}{5}a_1s_3s_6-\frac{8628}{25}a_1s_4s_5-\frac{744}{5}a_1s_4s_6-18b^2s_3s_4-\frac{744}{5}a_1s_5s_6-66a_1^2bs_5-24b^2s_2s_3-\\
-12bs_3^2s_6-42bs_3s_4^2-24bs_3s_5^2-6bs_3s_6^2-24bs_4^2s_5-12bs_4^2s_6-18bs_4s_5^2-6bs_4s_6^2-12bs_5^2s_6-\\
-6bs_5s_6^2-18a_1b^2s_4-6a_1b^2s_6-27a_1^2bs_6-27a_1s_4s_5^2-105a_1^2s_4s_5-144a_1s_2^2s_4-135a_1^2s_2s_5-\\
-186a_1^2s_2s_4-30a_1s_4^2s_6-15a_1s_5s_6^2-15a_1s_4s_6^2-15a_1s_3s_6^2-12a_1b^2s_5-30a_1s_2^2s_6-\\
-\frac{7062}{25}s_4x^2-\frac{211176}{625}s_4x-32s_5x^3-\frac{5058}{25}s_5x^2-\frac{148134}{625}s_5x-16s_6x^3-\frac{2004}{25}s_6x^2-\\
-120a_1^2bs_3-147a_1^2bs_2-42a_1s_4^2s_5-15a_1s_2s_6^2-237a_1^2s_2s_3-24a_1b^2s_3-78a_1s_3^2s_5-\\
-\frac{16956}{25}a_1x^2-\frac{478188}{625}a_1x-18b^2x^2-\frac{828}{25}b^2x-56bx^3-\frac{5214}{25}bx^2-\frac{102222}{625}bx-\\
-93a_1s_3s_4^2-6b^2s_2s_6-123a_1bs_2^2-48a_1s_3s_5^2-51a_1^2s_4s_6-108a_1s_3^2s_4-69a_1s_2s_5^2-\\
-30a_1s_5^2s_6-54s_2^3x-36s_3^3x-18s_4^3x-120a_1^3x-204a_1^2x^2-\frac{18984}{25}a_1^2x-124a_1x^3-\\
-\frac{63042}{625}s_6x-6s_3s_4^2s_5-102s_2^2x^2-\frac{9042}{25}s_2^2x-78s_3^2x^2-\frac{7038}{25}s_3^2x-54s_4^2x^2-\frac{5034}{25}s_4^2x-\\
-80s_2x^3-\frac{2214}{5}s_2x^2-\frac{67452}{125}s_2x-64s_3x^3-\frac{9066}{25}s_3x^2-\frac{274218}{625}s_3x-48s_4x^3-\\
-42a_1bs_5^2-15a_1bs_6^2-171a_1^2s_3s_4-51a_1^2s_3s_6-114a_1s_2s_4^2-96a_1bs_3^2-159a_1s_2s_3^2-
\end{multline*}

\begin{multline*}
-12s_3s_4^2s_6-6s_3s_4s_6^2-12s_3s_5^2s_6-6s_3s_5s_6^2-9s_4^2s_5s_6-9s_4s_5^2s_6-6s_4s_5s_6^2-\\
-198a_1s_2s_3s_5-60a_1s_2s_3s_6-168a_1s_2s_4s_5-60a_1s_2s_4s_6-60a_1s_2s_5s_6-126a_1s_3s_4s_5-\\
-3s_5^2s_6^2-\frac{2718}{25}a_1^2s_6-120a_1bs_2s_5-\frac{11034}{25}a_1s_2^2-\frac{8604}{25}a_1s_3^2-\frac{6174}{25}a_1s_4^2-\frac{3744}{25}a_1s_5^2-\\
-60a_1s_3s_5s_6-60a_1s_4s_5s_6-108bs_2s_3s_4-84bs_2s_3s_5-24bs_2s_3s_6-72bs_2s_4s_5-24bs_2s_4s_6-\\
-24bs_3s_5s_6-24bs_4s_5s_6-78s_2s_3s_4s_5-30s_2s_3s_4s_6-30s_2s_3s_5s_6-24bs_3s_4s_6-\\
-30s_2s_4s_5s_6-30s_5^2x^2-\frac{606}{5}s_5^2x-12s_6^2x^2-\frac{1002}{25}s_6^2x-\frac{211176}{625}x^2-3s_4^2s_6^2-2s_5^3s_6-\\
-144a_1^2s_2^2-111a_1^2s_3^2-78a_1^2s_4^2-45a_1^2s_5^2-18a_1^2s_6^2-68a_1s_2^3-46a_1s_3^3-24a_1s_4^3-\\
-\frac{26676}{25}s_2a_1x-\frac{21816}{25}s_3a_1x-\frac{16956}{25}s_4a_1x-\frac{12096}{25}s_5a_1x-\frac{972}{5}s_6a_1x-\frac{8322}{25}bs_2x-\\
-42a_1bs_2s_6-150a_1bs_3s_4-108a_1bs_3s_5-42a_1bs_3s_6-96a_1bs_4s_5-42a_1bs_4s_6-42a_1bs_5s_6-\\
-24s_3s_4s_5s_6-294a_1bs_2x-240a_1bs_3x-186a_1bs_4x-132a_1bs_5x-54a_1bs_6x-426a_1s_2s_3x-\\
-\frac{241047}{625}s_3s_2-\frac{311634}{625}s_3a_1-\frac{146508}{625}s_2^2-\frac{384174}{625}s_2a_1-\frac{1312}{5}a_1^3-52a_1^3b-15a_1^2b^2-\\
-12s_2s_4^3-12s_3^3s_4-8s_3^3s_5-4s_3^3s_6-15s_3^2s_4^2-6s_3^2s_5^2-3s_3^2s_6^2-6s_3s_4^3-3s_4^3s_6.
\end{multline*}
Thus, if $a_2-a_1\geqslant\frac{23}{25}$, then
\begin{multline*}
f(a_1,a_2,a_3,a_4,a_5,a_6,a_7,b)\leqslant\frac{735066}{390625}+2s_2s_5^3+4s_3s_5^3+3s_4^3s_5+6s_4^2s_5^2+6s_4s_5^3+2s_5^4-\\
-\frac{2662308}{15625}a_1-\frac{870744}{15625}s_5-\frac{1330266}{15625}s_4-\frac{1789788}{15625}s_3-\frac{449862}{3125}s_2\leqslant\\
\leqslant\frac{735066}{390625}+2s_5+4s_5+3s_5+6s_5+6s_5+2s_5-\\
-\frac{2662308}{15625}a_1-\frac{870744}{15625}s_5-\frac{1330266}{15625}s_4-\frac{1789788}{15625}s_3-\frac{449862}{3125}s_2=\\
=\frac{735066}{390625}-\frac{511369}{15625}s_5-\frac{2662308}{15625}a_1-\frac{1330266}{15625}s_4-\frac{1789788}{15625}s_3-\frac{449862}{3125}s_2<0.
\end{multline*}
because $2s_2+s_3+a_1-s_5\geqslant 1-s_1\geqslant\frac{2}{25}$, since $1+a_2+a_6\leqslant a_3+a_4+a_5$.

Let $f_2=f(a_1,a_2,a_3,a_4,a_5,a_6,1,b)$ and $g_2(a_1,s_1,s_2,s_3,s_4,s_5,b)=\widehat{f}_2$. Then
\begin{multline*}
g_2\Big(a_1,x+\frac{9}{10},s_2,s_3,s_4,s_5,b\Big)=\frac{523}{250}+9s_3s_4s_5^2+8s_5^3x+9a_1s_5^3+2bs_5^3+7s_2s_5^3+\\
+6s_3s_5^3+3s_4^3s_5+6s_4^2s_5^2+5s_4s_5^3+s_5^4+\frac{31}{5}s_5^3+9s_4s_5^2x+6s_2s_4s_5^2+3s_3s_4^2s_5+9a_1s_4s_5^2-\\
-\frac{66639}{500}a_1-\frac{51}{2}bs_5-\frac{6667}{50}x-\frac{12767}{500}s_5-\frac{26479}{500}s_4-\frac{40191}{500}s_3-\frac{53903}{500}s_2+\blacklozenge,
\end{multline*}
where $\blacklozenge$ is the following polynomial with negative coefficients:
\begin{multline*}
-66a_1bs_2s_5-96a_1bs_3s_4-54a_1bs_3s_5-42a_1bs_4s_5-174a_1s_2s_3s_4-114a_1s_2s_3s_5-\\
-36s_3s_4s_5x-\frac{613}{500}b-108a_1bs_2s_4-\frac{8082}{25}a_1^2-20a_1^3s_5-40a_1^3s_4-60a_1^3s_3-80a_1^3s_2-\\
-228a_1bs_2x-174a_1bs_3x-120a_1bs_4x-318a_1s_2s_3x-228a_1s_2s_4x-138a_1s_2s_5x-\\
-\frac{1341}{10}a_1^2s_5-66a_1bs_5x-\frac{483}{2}a_1^2s_4-\frac{3489}{10}a_1^2s_3-84a_1s_2s_4s_5-\frac{4563}{10}a_1^2s_2-150a_1bs_2s_3-\\
-198a_1s_3s_4x-108a_1s_3s_5x-78a_1s_4s_5x-132bs_2s_3x-96bs_2s_4x-60bs_2s_5x-84bs_3s_4x-\\
-138s_2s_3s_4x-90s_2s_3s_5x-66s_2s_4s_5x-\frac{2607}{50}bs_4-\frac{3939}{50}bs_3-48bs_3s_5x-\frac{5271}{50}bs_2-\\
-\frac{3189}{100}s_5^2-36bs_4s_5x-\frac{807}{10}s_5s_4-\frac{4881}{50}s_5s_3-\frac{5727}{50}s_5s_2-\frac{12843}{100}s_5a_1-\frac{8073}{100}s_4^2-\\
-48a_1s_3s_4s_5-72bs_2s_3s_4-48bs_2s_3s_5-36bs_2s_4s_5-24bs_3s_4s_5-36s_2s_3s_4s_5-\\
-\frac{1278}{5}bs_2x-\frac{17841}{100}s_2^2-\frac{46503}{100}s_2a_1-12a_1^4-202a_1^3-40a_1^3b-12a_1^2b^2-\frac{1803}{10}a_1^2b-\\
-60bs_3^2x-36bs_4^2x-12bs_5^2x-126s_3s_2x^2-90s_4s_2x^2-54s_5s_2x^2-78s_4s_3x^2-\\
-\frac{972}{5}bs_3x-\frac{666}{5}bs_4x-72bs_5x-\frac{2184}{5}s_3s_2x-318s_4s_2x-\frac{996}{5}s_5s_2x-\frac{1392}{5}s_4s_3x-\\
-42s_5s_3x^2-30s_5s_4x^2-147a_1^2bx-267a_1^2s_2x-201a_1^2s_3x-135a_1^2s_4x-69a_1^2s_5x-27a_1b^2-\\
-\frac{8919}{50}s_4s_3-\frac{1953}{10}s_4s_2-\frac{24063}{100}s_4a_1-24a_1s_5^2x-\frac{12957}{100}s_3^2-\frac{13803}{50}s_3s_2-\frac{35283}{100}s_3a_1-\\
-30a_1b^2x-147a_1bx^2-243s_2a_1x^2-183s_3a_1x^2-123s_4a_1x^2-63s_5a_1x^2-204a_1s_2^2x-144a_1s_3^2x-\\
-\frac{12807}{100}a_1b-\frac{3}{20}b^2-146x^3-20x^4-8s_2^4-3s_3^4-\frac{727}{10}s_2^3-\frac{232}{5}s_3^3-\frac{201}{10}s_4^3-30bs_5x^2-\\
-84a_1s_4^2x-\frac{108}{5}b^2s_2-24b^2s_2x-\frac{81}{5}b^2s_3-18b^2s_3x-\frac{54}{5}b^2s_4-12b^2s_4x-\frac{27}{5}b^2s_5-\\
-6b^2s_5x-102bs_2x^2-78bs_3x^2-54bs_4x^2-57s_2s_4^2x-21s_2s_5^2x-54s_3^2s_4x-30s_3^2s_5x-42s_3s_4^2x-\\
-6s_3s_5^2x-3s_4^2s_5x-\frac{1773}{5}a_1bx-\frac{4077}{5}s_2a_1x-\frac{3117}{5}s_3a_1x-\frac{2157}{5}s_4a_1x-\frac{1197}{5}s_5a_1x-\\
-18b^2s_2s_3-6b^2s_2s_5-42a_1bs_4^2-72a_1s_2^2s_5-60a_1^2s_3s_5-24a_1b^2s_2-6b^2s_3s_5-60a_1^2bs_4-\\
-\frac{798}{5}s_5s_3x-120s_5s_4x-105s_2^2s_3x-81s_2^2s_4x-57s_2^2s_5x-93s_2s_3^2x-\frac{864}{5}bs_2s_3-\frac{612}{5}bs_2s_4-\\
-3s_2s_4^2s_5-132a_1s_2^2s_3-12b^2s_3s_4-12b^2s_2s_4-6b^2s_4s_5-54bs_2^2s_3-42bs_2^2s_4-30bs_2^2s_5-\\
-72bs_2s_5-\frac{558}{5}bs_3s_4-\frac{306}{5}bs_3s_5-\frac{252}{5}bs_4s_5-\frac{1191}{5}s_2s_3s_4-159s_2s_3s_5-\frac{597}{5}s_2s_4s_5-\\
-\frac{402}{5}s_3s_4s_5-57s_2^2s_3s_4-42s_2^2s_3s_5-33s_2^2s_4s_5-42s_2s_3^2s_4-27s_2s_3^2s_5-33s_2s_3s_4^2-\\
-9s_2s_3s_5^2-6s_3^2s_4s_5-\frac{1386}{5}a_1bs_2-\frac{1053}{5}a_1bs_3-144a_1bs_4-\frac{387}{5}a_1bs_5-\frac{2661}{5}a_1s_2s_3-\\
-48bs_2s_3^2-30bs_2s_4^2-12bs_2s_5^2-30bs_3^2s_4-18bs_3^2s_5-24bs_3s_4^2-6bs_3s_5^2-6bs_4^2s_5-\\
-\frac{1926}{5}a_1s_2s_4-84bs_2^2x-\frac{1191}{5}a_1s_2s_5-\frac{1701}{5}a_1s_3s_4-\frac{966}{5}a_1s_3s_5-\frac{741}{5}a_1s_4s_5-33a_1^2bs_5-\\
-12a_1b^2s_4-45a_1^2s_4s_5-102a_1s_2^2s_4-75a_1^2s_2s_5-126a_1^2s_2s_4-6a_1b^2s_5-87a_1^2bs_3-
\end{multline*}

\begin{multline*}
-114a_1^2bs_2-6a_1s_4^2s_5-177a_1^2s_2s_3-18a_1b^2s_3-39a_1s_3^2s_5-15a_1bs_5^2-111a_1^2s_3s_4-\\
-\frac{2637}{5}a_1x^2-9b^2s_3^2-\frac{29673}{50}a_1x-15b^2x^2-27b^2x-44bx^3-\frac{819}{5}bx^2-\frac{3273}{25}bx-64s_2x^3-\frac{1689}{5}s_2x^2-\\
-3b^2s_5^2-72a_1s_2s_4^2-69a_1bs_3^2-117a_1s_2s_3^2-54a_1s_3s_4^2-96a_1bs_2^2-9a_1s_3s_5^2-69a_1s_3^2s_4-\\
-\frac{10188}{25}s_2x-48s_3x^3-\frac{1293}{5}s_3x^2-\frac{7746}{25}s_3x-32s_4x^3-\frac{897}{5}s_4x^2-\frac{5304}{25}s_4x-4s_3^3s_5-\\
-27a_1s_2s_5^2-43s_2^3x-26s_3^3x-9s_4^3x-100a_1^3x-8s_3^3s_4-168a_1^2x^2-6b^2s_4^2-\frac{2952}{5}a_1^2x-102a_1x^3-\\
-16s_5x^3-\frac{501}{5}s_5x^2-9s_3^2s_4^2-\frac{2862}{25}s_5x-81s_2^2x^2-\frac{1389}{5}s_2^2x-57s_3^2x^2-\frac{993}{5}s_3^2x-\\
-\frac{258}{5}a_1s_5^2-33a_1s_3^3-19s_2s_3^3-\frac{558}{5}bs_2^2-81bs_3^2-\frac{252}{5}bs_4^2-\frac{99}{5}bs_5^2-\frac{357}{2}s_2^2s_3-\frac{1389}{10}s_2^2s_4-\\
-54a_1s_2^3-\frac{993}{10}s_2^2s_5-\frac{1587}{10}s_2s_3^2-\frac{993}{10}s_2s_4^2-\frac{399}{10}s_2s_5^2-\frac{498}{5}s_3^2s_4-60s_3^2s_5-\frac{399}{5}s_3s_4^2-\\
-12a_1s_4^3-\frac{102}{5}s_3s_5^2-15a_1^2s_5^2-\frac{207}{10}s_4^2s_5-6s_2s_4^3-\frac{9}{10}s_4s_5^2-114a_1^2s_2^2-81a_1^2s_3^2-48a_1^2s_4^2-\\
-33s_4^2x^2-\frac{597}{5}s_4^2x-9s_5^2x^2-\frac{201}{5}s_5^2x-3s_3s_4^3-261x^2-\frac{1698}{5}a_1s_2^2-\frac{1218}{5}a_1s_3^2-\frac{738}{5}a_1s_4^2-\\
-22bs_2^3-12b^2s_2^2-14bs_3^3-6bs_4^3-27s_2^3s_3-22s_2^3s_4-17s_2^3s_5-36s_2^2s_3^2-24s_2^2s_4^2-12s_2^2s_5^2.
\end{multline*}
Thus, if $x\geqslant 0$, then
\begin{multline*}
g_2\Big(a_1,x+\frac{9}{10},s_2,s_3,s_4,s_5,b\Big)\leqslant\frac{523}{250}+9s_3s_4s_5^2+8s_5^3x+9a_1s_5^3+2bs_5^3+\\
+7s_2s_5^3+6s_3s_5^3+3s_4^3s_5+6s_4^2s_5^2+5s_4s_5^3+s_5^4+\frac{31}{5}s_5^3+9s_4s_5^2x+6s_2s_4s_5^2+3s_3s_4^2s_5+9a_1s_4s_5^2-\\
-\frac{66639}{500}a_1-\frac{51}{2}bs_5-\frac{6667}{50}x-\frac{12767}{500}s_5-\frac{26479}{500}s_4-\frac{40191}{500}s_3-\frac{53903}{500}s_2\leqslant\\
\leqslant\frac{523}{250}+9s_3s_4s_5^2+9a_1s_5^3+7s_2s_5^3+6s_3s_5^3+3s_4^3s_5+6s_4^2s_5^2+5s_4s_5^3+s_5^4+\frac{31}{5}s_5^3+\\
+6s_2s_4s_5^2+3s_3s_4^2s_5+9a_1s_4s_5^2-\frac{66639}{500}a_1-\frac{12767}{500}s_5-\frac{26479}{500}s_4-\frac{40191}{500}s_3-\frac{53903}{500}s_2\leqslant\\
\leqslant\frac{523}{250}+9s_4+9a_1+7s_2+6s_5+3s_4+6s_5+5s_4+s_5+\frac{31}{5}s_5+6s_2+3s_4+9a_1-\\
-\frac{66639}{500}a_1-\frac{12767}{500}s_5-\frac{26479}{500}s_4-\frac{40191}{500}s_3-\frac{53903}{500}s_2=\\
=\frac{523}{250}-\frac{16479}{500}s_4-\frac{57639}{500}a_1-\frac{47403}{500}s_2-\frac{3167}{500}s_5-\frac{40191}{500}s_3.
\end{multline*}
Hence, if $s_1=a_2-a_1\geqslant\frac{9}{10}$, then
\begin{multline*}
f(a_1,a_2,a_3,a_4,a_5,a_6,1,b)\leqslant\frac{523}{250}-\frac{16479}{500}s_4-\frac{57639}{500}a_1-\frac{47403}{500}s_2-\frac{3167}{500}s_5-\frac{40191}{500}s_3<0.
\end{multline*}
because $2s_4+2s_2+s_5+s_3+a_1\geqslant\frac{1}{10}$.
\end{proof}

\begin{lemma}
\label{lemma:Maple-P1-P1-d-2-f}
Suppose that $f$ is the polynomial~\eqref{equation:polynomial-P1xP1-d-1-f} and $a_2+a_5+a_6\geqslant 1+a_3+a_4$.
If~$a_2-a_1\geqslant \frac{23}{25}$, then $f<0$.
Similarly, if  $a_2-a_1\geqslant \frac{9}{10}$, then $f(a_1,a_2,a_3,a_4,a_5,a_6,1,b)<0$.
\end{lemma}

\begin{proof}
Let $g_1(a_1,s_1,s_2,s_3,s_4,s_5,s_6,b)=\widehat{f}$. Then
\begin{multline*}
g_1\Big(a_1,x+\frac{23}{25},s_2,s_3,s_4,s_5,s_6,b\Big)=\bigstar+\frac{697211}{390625}+\\
+2s_2s_4^3+\frac{9}{2}s_3^3s_4+\frac{27}{4}s_3^2s_4^2+7s_3s_4^3+5s_4^3s_5+\frac{3}{4}s_4^2s_5^2+s_3^4+3s_4^4+\frac{9}{2}s_3s_4^2s_5-\\
-\frac{7063}{125}s_5-\frac{1317297}{15625}s_4-\frac{1752697}{15625}s_3-\frac{2188097}{15625}s_2-\frac{525143}{3125}a_1,
\end{multline*}
where $\bigstar$ is the following polynomial with negative coefficients:
\begin{multline*}
-\frac{369}{2}a_1^2s_2s_3-144a_1bs_2^2-\frac{261}{2}a_1^2s_2s_4-117a_1^2s_3s_4-144a_1s_2^2s_3-93a_1s_2^2s_4-\\
-\frac{1992}{25}s_3s_5s_6-\frac{1497}{25}s_4s_5s_6-6s_4^2s_5s_6-\frac{9}{2}s_4s_5^2s_6-3s_4s_5s_6^2-\frac{9714}{25}a_1bs_2-\\
-72a_1^2s_3s_6-63a_1s_2^2s_6-\frac{21027}{25}a_1s_2s_3-66a_1s_2s_4^2-\frac{117}{2}a_1s_3^2s_4-45a_1s_3s_4^2-\\
-\frac{231}{2}a_1s_2s_3^2-105a_1^2s_2s_5-\frac{189}{2}a_1^2s_3s_5-78a_1s_2^2s_5-\frac{159}{2}a_1^2s_2s_6-\\
-27s_2^2s_3s_4-\frac{21}{2}s_2s_3^2s_4-6s_2s_3s_4^2-\frac{15519}{25}a_1s_2s_4-\frac{13773}{25}a_1s_3s_4-\frac{9771}{25}s_2s_3s_4-\\
-57s_3s_4s_5x-\frac{63}{2}s_2^2s_3s_5-24s_2^2s_4s_5-\frac{39}{2}s_2s_3^2s_5-27s_2s_3s_5^2-6s_2s_4^2s_5-\\
-84a_1^2s_4s_5-\frac{135}{2}a_1s_2s_5^2-\frac{105}{2}a_1s_3^2s_5-51a_1s_3s_5^2-27a_1s_4^2s_5-\frac{69}{2}a_1s_4s_5^2-\\
-\frac{8112}{25}a_1s_4s_5-\frac{6633}{25}s_2s_3s_5-\frac{5613}{25}s_2s_4s_5-\frac{4611}{25}s_3s_4s_5-102a_1bs_3^2-\\
-66a_1bs_4^2-42a_1bs_5^2-\frac{129}{2}a_1^2s_4s_6-\frac{117}{2}a_1^2s_5s_6-36a_1s_2s_6^2-\frac{93}{2}a_1s_3^2s_6-\\
-\frac{3621}{25}a_1s_5s_6-\frac{699}{5}s_2s_3s_6-\frac{2994}{25}s_2s_4s_6-\frac{2487}{25}s_2s_5s_6-\frac{2499}{25}s_3s_4s_6-\\
-27a_1s_3s_6^2-30a_1s_4^2s_6-18a_1s_4s_6^2-\frac{33}{2}a_1s_5^2s_6-9a_1s_5s_6^2-18a_1bs_6^2-\\
-\frac{33}{2}s_2s_4s_5^2-\frac{3}{2}s_3^2s_4s_5-\frac{15}{2}s_3s_4s_5^2-\frac{10428}{25}a_1s_2s_5-\frac{1854}{5}a_1s_3s_5-\\
-21bs_4^2s_5-12bs_4^2s_6-18bs_4s_5^2-6bs_4s_6^2-6bs_5^2s_6-3bs_5s_6^2-36s_2^2s_3s_6-\\
-\frac{7746}{25}a_1bs_3-\frac{5778}{25}a_1bs_4-\frac{3798}{25}a_1bs_5-\frac{6642}{25}bs_2s_3-198bs_2s_4-\frac{3246}{25}bs_2s_5-\\
-18s_2s_4^2s_6-12s_2s_4s_6^2-15s_3^2s_5s_6-12s_3s_4^2s_6-9s_3s_4s_6^2-\frac{21}{2}s_2s_5^2s_6-6s_2s_5s_6^2-\\
-\frac{15}{2}s_3s_5^2s_6-165a_1s_3s_5x-\frac{9}{2}s_3s_5s_6^2-\frac{5337}{25}a_1s_2s_6-\frac{4767}{25}a_1s_3s_6-\frac{4197}{25}a_1s_4s_6-
\end{multline*}

\begin{multline*}
-12bs_2s_6^2-48bs_3^2s_4-33bs_3^2s_5-18bs_3^2s_6-39bs_3s_4^2-24bs_3s_5^2-9bs_3s_6^2-\\
-\frac{194703}{2500}s_5^2-\frac{43749}{250}s_5s_4-\frac{48171}{250}s_5s_3-\frac{52593}{250}s_5s_2-\frac{322203}{1250}s_5a_1-\\
-93bs_2^2s_3-66bs_2^2s_4-45bs_2^2s_5-24bs_2^2s_6-78bs_2s_3^2-48bs_2s_4^2-30bs_2s_5^2-\\
-\frac{4536}{25}bs_3s_4-\frac{594}{5}bs_3s_5-\frac{2694}{25}bs_4s_5-30a_1b^2s_2-24a_1b^2s_3-18a_1b^2s_4-\\
-12a_1b^2s_5-24b^2s_2s_3-18b^2s_2s_4-12b^2s_2s_5-18b^2s_3s_4-12b^2s_3s_5-12b^2s_4s_5-\\
-6a_1b^2s_6-6b^2s_2s_6-6b^2s_3s_6-6b^2s_4s_6-6b^2s_5s_6-\frac{1818}{25}a_1bs_6-\frac{1542}{25}bs_2s_6-\\
-84a_1s_3s_4s_5-30s_2s_3s_4s_5-168a_1bs_2s_4-114a_1bs_2s_5-150a_1bs_3s_4-102a_1bs_3s_5-\\
-\frac{1404}{25}bs_3s_6-\frac{1266}{25}bs_4s_6-\frac{1128}{25}bs_5s_6-\frac{25136}{15625}b-\frac{2623497}{15625}x-\\
-90a_1bs_4s_5-96a_1s_2s_4s_6-84a_1s_2s_5s_6-78a_1s_3s_4s_6-66a_1s_3s_5s_6-48a_1s_4s_5s_6-\\
-60a_1bs_2s_6-\frac{200931}{625}s_4s_2-\frac{98109}{250}s_4a_1-\frac{164697}{250}a_1s_2-\frac{168966}{625}s_2^2-\frac{21}{4}a_1^4-\\
-210a_1s_2^2x-\frac{255}{2}s_4a_1x^2-42s_6s_5x^2-\frac{153}{2}s_2s_3^2x-\frac{465}{2}a_1s_2x^2-\frac{5889}{5}a_1s_2x-\\
-252a_1bs_3x-204bs_2s_3x-228a_1s_2s_4x-201a_1s_3s_4x-102s_2s_3s_4x-186a_1s_2s_5x-\\
-\frac{3429}{125}bs_6-\frac{4943}{25}x^3-\frac{33}{2}x^4-87a_1^2s_6x-84a_1s_5^2x-\frac{111}{2}s_2^2s_5x-\\
-\frac{4491}{25}s_6s_2x-57bs_5x^2-129bs_2^2x-\frac{105}{2}s_5s_4x^2-\frac{1323}{5}s_5s_4x-\frac{3453}{5}s_4a_1x-\\
-66a_1bs_6x-54bs_2s_6x-48bs_3s_6x-42bs_4s_6x-36bs_5s_6x-\frac{12963}{50}a_1^2s_5-\frac{9687}{25}a_1^2s_4-\\
-126a_1bs_5x-150bs_2s_4x-102bs_2s_5x-132bs_3s_4x-90bs_3s_5x-78bs_4s_5x-\\
-\frac{2619}{5}a_1^2s_3-\frac{16503}{25}a_1^2s_2-\frac{33978}{625}bs_5-\frac{50811}{625}bs_4-\frac{67656}{625}bs_3-\frac{84501}{625}bs_2-\\
-186a_1bs_4x-10s_2^4-\frac{682}{5}s_2^3-\frac{1889}{25}s_3^3-\frac{869}{25}s_4^3-\frac{1}{2}s_5^4-\frac{839}{50}s_5^3-\\
-33ba_1^2s_6-63ba_1^2s_5-93ba_1^2s_4-126ba_1^2s_3-159ba_1^2s_2-\frac{448453}{15625}s_6-\frac{798}{25}s_6^2-\\
-\frac{90861}{1250}s_6s_5-\frac{50961}{625}s_6s_4-\frac{22599}{250}s_6s_3-\frac{62034}{625}s_6s_2-\frac{153861}{1250}s_6a_1-\\
-66bs_2s_4s_5-30s_2^2s_4s_6-\frac{57}{2}s_2s_3^2s_6-\frac{33}{2}s_3^2s_4s_6-27s_2^2s_5s_6-18s_2s_3s_6^2-\\
-\frac{494667}{2500}s_3^2-\frac{538863}{1250}s_3s_2-\frac{131403}{250}s_3a_1-\frac{83892}{625}s_4^2-\frac{73743}{250}s_4s_3-\\
-228a_1bs_2s_3-159a_1s_2s_3s_4-135a_1s_2s_3s_5-\frac{198}{5}s_2s_6^2-111a_1s_2s_3s_6-114a_1s_2s_4s_5-\\
-84bs_4x^2-\frac{27}{2}s_5^2s_6x-144a_1^2s_4x-\frac{117}{2}s_6s_2x^2-15bs_6^2x-90bs_3^2x-\frac{3}{4}s_5^2s_6^2-\\
-\frac{99618}{625}ba_1-27a_1^3s_6-33a_1^3s_5-39a_1^3s_4-\frac{111}{2}a_1^3s_3-72a_1^3s_2-\frac{3276}{25}a_1^2s_6-
\end{multline*}

\begin{multline*}
-\frac{405}{2}a_1^2s_3x-261a_1^2s_2x-15s_4s_6^2x-\frac{153}{2}s_4s_2x^2-\frac{4671}{5}s_3a_1x-\frac{609}{10}a_1s_6^2-\\
-102a_1s_5s_6x-78s_2s_4s_6x-69s_2s_5s_6x-63s_3s_4s_6x-54s_3s_5s_6x-39s_4s_5s_6x-\\
-\frac{135}{2}s_4s_3x^2-\frac{207}{2}s_5a_1x^2-\frac{11586}{25}s_5a_1x-\frac{1527}{5}s_5s_3x-\frac{3489}{25}s_6s_4x-\\
-36a_1b^2x-\frac{138}{5}b^2s_2-30b^2s_2x-\frac{552}{25}b^2s_3-24b^2s_3x-\frac{414}{25}b^2s_4-18b^2s_4x-\frac{276}{25}b^2s_5-\\
-111s_3s_2x^2-\frac{17406}{25}s_3s_2x-114bs_3x^2-30s_2s_6^2x-180s_3a_1x^2-\frac{93}{2}s_6s_4x^2-\\
-\frac{231}{2}a_1^2s_5x-24s_4^2s_6x-42s_2s_4^2x-60s_5s_3x^2-\frac{105}{2}s_6s_3x^2-\frac{45}{2}s_3s_6^2x-\\
-\frac{11403}{50}s_5x-\frac{583059}{1250}s_3x-\frac{435009}{1250}s_4x-\frac{197991}{250}a_1x-\frac{731109}{1250}s_2x-\frac{177}{2}a_1^3x-\\
-192ba_1x^2-39s_3s_5^2x-\frac{51}{2}s_3s_4^2x-\frac{51}{2}s_4s_5^2x-\frac{75}{2}s_3^2s_6x-\frac{33}{2}s_4^2s_5x-\\
-\frac{1131}{5}s_2^2s_4-\frac{4116}{25}s_2s_4^2-\frac{4131}{25}s_3^2s_4-\frac{6723}{50}s_3s_4^2-\frac{195}{4}a_1^2s_5^2-\\
-60s_2^2s_4x-\frac{15}{2}s_5s_6^2x-\frac{7194}{25}bs_3x-\frac{5364}{25}bs_4x-\frac{3522}{25}bs_5x-\frac{828}{25}a_1b^2-\\
-96s_2s_3s_5x-144a_1s_2s_6x-129a_1s_3s_6x-90s_2s_3s_6x-144a_1s_4s_5x-81s_2s_4s_5x-\\
-12b^2s_5x-\frac{138}{25}b^2s_6-6b^2s_6x-192ba_1^2x-\frac{195}{2}s_2^2s_3x-\frac{135}{2}s_5s_2x^2-\frac{1731}{5}s_5s_2x-\\
-\frac{798}{5}s_6s_3x-36s_3^2s_5x-36bs_5^2x-57bs_4^2x-\frac{159}{2}s_6a_1x^2-\frac{5907}{25}s_6a_1x-144a_1s_3^2x-\\
-48s_2s_3s_4s_6-42s_2s_3s_5s_6-30s_2s_4s_5s_6-21s_3s_4s_5s_6-324a_1s_2s_3x-318a_1bs_2x-\\
-\frac{12819}{25}s_4s_2x-\frac{2982}{25}s_6s_5x-87a_1s_4^2x-30bs_6x^2-\frac{336}{5}bs_6x-51s_2^2s_6x-\\
-144bs_2x^2-\frac{9024}{25}bs_2x-\frac{69}{2}s_3^2s_4x-45a_1s_6^2x-\frac{105}{2}s_2s_5^2x-\frac{11682}{25}ba_1x-\\
-\frac{7677}{25}s_3^2x-45s_2^3x-25s_5x^3-\frac{387}{2}s_5x^2-\frac{43}{2}s_6x^3-\frac{2496}{25}s_6x^2-18s_3^3x-\\
-36bs_2s_4s_6-30bs_2s_5s_6-54bs_3s_4s_5-30bs_3s_4s_6-24bs_3s_5s_6-18bs_4s_5s_6-\\
-\frac{117}{4}s_4^2x^2-\frac{9741}{50}s_4^2x-3s_4^3x-33s_5^2x^2-\frac{3018}{25}s_5^2x-9s_5^3x-\frac{75}{4}s_6^2x^2-\\
-\frac{99}{2}s_6^2x-18b^2x^2-\frac{828}{25}b^2x-\frac{19911}{25}a_1^2x-\frac{101346}{625}bx-\frac{135141}{1250}s_6x-\\
-54a_1bs_3s_6-48a_1bs_4s_6-42a_1bs_5s_6-114bs_2s_3s_4-78bs_2s_3s_5-42bs_2s_3s_6-\\
-41s_3x^3-\frac{9729}{25}s_3x^2-\frac{57}{2}s_4x^3-\frac{7179}{25}s_4x^2-\frac{291}{4}s_2^2x^2-\frac{21993}{50}s_2^2x-\\
-\frac{2256}{5}s_4s_3x-\frac{99}{2}s_3^2x^2-\frac{879159}{2500}x^2-\frac{297}{10}s_3s_6^2-\frac{99}{5}s_4s_6^2-\frac{498}{25}s_5^2s_6-\\
-\frac{5307}{10}a_1s_2^2-29a_1s_3^3-26s_2^3s_3-\frac{99}{4}s_2^2s_3^2-7s_2s_3^3-\frac{18699}{50}a_1s_3^2-\frac{3177}{10}s_2^2s_3-
\end{multline*}

\begin{multline*}
-\frac{99}{10}s_5s_6^2-3b^2s_6^2-\frac{99}{5}bs_6^2-\frac{477}{4}a_1^2s_2^2-\frac{333}{4}a_1^2s_3^2-65a_1s_2^3-\\
-\frac{6672}{25}s_2s_3^2-\frac{207}{4}a_1^2s_4^2-8a_1s_4^3-12s_2^3s_4-9s_2^2s_4^2-\frac{12027}{50}a_1s_4^2-\\
-54a_1^3b-15a_1^2b^2-\frac{13821}{50}a_1^3-114a_1s_4s_6x-\frac{1077531}{2500}a_1^2-\frac{5991}{25}ba_1^2-\frac{147}{625}b^2-\\
-\frac{639}{4}a_1^2x^2-95a_1x^3-\frac{7107}{10}a_1x^2-58bx^3-\frac{5427}{25}bx^2-\frac{107}{2}s_2x^3-\frac{12279}{25}s_2x^2-\\
-\frac{23}{2}a_1s_5^3-13s_2^3s_5-\frac{75}{4}s_2^2s_5^2-\frac{13}{2}s_2s_5^3-s_3^3s_5-\frac{15}{2}s_3^2s_5^2-\\
-4s_3s_5^3-\frac{3}{2}s_4s_5^3-\frac{7389}{50}a_1s_5^2-\frac{7653}{50}s_2^2s_5-\frac{504}{5}s_2s_5^2-\frac{5631}{50}s_3^2s_5-\\
-\frac{2022}{25}s_3s_5^2-\frac{3609}{50}s_4^2s_5-\frac{1524}{25}s_4s_5^2-\frac{105}{4}a_1^2s_6^2-40bs_2^3-22bs_3^3-\\
-10bs_4^3-5bs_5^3-14s_2^3s_6-\frac{13}{2}s_3^3s_6-12s_2^2s_6^2-\frac{27}{4}s_3^2s_6^2-\frac{1998}{25}s_2^2s_6-\\
-60s_3^2s_6-\frac{1002}{25}s_4^2s_6-2s_4^3s_6-3s_4^2s_6^2-\frac{4167}{25}bs_2^2-\frac{609}{5}bs_3^2-\\
-\frac{2061}{25}bs_4^2-15b^2s_2^2-12b^2s_3^2-9b^2s_4^2-6b^2s_5^2-\frac{1203}{25}bs_5^2-s_5^3s_6.
\end{multline*}
Thus, if $s_1\geqslant\frac{23}{25}$, then
\begin{multline*}
g_1\Big(a_1,s_1,s_2,s_3,s_4,s_5,s_6,b\Big)\leqslant-\frac{7063}{125}s_5-\frac{1317297}{15625}s_4-\frac{1752697}{15625}s_3-\frac{2188097}{15625}s_2-\frac{525143}{3125}a_1+\\
+2s_2s_4^3+\frac{9}{2}s_3^3s_4+\frac{27}{4}s_3^2s_4^2+7s_3s_4^3+5s_4^3s_5+\frac{3}{4}s_4^2s_5^2+s_3^4+3s_4^4+\frac{9}{2}s_3s_4^2s_5+\frac{697211}{390625}\leqslant\\
\leqslant-\frac{7063}{125}s_5-\frac{1317297}{15625}s_4-\frac{1752697}{15625}s_3-\frac{2188097}{15625}s_2-\frac{525143}{3125}a_1+\\
+2s_2+\frac{9}{2}s_4+\frac{27}{4}s_3+7s_3+5s_5+\frac{3}{4}s_5+s_3+3s_4+\frac{9}{2}s_5+\frac{697211}{390625}=\\
=\frac{697211}{390625}-\frac{2156847}{15625}s_2-\frac{2400219}{31250}s_4-\frac{6088913}{62500}s_3-\frac{23127}{500}s_5-\frac{525143}{3125}a_1.
\end{multline*}
On the other hand,  we have $a_2+a_5+a_6\geqslant 1+a_3+a_4$,
so that $2s_4+s_5+s_3+s_1+a_1\geqslant 1$.
Thus, if $s_1\geqslant\frac{23}{25}$, then $2s_4+s_5+s_3+s_1+a_1\geqslant\frac{2}{25}$.
The latter inequality implies that
$$
\frac{697211}{390625}-\frac{2156847}{15625}s_2-\frac{2400219}{31250}s_4-\frac{6088913}{62500}s_3-\frac{23127}{500}s_5-\frac{525143}{3125}a_1<0.
$$
Thus, we see that $f(a_1,a_2,a_3,a_4,a_5,a_6,a_7,b)<0$ provided that $a_2-a_1>\frac{23}{25}=0.92$.

Let $f_2=f(a_1,a_2,a_3,a_4,a_5,a_6,1,b)$ and $g_2(a_1,s_1,s_2,s_3,s_4,s_5,b)=\widehat{f}_2$. Then
\begin{multline*}
g_2\Big(a_1,x+\frac{9}{10},s_2,s_3,s_4,s_5,b\Big)=\frac{16293}{8000}+\blacklozenge+\\
+6s_2s_3s_4s_5+3s_4^2s_5x+2s_4^3x+4s_2s_4^3+5s_3^3s_4+\frac{5}{2}s_3^3s_5+\frac{15}{2}s_3^2s_4^2+6s_3s_4^3+4s_4^3s_5+\\
+\frac{3}{2}s_4^2s_5^2+\frac{3}{4}s_3^4+2s_4^4+3s_2s_3^2s_4+\frac{3}{2}s_2s_3^2s_5+6s_2s_3s_4^2+6s_2s_4^2s_5+\frac{15}{2}s_3^2s_4s_5+9s_3s_4^2s_5-\\
-\frac{264327}{2000}a_1-\frac{13221}{100}x-\frac{52767}{2000}s_5-\frac{52767}{1000}s_4-\frac{9906}{125}s_3-\frac{105729}{1000}s_2,
\end{multline*}
where $\blacklozenge$ is the following polynomial with negative coefficients:
\begin{multline*}
-\frac{237}{2}a_1^2s_2s_3-102a_1bs_2^2-63a_1^2s_2s_4-54a_1^2s_3s_4-93a_1s_2^2s_3-42a_1s_2^2s_4-\\
-\frac{24423}{200}s_5a_1-\frac{678}{5}s_3^2-\frac{59301}{200}s_3s_2-\frac{72897}{200}s_3a_1-\frac{16521}{200}s_4^2-\\
-69a_1s_2s_3^2-\frac{63}{2}a_1^2s_2s_5-27a_1^2s_3s_5-21a_1s_2^2s_5-\frac{5493}{10}a_1s_2s_3-\\
-\frac{1029}{10}s_5s_4x-\frac{2007}{5}s_4a_1x-\frac{9}{2}s_3^2s_5x-15bs_5^2x-30bs_4^2x-6b^2s_4^2-\\
-3s_3s_4s_5x-\frac{1032}{5}s_2s_3s_4-\frac{45}{2}a_1^2s_4s_5-24a_1s_2s_5^2-\frac{21}{2}a_1s_3^2s_5-\\
-\frac{888}{5}a_1s_2s_5-\frac{309}{2}a_1s_3s_5-\frac{657}{5}a_1s_4s_5-\frac{516}{5}s_2s_3s_5-\frac{414}{5}s_2s_4s_5-\\
-\frac{627}{10}s_3s_4s_5-66a_1bs_3^2-36a_1bs_4^2-18a_1bs_5^2-60bs_2^2s_3-36bs_2^2s_4-18bs_2^2s_5-\\
-48bs_2s_3^2-24bs_2s_4^2-12bs_2s_5^2-24bs_3^2s_4-12bs_3^2s_5-18bs_3s_4^2-9bs_3s_5^2-6bs_4^2s_5-\\
-6bs_4s_5^2-\frac{1446}{5}a_1bs_2-\frac{1086}{5}a_1bs_3-\frac{726}{5}a_1bs_4-\frac{363}{5}a_1bs_5-\frac{924}{5}bs_2s_3-\\
-\frac{618}{5}bs_2s_4-\frac{309}{5}bs_2s_5-\frac{564}{5}bs_3s_4-\frac{282}{5}bs_3s_5-51bs_4s_5-3b^2s_5^2-\\
-\frac{177}{2}a_1s_3^2x-\frac{1233}{5}s_4s_3x-\frac{1437}{5}s_4s_2x-36a_1s_4^2x-102bs_2x^2-9b^2s_3^2-\\
-\frac{651}{25}bs_5-\frac{1302}{25}bs_4-\frac{1563}{20}bs_3-\frac{5211}{50}bs_2-8s_2^4-\frac{473}{5}s_2^3-\\
-\frac{1338}{5}bs_2x-9s_3^2s_4x-18s_2s_5^2x-\frac{1806}{5}ba_1x-\frac{387}{2}s_3^2x-34s_2^3x-\\
-108a_1bs_4x-54a_1bs_5x-84bs_2s_4x-42bs_2s_5x-72bs_3s_4x-36bs_3s_5x-\\
-\frac{33}{2}a_1s_3s_5^2-9a_1s_4s_5^2-3s_2^2s_3s_5-\frac{15}{2}s_2s_3s_5^2-3s_2s_4s_5^2-\\
-30bs_4s_5x-\frac{567}{5}a_1^2s_5-\frac{1134}{5}a_1^2s_4-\frac{6963}{20}a_1^2s_3-\frac{939}{2}a_1^2s_2-\\
-\frac{89}{2}s_3^3-\frac{71}{5}s_4^3-\frac{1}{4}s_5^4-\frac{133}{20}s_5^3-27ba_1^2s_5-54ba_1^2s_4-
\end{multline*}

\begin{multline*}
-84ba_1^2s_3-114ba_1^2s_2-\frac{903}{25}s_5^2-\frac{16521}{200}s_5s_4-\frac{18207}{200}s_5s_3-\frac{19893}{200}s_5s_2-\\
-\frac{18207}{100}s_4s_3-\frac{19893}{100}s_4s_2-\frac{24423}{100}s_4a_1-\frac{24237}{50}a_1s_2-\frac{4926}{25}s_2^2-\\
-\frac{9}{2}a_1^4-39a_1^3b-12a_1^2b^2-\frac{1014}{5}a_1^3-\frac{131553}{400}a_1^2-\frac{918}{5}ba_1^2-\frac{3}{20}b^2-\\
-\frac{3186}{25}ba_1-9a_1^3s_5-18a_1^3s_4-36a_1^3s_3-54a_1^3s_2-145x^3-\frac{55}{4}x^4-\frac{63}{2}a_1s_5^2x-\\
-12s_2^2s_5x-48bs_4x^2-72a_1^2s_4x-57bs_3^2x-\frac{267}{2}a_1^2s_3x-195a_1^2s_2x-33s_4s_2x^2-\\
-\frac{6183}{10}s_3a_1x-27s_4s_3x^2-\frac{63}{2}s_5a_1x^2-\frac{2007}{10}s_5a_1x-\frac{1233}{10}s_5s_3x-\\
-24a_1b^2s_2-18a_1b^2s_3-12a_1b^2s_4-6a_1b^2s_5-18b^2s_2s_3-12b^2s_2s_4-6b^2s_2s_5-\\
-\frac{141}{2}s_3s_2x^2-\frac{4479}{10}s_3s_2x-75bs_3x^2-\frac{237}{2}s_3a_1x^2-36a_1^2s_5x-\\
-\frac{27}{2}s_5s_3x^2-153a_1s_2^2x-63s_4a_1x^2-45s_2s_3^2x-174a_1s_2x^2-\frac{4176}{5}a_1s_2x-\\
-12s_2s_4^2x-\frac{657}{5}a_1s_4^2-\frac{618}{5}s_2^2s_4-\frac{414}{5}s_2s_4^2-\frac{831}{10}s_3^2s_4-\frac{627}{10}s_3s_4^2-\\
-144ba_1x^2-12s_3s_5^2x-3s_3s_4^2x-6s_4s_5^2x-24s_2^2s_4x-201bs_3x-\frac{672}{5}bs_4x-\\
-\frac{3}{2}s_3s_5^3-\frac{1}{2}s_4s_5^3-\frac{1257}{20}a_1s_5^2-\frac{309}{5}s_2^2s_5-\frac{201}{5}s_2s_5^2-\\
-\frac{336}{5}bs_5x-27a_1b^2-30a_1b^2x-\frac{108}{5}b^2s_2-24b^2s_2x-\frac{81}{5}b^2s_3-\\
-18b^2s_3x-\frac{54}{5}b^2s_4-12b^2s_4x-\frac{27}{5}b^2s_5-6b^2s_5x-144ba_1^2x-63s_2^2s_3x-\\
-\frac{33}{2}s_5s_2x^2-\frac{1437}{10}s_5s_2x-24bs_5x^2-90bs_2^2x-\frac{21}{2}s_5s_4x^2-\\
-24a_1s_2s_4s_5-12a_1s_3s_4s_5-96a_1bs_2s_4-48a_1bs_2s_5-84a_1bs_3s_4-\\
-\frac{4803}{20}a_1s_3^2-\frac{2037}{10}s_2^2s_3-33a_1s_2s_3s_5-\frac{327}{2}s_2s_3^2-\frac{45}{2}a_1^2s_4^2-4s_2^3s_4-\\
-24a_1s_2s_4^2-21a_1s_3^2s_4-12a_1s_3s_4^2-6s_2^2s_3s_4-\frac{1776}{5}a_1s_2s_4-309a_1s_3s_4-\\
-\frac{13}{2}s_5x^3-\frac{1641}{20}s_5x^2-10s_3^3x-\frac{21}{2}s_4^2x^2-\frac{1029}{10}s_4^2x-\\
-12s_5^2x^2-\frac{501}{10}s_5^2x-\frac{7}{2}s_5^3x-15b^2x^2-27b^2x-\frac{11817}{20}a_1^2x-\frac{13029}{100}bx-\\
-207a_1s_2s_3x-228a_1bs_2x-168a_1bs_3x-132bs_2s_3x-108a_1s_2s_4x-90a_1s_3s_4x-\\
-\frac{21579}{200}s_5x-\frac{32181}{100}s_3x-\frac{21579}{100}s_4x-\frac{120999}{200}a_1x-\frac{42783}{100}s_2x-\\
-72a_1^3x-\frac{513}{4}a_1^2x^2-\frac{153}{2}a_1x^3-\frac{10521}{20}a_1x^2-43bx^3-\frac{1671}{10}bx^2-\\
-41s_2x^3-\frac{3447}{10}s_2x^2-27s_3x^3-\frac{1272}{5}s_3x^2-13s_4x^3-\frac{1641}{10}s_4x^2-54s_2^2x^2-
\end{multline*}

\begin{multline*}
-12b^2s_3s_4-6b^2s_3s_5-6b^2s_4s_5-\frac{1257}{1000}b-150a_1bs_2s_3-66a_1s_2s_3s_4-\\
-36s_2s_3s_4x-54a_1s_2s_5x-45a_1s_3s_5x-18s_2s_3s_5x-36a_1s_4s_5x-12s_2s_4s_5x-\\
-\frac{1521}{5}s_2^2x-30s_3^2x^2-\frac{10677}{40}x^2-87a_1^2s_2^2-\frac{207}{4}a_1^2s_3^2-\\
-48a_1s_2^3-\frac{3717}{10}a_1s_2^2-\frac{33}{2}a_1s_3^3-18s_2^3s_3-15s_2^2s_3^2-\frac{7}{2}s_2s_3^3-\\
-\frac{75}{4}a_1^2s_5^2-\frac{9}{2}a_1s_5^3-2s_2^3s_5-6s_2^2s_5^2-\frac{5}{2}s_2s_5^3-\frac{3}{2}s_3^2s_5^2-\\
-42a_1bs_3s_5-36a_1bs_4s_5-60bs_2s_3s_4-30bs_2s_3s_5-24bs_2s_4s_5-18bs_3s_4s_5-\\
-\frac{831}{20}s_3^2s_5-\frac{303}{10}s_3s_5^2-\frac{213}{10}s_4^2s_5-\frac{102}{5}s_4s_5^2-\\
-28bs_2^3-13bs_3^3-4bs_4^3-2bs_5^3-123bs_2^2-\frac{843}{10}bs_3^2-51bs_4^2-12b^2s_2^2-\frac{45}{2}bs_5^2.
\end{multline*}
Thus, if $x\geqslant 0$, then $g_2(a_1,x+\frac{9}{10},s_2,s_3,s_4,s_5,b)$ does not exceed
\begin{multline*}
\frac{16293}{8000}+6s_2s_3s_4s_5+3s_4^2s_5x+2s_4^3x+4s_2s_4^3+5s_3^3s_4+\frac{5}{2}s_3^3s_5+\frac{15}{2}s_3^2s_4^2+6s_3s_4^3+\\
+4s_4^3s_5+\frac{3}{2}s_4^2s_5^2+\frac{3}{4}s_3^4+2s_4^4+3s_2s_3^2s_4+\frac{3}{2}s_2s_3^2s_5+6s_2s_3s_4^2+6s_2s_4^2s_5+\frac{15}{2}s_3^2s_4s_5+9s_3s_4^2s_5-\\
-\frac{264327}{2000}a_1-\frac{13221}{100}x-\frac{52767}{2000}s_5-\frac{52767}{1000}s_4-\frac{9906}{125}s_3-\frac{105729}{1000}s_2\leqslant\\
\leqslant\frac{16293}{8000}+6s_2+3x+2x+4s_2+5s_3+\frac{5}{2}s_3+\frac{15}{2}s_3+6s_3+\\
+4s_4+\frac{3}{2}s_5+\frac{3}{4}s_3+2s_4+3s_2+\frac{3}{2}s_2+6s_2+6s_2+\frac{15}{2}s_3+9s_3-\\
-\frac{264327}{2000}a_1-\frac{13221}{100}x-\frac{52767}{2000}s_5-\frac{52767}{1000}s_4-\frac{9906}{125}s_3-\frac{105729}{1000}s_2=\\
=\frac{16293}{8000}-\frac{264327}{2000}a_1-\frac{12721}{100}x-\frac{49767}{2000}s_5-\frac{46767}{1000}s_4-\frac{20499}{500}s_3-\frac{79229}{1000}s_2\leqslant\\
\leqslant\frac{16293}{8000}-\frac{264327}{2000}a_1-\frac{49767}{2000}s_5-\frac{46767}{1000}s_4-\frac{20499}{500}s_3-\frac{79229}{1000}s_2.
\end{multline*}
Thus, if $s_1\geqslant\frac{9}{10}$, then
\begin{multline*}
g_2(a_1,s_1,s_2,s_3,s_4,s_5,b)\leqslant\frac{16293}{8000}-\frac{264327}{2000}a_1-\frac{49767}{2000}s_5-\frac{46767}{1000}s_4-\frac{20499}{500}s_3-\frac{79229}{1000}s_2.
\end{multline*}
Moreover, if $s_1\geqslant\frac{9}{10}$, then $2s_4+2s_2+s_5+s_3+a_1\geqslant\frac{1}{10}$.
The latter inequality gives
$$
\frac{16293}{8000}-\frac{264327}{2000}a_1-\frac{49767}{2000}s_5-\frac{46767}{1000}s_4-\frac{20499}{500}s_3-\frac{79229}{1000}s_2<0.
$$
This shows that $f(a_1,a_2,a_3,a_4,a_5,a_6,1,b)<0$ provided that $a_2-a_1>\frac{9}{10}$.
\end{proof}

\end{document}